\documentclass[a4paper, reqno, 11pt, notitlepage]{amsart}
\usepackage{fullpage}
\usepackage[utf8]{inputenc}

\pdfoutput=1

\usepackage{latexsym}
\usepackage{amsmath}
\usepackage{amssymb}
\usepackage{amsthm}
\usepackage{amscd}
\usepackage{mathrsfs}
\usepackage[all]{xy}
\usepackage{graphicx}
\usepackage{comment}
\usepackage{arydshln}
\usepackage{mathtools}
\usepackage{dsfont}
\usepackage{adjustbox}
\usepackage{multicol}
\usepackage{quiver}
\usepackage{changepage} 
\usepackage[activate={true,nocompatibility},final,tracking=true,kerning=true,spacing=true]{microtype}
\microtypecontext{spacing=nonfrench}
\usepackage{eucal}
\usepackage{yfonts}

\numberwithin{equation}{subsection}

\makeatletter

\renewcommand{\tocsection}[3]{%
  \indentlabel{\@ifnotempty{#2}{\bfseries\ignorespaces#1 #2\quad}}\bfseries#3}

\renewcommand{\tocsubsection}[3]{%
  \indentlabel{\@ifnotempty{#2}{\ignorespaces#1 #2\quad}}#3}

\newcommand\@dotsep{4.5}
\def\@tocline#1#2#3#4#5#6#7{\relax
  \ifnum #1>\c@tocdepth 
  \else
    \par \addpenalty\@secpenalty\addvspace{#2}%
    \begingroup \hyphenpenalty\@M
    \@ifempty{#4}{%
      \@tempdima\csname r@tocindent\number#1\endcsname\relax
    }{%
      \@tempdima#4\relax
    }%
    \parindent\z@ \leftskip#3\relax \advance\leftskip\@tempdima\relax
    \rightskip\@pnumwidth plus1em \parfillskip-\@pnumwidth
    #5\leavevmode\hskip-\@tempdima{#6}\nobreak
    \leaders\hbox{$\m@th\mkern \@dotsep mu\hbox{.}\mkern \@dotsep mu$}\hfill
    \nobreak
    \hbox to\@pnumwidth{\@tocpagenum{\ifnum#1=1\bfseries\fi#7}}\par
    \nobreak
    \endgroup
  \fi}
\AtBeginDocument{%
\expandafter\renewcommand\csname r@tocindent0\endcsname{0pt}
}
\def\l@subsection{\@tocline{2}{0pt}{2.5pc}{5pc}{}}
\makeatother

    \makeatletter
    \def\paragraph{\@startsection{paragraph}{4}%
    \z@\z@{-\fontdimen2\font}%
    {\normalfont\bfseries}}
    \makeatother

\makeatletter
\def\subsubsection{\@startsection{subsubsection}{3}%
  \z@{.5\linespacing\@plus.7\linespacing}{-.5em}%
  {\normalfont\bfseries}}
\makeatother

\usepackage[T1]{fontenc}

\usepackage{tikz-cd}
\usepackage{nicefrac}
\usepackage{authblk}
\usepackage{textcomp}
\usepackage{graphicx}
\usepackage{framed} \usepackage{color}

\usepackage[pagebackref,hypertexnames=false]{hyperref} 
    \definecolor{darkblue}{rgb}{0,0,.85} 
    \definecolor{darkred}{rgb}{0.84,0,0}
    \hypersetup{colorlinks = true,
            linkcolor = darkblue,
            urlcolor  = darkblue,
            citecolor = darkred,
            anchorcolor = darkblue,
            bookmarksdepth=3}

\usepackage{accents}

\usepackage[shortlabels]{enumitem}

\usepackage[foot]{amsaddr}

\usepackage{bm}

\usepackage{stmaryrd}

\usepackage{graphics}

\usepackage{mathtools}

\newtheorem{thm}{Theorem}[section]
\newtheorem{prop}[thm]{Proposition}

\newtheorem{thmi}{Theorem}

\newtheorem{lem}[thm]{Lemma}
\newtheorem{claim}[thm]{Claim}
\newtheorem{cor}[thm]{Corollary}

\theoremstyle{definition}
\newtheorem{example}[thm]{Example}

\newtheorem{nota}[thm]{Notation}
\newtheorem{conve}[thm]{Convention}
\newtheorem{construction}[thm]{Construction}

\makeatletter
\newtheoremstyle{examplestyle}
  {1em}
  {1em}
  {\addtolength{\@totalleftmargin}{1.0em}
   \addtolength{\linewidth}{-1.0em}
   \parshape 1 1.0em \linewidth}
  {}
  {\bfseries}
  {.}
  {.5em}
  {}

\theoremstyle{examplestyle} 

\newtheorem{eg}[thm]{Example}

\newtheorem{rem}[thm]{Remark}
\newtheorem{remi}{Remark}

\newtheorem{defn}[thm]{Definition}

\newtheorem*{defni*}{Definition}

\DeclareMathOperator{\Spec}{Spec}

\DeclareMathOperator{\Gal}{Gal}
\DeclareMathOperator{\id}{id}

\DeclareMathOperator{\Spf}{Spf}

\DeclareMathOperator{\Aut}{Aut} 
\DeclareMathOperator{\Fil}{Fil}

\DeclareMathOperator{\Spa}{Spa}

\DeclareMathOperator{\Sht}{Sht}
\DeclareMathOperator{\Spd}{Spd}

\DeclareMathOperator{\Gr}{Gr}
\DeclareMathOperator{\GL}{GL}

\DeclareMathOperator{\Int}{Int}

\newcommand{\sht}{\mathrm{sht}}

\newcommand{\colim@}[2]{%
  \vtop{\m@th\ialign{##\cr
    \hfil$#1\operator@font colim$\hfil\cr
    \noalign{\nointerlineskip\kern1.5\ex@}#2\cr
    \noalign{\nointerlineskip\kern-\ex@}\cr}}%
}
\newcommand{\colim}{%
  \mathop{\mathpalette\colim@{}}\nmlimits@
}

\renewcommand{\sp}{\mathrm{sp}}

\newcommand{\Fal}{\mathrm{Fal}}

\newcommand{\Gmu}{\mathcal{G}\text{-}\mu}

\newcommand{\GDisp}{\mathcal{G}\text{-}\mathbf{Disp}}

\newcommand{\A}{\mathbb{A}}

\newcommand{\F}{\mathbb{F}}
\newcommand{\bG}{\mathbb{G}}

\newcommand{\bb}[1]{\mathbb{#1}}

\newcommand{\mc}[1]{\mathcal{#1}}
\newcommand{\mbb}[1]{\mathbb{#1}}

\newcommand{\mf}[1]{\mathfrak{#1}}

\usepackage{relsize}

\newcommand{\wh}{\widehat}
\newcommand{\wt}{\widetilde}

\newcommand{\Gm}{\mathbb{G}_{{m}}}

\newcommand{\GVect}{\mathcal{G}\text{-}\mathbf{Vect}}

\newcommand{\GLoc}{\mathcal{G}\text{-}\mathbf{Loc}}

\newcommand{\hyphen}{\mathchar`-}

\newcommand{\FGauge}{F\text{-}\mathbf{Gauge}^\mr{vect}}
\newcommand{\GFGauge}{\mathcal{G}\text{-}F\text{-}\mathbf{Gauge}^\mr{vect}}

\newcommand{\mb}[1]{\mathbf{#1}}
\newcommand{\Z}{\mathbb{Z}}

\newcommand{\Q}{\mathbb{Q}}

\newcommand{\C}{\mathbb{C}}
\newcommand{\univ}{\mathrm{univ}}

\renewcommand{\ll}{\llbracket}
\newcommand{\rr}{\rrbracket}

\renewcommand{\j}{\jmath}
\newcommand{\syn}{{\mr{syn}}}

\makeatletter
\renewcommand{\email}[2][]{%
  \ifx\emails\@empty\relax\else{\g@addto@macro\emails{,\space}}\fi%
  \@ifnotempty{#1}{\g@addto@macro\emails{\textrm{(#1)}\space}}%
  \g@addto@macro\emails{#2}%
}
\makeatother

\newcommand{\triv}{\mathrm{triv}}

\newcommand{\der}{\mathrm{der}}

\newcommand{\stacks}[1]{\cite[\href{https://stacks.math.columbia.edu/tag/#1}{Tag~#1}]{StacksProject}}

\newcommand{\an}{\mathrm{an}}

\newcommand{\cat}[1]{\mathbf{#1}}

\newcommand{\Sh}{\mathrm{Sh}}

\DeclareMathOperator{\Isom}{Isom}

\newcommand{\ms}[1]{\mathscr{#1}}

\newcommand{\fl}{\mathrm{fl}}

\newcommand{\proet}{\mathrm{pro\acute{e}t}}

\newcommand{\GSp}{\mathrm{GSp}}

\newcommand{\R}{\mathbb{R}}
\newcommand{\msr}{\mathscr}

\def\Item(#1){\item[\llap{(}\refstepcounter{enumi}$\bullet$] #1)}

\DeclareMathOperator*{\twolim}{2-lim}
\DeclareMathOperator*{\twocolim}{2-colim}

\DeclareMathAccent{\wtilde}{\mathord}{largesymbols}{"65}

\newcommand*\isomto{%
        \xrightarrow{\raisebox{-0.2 em}{\smash{\ensuremath{\sim}}}}%
    }

    \newcommand*\isomfrom{%
        \xleftarrow{\raisebox{-0.2 em}{\smash{\ensuremath{\sim}}}}%
    }

    \newcommand{\ov}[1]{\overline{#1}}
    \newcommand{\et}{\mathrm{\acute{e}t}}
   
    \newcommand{\ac}{\mathrm{ac}}
    
    \newcommand{\qsyn}{\mathrm{qsyn}}

    \newcommand{\qrsp}{\mathrm{qrsp}}

    \newcommand{\Ainf}{\mathrm{A}_\mathrm{inf}}

   \newcommand{\defeq}{\vcentcolon=}
    \newcommand{\efdeq}{=\vcentcolon}
    \newcommand{\be}{\begin{equation*}}
    \newcommand{\ee}{\end{equation*}}
    \newcommand{\bx}{\begin{equation*}\xymatrix}
    \newcommand{\ex}{\end{equation*}}
    \newcommand{\lbb}{\ll}
    
    \newcommand{\rbb}{\rr}
\usepackage{relsize}
\usepackage[bbgreekl]{mathbbol}
\usepackage{amsfonts}

\DeclareSymbolFontAlphabet{\mathbbl}{bbold}
\newcommand{\Prism}{{\mathlarger{\mathbbl{\Delta}}}}
\newcommand{\prism}{{\mathlarger{\mathbbl{\Delta}}}}
\newcommand{\crys}{\mathrm{crys}}
\newcommand{\strcrys}{{{\mathsmaller{\Prism}}\text{-}\mathrm{gr}}}
\newcommand{\smallprism}{{{\mathsmaller{\Prism}}}}
\newcommand{\smallN}{{{\mathsmaller{\mc{N}}}}}
\newcommand{\mr}{\mathrm}
\newcommand{\ad}{\mathrm{ad}}

\newcommand{\dR}{\mathrm{dR}}

\title{The prismatic realization functor for Shimura varieties of abelian type}
\author{Naoki Imai$^{(1)}$}
\address[1]{\scriptsize Graduate School of Mathematical Sciences, The University of Tokyo,
    3-8-1 Komaba, Meguro-ku, Tokyo, 153-8914, Japan}
\email[1]{\scriptsize naoki@ms.u-tokyo.ac.jp}
\author{Hiroki Kato$^{(2)}$}
\address[2]{\scriptsize 
Institut des Hautes Études
Scientifiques, 35 route de Chartres, 91440 Bures-sur-Yvette, France
}
\address[3]{\scriptsize Department of Mathematics
National University of Singapore, Level 4, Block S17, 10 Lower Kent Ridge Road, Singapore 119076}
\email[2]{\scriptsize hiroki@ihes.fr}
\author{Alex Youcis$^{(3)}$}

\email[3]{\scriptsize alex.youcis@gmail.com}

\date{\today}

\begin{document}
\begin{abstract}
    For the integral canonical model $\ms{S}_{\mathsf{K}^p}$ of a Shimura variety $\Sh_{\mathsf{K}_0\mathsf{K}^p}(\mb{G},\mb{X})$ of abelian type at hyperspecial level $K_0=\mc{G}(\Z_p)$, we construct a prismatic $F$-gauge model for the `universal' $\mc{G}(\Z_p)$-local system on $\Sh_{\mathsf{K}_0\mathsf{K}^p}(\mb{G},\mb{X})$. We use this to obtain several new results about the $p$-adic geometry of Shimura varieties, notably an abelian-type analogue of the Serre--Tate deformation theorem (realizing an expectation of Drinfeld in the abelian-type case) and a prismatic characterization of these models at individual level.
\end{abstract}

\maketitle

\tableofcontents

\newpage

\section*{Introduction}

Shimura varieties $\Sh_{\mathsf{K}}(\mb{G},\mb{X})$ are a class of varieties over a number field $\mb{E}=\mb{E}(\mb{G},\mb{X})$ associated to a reductive $\Q$-group $\mb{G}$, and a piece of ancillary Hodge-theoretic data $\mb{X}$, which sit at the intersection of differential geometry, algebraic geometry, and number theory. A guiding principle concerning Shimura varieties is that they should be moduli spaces of $\mb{G}$-motives.\footnote{For general $\mb{G}$ this is not quite correct, and one should instead consider $\mathbf{G}^c$-motives for a certain modified group $\mb{G}^c$. See \cite[\S3]{LanStrohII} and \S\ref{ss:etale-fiber-functor-def} for details. For the sake of simplicity, we ignore this subtlety in the introduction and assume $\mb{G}=\mb{G}^c$, but do not make this assumption in the main body of the article. \label{footnote:G^c-intro}} In particular, there ought to be a universal $\mb{G}$-motive $\omega_\mathrm{mot}$. While even the precise formulation of this is conjectural, for more well-behaved categories $\ms{C}$ which approximate the theory of motives one might still write down \emph{$\ms{C}$-realization functors}. In other words, $\mb{G}$-objects $\omega_\ms{C}$ in $\ms{C}$ which serve the role of $R_\ms{C}\circ\omega_\mr{mot}$ for $R_\ms{C}$ the $\ms{C}$-realization functor from motives to $\ms{C}$.

Over $\C$ this idea has been mostly realized. Namely, one can construct a functor
\begin{equation}\label{eq:MHM-realization}\tag{I.1}
\omega_{\mathsf{K},\mr{MHM}}\colon \cat{Rep}_{\Q}(\mb{G})\to \cat{MHM}(\Sh_{\mathsf{K}\C}(\mb{G},\mb{X})),
\end{equation}
called the \emph{MHM realization functor} (see \cite[\S2]{BurgosWildehaus}). Here for a smooth complex variety $X$ we denote by $\cat{MHM}(X)$ Saito's category of \emph{mixed Hodge modules} on $X$ (see \cite{Saito}), which serves as a very close approximation to the theory of $\Q$-motives over $X$. This is a remarkably powerful tool in studying the geometry of $\Sh_{\mathsf{K}}(\mb{G},\mb{X})_\C$ (e.g., see \cite{BurgosWildehaus}) but, up to non-trivial foundational issues, is easy to construct by the very definition of Shimura varieties. Indeed, Shimura varieties start out life as complex analytic spaces of a very Hodge-theoretic flavor, and only after quite sophisticated and inexplicit arguments obtain the structure of algebraic varieties over $\mb{E}$.

For applications of Shimura varieties to number theory one must understand Shimura varieties not just over $\C$, but over $p$-adic fields and their integer rings. Until recently, finding an analogue of \eqref{eq:MHM-realization} in this setting seemed completely out of reach since there was no good analogue for $\cat{MHM}(X)$. But, recent deep work of Drinfeld and Bhatt--Lurie on integral $p$-adic Hodge theory (e.g., see \cite{BhattNotes}) has provided a good analogue of $\cat{MHM}(X)$ over a $p$-adic formal scheme $\mf{X}$. Specifically, they construct the category $\mc{D}_\mr{qc}(\mf{X}^\mr{syn})$ of \emph{prismatic $F$-gauges on $\mf{X}$}, closely related to the \emph{syntomic cohomology} from \cite{BMS-THH}, and built upon work of Fontaine and Messing. This category promises to form a very close approximation to the category of $\Z_p$-motives over $\mf{X}$. 

However, finding a \emph{syntomic analogue} of \eqref{eq:MHM-realization} remains challenging. Now the analytic origins of Shimura varieties instead of helpful are a hindrance, as their Hodge-theoretic nature over $\C$ does not lend itself well to the $p$-adic Hodge-theoretic setting. This is further complicated by the inexplicit descent from $\C$ to $\mb{E}$. So, except in the rare cases that one can unconditionally get at the `true $\mb{G}$-motive' (e.g., if the Shimura datum is of so-called PEL type) it's not at all clear where to start. This makes the following theorem in the abelian-type setting significant.

\begin{thmi}\label{thmi:main} Let $(\mb{G},\mb{X})$ be of abelian type and $p$ be an odd prime. Set $K_0=\mc{G}(\Z_p)$ where $\mc{G}$ is a reductive $\Z_p$-model of $\mb{G}_{\Q_p}$. Then, for the integral canonical model $\ms{S}_{\mathsf{K}_0\mathsf{K}^p}(\mb{G},\mb{X})$ at a $p$-adic place of $\mb{E}$, there is a \emph{syntomic realization functor} on $\wh{\ms{S}}_{\mathsf{K}_0\mathsf{K}^p}(\mb{G},\mb{X})$, i.e., an exact $\Z_p$-linear $\otimes$-functor
\begin{equation*}
    \omega_{\mathsf{K}_0\mathsf{K}^p,\mr{syn}}\colon\cat{Rep}_{\Z_p}(\mc{G})\to \cat{Vect}(\wh{\ms{S}}_{\mathsf{K}_0\mathsf{K}^p}(\mb{G},\mb{X})^\mr{syn})\subseteq \mc{D}_\mr{qc}(\wh{\ms{S}}_{\mathsf{K}_0\mathsf{K}^p}(\mb{G},\mb{X})^\mr{syn}), 
\end{equation*}
which recovers the universal $\mc{G}(\Z_p)$-local system $\omega_{\mathsf{K}_0\mathsf{K}^p,\et}$ on the generic fiber.
\end{thmi}

There are many applications of Theorem \ref{thmi:main} discussed below. We highlight two here.
\begin{enumerate}[leftmargin=.5cm]
    \item \textbf{Serre--Tate theorem:} if $R\to R/I$ is a nilpotent thickening of $p$-nilpotent rings, then for an $R/I$-point $x$ of $\ms{S}_{\mathsf{K}_0\mathsf{K}^p}(\mb{G},\mb{X})$ the deformations of $x$ to an $R$-point of $\ms{S}_{\mathsf{K}_0\mathsf{K}^p}(\mb{G},\mb{X})$ are naturally in bijection with the deformations of the prismatic $F$-gauge with $\mc{G}$-structure $(\omega_{\mathsf{K}^p,\mr{syn}})_x$.
    \item \textbf{A characterization of $\ms{S}_{\mathsf{K}_0\mathsf{K}^p}(\mb{G},\mb{X})$ at individual level:} $\ms{S}_{\mathsf{K}_0\mathsf{K}^p}(\mb{G},\mb{X})$ is the unique smooth and separated model $\ms{X}$ of $\Sh_{\mathsf{K}_0\mathsf{K}^p}(\mb{G},\mb{X})$ which has a syntomic model of $\omega_{\mathsf{K}_0\mathsf{K}^p,\et}$ such that (1) holds, and such that $\wh{\ms{X}}_\eta$ is the (potentially) crystalline locus of $\omega_{\mathsf{K}_0\mathsf{K}^p,\et}$. We emphasize that this characterization works at individual level $\mathsf{K}_0\mathsf{K}^p$, not requiring consideration of the full prime-to-$p$ Hecke tower (i.e., letting $\mathsf{K}^p$ vary).
\end{enumerate}
For (1) the best previously-known versions of the Serre--Tate theorem (in this generality of $R$) hold in the PEL case, where (unlike the abelian-type case) one may leverage the existence of a $\mb{G}$-motive. With respect to (2), previously the only characterization of integral canonical models was for the entire family $\{\ms{S}_{\mathsf{K}_0\mathsf{K}^p}(\mb{G},\mb{X})\}_{\mathsf{K}^p}$ simultaneously, and involve a `(strong) extension criterion' which is less motivic in nature. We feel that both (1) and (2) make substantial progress towards understanding Shimura varieties of abelian type as parameterizing motivic objects.

\medskip

\paragraph*{Syntomic realization functor}\label{ss:intro-shim-var}

Let $(\mb{G},\mb{X})$ be a Shimura datum of abelian type with reflex field $\mb{E}$. Fix a prime $p>2$ and let $E$ be the completion of $\mb{E}$ at a $p$-adic place. Set $G=\mb{G}_{\Q_p}$, and fix a reductive $\Z_p$-model $\mc{G}$ of $G$, letting $\mathsf{K}_0=\mc{G}(\Z_p)$ be the associated hyperspecial subgroup. For $\mathsf{K}=\mathsf{K}_p\mathsf{K}^p\subseteq\mb{G}(\A_f)$ a compact open subgroup, write $\Sh_{\mathsf{K}}$ for $\Sh_{\mathsf{K}}(\mb{G},\mb{X})_E$. Then,
\begin{equation*}
\varprojlim_{\scriptscriptstyle \mathsf{K}_p\subseteq \mathsf{K}_0} \Sh_{\mathsf{K}_p\mathsf{K}^p}\to\Sh_{\mathsf{K}_0\mathsf{K}^p},
\end{equation*}
is a $\underline{\mathsf{K}_0}$-torsor on the pro-\'etale site of $\Sh_{\mathsf{K}_0\mathsf{K}^p}$, and we let 
\begin{equation*}
\omega_{\mathsf{K}^p,\et}\colon \cat{Rep}_{\Z_p}(\mc{G})\to \cat{Loc}_{\Z_p}(\Sh_{\mathsf{K}_0\mathsf{K}^p}),
\end{equation*}
be the associated exact $\Z_p$-linear $\otimes$-functor, an object of $\GLoc_{\Z_p}(\Sh_{\mathsf{K}_0\mathsf{K}^p})$. Here, for an exact $\Z_p$-linear $\otimes$-category $\mc{C}$ we write $\mc{G}\text{-}\mc{C}$ for the category of exact $\Z_p$-linear $\otimes$-functors $\omega\colon \cat{Rep}_{\Z_p}(\mc{G})\to \mc{C}$.

Let $\ms{S}_{\mathsf{K}^p}$ be the integral canonical model of $\Sh_{\mathsf{K}_0\mathsf{K}^p}$ over $\mc{O}_E$ as in \cite{KisIntShab} and $\wh{\ms{S}}_{\mathsf{K}^p}$ its $p$-adic completion. One may then consider the open subspace $(\wh{\ms{S}}_{\mathsf{K}^p})_\eta\subseteq \Sh_{\mathsf{K}_0\mathsf{K}^p}^\an$, and define 
\begin{equation*}
\omega_{\mathsf{K}^p,\an}\colon\cat{Rep}_{\Z_p}(\mc{G})\to\cat{Loc}_{\mathbb{Z}_p}((\wh{\ms{S}}_{\mathsf{K}^p})_\eta),\qquad \xi\mapsto \omega_{\mathsf{K}^p,\an}(\xi)\defeq \omega_{\mathsf{K}^p,\et}(\xi)^\an|_{(\wh{\ms{S}}_{\mathsf{K}^p})_\eta},
\end{equation*}
a $\mc{G}(\Z_p)$-local system on $(\wh{\ms{S}}_{\mathsf{K}^p})_\eta$, i.e., an object of the category $\GLoc_{\Z_p}((\wh{\ms{S}}_{\mathsf{K}^p})_\eta)$. 

In \cite[Definition 2.27]{IKY1}, we define when a $\mc{G}(\mathbb{Z}_p)$-local system on a smooth rigid analytic variety $X$ has \emph{prismatically good reduction} relative to smooth formal model $\mf{X}$ of $X$. While it requires non-trivial technical input, we show that $\omega_{\mathsf{K}^p,\mr{an}}$ has prismatically good reduction relative to $\wh{\ms{S}}_{\mathsf{K}^p}$ by reducing to the Siegel case (see Theorem \ref{thm:main-Shimura-theorem-abelian-type-case}). Using the main results of \cite{IKY1} we are then able to deduce the existence of a \emph{prismatic realization functor}
\begin{equation*}
\omega_{\mathsf{K}^p,\smallprism}\colon \cat{Rep}_{\Z_p}(\mc{G})\to \cat{Vect}^\varphi((\wh{\ms{S}}_{\mathsf{K}^p})_\smallprism),
\end{equation*}
where the target is the category of prismatic $F$-crystals on $\wh{\ms{S}}_{\mathsf{K}^p}$ (see \cite{BhattScholzeCrystals}), unique with respect to the property that $T_\et\circ \omega_{\mathsf{K}^p,\smallprism}\simeq \omega_{\mathsf{K}^p,\mr{an}}$, where $T_\et$ is the \'etale realization functor from \cite{GuoReinecke}. But, it is initially quite unclear whether $\omega_{\mathsf{K}^p,\smallprism}$ can be upgraded to take values in $\cat{Vect}((\wh{\ms{S}}_{\mathsf{K}^p})^\mr{syn})$.

\begin{remi} Our construction of $\omega_{\mathsf{K}^p,\smallprism}$ utilizes abstract $p$-adic Hodge theory from \cite{IKY1}, ultimately relying on \cite{GuoReinecke} or \cite{DLMS}. In \cite{NieThesis}, Nie constructs `absolute Hodge cycles' in the prismatic cohomology of good-reduction abelian varieties over $p$-adic fields. This method currently only works over a point and with inexplicit restrictions on $p$. If these conditions were removed one might alternatively attempt to construct $\omega_{\mathsf{K}^p,\smallprism}$ via this method (at least in the Hodge-type case). This would yield an approach closer in spirit to that used in \cite{KisIntShab}.
\end{remi}

In \S\ref{s:G-objects-prismatic-crystals} we show that for a smooth formal $\mc{O}_E$-scheme $\mf{X}$ there is a bi-exact $\Z_p$-linear $\otimes$-equivalence
\begin{equation}\label{eq:intro-lff-gauge-equiv}\tag{I.2}
    \cat{Vect}(\mf{X}^\mr{syn})\isomto \cat{Vect}^{\varphi,\mr{lff}}(\mf{X}_\smallprism), 
\end{equation}
(see Proposition \ref{prop:F-gauge-lff-equiv}), which is proven to be bi-exact in \cite{IKY3} using our integral analogue $\bb{D}_\mr{crys}$ of $D_\mr{crys}$ from op.\@ cit. Here $\cat{Vect}^{\varphi,\mr{lff}}(\mf{X}_\smallprism)$ is the category of prismatic $F$-crystals $(\mc{E},\varphi_\mc{E})$ on $\mf{X}$ which are \emph{locally filtered free (lff)}: the Nygaard filtration on the Frobenius pullback $\phi^\ast\mc{E}$ is locally free (see Definition \ref{defn:lff}). 

Thus, to upgrade $\omega_{\mathsf{K}^p,\smallprism}$ to an object of $\GVect((\wh{\ms{S}}_{\mathsf{K}^p})^{\mr{syn}})$, it suffices to show that $\omega_{\mathsf{K}^p,\smallprism}$ takes values in $\cat{Vect}^{\varphi,\mr{lff}}((\wh{\ms{S}}_{\mathsf{K}^p})_\smallprism)$. One can even hope for a stronger property: $\omega_{\mathsf{K}^p,\smallprism}$ is of type $-\mu_h$. Here $\mu_h$ is the Hodge cocharacter of $(\mb{G},\mb{X})$ and being of type $-\mu_h$ means that locally the Frobenius $\varphi_\mc{E}$ is in the double coset defined by $-\mu_h$ (see \cite[Definition 3.12]{IKY1}). By an Artin-approximation-like argument, it suffices to show $\omega_{\mathsf{K}^p,\smallprism}$ is of type $-\mu_h$ over the complete local rings $\wh{\mc{O}}_{\ms{S}_{\mathsf{K}^p},x}$ for each $\overline{\mathbb{F}}_p$-point $x$ of $\ms{S}_{\mathsf{K}^p}$. We achieve this by comparing the pullback of $\omega_{\mathsf{K}^p,\smallprism}$ to $\wh{\mc{O}}_{\ms{S}_{\mathsf{K}^p},x}$ and a construction of Ito from \cite{Ito2} using our integral analogue $\bb{D}_\mr{crys}$ of $D_\mr{crys}$. 

Thus, by the equivalence in \eqref{eq:intro-lff-gauge-equiv} we obtain a \emph{syntomic realization functor} as in Theorem \ref{thmi:main}. In fact, one may use an enhancement of \eqref{eq:intro-lff-gauge-equiv} (see Proposition \ref{prop: F gauge type mu equals F crystal type mu}) to show this syntomic realization functor is of type $-\mu_h$ in an appropriate sense (see Definition \ref{defn: F-gauge of G mu structure}).

\begin{thmi}[{see Theorem \ref{thm:prismatic-F-gauge-realization}}]\label{thmi:main-refinement} There exists a unique $\mc{G}$-object $\omega_{\mathsf{K}^p,\mr{syn}}$ of $\cat{Vect}((\wh{\ms{S}}_{\mathsf{K}^p})^\mr{syn})$ of type $-\mu_h$ such that $T_\et\circ\omega_{\mathsf{K}^p,\syn}\simeq \omega_{\mathsf{K}^p,\an}$.
\end{thmi}

\begin{remi} In \cite{ShendeRham}, Shen constructs a \emph{de Rham $F$-gauge} on $(\ms{S}_{\mathsf{K}^p}
)_{\ov{\bb{F}}_p}$ when $(\mb{G},\mb{X})$ is of Hodge type. In the language of the next section, this is equivalent to a map $(\ms{S}_{\mathsf{K}^p
})_{\ov{\bb{F}}_p}\to \mr{BT}^{\mc{G},-\mu_h}_1$, whereas our syntomic realization functor is a map $\wh{\ms{S}}_{\mathsf{K}^p}\to \mr{BT}^{\mc{G},-\mu_h}_\infty$. 
That said, Shen is able to describe this object in more down-to-earth terms using the de Rham cohomology of abelian varieties and the BGG complex. Our construction recovers Shen's after base changing along $\Spec(\overline{\mathbb{F}}_p)\to \Spf(\mc{O}_E)$ and truncating from an $\infty$-truncated (i.e., untruncated) object to a $1$-truncated one, i.e., base changing the composition $\wh{\ms{S}}_{\mathsf{K}^p}\to \mr{BT}^{\mc{G},-\mu_h}_\infty\to \mr{BT}^{\mc{G},-\mu_h}_1$ to $\ov{\bb{F}}_p$.
\end{remi}

\paragraph*{Syntomic integral canonical models and the Serre--Tate theorem}

The family of integral canonical models $\{\ms{S}_{\mathsf{K}^p}\}_{\mathsf{K}^p}$ are uniquely characterized \emph{as a system} by a \emph{strong extension property}: for every regular formally smooth $\mc{O}_E$-algebra $R$, one has\begin{equation*}
\varprojlim_{\scriptscriptstyle\mathsf{K}^p}\ms{S}_{\mathsf{K}^p}(R)= \varprojlim_{\scriptscriptstyle\mathsf{K}^p}\Sh_{\mathsf{K}_0\mathsf{K}^p}(R[\nicefrac{1}{p}]).
\end{equation*}
While this characterization is sufficient for many applications, it is incapable of characterizing the models $\ms{S}_{\mathsf{K}^p}$ for individual levels $\mathsf{K}^p$, and is far from a direct moduli-theoretic characterization.

Given our guiding principle for $\Sh_{\mathsf{K}_0\mathsf{K}^p}$, it is natural to expect that a canonical model $\ms{S}_{\mathsf{K}^p}$ of $\Sh_{\mathsf{K}^p\mathsf{K}_0}$ should be a moduli space of $\mc{G}$-motives in some sense. If one thinks of prismatic $F$-gauges as being the `syntomic realization' of (and good approximation to) the category of $\mathbb{Z}_p$-motives, it seems not unreasonable to expect a characterization of such models in syntomic terms.

To make this precise, we use the moduli stack $\mr{BT}^{\mc{G},-\mu_h}_\infty$ of prismatic $F$-gauges with $\mc{G}$-structure of type $-\mu_h$ suggested by Drinfeld, and developed by Gardner--Madapusi in \cite{GMM}.

\begin{defni*}[{see Definition \ref{defn:syntomic-integral-canonical-model}}]\label{defni:prismatic-integral-canonical-model} A \emph{syntomic integral canonical model} of $\Sh_{\mathsf{K}_0\mathsf{K}^p}$ is a smooth and separated $\mc{O}_E$-model $\ms{X}_{\mathsf{K}^p}$ of $\Sh_{\mathsf{K}_0\mathsf{K}^p}$ such that 
\begin{enumerate}
\item $(\wh{\ms{X}}_{\mathsf{K}^p})_\eta$ is the potentially crystalline locus of $\omega_{\mathsf{K}^p,\et}$,
\item there exists a syntomic model of $\omega_{\mathsf{K}^p,\an}$ of type $-\mu_h$ such that the resulting map
\begin{equation}\label{eq:Drinfeld-conj-map}\tag{I.3}
    \rho_{\mathsf{K}^p}\colon \wh{\ms{X}}_{\mathsf{K}^p}\to \mr{BT}^{\mc{G},-\mu_h}_\infty
\end{equation}
is formally \'etale.
\end{enumerate}
\end{defni*}

Our `motivic' characterization of $\ms{S}_{\mathsf{K}^p}$ is then the following which additionally realizes an expectation of Drinfeld (see \cite[\S4.3.3]{DrinfeldTowardsShimurian}) for abelian-type Shimura varieties.

\begin{thmi}[{see Theorem \ref{thm:Drinfeld-conjecture}\label{thmi:intro-characterization} and Theorem \ref{thm:syntomic-characterization}}]\label{thmi:prismatic-integral-canonical-model} The integral canonical model $\ms{S}_{\mathsf{K}^p}$ is a syntomic integral canonical model, and it is the unique such model.
\end{thmi}

Of course, the syntomic model of $\omega_{\mathsf{K}^p,\an}$ on $\wh{\ms{S}}_{\mathsf{K}^p}$ realizing the map $\rho_{\mathsf{K}^p}$ for the integral canonical model $\ms{S}_{\mathsf{K}^p}$ is the syntomic realization functor $\omega_{\mathsf{K}^p,\mr{syn}}$. And the formal \'etaleness of this map $\rho_{\mathsf{K}^p}$ can be reinterpeted in terms of a \emph{Serre--Tate theorem}.

\begin{thmi}[{Serre--Tate theorem for abelian-type Shimura varieties, see Theorem \ref{thm:Drinfeld-conjecture}}]\label{thmi:Serre--Tate} Let $R$ be a $p$-nilpotent ring and $R\to R/I$ a nilpotent thickening. Then, the diagram \begin{equation*}
    \begin{tikzcd}[sep=2.25em]
	{\ms{S}_{\mathsf{K}^p}(R)} & {\mr{BT}^{\mc{G},-\mu_h}_\infty(R)} \\
	{\ms{S}_{\mathsf{K}^p}(R/I)} & {\mr{BT}^{\mc{G},-\mu_h}_\infty(R/I)}
	\arrow["{\rho_{\mathsf{K}^p}}", from=1-1, to=1-2]
	\arrow[from=1-1, to=2-1]
	\arrow[from=1-2, to=2-2]
	\arrow["{\rho_{\mathsf{K}^p}}"', from=2-1, to=2-2]
\end{tikzcd}
\end{equation*}
is Cartesian. In other words, the deformations of an $R/I$-point $x$ of $\ms{S}_{\mathsf{K}^p}$ to an $R$-point of $\ms{S}_{\mathsf{K}^p}$ are canonically in bijection with the deformations of $\rho_{\mathsf{K}^p}(x)$ to an $R$-point of $\mr{BT}^{\mc{G},-\mu_h}_\infty$.
\end{thmi}

\begin{remi} In \cite{PappasRapoportI}, Pappas--Rapoport conjecture a method to characterize a system of integral models of Shimura varieties in terms of shtukas which would apply even at non-hyperspecial parahoric levels $\mathsf{K}_0$. But, this approach only characterizes the full system (and cannot work at individual level $\mathsf{K}_0\mathsf{K}^p$). Moreover, as the theory of shtukas is based off the theory of $v$-sheaves which cannot distinguish between a ring and its reduced quotient, it is impossible to use the approach of Pappas--Rapoport to obtain a result like Theorem \ref{thmi:Serre--Tate}.

To understand the relationship between our work and that of \cite{PappasRapoportI}, we observe the following. In \cite[\S3.3.4]{IKY1} we construct the shtuka realization of a $\mc{G}$-object in prismatic $F$-crystals (or $F$-gauges) over $\ms{X}$. Together with Theorem \ref{thmi:intro-characterization} one is able to recover this `universal shtuka' at hyperspecial level, giving a verification of the Pappas--Rapoport conjecture in this case. 
\end{remi}

\medskip

\paragraph*{Cohomological application} At the end of this introduction, we list some further applications of Theorems \ref{thmi:main-refinement}, \ref{thmi:intro-characterization}, and \ref{thmi:Serre--Tate}. But, we wish to highlight here an immediate \emph{cohomological consequence}.

It is a well-established conjecture that the cohomology spaces $H^i_\et((\Sh_{\mathsf{K}_0\mathsf{K}^p})_{\ov{\Q}},\omega_{\mathsf{K}^p,\et}(\xi)[\nicefrac{1}{p}])$ are meant to realize the global Langlands correspondence for $G$ (e.g., see \cite{KottwitzAnnArbor}). Thus, information about these cohomology spaces should have implications for the local and global Langlands programs. To this end, Theorem \ref{thmi:main-refinement} recovers a result of Lovering (see Proposition \ref{prop:coholomogy-crystalline}) that says that these cohomology spaces are crystalline when the Shimura variety is proper. 

That said, in recent years much attention has been given to richer refinements of the classical local Langlands program, which concerns $\ell$-adic representations of $p$-adic Galois groups, allowing instead $p$-adic or even mod-$p$ representations. For these purposes $\mathbb{Z}_p$-refinements of the cohomology of Shimura varieties are required, and Theorem \ref{thmi:main-refinement} allows us to prove results in that direction.

To this end, let $n$ be an element of $\bb{N}\cup\{\infty\}$. We say that an object $\Lambda$ of $\cat{Rep}_{\Z/p^n}(\Gal(\ov{E}/E))$ has \emph{syntomically good reduction} if there exists an $F$-gauge $\mc{V}$ in $\cat{Vect}(\mc{O}_E^\mr{syn}/p^n)$ such that $T_\et(\mc{V})\simeq \Lambda$. This is a refinement of the notion of being crystalline (cf.\@ \cite[Theorem 6.6.13]{BhattNotes}).\footnote{In general the \'etale realization functor is \emph{not fully faithful} on $\cat{Vect}(\mc{O}_E^\mr{syn}/p^n)$, and so one should really view being syntomically good reduction as a piece of data instead of a condition. Thus, Theorem \ref{thm:gauge-propert-coh} should perhaps be viewed as giving a `canonical' syntomic model for the \'etale cohomology of various $\mathbb{Z}/p^n$-local systems on $\Sh_{\mathsf{K}_0\mathsf{K}^p}$.\label{footnote:syntomically-good-reduction}}

\begin{thmi}[{see Theorem \ref{thm:gauge-propert-coh}}]\label{thmi:coh-syntomically-good-reduction} Suppose that $\mr{Sh}_{\mathsf{K}_0\mathsf{K}^p}$ is proper. Then, for any object $\xi$ of $\cat{Rep}_{\Z_p}(\mc{G})$ such that the ${\mu}_h$-weights of $\xi[\nicefrac{1}{p}]$ are in $[0,p-3-i]$, the $\Gal(\ov{E}/E)$-representation $H^i_\et((\Sh_{\mathsf{K}_0\mathsf{K}^p})_{\ov{\Q}_p},\omega_{\mathsf{K}^p,\et}(\xi)/p^n)$ has syntomically good reduction. In fact, there is an isomorphism
 \begin{equation*}
     H^i_\et((\Sh_{\mathsf{K}_0\mathsf{K}^p})_{\ov{\Q}_p},\omega_{\mathsf{K}^p,\et}(\xi)/p^n)\simeq T_\et\left(\mc{H}^i_\mr{syn}(\wh{\ms{S}}_{\mathsf{K}^p}/\mc{O}_E,\omega_{\mathsf{K}^p,\mr{syn}}(\xi))/p^n)\right).
 \end{equation*}
\end{thmi}

\medskip

\paragraph*{Enhancement of Lovering's crystalline realization functor} We return now to the setting of abelian-type Shimura varieties. In \cite{LoveringFCrystals}, Lovering constructs a functor
\begin{equation*}
\omega_{\mathsf{K}^p,\crys}\colon\cat{Rep}_{\Z_p}(\mc{G})\to \cat{VectF}^{\varphi,\mr{div}}((\wh{\ms{S}}_{\mathsf{K}^p})_\crys)
\end{equation*}
called the \emph{crystalline realization functor}, where the target is the category of strongly divisible filtered $F$-crystals on $\wh{\ms{S}}_{\mathsf{K}^p}$. There is a natural identification
\begin{equation*}
D_\crys\circ \omega_{\mathsf{K}^p,\et}[\nicefrac{1}{p}]\isomto\omega_{\mathsf{K}^p,\crys}[\nicefrac{1}{p}].
\end{equation*}
Moreover, he shows that the lattices $\omega_{\mathsf{K}^p,\et}(\xi)$ and $\omega_{\mathsf{K}^p,\crys}(\xi)$ are matched by Fontaine--Laffaille theory when it applies (i.e.\@, when $\omega_{\mathsf{K}^p,\et}(\xi)$ has Hodge--Tate weights in $[0,p-2]$). Lovering's functor has found multiple applications (e.g., in \cite{Lee} and \cite{ShenZhang}). 

One main motivation for the construction of $\omega_{\mathsf{K}^p,\mr{syn}}$ was to refine this construction to the prismatic/syntomic setting and to remove the weight restrictions on the lattice-comparison aspects of Lovering's results. We can make precise the fact that $\omega_{\mathsf{K}^p,\mr{syn}}$ recovers $\omega_{\mathsf{K}^p,\mr{crys}}$ and allows one to remove the weight restrictions on Lovering's results using the integral analogue $\mathbb{D}_\mr{crys}$ of $D_\crys$ from \cite{IKY3}.

\begin{thmi}[{see Theorem \ref{thm:prismatic-crystalline-comparison}}]\label{thmi:Dcrys-matching} There are canonical identifications 
\begin{equation*}
\bb{D}_\crys\circ \omega_{\mathsf{K}^p,\mr{syn}}=\bb{D}_\crys\circ \omega_{\mathsf{K}^p,\smallprism}\isomto \omega_{\mathsf{K}^p,\crys}.
\end{equation*}
In particular for all $\xi$, the lattices $\omega_{\mathsf{K}^p,\et}(\xi)$ and $\omega_{\mathsf{K}^p,\crys}(\xi)$ are matched by $\bb{D}_\crys$.
\end{thmi}

\begin{remi} The functor $\bb{D}_\crys\colon \cat{Vect}((\wh{\ms{S}}_{\mathsf{K}_p})^\mr{syn})\to \cat{VectF}^{\varphi,\mr{div}}((\wh{\ms{S}}_{\mathsf{K}^p})_\crys)$ from \cite{IKY3} is an equivalence on the Fontaine--Laffaille range, i.e. when restricted to objects with `weights' in $[0,p-2]$. But, it fails to be fully faithful outside of that range. As $\omega_{\mathsf{K}^p,\mr{crys}}$ (being a tensor functor) necessarily takes values outside of the Fontaine--Laffaille range, the existence of $\omega_{\mathsf{K}^p,\mr{syn}}$ satisfying Theorem \ref{thmi:Dcrys-matching} is far from formal.
\end{remi}

\paragraph*{Further applications} Finally, we mention some further applications of the above results.

\begin{itemize}[leftmargin=.18in]
\item In the forthcoming work \cite{MadapusiDerivedCycles}, Theorem \ref{thmi:main-refinement} is applied to understand derived cycles on Shimura varieties with applications to special values of $L$-functions/automorphic forms.
\item In the forthcoming work \cite{MadapusiLee}, Theorem \ref{thmi:main-refinement}  is used to produce $p$-integral Hecke operators on Shimura varieties of abelian type. This is used to solve a conjecture of Fakhruddin--Pilloni and to give a conceptual construction of the fiber product diagram conjectured by Scholze, and studied in \cite{ZhangThesis} and \cite{DvHKZ2}.
\item Theorem \ref{thmi:main-refinement} gives rise to a smooth morphism $\zeta_{\mathsf{K}^p}\colon \ms{S}_{\mathsf{K}^p}(\mb{G},\mb{X})_k\to \mathcal{G}\text{-}\mathsf{Zip}^{-\mu_h}$, the zip period map, where $k$ is the residue field $E$ and the target is the category of $(\mc{G},\mu_h)$-zips (see Theorem \ref{thm:GZip-map}). This is applied in \cite{Reppen} to study the coherent cohomology of Shimura varieties.
\item In \cite{Inoue}, the ideas developed here are applied to construct a prismatic realization functor for toroidal compactifications of integral canonical models of Shimura varieties of Hodge type.
\item In \cite{Yan}, Theorem \ref{thmi:main-refinement} is further applied to obtain a refinement of the zip period map.
\end{itemize}

\medskip

\paragraph*{Acknowledgments} The authors would like to heartily thank Bhargav Bhatt, Alexander Bertoloni Meli, Patrick Daniels, Ian Gleason, Haoyang Guo, Pol van Hoften, Kazuhiro Ito, Wansu Kim, Teruhisa Koshikawa, Akhil Mathew, Emanuel Reinecke, Peter Scholze, Sug Woo Shin, Takeshi Tsuji, and Qixiang Wang, for helpful discussions. The authors would particularly like to thank Keerthi Madapusi for pointing out several subtleties concerning derived algebraic geometry in an earlier version of this article. Part of this work was conducted during a visit to the Hausdorff Research Institute for Mathematics, funded by the Deutsche Forschungsgemeinschaft (DFG, German Research Foundation) under Germany's Excellence Strategy – EXC-2047/1 – 390685813. This work was supported by JSPS KAKENHI Grant Numbers 22KF0109, 22H00093 and 23K17650, the European Research Council (ERC) under the European Union’s Horizon 2020 research and innovation programme (grant agreement No. 851146), and funding through the Max Planck Institute for Mathematics in Bonn, Germany (report numbers MPIM-Bonn-2022, MPIM-Bonn-2023, MPIM-Bonn-2024). 

\medskip

\paragraph*{Notation and conventions}

\begin{itemize}[leftmargin=.18in,label=$\diamond$]
    \item The symbol $p$ will always denote a (rational) prime.
    \item The functor $\bb{D}_\mr{crys}$ is the integral analogue of $D_\mr{crys}$ from \cite{IKY3}, and we use the notation and conventions concerning various categories of (crystalline) $F$-crystals as in \cite[\S2.1]{IKY3}.
     \item By a \emph{(derived) formal stack over $\Z_p$} we mean a stack $\mc{X}$ on the big fpqc site of $p$-nilpotent (animated) rings $R$, which we may view as a stack on the adic flat site $\Spf(\Z_p)_\mr{fl}^\mr{adic}$ (see \cite[\S A.4]{IKY1}) by declaring $\mc{X}(R)\defeq \lim \mc{X}(R/p^n)$. For a formal stack $\mc{X}$ over $\Z_p$ we have the induced stack $\mc{X}_n$ on the fpqc site of $\Spec(\Z/p^n)$, and we say that $\mc{X}$ is a \emph{$p$-adic formal Artin stack} if each $\mc{X}_n$ is an Artin stack over the fpqc site of $\Spec(\Z/p^n)$.
     \item For a property $P$ of morphisms of schemes (resp.\@ Artin stacks), an adic morphism of formal schemes $\mf{X}\to\mf{Y}$ where $\mf{Y}$ has an ideal sheaf of definition $\mc{I}$ (resp.\@ a morphism of $p$-adic formal Artin stacks $\mc{X}\to \mc{Y}$), is \emph{adically $P$ (or $\mc{I}$-adically $P$)} if the reduction modulo $\mc{I}^n$ (resp.\@ the morphism $\mc{X}_n\to\mc{Y}_n$) is $P$ for all $n$. If $A\to B$ is an adic morphism of rings with the $I$-adic topology, for $I\subseteq A$ an ideal, then we make a similar definition.
     \item For an Artin stack $\ms{X}$ on the fpqc site of $\Spec(\Z_p)$ and a subset $T\subseteq |\ms{X}|$ we define the \emph{formal completion along $T$} to be the substack on the fpqc site of $\Spec(\Z_p)$ to be
     \begin{equation*}
         \wh{\ms{X}}_T(R)\defeq \left\{f\colon \Spec(R)\to \ms{X}:f(|\Spec(R)|)\subseteq T\right\}.
     \end{equation*}
    If $T$ is clear from context we shall often omit it from the notation. 
    \item For a morphism $f\colon X\to Y$ we denote $R^if_\ast\underline{\Z_p}$ and $R^if_\ast\underline{\Q_p}$ by $\mc{H}^i_{\Z_p}(X/Y)$ and $\mc{H}^i_{\Q_p}(X/Y)$, respectively. Similar notation will be applied for other cohomology theories.
    \item Our notation and conventions concerning derived algebraic geometry are as in \cite{GMM}.
    \item For a non-archimedean field $K$, a \emph{rigid $K$-space} $X$ is an adic space locally of finite type over $K$. We denote the set of \emph{classical points} by $|X|^\mathrm{cl}:=\left\{x\in X: [k(x):K]<\infty\right\}$.
    \item For an $R$-module $M$ and an ideal $I\subseteq R$ (resp.\@ principal ideal $(a)\subseteq R$) we often write $M/I$ (resp.\@ $M/a$) as shorthand for $M/IM$ (resp.\@ $M/aM$).
    \item A filtration always means a decreasing and exhaustive $\Z$-filtration (i.e.,  $\bigcup_{i\in \Z}\Fil^i=M$).
    \item For a ring $A$ and an element $a$ of $A$, denote by $\Fil^\bullet_a$ the filtration with $\Fil^r_a=a^r A$ for $r> 0$, and $\Fil^r_a=A$ for $r\leqslant 0$. Define $\Fil^\bullet_\triv\defeq \Fil^\bullet_0$.
    \item A filtration of (sheaves of) modules is \emph{locally split} if its graded pieces are locally free.
    \item For an $\bb{F}_p$-algebra $R$ (resp.\@ $\bb{F}_p$-scheme $X$), we denote by $F_R$ (resp.\@ $F_X$) its absolute Frobenius.
\end{itemize}

\section{The Tannakian framework for prismatic \texorpdfstring{$F$-gauges}{F-gauges}}\label{s:G-objects-prismatic-crystals}

In this section we discuss the Tannakian aspects of Drinfeld and Bhatt--Lurie's theory of prismatic $F$-gauges and compare it with the Tannakian theory of prismatic $F$-crystals as developed in \cite{IKY1}. We refer the reader to \cite[\S1]{IKY1}, \cite{BhattNotes}, and \cite[Appendix A]{IKY1} for preliminary discussions of prismatic $F$-crystals, Tannakian theory, and stack-theoretic notions, respectively.

\subsection{Some preliminaries on graded and filtered algebra} The extra structure present in a prismatic $F$-gauge versus a prismatic $F$-crystal is that of a filtration with good properties. So, we first describe some general results in the algebra of filtered rings/modules.

\subsubsection{Basic definitions and results} We use standard terminology concerning filtered rings $(R,\Fil^\bullet_R)$ on topoi $\ms{T}$ and filtered modules over them (e.g., see \cite[Definition 10]{Tsu20}), and only comment on two pieces of terminology/notation not explicitly stated there:

\begin{itemize}[leftmargin=.5cm]
    \item For filtered modules $(M,\Fil^\bullet_M)$ and $(N,\Fil^\bullet_N)$ over $(R,\Fil^\bullet_R)$ the \emph{filtered tensor product} 
    \begin{equation*}
    (M,\Fil^\bullet_M)\otimes_{(R,\Fil^\bullet_R)}(N,\Fil^\bullet_N)
\end{equation*}
    is the module $M\otimes_R N$ equipped with the filtration where
\begin{equation}\label{eq:filt-module-tensor-product}
    \Fil^r_{M\otimes_R N}\defeq \sum_{a+b=r}\mathrm{im}(\Fil^a_M\otimes_R \Fil^b_N\to M\otimes_R N).
\end{equation}
When $(N,\Fil^\bullet_N)=(S,\Fil^\bullet_S)$ is a filtered ring with the structure of a filtered module over $(R,\Fil^\bullet_R)$ via a filtered ring map, then this tensor product is a filtered module over $(S,\Fil^\bullet_S)$.
\item A \emph{filtered crystal} over $(R,\Fil^\bullet_R)$ is a filtered module $(M,\Fil^\bullet_M)$ over $(R,\Fil^\bullet_R)$ such that for every morphism $T\to T'$ in $\ms{T}$ the natural morphism
    \begin{equation*}
        (M(T'),\Fil^\bullet_M(T'))\otimes_{(R(T'),\Fil^\bullet_R(T'))} (R(T),\Fil^\bullet_R(T))\to (M(T),\Fil^\bullet_M(T))
    \end{equation*}
    is an isomorphism of filtered modules over $(R(T),\Fil^\bullet_R(T))$.\footnote{For a filtered ring $(R,\Fil^\bullet_R)$, we observe that there is a natural equivalence between filtered modules over $(R,\Fil^\bullet_R)$ and filtered crystals for the filtered ring $(\mc{O}_{\Spec(R)},\wt{\Fil^\bullet_R})$ in the Zariski topos on $\Spec(R)$. Thus, hereon out we will treat the theory of filtered modules over a ring as a special case of the theory of filtered crystals.}
\end{itemize}

With the obvious notion of morphisms, denote by $\cat{MF}(R,\Fil^\bullet_R)$ the category of filtered crystals over $(R,\Fil^\bullet_R)$. Note that $\cat{MF}(R,\Fil^\bullet_R)$ has the structure of an exact $R$-linear $\otimes$-category where we define the tensor product by the same formula as in \eqref{eq:filt-module-tensor-product}, and where
    \begin{equation*}
        0\to (M_1,\Fil^\bullet_{M_1})\to (M_2,\Fil^\bullet_{M_2})\to (M_3,\Fil^\bullet_{M_3})\to 0,
    \end{equation*}
    is exact if
    \begin{equation*}
        0\to \Fil^r_{M_1}\to \Fil^r_{M_2}\to \Fil^r_{M_3}\to 0
    \end{equation*}
is an exact sequence of $R$-modules for every $r$ in $\Z$.

The following freeness condition will play an important role in our paper.

\begin{defn}[{cf.\@ \cite[Definition 10]{Tsu20}}]\label{defn: filtered basis} Let $\ms{T}$ be a topos with final object $\ast$, and $(R,\Fil^\bullet_R)$ a filtered ring in $\ms{T}$. A filtered module $(M,\Fil^\bullet_M)$ over $(R,\Fil^\bullet_R)$ is \emph{(finite) free} if there exists a \emph{filtered basis}: a collection $(e_\nu,r_\nu)_{\nu=1}^n$ where $(e_\nu)_{\nu=1}^n$ is a basis of $M$ as an $R$-module, and $r_\nu$ are integers, such that 
\begin{equation}\label{eq:filtered-basis}
    \Fil^r_M=\sum_{\nu=1}^n\Fil^{r-r_\nu}_R\cdot e_\nu.
\end{equation}
We say that $(M,\Fil^\bullet_M)$ is \emph{locally filtered free (lff)} if there exists a cover $\{T_i\to \ast\}$ such that the restriction of $(M,\Fil^\bullet_M)$ to each slice topos $\ms{T}/T_i$ is (finite) free for all $i$. We denote the full subcategory of $\cat{MF}(R,\Fil^\bullet_R)$ consisting of lff objects by $\cat{MF}^\mr{lff}(R,\Fil^\bullet_R)$.
\end{defn}

For a filtered crystal $(M,\Fil^\bullet_M)$ over the filtered ring $(R,\Fil^\bullet_R)$ in a topos $\ms{T}$, we recall that the $r^\text{th}$-graded piece (for $r$ in $\Z$) is defined as the $R$-module
\begin{equation*}
    \mr{Gr}^r(M,\Fil^\bullet_M)=\mr{Fil}^r_M/\Fil^{r+1}_M.
\end{equation*}

\begin{eg} When $\Fil^\bullet_R=\Fil^\bullet_\triv$ a filtered crystal $(M,\Fil^\bullet_M)$ is (locally) filtered free over $(R,\Fil^\bullet_\triv)$ if and only if finitely many graded pieces of $\Fil^\bullet_M$ are non-zero and the graded pieces are (locally) free $R$-modules, i.e.\@, that $\Fil^\bullet_M\subseteq M$ is a locally split filtration.
\end{eg}

It is not hard to show that the category $\cat{MF}^\mr{lff}(R,\Fil^\bullet_R)$ is closed under tensor product, and so inherits the structure of an exact $R$-linear $\otimes$-subcategory from $\cat{MF}(R,\Fil^\bullet_R)$.

We end by making an elementary, but useful, observation about short exact sequences of filtered modules. First recall that a map $f\colon (M_1,\Fil^\bullet_1)\to (M_2,\Fil^\bullet_2)$ of filtered $(R,\Fil^\bullet_R)$-modules is called \emph{strict} if the equality $f(\Fil^j_1)=f(M)\cap \Fil^j_2$ for all $j$ (cf.\@ \stacks{0120} and \stacks{05SI}). 
\begin{lem}\label{lem:filtration-equiv} Let $R$ be a ring, and let 
\begin{equation}\label{eq:exact-seq-filtered-modules}
    0\to (M_1,\Fil_1^\bullet)\to (M_2,\Fil_2^\bullet)\to (M_3,\Fil_3^\bullet)\to 0,
\end{equation}
be a sequence of filtered $(R,\Fil^\bullet_\mr{triv})$-modules, which is a short exact sequence on the underlying $R$-modules. Then, the following are equivalent:
\begin{enumerate}
    \item the maps $(M_1,\Fil_1^\bullet)\to (M_2,\Fil_2^\bullet)$ and $(M_2,\Fil_2^\bullet)\to (M_3,\Fil_3^\bullet)$ are strict,
    \item the sequence 
    \begin{equation}\label{eq:filt-exact-seq-1}
        0\to \Fil^j_1\to\Fil^j_2\to\Fil^j_3\to 0
    \end{equation}
    is exact for all $j$ (i.e., \eqref{eq:exact-seq-filtered-modules} is an exact sequence of filtered $R$-modules).
\end{enumerate}
Suppose further that each of these filtered modules is strictly exhaustive 
(i.e.\@, $\Fil^j=M$ for a small enough $j$), then (1) and (2) are further equivalent to
\begin{enumerate}
\setcounter{enumi}{2}
    \item the sequence 
    \begin{equation}\label{eq:filt-exact-seq-2}
        0\to \mr{Gr}^j(\Fil^\bullet_1)\to\mr{Gr}^j(\Fil^\bullet_2)\to\mr{Gr}^j(\Fil^\bullet_3)\to 0
    \end{equation}
    is exact for all $j$. 
\end{enumerate}
\end{lem}
\begin{proof}
The equivalence of (1) and (2) follows from the definition of strictness. By the snake lemma, (2) implies (3). We can check the converse by induction on $j$ using the snake lemma.
\end{proof}

\subsubsection{The Rees algebra construction} We now recall the Rees construction for a filtered ring/module, and relate it to the notion of lff filtered modules. 

We begin by reviewing some terminology concerning graded rings.
\begin{itemize}[leftmargin=.5cm]
    \item A \emph{graded ring} is a ring $R$ together with a decomposition $R=\bigoplus_{r\in \Z}R_r$ as abelian groups such that $R_r\cdot R_s\subseteq R_{r+s}$ for all $r$ and $s$ in $\Z$. For a ring $A$, we say that $R$ is a \emph{graded $A$-algebra} if there is a ring map $A\to R$ with image in $R_0$.
     \item A \emph{graded module} $M$  over $R=\bigoplus_{r\in \Z}R_r$ consists of an $R$-module $M$, and a decomposition $M=\bigoplus_{r\in \Z}M_r$ as abelian groups such that $R_r\cdot M_s\subseteq M_{r+s}$ for all $r$ and $s$ in $\Z$.
    \item We say that a graded module $M=\bigoplus_{r\in\Z}M_r$ over $R=\bigoplus_{r\in \Z}R_r$ is \emph{finite projective} if its underlying $R$-module is finite projective.\footnote{See \cite[Lemma 3.0.1]{LauHigher} for why this terminology is unambiguous}
    \item A graded ring map $R\to S$ is a ring map $f\colon R\to S$ with $f(R_r)\subseteq S_r$ for all $r$. 
    \item For graded $R$-modules $M$ and $N$, we define their \emph{graded tensor product} by declaring
\begin{equation*}
   (M\otimes_R N)_n=\left\{\sum_i m_i\otimes n_i \in M\otimes_R N: m_i\in M_r,\, n_i\in N_s,\text{ and }r+s=n\right\}.
\end{equation*}
When $N=S$ is a graded ring equipped with the structure of a graded $R$-module via a graded ring map, then this tensor product is a graded $S$-module.
\end{itemize}

With the obvious notion of morphisms, denote by $\cat{MG}(R)$ the category of graded modules over $R=\bigoplus_{r\in \Z}R_r$. Note that $\cat{MG}(R)$ has the structure of an exact $R$-linear $\otimes$-category where
    \begin{equation*}
        0\to \bigoplus_{r\in\Z}M_{1,r}\to \bigoplus_{r\in\Z}M_{2,r}\to \bigoplus_{r\in\Z}M_{3,r}\to 0,
    \end{equation*}
is said to be exact if
    \begin{equation*}
        0\to M_{1,r}\to M_{2,r}\to M_{3,r}\to 0
    \end{equation*}
is an exact sequence of abelian groups for every $r$ in $\Z$.
    
We denote by $\cat{MG}^\mr{fp}(R)$ the full subcategory of $\cat{MG}(R)$ consisting of finite projective objects. This is clearly closed under tensor products, and therefore $\cat{MG}^\mr{fp}(R)$ inherits the structure of an exact $R$-linear $\otimes$-subcategory from $\cat{MG}(R)$.

For a filtered ring $(R,\Fil^\bullet_R)$ we now wish to relate $\cat{MF}^\mr{lff}(R,\Fil^\bullet_R)$ to $\cat{MG}^\mr{fp}(S)$ for a certain graded ring $S$ associated to $R$ which we now discuss.

\begin{defn}\label{defn: Rees construction} Let $(R,\Fil^\bullet_R)$ be a filtered ring. Then, its \emph{Rees algebra} (e.g., see \cite[Chapter I, \S4.3, Definition 5]{LvO}) is the graded ring
\begin{equation*}
    \mr{Rees}(\Fil^\bullet_R)\defeq \bigoplus_{r\in \Z}\Fil^r_R t^{-r}\subseteq R[t^{\pm 1}].
\end{equation*}
For a filtered $R$-module $(M,\Fil^\bullet_M)$ we define its \emph{Rees module} (see loc.\@ cit.\@) to be
\begin{equation*}
    \mr{Rees}(\Fil^\bullet_M)\defeq \bigoplus_{r\in\Z}\Fil^r_M  t^{-r}\subseteq M\otimes_R R[t^{\pm1}],
\end{equation*}
which is a graded $\mr{Rees}(\Fil^\bullet_R)$-module.
\end{defn}

Suppose that $(R,\Fil^\bullet_R)\to (S,\Fil^\bullet_S)$ is a map of filtered rings, and $(M,\Fil^\bullet_M)$ is a filtered $(R,\Fil^\bullet_R)$-module. Then, by functoriality of the Rees algebra construction we obtain a natural map $\mr{Rees}(\Fil^\bullet_R)\to \mr{Rees}(\Fil^\bullet_S)$ of graded rings. One thus obtains a canonical morphism
\begin{equation}\label{eq:Rees-tensor}
    \mr{Rees}(\Fil^\bullet_M)\otimes_{\mr{Rees}(\Fil^\bullet_R)}\mr{Rees}(\Fil^\bullet_S)\to \mr{Rees}(\Fil^\bullet_{M\otimes_R S}),
\end{equation}
of graded $\mr{Rees}(\Fil^\bullet_S)$-modules. This map is an isomorphism if the source has no non-trivial $t$-torsion 
(e.g., if $\mr{Rees}(\Fil^\bullet_M)$ is a flat module over $\mr{Rees}(\Fil^\bullet_R)$). Indeed, it suffices to verify this map induces an isomorphism after applying the inverse to the equivalence in \cite[Chapter I, \S4.3, Proposition 7]{LvO} (see also \cite[Chapter I, \S4.3, Observation 6 (a)]{LvO}). But, this is trivial.

Observe that as $\mr{Rees}(\Fil^\bullet_R)$ is a graded $R$-algebra, we have a natural action of the group $R$-scheme $\bb{G}_{m,R}=\Spec(R[x^{\pm 1}])$ on $\Spec(\mr{Rees}(\Fil^\bullet_R))$ corresponding to the coaction map 
\begin{equation*}
    \mr{Rees}(\Fil^\bullet_R)\to \mr{Rees}(\Fil^\bullet_R)\otimes_R R[x^{\pm 1}],
\end{equation*}
uniquely specified by declaring that an element $a$ in $\Fil^r_R t^{-r}$ in $\mr{Rees}(\Fil^\bullet_R)$ maps to $a\otimes x^r$. We may then consider the Artin stack over $R$ given by
\begin{equation*}
    \mc{R}(\Fil^\bullet_R)\defeq [\Spec(\mr{Rees}(\Fil^\bullet_R))/\bb{G}_{m,R}],
\end{equation*} 
called the \emph{Rees stack} of $(R,\mr{Fil}^\bullet_R)$. Given a graded $\mr{Rees}(\Fil^\bullet_R)$-module $M=\bigoplus_{r\in\Z}M_r$ there is a natural action of $\bb{G}_{m,R}$ on $M$ corresponding to the coaction map
\begin{equation*}
    M\to M\otimes_R R[x^{\pm 1}]
\end{equation*}
defined in the analogous way. This defines a quasi-coherent sheaf on $\mc{R}(\Fil^\bullet_R)$ by the construction in \stacks{06WT}. We denote this functor by $(-)/\bb{G}_{m,R}$.

Finally, suppose that $R$ is $J$-adically complete with respect to a finitely generated ideal $J\subseteq R$. We define the \emph{completed Rees stack} $\wh{\mc{R}}(\Fil^\bullet_R)$ (leaving the ideal $J$ implicit) to be the completion of $\mc{R}(\Fil^\bullet_R)$ along $\mc{R}(\Fil^\bullet_R)\times_{\Spec(R)}\Spec(R/J)$. We then have a natural pullback functor
\begin{equation*}
    \wh{(-)}\colon \cat{Vect}(\mc{R}(\Fil^\bullet_R))\to\cat{Vect}(\wh{\mc{R}}(\Fil^\bullet_R))=\twolim \cat{Vect}(\mc{R}(\Fil^\bullet_R)\times_{\Spec(R)}\Spec(R/J^n)),
\end{equation*}
which is an $R$-linear $\otimes$-functor. 

\begin{prop}\label{prop:rees-equiv} Suppose that $(R,\Fil^\bullet_R)$ is a filtered ring. Then, the functors
\begin{equation*}
    \cat{MF}^\mr{lff}(R,\Fil^\bullet_R)\to \cat{MG}^\mr{fp}(\mr{Rees}(\Fil^\bullet_R)),\qquad (M,\Fil^\bullet_M)\mapsto \mr{Rees}(\Fil^\bullet_M),
\end{equation*}
and
\begin{equation*}
    (-)/\bb{G}_m\colon \cat{MG}^\mr{fp}(\mr{Rees}(\Fil^\bullet_R))\to\cat{Vect}(\mc{R}(\Fil^\bullet_R))
\end{equation*}
are $2$-funtorial bi-exact $R$-linear $\otimes$-equivalences. Suppose further that $R$ is $J$-adically complete with respect to a finitely generated ideal $J\subseteq R$ and that $\Fil^i_R\subset R$ is closed with respect to the $J$-adic topology, for every $i$. Then, the functor
\begin{equation*}
    \wh{(-)}\colon \cat{Vect}(\mc{R}(\Fil^\bullet_R))\to\cat{Vect}(\wh{\mc{R}}(\Fil^\bullet_R))
\end{equation*}
is a $2$-functorial bi-exact $R$-linear $\otimes$-equivalence. 
\end{prop}

Before proving this proposition, we first establish that the Rees construction preserves and reflects locally-free-like conditions. More precisely, we have the following. 

\begin{prop}\label{prop:Rees-lff-proj}
    Let $(R,\Fil_R^\bullet)$ be a filtered ring. Then a filtered module $(M,\Fil_M^\bullet)$ over $(R,\Fil_R^\bullet)$ is lff if and only if its Rees module $\mr{Rees}(\Fil_M^\bullet)$ is finite projective over $\mr{Rees}(\Fil_R^\bullet)$. 
\end{prop}
\begin{proof}
    Since the problem is local on $R$,\footnote{Indeed, this follows from the observation that if $y$ is an element of $R$, then the Rees module for $(R,\Fil^\bullet_R)|_{D(y)}$, where $D(y)$ is the non-vanishing locus of $y$, is canonically identified with $\mr{Rees}(R,\Fil^\bullet_R)[\nicefrac{1}{y}]$.} we may assume that $R$ is a local ring. Then $(M,\Fil_M^\bullet)$ being lff is equivalent to it being finite free, and hence easily implies that $\mr{Rees}(\Fil_M^\bullet)$ is free. 

    We show the converse. When $\Fil_R^\bullet=\Fil_{(1)}^\bullet$, i.e., when $\Fil_R^i=R$ for all $i$ in $\Z$, the assertion is obvious. So, we assume $\Fil^\bullet_R\ne \Fil_{(1)}^\bullet$, i.e.,  that $\Fil^1_R$ is contained in the maximal ideal of the local ring $R$. 
    We first observe that the assertion holds when $\Fil^\bullet_R=\Fil_\mr{triv}^\bullet$, i.e., when $\Fil^1_R=0$. Indeed, note that $\bigoplus_{r\in\Z}\Gr^r_{\Fil_M}$, being the specialization of $\mr{Rees}(\Fil^\bullet_M)$ to $t=0$ (see \cite[Chapter I, \S4.3, Proposition 7]{LvO}), is finite projective over $R$, and hence so is each $\Gr^r_{\Fil_M}$. As $R$ is local, this implies that each $\Gr^r_{\Fil_M}$ is finite free, and hence $(M,\Fil_M^\bullet)$ is filtered free. 

    In general, consider the map $(R,\Fil_R^\bullet)\to (R/\Fil^1_R,\Fil_\mr{triv}^\bullet)$ and the corresponding map 
    \begin{equation*}
        \mr{Rees}(\Fil_R^\bullet)\to \mr{Rees}(\Fil_\mr{triv}^\bullet)\simeq( R/\Fil_R^1)[t].
    \end{equation*} 
    We observe that the scalar extension along this map is described as follows.
    \begin{lem}\label{lem: Rees tensor}
        We have a canonical isomorphism 
        \be
        \mr{Rees}(\Fil_M^\bullet)\otimes_{\mr{Rees}(\Fil_R^\bullet)}\mr{Rees}(\Fil_\mr{triv}^\bullet)
        \simeq\bigoplus_{r\in \Z}\frac{\Fil^r_M}{\sum_{i\geqslant1}\Fil^i_R\cdot\Fil^{r-i}_M} t^{-r}
        \ee
        of graded $\mr{Rees}(\Fil_\mr{triv}^\bullet)$-modules, where the multiplication-by-$t$ map on the right-hand side is defined by the canonical maps induced by $\Fil^{r+1}_M\to \Fil_M^r$. 
    \end{lem}
    \begin{proof}
        This follows by considering the short exact sequence 
        \be
        0\to \bigoplus_{i\geqslant 0}\Fil_R^1 t^i\oplus \bigoplus_{i\leqslant -1}\Fil^{-i}_R t^i
        \to \underbrace{\bigoplus_{i\in\Z}\Fil_R^{-i}t^i}_{\mr{Rees}(\Fil_R^\bullet)}\to \underbrace{\bigoplus_{i\geqslant0}R/\Fil^1_R t^i}_{\mr{Rees}(\Fil_\mr{triv}^\bullet)}\to 0,
        \ee 
        and tensoring it with $\mr{Rees}(\Fil_M^\bullet)$ over $\mr{Rees}(\Fil_R^\bullet)$.
    \end{proof} 
    Now assume that $\mr{Rees}(\Fil^\bullet_M)$ is finite projective over $\mr{Rees}(\Fil^\bullet_R)$. Then the tensor product $\mr{Rees}(\Fil_M^\bullet)\otimes_{\mr{Rees}(\Fil_R^\bullet)}\mr{Rees}(\Fil_\mr{triv})$ is finite projective over $\mr{Rees}(\Fil^\bullet_\mr{triv})$. In particular, the multiplication-by-$t$ map 
    \be
    \ov{\Fil}^{r+1}\defeq \frac{\Fil^{r+1}_M}{\sum_{i\geqslant1}\Fil^i_R\cdot\Fil^{r+1-i}_M}
    \to
    \ov{\Fil}^r\efdeq \frac{\Fil^r_M}{\sum_{i\geqslant1}\Fil^i_R\cdot\Fil^{r-i}_M}
    \ee
    is injective. This implies by \cite[Chapter I, \S4.3, Proposition 7]{LvO} that this graded $\mr{Rees}(\Fil^\bullet_\mr{triv})$-module is the Rees module of some module with filtration of the form $(\ov{M},\ov{\Fil}^\bullet)$ over $(R/\mr{Fil}^1_R,\mr{Fil}^\bullet_\mr{triv})$. By loc.\@ cit.\@, the underlying module $\ov{M}$ is given by $M/\Fil^1_R\cdot M$. Since we have already seen that the assertion holds in the case where $\Fil_R^\bullet=\Fil^\bullet_\mr{triv}$, we deduce that $(\ov{M},\ov{\Fil}^\bullet)$ is actually free. 
    
    Let $(\bar e_\nu,r_\nu)_{\nu=1}^n$ be a filtered basis of $(\ov{M},\ov{\Fil}^\bullet)$ over $(R/\Fil^1_R,\Fil^\bullet_\mr{triv})$, and take a lift $(e_\nu)_{\nu=1}^n$ with $e_\nu$ in $\Fil^{r_\nu}$, which is a basis of the $R$-module $M$. 
    \begin{claim}\label{claim:basis-lift}
        The tuple $(e_\nu,r_\nu)_{\nu=1}^n$ is a filtered basis of $(M,\Fil_M^\bullet)$ over $(R,\Fil_R^\bullet)$. In other words,
        \be\Fil^r_M=\sum_{\nu=1}^n\Fil_R^{r-r_\nu}e_\nu,
        \ee
        for every $r$ in $\Z$. 
    \end{claim}
    \begin{proof}
        When $r\leqslant r_\nu$ for every $\nu$, we have $\Fil^r_M=M$ by Nakayama's lemma, and this is clearly equal to the right-hand side of the claimed equality. 

        We show the equality by induction on $r$. Fix an $r$ in $\Z$ and assume that the equality holds for $r$. We show $\Fil^{r+1}_M=\sum_{\nu=1}^n\Fil_R^{r+1-r_\nu}e_\nu$. The right-hand side is evidently contained in the left. Conversely, consider $x=\sum_\nu a_\nu e_\nu$ in $\mr{Fil}^{r+1}_M$. Since $x$ is in $\Fil^r_M$, we know by induction that $a_\nu$ belongs to $\Fil^{r-r_\nu}_R$. For $\nu$ with $r=r_\nu$, we have that $\bar e_\nu$ is in $\ov{\Fil}^{r}\setminus\ov{\Fil}^{r+1}$, and hence that $a_\nu$ belongs to $\Fil^1_R=\Fil^{r+1-r_\nu}_R$. 
        We now consider $x_1=\sum_{\nu:r>r_\nu}a_\nu e_\nu$, which we know is both in $\Fil^{r+1}_M$ and $\sum_{i\geqslant 1}\Fil_R^i\Fil^{r-i}_M$. Thus, by the injectivity of the map $\ov{\Fil}^{r+1}\to \ov{\Fil}^r$, we get that $x_1$ is in $\sum_{i\geqslant 1}\Fil_R^i\Fil_M^{r+1-i}$, which by induction, implies that $a_\nu$ is in $\Fil^{r+1-r_\nu}_R$ for $\nu$ with $r>r_\nu$ . 
    \end{proof}
    In particular, this claim implies the filtered module $(M,\Fil^\bullet_M)$ is filtered free as desired.
\end{proof}

\begin{rem}\label{rem: Rees-lff-proj}
    The proof of Proposition \ref{prop:Rees-lff-proj} shows the following. Assume that $\Fil^1_R$ is contained in the Jacobson radical of $R$ and that $\mr{Rees}(\Fil^\bullet_M)\otimes_{\mr{Rees}(\Fil^\bullet_R)}\mr{Rees}(\Fil_\mr{triv}^\bullet)$ is free over $\mr{Rees}(\Fil_\mr{triv}^\bullet)\simeq R/\Fil^1_R[t]$. Then $\mr{Rees}(\Fil^\bullet_M)$ is free over $\mr{Rees}(\Fil^\bullet_R)$, or equivalently, the filtered module $(M,\Fil_M^\bullet)$ is filtered free over $(R,\Fil_R^\bullet)$. 
\end{rem}

\begin{proof}[Proof of Proposition \ref{prop:rees-equiv}] Given Proposition \ref{prop:Rees-lff-proj}, the fact that the first functor is an $R$-linear $\otimes$-equivalence follows from \cite[Chapter I, \S4.3, Proposition 7]{LvO} together with the isomorphism given by \eqref{eq:Rees-tensor}. Moreover, an explicit quasi-inverse to the first functor is given by by taking a graded $\mr{Rees}(\Fil^\bullet_R)$-module $N=\bigoplus_{r\in\Z}N_r$ to the module $M=N/(t-1)$ with
\begin{equation*}
    \Fil^r_M\defeq (N_r+(t-1)N)/(t-1)N.
\end{equation*}
The fact that this functor and its quasi-inverse  are exact follows from Lemma \ref{lem:filtration-equiv}, \cite[Chapter I, \S4.3, Proposition 8]{LvO}, together with the observation that an exact sequence of projective $\mr{Rees}(\Fil^\bullet_R)$-modules is split and thus exactness is preserved by any additive functor. 

The fact that the second functor is an exact $R$-linear $\otimes$-equivalence follows from the general theory of stacks (e.g., see again \stacks{06WT}). Its bi-exactness is clear by inspection.

Finally, to show the last functor is a bi-exact $R$-linear $\otimes$-equivalence we begin by defining a functor $\cat{Vect}(\wh{\mc R}(\Fil^\bullet_R))\to \cat{Vect}(\mc R(\Fil^\bullet_R))$. Noting that there is a canonical identification 
\begin{equation*}
    \mc R(\Fil^\bullet_R)\times_{\Spec(R)}\Spec(R/J^n)\simeq \left[\Spec\left(\mr{Rees}(\Fil^\bullet_R)\otimes_RR/J^n\right)/\bb{G}_{m,R}\right],
\end{equation*} 
so giving such a functor is equivalent to giving a functor 
\begin{equation}\label{eq:2-lim-functor}
    \twolim_n\cat{MG}^\mr{fp}(\mr{Rees}(\Fil^\bullet_R)/J^n)\to \cat{MG}^\mr{fp}(\mr{Rees}(\Fil^\bullet_R)),
\end{equation}
which we now do.

In the following, we denote by $\mc O_n(d)$ the graded $\mr{Rees}(\Fil^\bullet_R)/J^n$-module with underlying $\mr{Rees}(\Fil^\bullet_R)/J^n$-module free of rank one, generated by a homogeneous element of degree $d$. 
    \begin{claim}\label{claim:limit-of-fp}
        Let $(M_n=\bigoplus_{i\in\Z}M_{n,i})_{n\geqslant 1}$ be an object of the source of \eqref{eq:2-lim-functor}. Then, the graded $\mr{Rees}(\Fil_R^\bullet)$-module $M\defeq \varprojlim_n M_n\simeq \bigoplus_{i\in \Z}(\varprojlim_n M_{n,i})$ is finite projective.\footnote{Note that this first inverse limit is taken in the category of graded $\mathrm{Rees}(\Fil^\bullet_R)$-modules, which is indeed computed via this last direct sum.} Moreover, the natural map $M\otimes_RR/J^n\to M_n$ is an isomorphism. 
    \end{claim}
    \begin{proof}
        By taking a finite set of homogeneous generators of $M_1$, we may produce a graded surjection $\bigoplus_{j=1}^r\mc O_1(d_j)\to M_1$ for some $r$ and some $d_1,\ldots,d_r$ which induces generators $m_{1,1},\ldots,m_{1,r}$ of $M_1$. This surjection of graded $\mr{Rees}(\Fil^\bullet_R)/J$-modules is split by the projectivity of $M_1$.

        Choosing compatible homogeneous lifts $m_{n,i}$ of $m_{1,j}$ for each $n\geqslant 1$ and $j=1,\ldots,r$ we obtain a compatible system of homogeneous maps $(\bigoplus_{j=1}^r\mc O_n(d_j)\to M_n)_{n}$ of $\mr{Rees}(\Fil^\bullet_R)/J^n$-modules. As $J$ is nilpotent in $R/J^n$ we further see by Nakayama's lemma that these maps are surjective for all $n$ and thus, using the projectivity of $M_n$, split surjections. By the assumption that $\Fil^i_R$ is closed in $R$ for all $i$, the inverse limit $\lim_n\bigoplus_j\mc O_n(d_j)$ is identified with $\bigoplus_j\mc O(d_j)$ as a graded $\mr{Rees}(\Fil^\bullet_R)$ module, which is finite projective. 
        As we have a split surjection $\bigoplus_j\mc O(d_j)\to M$, the inverse limit $M$ is also finite projective as desired.
        
        Finally, thanks to the system of split surjection constructed above, the natural morphism $M\otimes_RR/J^n\to M_n$ being an isomorphism is reduced to case when $M=\mc{O}(d_j)$, which is clear. 
    \end{proof}
    Claim \ref{claim:limit-of-fp} allows us to produce a functor as in \eqref{eq:2-lim-functor} which we claim is quasi-inverse to $\widehat{(-)}$. On filtered finite free modules (i.e., the modules of the form $\bigoplus_j\mc O(d_j)$), this follows as $\Fil^i_R\subseteq R$ is closed for all $i$. In general, let $M$ be an object of $\cat{MG}^\mr{fp}(\mc R(\Fil^\bullet_R))$. Then showing that $M\to \lim_nM\otimes_RR/J^n$ is an isomorphism can be reduced to the filtered finite free case by taking a split surjection $\bigoplus_j\mc O(d_j)\to M$. That the other composition is naturally isomorphic to the identity follows from Claim \ref{claim:limit-of-fp}. The fact that both of these functors are exact $R$-linear $\otimes$-functors is clear by inspection.
\end{proof}

\subsection{Prismatic \texorpdfstring{$F$-gauges}{F-gauges} in vector bundles} We now recall the two ways of describing prismatic $F$-gauges in vector bundles: in terms of Rees algebras and in terms of formal stacks.

\medskip

\paragraph*{Description in terms of Rees algebras} We begin by recalling the natural filtration on a prism used to construct the operative Rees algebras for prismatic $F$-gauges.

\begin{defn} For a prism $(A,I)$, we define the \emph{Nygaard filtration} on $A$ as follows:
\begin{equation*}
    A\supseteq \Fil^r_\mathrm{Nyg}(A,I)\defeq \begin{cases} \phi_A^{-1}(I^r) & \mbox{if}\quad r\geqslant 0\\ A & \mbox{if}\quad r<0.\end{cases}
\end{equation*}
We often write $\Fil^\bullet_\mr{Nyg}(A)$ or just $\mr{Fil}^\bullet_\mr{Nyg}$ when there is no chance for confusion. 
\end{defn}

\begin{lem}\label{lem:Nygaard-closed} Let $(A,I)$ be a bounded prism. Then, $\mr{Fil}^r_\mr{Nyg}(A,I)\subseteq A$ is closed for all $r$.
\end{lem}
\begin{proof} It suffices to show that the ideals $I^r$ are closed, and that $\phi_A\colon A\to A$ is $(p,I)$-adically continuous. To see the former, it suffices by \cite[\S23.B]{MatsumuraCommAlg} to show that $A/I^r$ is $J$-adically complete with $J=(p,I)A/I^r$. But, as $J^r\subseteq pA/I^r\subseteq J$ we have that the $J$-adic and $p$-adic topologies coincide. Thus, the claim follows from \cite[Lemma 1.2]{IKY1}. For the latter claim, it suffices to observe that $\phi_A$ stabilizes $(p,I)$ (see \cite[Lemma 3.4]{GuoReinecke}).
\end{proof}

Let $R$ be a qrsp ring (see \cite[Definition 4.20]{BMS-THH}), and consider the initial object $(\Prism_R,I_R)$ of $R_\smallprism$ (see \cite[Proposition 7.2]{BhattScholzePrisms}). Associated to $\mr{Rees}(\Fil^\bullet_\mr{Nyg}(\Prism_R))$ are the following two maps:
\begin{enumerate}
    \item the map of $\Prism_R$-algebras
    \begin{equation*}
        \tau\colon\mr{Rees}(\Fil^\bullet_\mr{Nyg}(\Prism_R))\to \Prism_R,\qquad t\mapsto 1, 
    \end{equation*}
    \item the graded homomorphism
    \begin{equation*}
        \sigma\colon\mr{Rees}(\Fil^\bullet_\mr{Nyg}(\Prism_R))\to \bigoplus_{r\in \Z}I_R^r  t^{-r},\qquad \sum_r a_rt^{-r}\mapsto \sum_r \phi(a_r)t^{-r},
    \end{equation*}
    where the target is considered as a graded $\Prism_R$-subalgebra of $\Prism_R[\nicefrac{1}{I_R}][t^{\pm 1}]$.
\end{enumerate}

\begin{rem}\label{rem:0th-part-equiv} 
The functor 
\begin{equation*}
    \cat{MG}\left(\bigoplus_{r\in \Z}I_R^r  t^{-r}\right)\to \prism_R\text{-}\cat{Mod},\qquad M=\bigoplus_{r\in\Z}M_r\mapsto M_0,
\end{equation*}
is a bi-exact $\Prism_R$-linear $\otimes$-equivalence whose quasi-inverse is given by 
\begin{equation*}
    L\mapsto L\otimes_{\Prism_R}\left(\bigoplus_{r\in \Z}I_R^r t^{-r}\right),
\end{equation*}
with the obvious grading. In fact, $[\Spec(\bigoplus_{r\in\Z}I^r_R t^{-r})/\bb{G}_{m,\smallprism_R}]$ is isomorphic to $\Spec(\Prism_R)$.
\end{rem}

We now give the Rees-algebra-theoretic definition of prismatic $F$-gauges over $R$.

\begin{defn}[{Drinfeld and Bhatt--Lurie}]\label{defn:F-gauges-Rees} Let $R$ be a qrsp ring. A \emph{prismatic $F$-gauge (in vector bundles)} over $R$ is a pair $(M,\varphi_M)$ where $M$ is an object of $\cat{MG}^\mr{fp}(\mr{Rees}(\Fil^\bullet_\mr{Nyg}(\Prism_R)))$ and $\varphi_M$ is an isomorphism of $\Prism_R$-modules $ (\sigma^\ast M)_0\isomto \tau^\ast M$.
\end{defn}

A \emph{morphism} $(M_1,\varphi_{M_1})\to (M_2,\varphi_{M_2})$ of prismatic $F$-gauges over $R$ is a morphism $f\colon M_1\to M_2$ of graded $\mr{Rees}(\Fil^\bullet_\mr{Nyg}(\Prism_R))$-modules such that $\varphi_{M_2}\circ \sigma^\ast(f)=\tau^\ast(f)\circ \varphi_{M_1}$. We denote the category of prismatic $F$-gauges in vector bundles over $R$ by $\FGauge(R)$. The category $\FGauge(R)$ is an exact $\Z_p$-linear $\otimes$-category with structure essentially inherited from $\cat{MG}^\mr{fp}(\mr{Rees}(\Fil^\bullet_\mr{Nyg}(\Prism_R))$, but where we decree that
\begin{equation*}
    \varphi_{M_1\otimes_{\mr{Rees}(\Fil^\bullet_\mr{Nyg}(\Prism_R))}M_2}\defeq \varphi_{M_1}\otimes\varphi_{M_2},
\end{equation*}
(which makes sense as $(\sigma^\ast(M_1\otimes_{\Prism_R}M_2))_0\simeq (\sigma^\ast M_1)_0\otimes_{\Prism_R}(\sigma^\ast M_2)_0$ by Remark \ref{rem:0th-part-equiv}).

A morphism of qrsp rings $R\to S$ gives a map $\mr{Rees}(\mr{Fil}^\bullet_\mr{Nyg}(\Prism_R))\to \mr{Rees}(\mr{Fil}^\bullet_\mr{Nyg}(\Prism_S))$ of graded rings compatible with both $\tau$ and $\sigma$. Thus, base extension provides an exact $\bb{Z}_p$-linear $\otimes$-functor $\FGauge(R)\to \FGauge(S)$. So, if $\mf{X}$ is a quasi-syntomic $p$-adic formal scheme, $\FGauge$ forms a natural prestack on $\mf{X}_\qsyn$ which is a stack by \cite[Proposition 2.29]{GuoLi}. 

\begin{defn}[Drinfeld, Bhatt--Lurie]\label{defn:prismatic-F-gauges-Rees-algebras} For a quasi-syntomic $p$-adic formal scheme $\mf{X}$, the category of \emph{prismatic $F$-gauges (in vector bundles)} over $\mf{X}$ is given by the $2$-limit
\begin{equation*}
    \FGauge(\mf{X})=\twolim_{\Spf(R)\in\mf{X}_\qrsp}\FGauge(R),
\end{equation*}
equipped with the structure of an exact $\Z_p$-linear $\otimes$-category defined term-by-term.
\end{defn}

Suppose that $\mc{G}$ is a smooth group $\Z_p$-scheme. One can then make sense of the category $\mc{G}\text{-}\mc{C}$ of $\mc{G}$-objects in a $\Z_p$-linear $\otimes$-category $\mc{C}$ (e.g., see \cite[\S A.5]{IKY1}).

\begin{defn}\label{defn:prismatic-F-gauges-with-G-structure-Rees-algebras} For a quasi-syntomic $p$-adic formal scheme $\mf{X}$, the category $\GFGauge(\mf{X})$ of \emph{prismatic $F$-gauges with $\mc{G}$-structure over $\mf{X}$} is the category of $\mathcal{G}$-objects in $\FGauge(\mf{X})$.
\end{defn}

\medskip

\paragraph*{Prismatic $F$-gauges in terms of formal stacks} We now compare Definition \ref{defn:prismatic-F-gauges-with-G-structure-Rees-algebras} to the notion of a prismatic $F$-gauge with $\mc{G}$-structure which implicitly appears in \cite{BhattNotes}. Throughout this section we fix a bounded $p$-adic formal scheme $\mf{X}$.

Attached to $\mf{X}$ are the following formal stacks over $\mathbb{Z}_p$:
\begin{itemize}[leftmargin=.5cm]
    \item the \emph{prismatization} $\mf{X}^\smallprism$ as in \cite[Construction 7.1]{BhattLuriePrismatization} (cf.\@ \cite[Definition 5.1.6]{BhattNotes}) 
    classifying Cartier--Witt divisors (see \cite[Definition 5.1.3]{BhattNotes}) which is equipped with a Frobenius $F_\mf{X}\colon \mf{X}^\smallprism\to\mf{X}^\smallprism$ (see \cite[Remark 5.1.10]{BhattNotes}),
    \item the \emph{Nygaard filtered prismatization} $\mf{X}^\smallN$ as in \cite[Definition 5.3.10]{BhattNotes} and \cite[6.4]{GMM} (which classifies filtered Cartier--Witt divisors as in \cite[Definition 5.3.1]{BhattNotes}) which has a structure map $\pi_\mf{X}\colon \mf{X}^\smallN\to \mf{X}^\smallprism$,
    \item  the \emph{Hodge embedding} and \emph{de Rham embedding} $j_{\mf{X},\mr{HT}}$ and $j_{\mf{X},\mr{dR}}$ (see loc.\@ cit.\@) which are open embeddings $\mf{X}^\smallprism\hookrightarrow \mf{X}^\smallN$.
\end{itemize}
One has the equalities $\pi_\mf{X}\circ j_{\mf{X},\mr{dR}}=\id_{\mf{X}^\smallprism}$ and $\pi_\mf{X}\circ j_{\mf{X},\mr{HT}}=F_\mf{X}$. Whenever $\mf{X}$ is clear from context, we shall omit the decoration of $\mf{X}$ on these maps.

As in \cite[Definition 6.1.1]{BhattNotes}, we define the formal stack $\mf{X}^\mr{syn}$ over $\mathbb{Z}_p$, the \emph{syntomification} of $\mf{X}$, and the maps $j_\smallN$ and $j_\smallprism$ so that the following diagram is cocartesian:
\begin{equation*}
    \begin{tikzcd}[sep=huge]
	{\mf{X}^\smallprism\sqcup\mf{X}^\smallprism} & {\mf{X}^\smallN} \\
	{\mf{X}^\smallprism} & {\mf{X}^\mr{syn}.}
	\arrow["{\mr{taut.}}"', from=1-1, to=2-1]
	\arrow["{j_\mr{HT}\sqcup j_\mr{dR}}", from=1-1, to=1-2]
	\arrow["{j_\smallprism}"', from=2-1, to=2-2]
	\arrow["{j_\smallN}", from=1-2, to=2-2]
	\arrow["\lrcorner"{anchor=center, pos=0.125, rotate=180}, draw=none, from=2-2, to=1-1]
\end{tikzcd}
\end{equation*}
If $\mf{X}=\Spf(R)$ we shorten the notation of these objects to $R^\smallprism$, $R^\smallN$, and $R^\mr{syn}$. These constructions can be further extended to the case when $R$ is a $p$-complete animated ring (see \cite{GMM}).

The following shows that these formal stacks over $\mathbb{Z}_p$ are more manageable when $\mf{X}$ is quasi-syntomic and, in particular, are classical (i.e., don't have non-trivial derived structure).

\begin{prop}[Bhatt--Lurie]\label{prop:syn-for-qsyn} Suppose that $\mf{X}$ is quasi-syntomic. Then, one has a canonical identification of formal stacks over $\mathbb{Z}_p$:
    \begin{equation*}
        \mf{X}^\smallprism\simeq \twocolim_{\Spf(R)\in\mf{X}_\qrsp}\Spf(\Prism_R),\qquad \mf{X}^\smallN=\twocolim_{\Spf(R)\in\mf{X}_\qrsp}\wh{\mc{R}}(\Fil^\bullet_\mr{Nyg}(\Prism_R)).
    \end{equation*}
 where for each $R$ the topology on $\Prism_R$ and the completion $\wh{\mc{R}}(\Fil^\bullet_\mr{Nyg}(\Prism_R))$ are in terms of the $(p,I_R)$-adic topology. Moreover, under these identifications $j_\dR$ and $j_\mr{HT}$ are obtained by taking the colimit over $R$ of the maps $\tau$ and $\sigma$, respectively.

\end{prop}
\begin{proof} When $\mf{X}=\Spf(R)$ is qrsp, this follows from \cite[Theorem 7.17]{BhattLuriePrismatization}, \cite[Theorem 6.11.5]{GMM} and \cite[Remark 5.5.5]{BhattNotes}. In general, it suffices to observe that if $\mc{S}$ is a quasi-syntomic cover over $\mf{X}$ with each constituent of the form $\Spf(R)$ for a qrsp ring $R$ then
\begin{equation*}
    \mf{X}^\smallprism=\twocolim_{S\in\mc{S}^\bullet}S^\smallprism,\qquad \mf{X}^\smallN=\twocolim_{S\in\mc{S}^\bullet}S^\smallN
\end{equation*}
where $\mc{S}^\bullet$ is the \v{C}ech nerve of $\mc{S}$ in the topos of formal stacks over $\Z_p$, which follows from the covering properties discussed in \cite[Proposition 7.5]{BhattLuriePrismatization} and \cite[Corollary 6.12.8]{GMM} (cf.\@ \cite[Remark 5.5.18]{BhattNotes}). But, as $\Prism_{R_1}\widehat{\otimes}_{\Prism_{R_2}}\Prism_{R_3}\simeq \Prism_{R_1\widehat{\otimes}_{R_2}R_3}$ (see \cite[Proposition 3.30]{AnschutzLeBrasDD}) and a similar tensor-product compatibility
holds for completed Rees stacks, the claim follows.
\end{proof}

Combining this with Proposition \ref{prop:rees-equiv} and Lemma \ref{lem:Nygaard-closed} we obtain the following.

\begin{cor}\label{cor:Rees-ringed-stack-comp} Suppose that $\mf{X}$ is quasi-syntomic and that $\mc{G}$ is a smooth group $\Z_p$-scheme. Then, there are natural bi-exact $\Z_p$-linear $\otimes$-equivalences
\begin{equation*}
    \FGauge(\mf{X})\isomto \cat{Vect}(\mf{X}^\mr{syn})\isomto \twolim_{\Spf(R)\in\mf{X}_\qrsp}\cat{Vect}(R^\syn). 
\end{equation*}
and natural equivalences
\begin{equation*}
    \GFGauge(\mf{X})\isomto\GVect(\mf{X}^\mr{syn})\isomto \twolim_{\Spf(R)\in\mf{X}_\qrsp}\GVect(R^\syn)
\end{equation*}
\end{cor}

Due to Corollary \ref{cor:Rees-ringed-stack-comp}, we shall use the notation $\FGauge(\mf{X})$ (resp.\@ $\GFGauge(\mf{X})$) and $\cat{Vect}(\mf{X}^\mr{syn})$ (resp.\@ $\GVect(\mf{X}^\mr{syn})$) interchangeably when $\mf{X}$ is quasi-syntomic.

\begin{rem}\label{rem:syn-twists} For a (classical) formal stack $\mc{X}$ over $\Z_p$ a \emph{vector bundle} on $\mc{X}$ is a vector bundle on $(\mc{X}_{\mr{fpqc}},\mc{O}_{\mc{X}})$ where $\mc{X}_\mr{fpqc}$ is as in \stacks{06NU} and $\mc{O}_\mc{X}$ is as in \stacks{06TU}. Formally, one has a bi-exact $\Z_p$-linear $\otimes$-equivalence
\begin{equation*}
    \cat{Vect}(\mc{X})\simeq \twolim_{\Spec(R)\to \mc{X}}\cat{Vect}(R),
\end{equation*}
where $R$ is a (variable) $p$-nilpotent ring (one can also replace this with $\Spf(R)\to\mc{X}$ where now $\Spf(R)$ is an object of $\Spf(\Z_p)^\mr{adic}_\fl$), where the right-hand side is endowed with the term-by-term exact $\Z_p$-linear $\otimes$-structure. So one may formally apply \cite[Theorem A.18]{IKY1} to deduce that if $(\Lambda_0,\mathds{T}_0)$ is a tensor package for $\mc{G}$ (see \cite[\S A.5]{IKY1}) then $\GVect(\mc{X})\simeq \cat{Twist}_{\mc{O}_\mc{X}}(\Lambda_0,\mathds{T}_0)$. In particular, this applies when $\mc{X}=\mf{X}^\mr{syn}$ for $\mf{X}$ a quasi-syntomic $p$-adic formal scheme.
\end{rem}

\subsection{Relationship to prismatic \texorpdfstring{$F$-crystals}{F-crystals}}\label{ss:forgetful-functor} We now clarify the relationship between prismatic $F$-crystals and prismatic $F$-gauges on a base formal $\mc{O}_K$-scheme. We refer the reader to \cite[\S1.1]{IKY1} for standard terminology and notation concerning base formal schemes.

\begin{nota}\label{nota:OK} We fix the following notation:
\begin{itemize}[leftmargin=.1in]
\item $k$ is a perfect extension of $\bb{F}_p$, $W\defeq W(k)$, and $K_0\defeq \mr{Frac}(W)$,
\item $K$ is a finite totally ramified extension of $K_0$, with ring of integers $\mc{O}_K$ and ramification index $e$,
\item $\pi$ is a uniformizer of $K$ and $E=E(u)$ in $W[u]$ is the minimal polynomial for $\pi$ over $K_0$,
\item for a formally framed base $\mc{O}_K$-algebra $R$ we set $(\mf{S}_R,(E))$ to be the Breuil--Kisin prism.
\end{itemize}
\end{nota}

\begin{construction}[{\cite[Remark 6.3.4]{BhattNotes}}]\label{const:forgetful-functor} Let $\mf{X}$ be a quasi-syntomic $p$-adic formal scheme. Then, there is a natural $\Z_p$-linear exact $\otimes$-functor
\begin{equation*}
    \mathrm{R}_\mf{X}\colon \cat{Perf}(\mf{X}^\mr{syn})\to\cat{D}^\varphi_\mr{perf}(\mf{X}_\smallprism), \qquad \mc{V}\mapsto \mathrm{R}_\mf{X}(\mc{V})=(\mc{E},\varphi_\mc{E}),
\end{equation*}
which we imprecisely call the \emph{forgetful functor}, constructed as follows.

\medskip

\noindent\textbf{Step 1:} Set $\mc{E}^\smallprism\defeq j_\mr{dR}^\ast\mc{V}$, which we interpret as an object $\mc{E}$ of $\cat{Perf}(\mf{X}_\smallprism)$ via \cite[Theorem 6.5]{BhattLuriePrismatization} so then $F^\ast\mc{E}^\smallprism$ corresponds to $\phi^\ast\mc{E}$. More explicitly, for an object $(A,I)$ of $\mf{X}_\smallprism$, one can build a morphism $\rho_{(A,I)}\colon \Spf(A)\to\mf{X}^\smallprism$ as in \cite[Construction 3.10]{BhattLuriePrismatization} and then $\mc{E}(A,I)=\rho_{(A,I)}^\ast\mc{E}^\smallprism$.

\medskip

\noindent\textbf{Step 2:} Observe that there are natural morphisms
\begin{equation*}
    j_\mr{HT}^\ast\mc{V}\leftarrow F^\ast\pi_{\ast}\mc{V}\to F^\ast j_\mr{dR}^\ast \mc{V},
\end{equation*}
where the second map is the pullback along $F$ of the natural map $\pi_{\ast}\mc{V}\to j_\mr{dR}^\ast\mc{V}$ coming from the fact that $j_\mr{dR}$ is a section of $\pi$, and the first map is obtained by adjunction from the map $\pi_{\ast}\mc{V}\to F_{\ast}j_\mr{HT}^\ast\mc{V}$ using the fact that $\pi\circ j_\mr{HT}=F$. By \cite[Remark 6.3.4]{BhattNotes}, this induces an isomorphism after inverting the invertible ideal $\mc{I}_{\mf{X}^\smallprism}\subseteq \mc{O}_{\mf{X}^\smallprism}$ (see \cite[Construction 5.1.18]{BhattNotes}). 

\medskip

\noindent\textbf{Step 3:} By the construction of $\mf{X}^\mr{syn}$, we have a canonical identification $j_\mr{HT}^\ast\mc{V}\simeq j_\mr{dR}^\ast\mc{V}$. Thus, altogether, we get an isomorphism
\begin{equation*}
F^\ast\mc{E}^\smallprism[\nicefrac{1}{\mc{I}_{\mf{X}^\smallprism}}]\simeq j^\ast_\mr{HT}\mc{V}[\nicefrac{1}{\mc{I}_{\mf{X}^\smallprism}}]\simeq j_\mr{dR}^\ast\mc{V}[\nicefrac{1}{\mc{I}_{\mf{X}^\smallprism}}]=\mc{E}^\smallprism[\nicefrac{1}{\mc{I}_{\mf{X}^\smallprism}}].
\end{equation*}
which corresponds to an isomorphism $\varphi_\mc{E}\colon \phi^\ast\mc{E}[\nicefrac{1}{\mc{I}_\smallprism}]\to \mc{E}[\nicefrac{1}{\mc{I}_\smallprism}]$.  

\medskip

Finally, we observe by construction that $\mr{R}_\mf{X}$ restricts to give a functor 
\begin{equation}\label{eq:forgetful-functor-vbs}
\mr{R}_\mf{X}\colon \cat{Vect}(\mf{X}^\mr{syn})\to \cat{Vect}^\varphi(\mf{X}_\smallprism).
\end{equation}
\end{construction}

\begin{rem} When $\mf{X}=\Spf(R)$ for a qrsp ring $R$, one may understand $\mr{R}_{\mf{X}}$ as in \eqref{eq:forgetful-functor-vbs} as the functor $\FGauge(R)\to \cat{Vect}^\varphi(\Prism_R)$ explicitly described in \cite[Propsition 8.1.9]{Ito1}.
\end{rem}

In \cite[Corollary 2.31]{GuoLi} (see also \cite[Propsition 8.1.9]{Ito1}), it is shown that $\mr{R}_{\mf{X}}$ is fully faithful. We now wish to describe the essential image when $\mf{X}$ is a base formal $\mc{O}_K$-scheme.

\begin{defn}\label{defn:Nygaard-filtration} We make the following definitions.
\begin{enumerate}[leftmargin=.7cm]
    \item For a prism $(A,I)$, and an object $(M,\varphi_M)$ of $\cat{Vect}^\varphi(A,I)$, set 
\begin{equation*}
    \Fil^r_\mathrm{Nyg}(\phi_A^\ast M)\defeq \left\{x\in\phi_A^\ast M:\varphi_M(x)\in I^rM\right\},
\end{equation*}
which defines a filtration $\Fil^\bullet_\mr{Nyg}(\phi_A^\ast M)$ by $A$-submodules, called the \emph{Nygaard filtration}.
\item For a quasi-syntomic $p$-adic formal scheme $\mf{X}$ and an object $(\mc{E},\varphi_\mc{E})$ of $\cat{Vect}^\varphi(\mf{X}_\smallprism)$, we define the filtration $\Fil^\bullet_\mr{Nyg}(\phi^\ast\mc{E})\subseteq \phi^\ast\mc{E}$ by $\mc{O}_\smallprism$-submodules, called the \emph{Nygaard filtration}, so that $\Fil^\bullet_\mr{Nyg}(\phi^\ast\mc{E})(A,I)=\Fil^{\bullet}_\mr{Nyg}(\phi^\ast\mc{E}(A,I))$, functorially in an object $(A,I)$ of $\mf{X}_\smallprism$.\footnote{See \cite[Remark 1.13]{IKY3} for a remark about the terminology `Nygaard filtration'.} 
\end{enumerate}
\end{defn} 

By design $(\phi^\ast\mc{E},\Fil^\bullet_\mr{Nyg}(\phi^\ast\mc{E}))$ is a filtered module over $(\mc{O}_\smallprism,\Fil^\bullet_{\mc{I}_\smallprism})$.

\begin{defn}\label{defn:lff} Let $\mf{X}$ be a quasi-syntomic $p$-adic formal scheme. We call a prismatic $F$-crystal $(\mc{E},\varphi_\mc{E})$ on $\mf{X}$ \emph{locally filtered free (lff)} if $(\phi^\ast\mc{E},\Fil^\bullet_\mr{Nyg}(\phi^\ast\mc{E}))$ is lff over $(\mc{O}_\smallprism,\Fil^\bullet_{\mc{I}_\smallprism})$. 
\end{defn}

\begin{rem} Note that $(\phi^\ast\mc{E},\Fil^\bullet_\mr{Nyg}(\phi^\ast\mc{E}))$ is, a priori, only a filtered module over $(\mc{O}_\smallprism,\Fil^\bullet_{\mc{I}_\smallprism})$ and not a filtered crystal. But, in the lff case this is true (see \cite[Proposition 3.1.13]{Ito1}).
\end{rem}

Denote the full subcategory of $\cat{Vect}^\varphi(\mf{X}_\smallprism)$ consisting of lff objects by $\cat{Vect}^{\varphi,\mr{lff}}(\mf{X}_\smallprism)$. It is stable under tensor products, and so it inherits the structure of an exact $\Z_p$-linear $\otimes$-category.

\begin{example}\label{ex:mu-lff} Let $\mf{X}$ be a base formal $W$-scheme. If $\omega$ belongs to $\GVect^{\varphi,\mu}(\mf{X}_\smallprism)$ for a cocharacter $\mu\colon \mbb G_{m,W}\to \mc G_W$ (see \cite[Definition 3.12]{IKY1} or Definition \ref{defn:type-mu-crystals} below), then $\omega(\Lambda)$ belongs to $\cat{Vect}^{\varphi,\mr{lff}}(\mf{X}_\smallprism)$ for all objects $\Lambda$ of $\cat{Rep}_{\Z_p}(\mc{G})$, as can be easily checked by hand.
\end{example}

\begin{prop}\label{prop:lff-equiv-conditions}
    Let $\mf X$ be a bounded $p$-adic formal scheme, and $(\mc{E},\varphi_\mc{E})$ a prismatic $F$-crystal on $\mf{X}$. Then, the following are equivalent:
    \begin{enumerate}
        \item the prismatic $F$-crystal $(\mc{E},\varphi_\mc{E})$ is lff,
    \item for any object $(A,I)$ of $\mf X_\smallprism$, the filtered module $(\phi^*\mc E(A,I),\Fil^\bullet_\mr{Nyg}(\phi^*\mc E)(A,I))$ is lff over the filtered ring $(A,\Fil_I^\bullet)$. 
\end{enumerate}
If $\mf X$ is a base formal $\mc O_K$-scheme, the above conditions are additionally equivalent to the following:
    \begin{enumerate}
        \setcounter{enumi}{2}
        \item there exists an open cover $\{\Spf(R_i)\}$ of $\mf{X}$ with each $R_i$ a base (formally framed) $\mc{O}_K$-algebra, such that if $\mf{M}_i\defeq \mc{E}(\mf{S}_{R_i},(E))$ then $(\phi^\ast\mf{M}_i,\Fil^\bullet_\mr{Nyg}(\phi^\ast\mf{M}_i))$ is lff over $(\mf{S}_{R_i},\Fil^\bullet_E)$,
        \item there exists a quasi-syntomic cover $\{\Spf(S_i)\to\mf{X}\}$ with each $S_i$ qrsp, such that if $\mc{M}_i\defeq \mc{E}(\Prism_{S_i},I_{S_i})$ then $(\phi^\ast\mc{M}_i,\Fil^\bullet_\mr{Nyg}(\phi^\ast\mc{M}_i))$ is lff over $(\Prism_{S_i},\Fil^\bullet_{I_{S_i}})$.
    \end{enumerate}
\end{prop}

\begin{proof}
    Clearly (2) implies (1). For the converse it suffices to show that the lff condition is flat local on an object of $\mf X_\smallprism$. This follows from the fact that for a $(p,I)$-adically faithfully flat map of prisms $(A,I)\to (B,IB)$ the natural map $\mr{Rees}(\Fil_I^\bullet A)\to \mr{Rees}(\Fil_{IB}^\bullet B)$ is $(p,I)$-adically faithfully flat. Indeed, the natural map $I^n\otimes_AB\to I^nB$ is an isomorphism as $IB$ is an invertible ideal, and hence $\mr{Rees}(\Fil_I^\bullet A)\otimes_AB\isomto \mr{Rees}(\Fil^\bullet_{IB}B)$. The final claims concerning base formal $\mc O_K$-schemes follows from \cite[Propositions 1.11 and 1.16]{IKY1}.
\end{proof}

We now come to the precise relationship between prismatic $F$-crystals and prismatic $F$-gauges (compare with \cite[Corollary 8.2.13]{Ito1}). In the following, let $\mc{G}$ be a smooth group $\Z_p$-scheme.

\begin{prop}\label{prop:F-gauge-lff-equiv}
    Let $\mf X$ be a quasi-syntomic $p$-adic formal scheme. Then the essential image of $\mathrm{R}_\mf{X}$ is contained in $\cat{Vect}^{\varphi,\mr{lff}}(\mf X_\smallprism)$.  
    If $\mf X$ is a base formal $\mc O_K$-scheme, then $\mathrm{R}_\mf{X}$  
    induces a bi-exact $\Z_p$-linear $\otimes$-equivalence 
    \begin{equation*}
        \mathrm{R}_\mf{X}\colon \cat{Vect}(\mf{X}^\mr{syn})\isomto\cat{Vect}^{\varphi,\mr{lff}}(\mf X_\smallprism).
    \end{equation*}
    In particular, $\mathrm{R}_\mf{X}$ induces an equivalence
    \begin{equation*}\GVect(\mf{X}^\syn)\isomto\GVect^{\varphi,\mr{lff}}(\mf{X}_\smallprism).\end{equation*}
\end{prop}
\begin{proof} Let $\Pi_\mf X\colon \cat{Vect}^\varphi(\mf X_\smallprism)\to \cat{Perf}(\mf X^\mr{syn})$ denote the functor from \cite[Theorem 2.31]{GuoLi}.\footnote{While this functor is only constructed in loc.\@ cit.\@ when $\mf{X}$ is smooth over $\mc{O}_K$, the construction goes through, mutatis mutandis, for $\mf{X}$ a base scheme using the cover from \cite[Lemma 1.15]{IKY1}.} 
    Let $(\mc E,\varphi_\mc{E})$ be an object of $\cat{Vect}^{\varphi}(\mf X_\smallprism)$. We first prove the following claim. 
    \begin{claim}\label{claim: lff and vect syn}
        The object $\wt{\mc E}\defeq\Pi_\mf X(\mc E,\varphi_\mc{E})$ is in $\cat{Vect}(\mf X^\mr{syn})$ if and only if $(\mc E,\varphi_\mc{E})$ is in $\cat{Vect}^{\varphi,\mr{lff}}(\mf X_\smallprism)$. 
    \end{claim}
    \begin{proof}[Proof of Claim \ref{claim: lff and vect syn}]
        For each perfectoid ring $S$ with a map $\Spf(S)\to \mf X$, consider the conditions:
    \begin{enumerate}
        \item the restriction $\wt{\mc E}|_{S^{\smallN}}$ is a vector bundle on $S^{\smallN}$,
        \item $\Fil_\mr{Nyg}^\bullet(\phi^*\mc E(\Ainf(S)))$ is lff over $(\Ainf(S),\Fil^\bullet_{\tilde{\xi}})$ (where $\tilde\xi$ is as in \cite[\S1.1.1]{IKY1}).     
    \end{enumerate}
        Then $\wt{\mc E}$ being in $\cat{Vect}(\mf X^\mr{syn})$ is equivalent to (1) being satisfied for any such $S$ by \cite[Proposition 1.11 and Lemma 1.15]{IKY1} and \cite[Remark 5.5.18]{BhattNotes}. On the other hand by Proposition \ref{prop:lff-equiv-conditions}, $(\mc E,\varphi_\mc{E})$ being lff is equivalent to (2) being satisfied for any such $S$. By Proposition \ref{prop:Rees-lff-proj}, conditions (1) and (2) are equivalent, as desired.
    \end{proof}

    This together with the construction of $\Pi_\mf X$ and \cite[Proposition 2.52]{GuoLi} gives an isomorphism 
    \begin{eqnarray}\label{eq: lff vect syn}
        (\Pi_\mf X\circ \mr R_\mf X)(\mc F)\simeq \mc F
    \end{eqnarray}
    for any $\mc{F}$ an object of $\cat{Vect}(\mf{X}^\syn)$. Thus, $\mr{R}_\mf{X}$ induces a functor $\cat{Vect}(\mf X^\syn)\to \cat{Vect}^{\varphi,\mr{lff}}(\mf X_\smallprism)$. 
    Again by Claim \ref{claim: lff and vect syn}, the functor $\Pi_\mf X$ induces a functor $\cat{Vect}^{\varphi,\mr{lff}}(\mf X_\smallprism)\to \cat{Vect}(\mf X^\syn)$. These functors are quasi-inverse to each other: for an object $\mc F$ in $\cat{Vect}(\mf X^\syn)$, we have the isomorphism (\ref{eq: lff vect syn}); on the other hand, for an object $(\mc E,\varphi_\mc{E})$ of $\cat{Vect}^{\varphi,\mr{lff}}(\mf X_\smallprism)$, we have a functorial isomorphism $(\mc E,\varphi_\mc{E})\simeq (R_\mf X\circ\Pi_\mf X) (\mc E,\varphi_\mc{E})$ by the constructions of $\mr{R}_\mf{X}$ and $\Pi_\mf{X}$. This proves the claim that $R_\mf{X}$ is an equivalence, and its bi-exactness is \cite[Proposition 2.17]{IKY3}. 
\end{proof}

\subsection{Prismatic \texorpdfstring{$F$-gauges}{F-gauges} with \texorpdfstring{$\mc G$-structure}{G-structure} of type \texorpdfstring{$\mu$}{mu}}\label{ss:F-gauges-of-type-mu}
In this subsection, we introduce the notion of prismatic $F$-gauge with $\mc G$-structure of type $\mu$.

\begin{conve}\label{conve: BGm} We make the following three conventions throughout this section.
\begin{enumerate}[leftmargin=.3in]
    \item The  
\emph{tautological line bundle} $\mc O_{B\bb{G}_m}\{1\}$ on $B\bG_{m}$ corresponds to 
the trivial line bundle on $\Spec(\Z)$ with the $\bb{G}_{m}$-action given by the \emph{inverse} of the natural scalar multiplication of $\mathbb{G}_{m}$. 
So $\mc O_{B\Gm}\{1\}$ corresponds to the graded line bundle $\Z(-1)$ concentrated in degree $-1$. \footnote{This convention agrees with \cite[Construction 2.2.1]{BhattNotes} and \cite{GMM}, but is opposite to that in \cite{Ito1}.} 
\item  We normalize the isomorphism $B\bb{G}_{m}\isomto \cat{Pic}$ of stacks via $\id_{B\Gm}\mapsto \mc O_{B\Gm}\{1\}$.
\footnote{Here $\cat{Pic}$ is the Picard stack, denoted by $\mc{P}ic_{\Z/\Z}$ in \stacks{0372}.} So $L$ in $\cat{Pic}(\Spec(S))$ corresponds to the natural map $\Spec(S)\isomto\Spec(\bigoplus_{i\in \Z}L^{\otimes i}t^i)/\Gm\to B\Gm$, where $t$ sits in degree $1$.
\item We identify $B\Gm$ and the the moduli stack of $\Gm$-torsors in the standard way.
\end{enumerate}
\end{conve}

\begin{nota}\label{nota:1.4} In addition to Notation \ref{nota:OK} we fix the following notation:
\begin{itemize}[leftmargin=.3in]
    \item $\mc{G}$ is a smooth affine group $\Z_p$-scheme,
    \item $\mu\colon \bb{G}_{m,W}\to \mc{G}_W$ is a $1$-bounded cocharacter (see \cite[Definition 6.3.1]{LauHigher}),
    \item for a line bundle $\mc L$ on a stack $\mathscr{X}$ together with a map $a\colon \msr X\to \Spf(W)$, we denote the induced $\mc G$-torsor on $\msr X$ via $\mu$ by $\mu_*(\mc L,a)$ (we omit the $a$ if it is clear from context or write $\mr{can.}$ to emphasize that it is the structure map which is canonical in the given context),
    \item for $b\colon \mc L\isomto\mc L'$ we denote the induced isomorphism by $\mu_\ast(b)\colon \mu_*(\mc L,a)\isomto\mu_*(\mc L',a)$. 
\end{itemize}
\end{nota}

Additionally, we recall the following standard notions of Breuil--Kisin twists on a bounded $p$-adic formal scheme $\mf{X}$. Denote by $\mc{O}_{\mf{X}^\smallprism}\{1\}$ the pullback of $\mc{O}_{\Z_p^\smallprism}\{1\}$ from \cite[Definition 4.9.4]{Drinfeld} (cf.\ \cite[Construction 2.2.11]{BhattLurieAbsolute} and \cite[Remark 5.1.19]{BhattNotes}) along $\mf{X}^\smallprism\to\Z_p^\smallprism$. Similarly, let $\mc{O}_{\mf{X}^\smallN}\{1\}$ be the pullback to $\mf{X}^\smallN$ of $\mc{O}_{\Z_p^\smallN}\{1\}\defeq\pi^*\mc O_\smallprism\{1\}\otimes t^*\mc O_{B\Gm}\{-1\}$ from \cite[Remark 5.5.15]{BhattNotes}.\footnote{The definition of $\mc O_{\smallN}\{1\}$ agrees with that in \cite{GMM}: $\A^1/\Gm$ is formed using the \emph{inverse} of the natural action so the degree of the Rees variable $t$ is $-1$ in \cite{GMM}, and $\mc O_{B\Gm}\{1\}$ corresponds to $\Z(1)$; while we use the natural action so here $\deg(t)=1$, and $\mc O_{B\Gm}\{1\}$ corresponds to $\Z(-1)$, as in \cite{BhattNotes}.}
 If $\mf{X}$ is clear from context we shorten $\mc{O}_{\mf{X}^\smallprism}\{1\}$ and $\mc{O}_{\mf{X}^\smallN}\{1\}$ to $\mc{O}_\smallprism\{1\}$ and $\mc{O}_\smallN\{1\}$, respectively.

Let notation be as in Notation \ref{nota:1.4}. We define 
\begin{equation*}
    \mc{P}_\mu\defeq \mu_\ast^{-1}(\mc{O}_{B\Gm}\{1\}),
\end{equation*}
a $\mc{G}$-torsor on $B\bb{G}_{m,W}$. Equivalently, $\mc{P}_\mu$ corresponds to the trivial $\mc G$-torsor on $\Spec(W)$ with the action of $\bb{G}_{m,W}=\Spec(W[z^{\pm 1}])$ given by left multiplication by $\mu(z)^{-1}$ in $\mc G(W[z^{\pm 1}])$. 
Denote by $\mu^{(-1)}$ the base change of $\mu$ along $\phi_W^{-1}\colon W\to W$ and by 
$\mc P_{\mu,{\smallN}}$ the pullback of $\mc P_{\mu^{(-1)}}$ by the map $W^{\smallN}\to B\bG_{m,W}$ that corresponds to the Breuil--Kisin twist $\mc O_{W^{\smallN}}\{1\}$.

\begin{defn}\label{defn: F-gauge of G mu structure}
    Fix an $n$ in $\mathbb N\cup\{\infty\}$. Let $\mf X$ be a bounded $p$-adic formal $W$-scheme. Write $\mf{X}^\syn_n\defeq \mf{X}^\syn\otimes^\bb{L}_{\Z_p}(\Z/p^n)$, and similarly for $\mf{X}^\smallN_n$.
    \begin{enumerate}[leftmargin=.3in]
        \item An \emph{$n$-truncated {prismatic $F$-gauge with $\mc G$-structure}} on $\mf X$ is a $\mc G$-torsor  $\mc F$ on the (derived) formal stack $\mf X^\mr{syn}_n$ over $\mathbb{Z}_p$ (giving $\mf X^\mr{syn}$ if $n=\infty$, whence we drop the prefix `$\infty$-truncated'). 
    \item An $n$-truncated prismatic $F$-gauge with $\mc G$-structure $\mc F$ is \emph{of type $\mu$} if the restriction to $\mf X^\smallN_n$ (again giving $\mf X^{\smallN}$ itself when $n=\infty$) is of type $\mu$, i.e., flat locally (equiv.\@ quasi-syntomically locally) on $\mf X$, the restriction $\mc F|_{\mf X^\smallN_n}$ is isomorphic to $\mc P_{\mu,{\smallN}}|_{\mf X^\smallN_n}$. 
    \end{enumerate}
Denote by $\cat{Tors}_{\mc{G},n}^\mu(\mf{X}^\mr{syn})$ the $\infty$-groupoid of $n$-truncated prismatic $F$-gauges with $\mc G$-structure of type $\mu$, a full $\infty$-subgroupoid of $\cat{Tors}_\mc{G}(\mf{X}^\syn_n)$. If $n=\infty$ we drop it from the notation.\footnote{When $\mf{X}$ is quasi-syntomic, observe that $\cat{Tors}_{\mc{G}}(\mf{X}^\syn)$ is a groupoid.} 
\end{defn}

\begin{rem}
    The Frobenius twist in the definition of $\mc{P}_{\mu,\smallN}$ is necessary for $\mc P_{\mu,{\smallN}}$ to have filtration of type $\mu$ (cf.\@  Proposition \ref{prop: F gauge type mu equals F crystal type mu}). 
\end{rem}

\begin{rem}\label{rem:GVect-tors-syn} Similarly to Remark \ref{rem:syn-twists}, if $\mf{X}$ is quasi-syntomic, one may identify $\cat{Tors}_\mc{G}(\mf{X}^\mr{syn})$ with $\GVect(\mf{X}^\syn)$, and consequently $\cat{Tors}^\mu_{\mc{G}}(\mf{X}^\mr{syn})$ may be identified with a full subcategory of $\GVect(\mf{X}^\syn)$ which we denote by $\GVect^\mu(\mf{X}^\syn)$. In particular, if $\mf{X}$ is quasi-syntomic, combining Proposition \ref{prop:syn-for-qsyn} and Corollary \ref{cor:Rees-ringed-stack-comp} we have natural identifications
\begin{equation*}
    \cat{Tors}_\mc{G}(\mf{X}^\syn)\simeq \twolim_{\Spf(R)\in\mf{X}_\qrsp}\cat{Tors}_\mc{G}({R}^\syn),\quad \cat{Tors}_\mc{G}(\mf{X}^\smallN)\simeq\twolim_{\Spf(R)\in\mf{X}_\qrsp}\cat{Tors}_\mc{G}(\wh{\mc{R}}(\Fil^\bullet_\mr{Nyg}(\Prism_R))).
\end{equation*}
\end{rem}

We first compare the notion of $n$-truncated prismatic $F$-gauges with $\mc{G}$-structure (of type $\mu$) to the notion of prismatic $\mc{G}$-torsors with $F$-structure of type $\mu$ as in \cite[Definition 3.12]{IKY1}, which we now recall using a slightly different presentation. 

Let $(A,\Fil^\bullet)$ be a filtered ring. Then, we have the natural closed immersion \begin{equation*} \iota\colon B\Gm\times\Spec(A/\Fil^1)\to \mc R(\Fil^\bullet),
\end{equation*}
defined by the surjection of graded rings 
\begin{eqnarray*}
    \mr{Rees}(\Fil^\bullet)\to \mr{Rees}(\Fil^\bullet_\mr{triv}(A/\Fil^1))=A/\Fil^1[t]\xrightarrow{t\mapsto 0} A/\Fil^1.
\end{eqnarray*} 

\begin{defn}
    We define the \emph{punctured Rees stack} $\mc R^\circ(\Fil^\bullet)$ for $(A,\Fil^\bullet)$ to be the open substack of the Rees stack $\mc R(\Fil^\bullet)$ given as the complement of the above closed embedding $\iota$. 
\end{defn}
In the situation that $\Fil^\bullet=\Fil_I^\bullet$ for some invertible ideal $I$, the punctured Rees stack $\mc R^\circ(\Fil^\bullet_I)$ is also obtained as follows. Consider the open embeddings
\begin{equation*}
    \Spec(A)\isomto\left[\Spec\left(\bigoplus_{i\in\Z}I^{-i}t^i\right)/\Gm\right]\to \mc R(\Fil^\bullet_I),\quad \Spec(A)\isomto \{t\ne 0\}\subseteq \mc{R}(\Fil^\bullet_I).
\end{equation*}  
These two open embeddings 
induce an isomorphism 
    \begin{eqnarray*}
        \Spec(A)\sqcup_{\Spec(A[\nicefrac{1}{I}])} \Spec(A)\isomto \mc R^\circ(\Fil_I^\bullet).
    \end{eqnarray*}
In particular, we obtain a $2$-functorial identification of $\cat{Tors}_\mc{G}(\mc{R}^\circ(\Fil^\bullet_I))$ with the category of triples $(Q,P,\psi\colon Q[\nicefrac{1}{I}]\isomto P[\nicefrac{1}{I}])$ where $Q$ and $P$ are $\mc G$-torsors on $A$ and $\psi$ is an isomorphism of $\mc G$-torsors on $A[\nicefrac{1}{I}]$, with the obvious notion of morphisms. We fix the direction of the isomorphism $\psi$ so that the restriction of such a triple to $\{t\ne 0\}$ gives the $\mc G$-torsor $P$.

\begin{defn}\label{defn:type-mu-crystals} Let $(A,I)$ be an object of $W_\smallprism$ (in particular, $A$ is naturally a $W$-algebra) and $\mu\colon \bG_{m,W}\to \mc G_W$ be a cocharacter. 
    \begin{enumerate}[leftmargin=.3in]
        \item We say a $\mc G$-torsor on $\mc R^\circ(\Fil_I^\bullet)$ presented as  $(Q,P,\psi\colon Q[\nicefrac{1}{I}]\isomto P[\nicefrac{1}{I}])$ is \emph{of type} $\mu$ if there exists a $(p,I)$-adically faithfully flat cover $A\to A'$ such that $IA'$ is principal, and there exists trivializations $\theta\colon \mc G_{A'}\isomto Q$ and $\theta'\colon \mc G_{A'}\isomto P$ such that the isomorphism $\theta'^{-1} \circ\psi\circ\theta$ is given by left multiplication by $\mu(d)$ for a generator $d$ of $IA'$.
        \item We say a $\mc G$-torsor $\mc P$ on $\mc R(\Fil_I^\bullet)$ is \emph{of type} $\mu$ if there exists a $(p,I)$-adically flat cover $A\to A'$ such that $\mc P$ restricted to $\mc R(\Fil_{IA'}^\bullet)$ is isomorphic to the \emph{$\mu$-typical $\mc G$-torsor}, i.e., the pullback of $\mc P_{\mu}$ along the natural map $\mc R(\Fil_{IA'}^\bullet)\to B\bG_{m,W}$. 
    \end{enumerate}
We denote these categories by $\cat{Tors}_{\mc{G}}^\mu(\mc{R}^\circ(\Fil^\bullet_I))$ and $\cat{Tors}_\mc{G}^\mu(\mc{R}(\Fil^\bullet_I))$, respectively.
\end{defn}

\begin{lem}[{cf.\ \cite[Proposition 6.6.3]{BhattNotes}}]\label{lem: type mu on punctured Rees stack}
    Let $(A,I)$ be a prism. 
    \begin{enumerate}[leftmargin=.3in]
        \item Restriction along $j\colon \mc R^\circ(\Fil_I^\bullet)\to \mc R(\Fil_I^\bullet)$ induces a fully faithful functor  
        \begin{eqnarray*}
            j^*\colon \cat{Tors}_\mc G(\mc R(\Fil_I^\bullet))\to \cat{Tors}_\mc G(\mc R^\circ(\Fil_I^\bullet)).
        \end{eqnarray*}
    \end{enumerate}
    Fix a $W$-structure on $(A,I)$ (i.e., let it be an object of $W_\smallprism$), and a cocharacter $\mu\colon \bG_{m,W}\to \mc G_W$. 
    \begin{enumerate}[leftmargin=.3in]\setcounter{enumi}{1}
        \item The functor from (1) induces an equivalence 
        \begin{eqnarray}\label{eq:equiv-type-mu}
            \cat{Tors}_{\mc G}^{\mu}(\mc R(\Fil_I^\bullet))\isomto\cat{Tors}_{\mc G}^{\mu}(\mc R^\circ(\Fil_I^\bullet)). 
        \end{eqnarray} 
    \end{enumerate}
\end{lem}
\begin{proof}
To prove Claim (1), we observe that by the Tannakian formalism it suffices to prove the claim for $\mc{G}=\GL_{n,\Z_p}$. Then it suffices to prove that 
\begin{equation*}
    j^\ast\colon \cat{Vect}(\mc{R}(\Fil^\bullet_I))\to\cat{Vect}(\mc{R}^\circ(\Fil^\bullet_I)),
\end{equation*}
is fully faithful. Using \stacks{06WT}, we are reduced to showing that pullback along
\begin{equation*}
    \{t\ne 0\}\cup \Spec\left(\bigoplus_{i\in\Z}I^{-i}t^i\right)\to \Spec(\mr{Rees}(\Fil^\bullet_I))
\end{equation*}
induces a fully faithful functor between the categories of $\bb{G}_m$-equivariant vector bundles on source and target. But, the conditions of \cite[Lemma 7.2.7]{CesnaviciusScholze} are satisfied, and so the claim follows by applying loc.\@ cit.\@ with $Y=\mc{H}om(\mc{V},\mc{V}')$ for vector bundles $\mc{V}$ and $\mc{V}'$ on $\mc{R}(\Fil^\bullet_I)$.

    For Claim (2), first observe that $\mc P_{\mu}|_{\mc R^\circ(\Fil^\bullet_I)}$ is of type $\mu$: using notation \ref{nota:1.4}, it is presented as
    \begin{equation*}(\mu_*(I,\mr{can.}),\mu_*(A,\mr{can.}),\psi\defeq\mu_*(\id\colon I\cdot A[\nicefrac{1}{I}]\isomto A[\nicefrac{1}{I}])).
    \end{equation*}
    Replacing $A$ by a Zariski cover on which $I$ becomes principal, we may pick a generator $d$ of $I$. Then, with the tautological trivialization $\theta\colon \mc G\isomto\mu_*(A,\mr{can.})$ and the trivialization 
    \begin{equation*}\theta'\colon \mc G=\mu_*(A,\mr{can.})\xrightarrow{\mu_*(d\cdot)}\mu_*(I,\mr{can.})
    \end{equation*}
    determined by $d$, the isomorphism ${\theta'}^{-1}\circ \psi\circ \theta$ is given by $\mu(d)$. 
    Thus, the functor $j^*$ does indeed carry objects of type $\mu$ to those of type $\mu$. We then only have to show that \eqref{eq:equiv-type-mu} is essentially surjective. Let $\mc P^\circ=(Q,P,\psi\colon Q[\nicefrac{1}{I}]\isomto P[\nicefrac{1}{I}])$ be an object of $\cat{Tors}_\mc G^\mu(\mc R^\circ(\Fil_I^\bullet))$. 
    By definition, there exists a $(p,I)$-adically faithfully flat cover $A\to A'$ such that $IA'$ is principal and there exist trivializations $\theta\colon \mc G_{A'}\isomto Q_{A'}$ and $\theta'\colon \mc G_{A'}\isomto P_{A'}$ such that $\theta'^{-1} \circ \psi\circ\theta$ is given by left multiplication of $\mu(d)$ for a generator $d$ of $IA'$. We consider the $\mu$-typical $\mc G$-torsor $\mc P'\defeq\mc P_{\mu,\mc R(\Fil_{IA'}^\bullet)}$ on $\mc R(\Fil^\bullet_{IA'})$. Note $\theta$ and $\theta'$ induce an isomorphism $\mc P'|_{\mc R^\circ(\Fil_{IA'}^\bullet)}\isomto \mc P^\circ_{A'}$. Then the descent datum 
    \begin{equation*}
        \sigma\colon \mc P^\circ_{A'}\otimes_{A',\mr{pr}_1^*}(A'\otimes_AA')\isomto \mc P^\circ_{A'}\otimes_{A',\mr{pr}_2^*}(A'\otimes_AA')
    \end{equation*} 
    for $\mc P^\circ$ along $A\to A'$ induces, by the full faithfulness from Claim (1), a descent datum on $\mc P'$, which then gives a $\mc G$-torsor $\mc{P}$ on $\mc R(\Fil_I^\bullet)$ such that $j^\ast\mc{P}\simeq \mc{P}^\circ$ as desired.
\end{proof}

To relate the above consideration to the type $\mu$ condition for prismatic $F$-crystals, we consider the two $A$-lattices associated to an object of $\cat{Tors}^\varphi_\mc{G}(A,I)$. More precisely, we consider the functor 
\begin{eqnarray*}
    \mr{Lat}_\mc G\colon \cat{Tors}_\mc{G}^\varphi(A,I)\to \cat{Tors}_\mc G(\mc R^\circ(\Fil_I^\bullet))
\end{eqnarray*} 
by sending $(\mc{P},\varphi_\mc{P})$ to $(\phi^*\mc{P},\mc{P},\varphi_{\mc{P}}\colon \phi^*\mc{P}[\nicefrac{1}{I}]\isomto \mc{P}[\nicefrac{1}{I}])$. 

\begin{defn}[{cf.\@ \cite[Definition 3.12]{IKY1}}]\label{defn: type mu F crystals}
    Let $\mf X$ be a quasi-syntomic $p$-adic formal scheme and $(\mc P,\varphi_{\mc{P}})$ be an object of $\cat{Tors}^\varphi_\mc{G}(\mf X_\smallprism)$. We say that $(\mc P,\varphi_{\mc{P}})$ is \emph{of type }$\mu$ if for any object $(A,I)$ of the site $\mf X_\smallprism$, the object $\mr{Lat}_\mc G(\mc P,\varphi_{\mc{P}})$ of $\cat{Tors}_\mc{G}(\mc R^\circ(\Fil_I^\bullet))$ is of type $\mu$. Denote the full subcategory of $\cat{Tors}_\mc{G}^\varphi(\mf{X}_\smallprism)$ of objects of type $\mu$ by $\cat{Tors}_\mc{G}^{\varphi,\mu}(\mf{X}_\smallprism)$. 
\end{defn}

Let $\mf{X}$ be a quasi-syntomic $p$-adic formal scheme over $W$ and consider the forgetful functor
\begin{equation*}
    \mathrm{R}_\mf{X}\colon \cat{Tors}_\mc{G}(\mf{X}^\mr{syn})\to\cat{Tors}_\mc{G}^\varphi(\mf{X}_\smallprism),
\end{equation*}
from Construction \ref{const:forgetful-functor}, which can also be interpreted in terms of $\mc{G}$-objects as in Remark \ref{rem:GVect-tors-syn}. The result below says this induces an equivalence on objects of type $\mu$ if $\mf{X}$ is sufficiently regular.

\begin{prop}\label{prop: F gauge type mu equals F crystal type mu}
    Let $\mf X$ be a quasi-syntomic $p$-adic formal scheme over $W$. 
    \begin{enumerate}[leftmargin=.3in]
        \item Let $\mc P$ be an object of $\cat{Tors}_\mc{G}(\mf{X}^\syn)$ of type $\mu$. Then the associated object $(\mc P_\smallprism,\varphi_{\mc{P}_\smallprism})=\mathrm{R}_\mf{X}(\mc{P})$ of $\cat{Tors}_\mc{G}^\varphi(\mf{X}_\smallprism)$ is also of type $\mu$. 
        \item Assume that $\mf X$ is either (a) $\Spf(R)$ for a perfectoid ring $R$ or (b) a base formal scheme over $\Spf(W)$. Then the forgetful functor induces an equivalence of categories 
        \begin{eqnarray*}
            \mr{R}_\mf{X}\colon \cat{Tors}_\mc{G}^\mu(\mf X^\syn)\isomto \cat{Tors}_\mc{G}^{\varphi,\mu}(\mf X_\smallprism).
        \end{eqnarray*}
    \end{enumerate}
\end{prop}

\begin{proof} 
For Claim (1), by the definitions of the type $\mu$ conditions, 
we may assume that $\mf X=\Spf(R)$ for a qrsp ring $R$. We consider the following commutative diagram of categories
    \bx{
        \cat{Tors}_\mc{G}(R^\syn) \ar[r]^-{\mc P\mapsto \mc P_\smallprism}\ar[d]_-{\mc{P}\mapsto \phi^*\mc{P}_\smallN}
        & \cat{Tors}_\mc{G}^\varphi(\Prism_R,I_R)\ar[d]^-{\mr{Lat}_\mc G}
        \\ \cat{Tors}_\mc{G}(\mc R(\Fil_{I_R}^\bullet)) \ar[r]_-{j^*}
        & \cat{Tors}_\mc{G}(\mc R^\circ(\Fil_{I_R}^\bullet)),
    }\ex
    where the left vertical arrow is the natural pullback functor along the composition
    \begin{equation*}
        \mc{R}(\Fil^\bullet_{I_R})\xrightarrow{\phi}\mc R(\Fil_\mr{Nyg}^\bullet)= R^\smallN\to R^\syn. 
    \end{equation*}
    Recall that $\mc P_\smallprism$ being of type $\mu$ means that $\mr{Lat}_\mc G(\mc P_\smallprism)$ is of type $\mu$, which, by the commutative diagram, happens if and only if $j^*\phi^*\mc P_{\smallN}$ is of type $\mu$. Now, by Lemma \ref{lem: type mu on punctured Rees stack}, this is equivalent to $\phi^*\mc P_{\smallN}$ being of type $\mu$. On the other hand, $\mc P$ being of type $\mu$ means that $\mc P_{\smallN}$ is of type $\mu$. 
    As scalar extension preserves the type $\mu$ condition, Claim (1) follows. 

    To prove Claim (2), observe that as in Example \ref{ex:mu-lff} any object of $\cat{Tors}_\mc{G}^{\varphi,\mu}(\mf{X}_\smallprism)$ corresponds to an object of $\GVect^{\varphi,\mr{lff}}(\mf{X}_\smallprism)$, and so, by Proposition \ref{prop:F-gauge-lff-equiv} and its proof, we only have to show that the converse of Claim (1) also holds under the assumption of Claim (2). We may now assume that $\mf X=\Spf(R)$ for a perfectoid ring $R$ (in case (b), take a perfectoid cover as in \cite[Lemma 1.15]{IKY1}). Then the assertion follows since $\phi$ is now an isomorphism. 
\end{proof}

\subsection{Relationship to the theory of displays}

To utilize the works \cite{Ito1} and \cite{Ito2}, we need to compare the notion of prismatic $F$-gauge with $\mc G$-structure of type $\mu$ to that of prismatic $\mc G$-$F$-gauge of type $\mu$ introduced by Ito in \cite[Definition 8.2.5]{Ito1}. We maintain the notation from Notation \ref{nota:1.4} throughout.

\subsubsection{Prismatic \texorpdfstring{$\mc{G}\text{-}\mu$}{Gmu}-displays}
Let $R$ be a quasi-syntomic $W$-algebra. Following \cite{Ito1}, for an object $(A,I)$ of $R_\smallprism$, and a generator $d$ of $I$, we define the sheaf
\begin{equation*}
    \mc{G}_{\mu,(A,I)}\colon \Spec(A)_\et\to\cat{Grp},\qquad B\mapsto \left\{g\in\mc{G}(B):\mu(d)g\mu(d)^{-1}\in \mc{G}(B)\subseteq \mc{G}(B[\nicefrac{1}{d}])\right\},
\end{equation*}
which does not depend on $d$. For $d$ generating $I$, define an action of $\mc{G}_{\mu,(A,I)}$ on $\mc{G}_d\defeq \mc{G}$ by
\begin{equation*}
    \mc{G}_d\times \mc{G}_{\mu,(A,I)}\to\mc{G}_d,\qquad (x,g)\mapsto g^{-1}x\phi(\mu(d)g\mu(d)^{-1}).
\end{equation*}
For another generator $d'$ of $I$ there exists a unique unit $u$ of $A$ with $d=ud'$ and the morphism $\mc{G}_d\to\mc{G}_{d'}$ given by sending $x$ to $x\phi(\mu(u))$ is a $\mc{G}_{\mu,(A,I)}$-equivariant isomorphism of sheaves. Thus, $\mc{G}_{\smallprism,(A,I)}\defeq \varprojlim \mc{G}_d$, is a sheaf of sets carrying a canonical action of $\mc{G}_{\mu,(A,I)}$. 

As in \cite{Ito1}, a \emph{$\mc{G}\text{-}\mu$-display} on $(A,I)$ is a pair $(\mc{Q}_{(A,I)},\alpha_{\mc{Q}_{(A,I)}})$ where $\mc{Q}_{(A,I)}$ is a $\mc{G}_{\mu,(A,I)}$-torsor and $\alpha_{\mc{Q}_{(A,I)}}\colon \mc{Q}\to \mc{G}_{\smallprism,(A,I)}$ is a $\mc{G}_{\mu,(A,I)}$-equivariant map of sheaves. There is an evident notion of morphism of $\Gmu$-displays on $(A,I)$, and we denote by $\GDisp_\mu(A,I)$ the category of prismatic $\Gmu$-displays on $(A,I)$. For a morphism $(A,I)\to (B,J)$, where both $I$ and $J$ are principal, there is an obvious pullback morphism $\GDisp_\mu(A,I)\to \GDisp_\mu(B,J)$ and Ito defines a \emph{prismatic $\Gmu$-display} on $R$ to be an object $(\mc{Q},\alpha_\mc{Q})$  of the category 
\begin{equation*}
    \GDisp_\mu(R_\smallprism)\defeq \twolim_{\scriptscriptstyle (A,I)\in R_\smallprism}\GDisp_\mu(A,I),
\end{equation*}
which makes sense as every object $(A,I)$ of $R_\smallprism$ has a cover $(A,I)\to (B,J)$ where $J$ is principal.

This definition extends, mutatis mutandis, to a quasi-syntomic formal scheme, and we denote the resulting category by $\GDisp_\mu(\mf{X}_\smallprism)$.

\subsubsection{Relationship to prismatic \texorpdfstring{$\mc{G}$}{G}-torsors with \texorpdfstring{$F$}{F}-structure}\label{p:G-mu-displays-and-prismatic-G-torsors-with-F-structure} We now show that $\Gmu$-displays are precisely the prismatic $\mc{G}$-torsors with $F$-structure bounded by $\mu$. More precisely, we construct
\begin{equation*}
    \GDisp_\mu(A,(d))\isomto \cat{Tors}_{\mc{G}}^{\varphi,\mu}(A,(d)),
\end{equation*}
functorial in an object $(A,(d))$ of $R_\smallprism$.

We first construct an equivalence between the category of banal (see \cite[\S5.1]{Ito1}) $\Gmu$-displays on $(A,(d))$ and the full-subcategory of $\cat{Tors}_{\mc{G}}^{\varphi,\mu}(A,(d))$ consisting of those $(\mc{A},\varphi_\mc{A})$ with $\mc{A}$ trivializable. Let $(\mc G_{\mu,(A,I)},X\colon \mc G_{\mu,(A,I)}\to \mc G_{\smallprism,(A,I)})$ be a banal $\Gmu$-display on $(A,I)$. 
Choose a generator $d$ of $I$, and write $X_d$ in $\mc G(A)$ for the $d$-component of $X(1)$. Define $\mc{P}_{X,d}$ to be the trivial $\mc{G}_A$-torsor and $\varphi_{\mc{P}_{X,d}}$ to be the composition 
\be
\phi^*\mc{G}_{A[\nicefrac{1}{d}]}=\mc{G}_{A[\nicefrac{1}{d}]}\isomto \mc{G}_{A[\nicefrac{1}{d}]}\isomto \mc{G}_{A[\nicefrac{1}{d}]},
\ee
where the first map is left multiplication by $X_d$ and the second map is left multiplication by $\mu(d)$. For another generator $d'$ of $I$ with $d=ud'$ with $u\in A$ a unit, the left multiplication by $\mu(u)$ defines an isomorphism $(\mc{P}_{X,d'},\varphi_{\mc{P}_{X,d'}})\isomto (\mc{P}_{X,d},\varphi_{\mc{P}_X,d})$. Define $(\mc{P}_X,\varphi_{\mc{P}_X})$ to be the inverse limit $\varprojlim_{d}(\mc{P}_{X,d},\varphi_{\mc{P}_{X,d}})$ with transition maps given by $\mu(u)$.

The pair $(\mc{P}_X,\varphi_{\mc{P}_X})$ is then seen to be an object of $\cat{Tors}_{\mc{G}}^{\varphi,\mu}(A,(d))$ whose underlying $\mc{G}$-torsor is trivalizable. This defines a functor as for an element $g$ of  $\mc G_{\mu,(A,(d))}$, we have an induced morphism $(\mu(d)g\mu(d)^{-1})_d\colon \mc{P}_{X\cdot g}\to \mc{P}_{X}$, which is functorial and preserves the Frobenius structures by construction. This functor is clearly fully faithful.

Stackifying this association gives us a functor 
\begin{equation*}
    \GDisp_\mu(A,(d))\to \cat{Tors}_{\mc{G}}^{\varphi,\mu}(A,(d)), 
\end{equation*}
which is compatible in $\mc{G}$ and $(A,(d))$. So, for a quasi-syntomic $W$-algebra $R$, we obtain a functor 
\begin{equation}\label{eq:display-torsor}
\GDisp_\mu(R_\smallprism)\to\cat{Tors}_{\mc{G}}^{\varphi,\mu}(R_\smallprism),
\end{equation}
functorial in $\mc{G}$ and $R$.

\begin{prop}\label{prop:display-torsor-equiv} The functor \eqref{eq:display-torsor} defines an equivalence of categories
\begin{equation*}\GDisp_\mu(R_\smallprism)\isomto\cat{Tors}_{\mc{G}}^{\varphi,\mu}(R_\smallprism)
\end{equation*}
functorial in $\mc{G}$ and $R$.
\end{prop}
\begin{proof} As we have already observed, this functor is fully faithful, it remains to show that it is essentially surjective. As both the source and target are stacks on $\mf{X}_\smallprism$, this fully faithfulness allows us to reduce ourselves to showing that the functor is essentially surjective on banal objects over some $(A,(d))$. Let $(\mc{P},\varphi_{\mc{P}})$ be an object of $\cat{Tors}_{\mc G}^{\varphi,\mu}(A,(d))$ 
with $\mc{P}$ trivializable. Then by definition, $\varphi_{\mc{P}}$ is defined by $Y\mu(d)X$ for some $X$ and $Y$ in $\mc G(A)$. But left multiplication by $Y$ then defines an isomorphism $\mc{P}_{X\phi(Y),d}\to \mc{P}$ in $\cat{Tors}_{\mc G}^{\varphi,\mu}(A,(d))$, 
from where the claim follows.
\end{proof}

\subsubsection{\texorpdfstring{$\mc{G}\text{-}F$}{G-F}-gauges of type \texorpdfstring{$\mu$}{mu}}\label{sss: G F gauges}

As in \cite[\S4.1]{Ito1}, set $A_\mc G$ to be the ring $\mc O(\mc G_W)$ and denote by $A_\mc G=\bigoplus_{i\in \Z}A_{\mc G,i}$ the weight decomposition with respect to the $\mu$-conjugation $\mu(z)g\mu(z)^{-1}$.\footnote{We take the opposite convention from \cite{Ito1}: the weight decomposition in loc.\ cit.\ is taken with respect to the \emph{inverse} $\mu$-conjugation. We do so as the variable $t$ is declared to have degree $-1$ in \cite[Definition 8.1.1]{Ito1}.} As $\mc{G}$ is defined over $\Z_p$ there is a natural identification $\phi^\ast A_\mc{G}=A_\mc{G}$, and so we obtain a decomposition
\begin{equation*}
    A_\mc{G}=\bigoplus_{i\in\Z}A_{\mc{G},i}^{(-1)},
\end{equation*}
where, by definition, $A_{\mc{G},i}^{(-1)}$ is the base change of $A_{\mc{G},i}$ along $\phi^{-1}\colon W\to W$.

For a qrsp ring $R$, set
\begin{equation*}
    \mc G_{\mu,{\smallN}}(R)\defeq \left\{g\in \mc G(\Prism_R): g^*\colon A_\mc G\to \Prism_R\text{ satisfies } g^\ast(A_{\mc G,i}^{(-1)})\subseteq \Fil^{-i}_\mr{Nyg}\Prism_R \text{ for all }i\in\Z\right\}. 
\end{equation*}
For a generator $d$ of $I_R\subset \Prism_R$, let $\mc G(\Prism_R)_{{\smallN},d}$ denote the group $\mc G(\Prism_R)$ with the right $\mc G_{\mu,{\smallN}}(R)$-action given by $X\cdot g\defeq g^{-1}X\sigma_{\mu,{\smallN},d}(g)$, where $\sigma_{\mu,{\smallN},d}(g)\defeq \mu(d)\phi(g)\mu(d)^{-1}$. We define $\mc G(\Prism_R)_{{\smallN}}$ to be the inverse limit $\varprojlim_{d}\mc G(\Prism_R)_{{\smallN},d}$, where $d$ runs over the set of generators of $I_R$, and when $d=ud'$ for a unit $u\in\Prism_R$ the transition map $\mc G(\Prism_R)_{{\smallN},d}\to \mc G(\Prism_R)_{{\smallN},d'}$ is given by $X\mapsto X\mu(u)$. 

Let $\mf X$ be a quasi-syntomic $p$-adic formal scheme. By \cite[Lemma 8.2.4]{Ito1}, the groups $\mc{G}_{\mu,\smallN}(R)$ form a sheaf of groups $\mc G_{\mu,{\smallN}}$ on $\mf X_\qrsp$ together with an action on the sheaf of sets $\mc G_{\smallprism,{\smallN}}$ defined via the association $R\mapsto \mc G(\Prism_R)_{{\smallN}}$.

\begin{defn}[{\cite[Definition 8.2.5]{Ito1}}] Let $\mf{X}$ be a quasi-syntomic $p$-adic formal scheme. A \emph{prismatic $\mc G$-$F$-gauge of type $\mu$} on $\mf X$ is a pair $(\msr Q,\alpha_\msr Q)$ consisting of a $\mc G_{\mu,{\smallN}}$-torsor $\msr Q$ on $\mf{X}_\qrsp$ and a $\mc G_{\mu,{\smallN}}$-equivariant map of sheaves $\alpha_\msr Q\colon \msr Q\to \mc G_{\smallprism,{\smallN}}$.
\end{defn}

Prismatic $\mc G$-$F$-gauges of type $\mu$ form a groupoid $\mc G\hyphen F\hyphen\cat{Gauge}_\mu(\mf{X})$. There is a natural functor
\begin{equation*}
    \mathrm{R}_\mf{X}^\mr{disp}\colon \mc{G}\text{-}F\text{-}\cat{Gauge}_\mu(\mf{X})\to\GDisp_\mu(\mf{X}),
\end{equation*}
(see \cite[Proposition 8.2.11]{Ito1}).

\begin{rem}\label{lem: the group G mu}
There exists a canonical isomorphism 
    \begin{eqnarray*}
        \left[\Spec\left(\bigoplus_{i\in\Z}A_{\mc G,i}\right)/\bb{G}_{m,W}\right]\isomto \underline{\Aut}(\mc P_\mu)
    \end{eqnarray*}
of group stacks over $B\bG_{m,W}$. 
    In fact, the action of $\bb{G}_{m,W}$ on $\mc{G}_W$ by $\mu$-conjugation is identified with the natural action of $\bb{G}_{m,W}$ on the group scheme $\underline{\Aut}(\mc{P}_\mu|_{\Spf(W)})$. 
    In particular, for a qrsp ring $R$, we have a canonical identification 
    \begin{eqnarray*}
        \mc G_{\mu,{\smallN}}(R)\isomto \underline{\Aut}(\mc P_{\mu^{(-1)}})(R^{\smallN})\defeq \mr{Map}_{B\bb{G}_{m,W}}(R^{\smallN},\underline{\Aut}(\mc P_{\mu^{(-1)}})),
    \end{eqnarray*}
    where $R^{\smallN}$ is regarded as a stack over $B\bb{G}_{m,W}$ by the Rees structure map $t\colon R^{\smallN}\to B\bb{G}_{m,W}$. 
\end{rem}

\subsubsection{Comparison to \texorpdfstring{$\mc G$}{G}-torsors on the syntomification}

We aim to prove the following. 

\begin{prop}\label{prop: BT equals Ito's}
    Let $\mf X$ be a quasi-syntomic $p$-adic formal scheme. Then, there exists an equivalence of categories 
    \begin{eqnarray*}
        \cat{Tors}^\mu_\mc G(\mf X^\syn)\isomto\mc G\hyphen F\hyphen\cat{Gauge}_\mu(\mf X),
    \end{eqnarray*}
    $2$-bi-functorial in $\mf X$ and $\mc{G}$. 
\end{prop}

To construct this equivalence we give a different notion of prismatic $\mc G$-$F$-gauge of type $\mu$ (see Definition \ref{defn: choice free displays}) and show it is equivalent to that in \cite{Ito1} (see Proposition \ref{prop: choice free display equals display}). 

To this end, for each qrsp ring $R$ we let $L_R$ denote the invertible $\Prism_R$-module that corresponds to $\mc O_\smallprism\{1\}^{\vee}$ via the canonical isomorphism $R^\smallprism\simeq\Spf(\Prism_R)$ from Proposition \ref{prop:syn-for-qsyn}, which comes with a canonical isomorphism $\phi^*L_R\isomto I_R\otimes_{\Prism_R}L_R$.

Consider the perfectoid ring $\mc O\defeq\Z_p\langle \mu_{p^\infty}\rangle$. Then $L\defeq L_{\mc O}$ is generated by $[\varepsilon]-1$,  
where $\varepsilon$ in $\mc O^\flat$ is as in \cite[Notation 1.1]{IKY1}. Then, for a qrsp ring $R$ that admits a map from $\mc O$, the expression $\Prism_R[\nicefrac{1}{L}]$ makes sense.

\begin{defn}
    For a qrsp $\mc{O}$-algebra $R$, we define the following groups:
    \begin{equation*}
        \begin{aligned}\mc G'_{\mu,{\smallN}}(R) &\defeq \left\{g\in \mc G(\Prism_R[\nicefrac{1}{L}]):  g^*(A_{\mc G,i}^{(-1)})\subseteq \Fil^{-i}_\mr{Nyg}\cdot L^{i} \right\}, \\
         \mc G'_{\mu,\smallprism}(R) & \defeq \left\{g\in \mc G(\Prism_R[\nicefrac{1}{L}]): g^*(A_{\mc G,i}^{(-1)})\subseteq L^{i}\right\},
         \\ \mc G''_{\mu,\smallprism}(R)&\defeq\left\{g\in \mc G(\smallprism_R[\nicefrac{1}{L}]): g^*(A_{\mc G,i})\subseteq L^{i}\right\}.
         \end{aligned}
    \end{equation*}
\end{defn}

\begin{lem}\label{lem: global section of G mu}
    There is a canonical commutative diagram 
    \bx{
        \Aut(\tau^*\mc P_{\mu,{\smallN}})\ar[d]^-\sim
        & \Aut(\mc P_{\mu,{\smallN}}) \ar[r]^-{\sigma^*}\ar[l]_-{\tau^*}\ar[d]^-\sim
        & \Aut(\sigma^*\mc P_{\mu,{\smallN}}) \ar[d]^-\sim
        \\ \mc G'_{\mu,\smallprism}(R)
        & \mc G'_{\mu,{\smallN}}(R)\ar[r]^\phi\ar[l]
        &\mc G''_{\mu,\smallprism}(R).
    }\ex
\end{lem}

\begin{proof}
    We first note that, by Remark \ref{lem: the group G mu}, the group $\Aut(\mc P_{\mu,{\smallN}})$ is canonically identified with the group $\mr{Map}_{B\bG_{m,W}}(R^{\smallN}\{1\},\Spec(\bigoplus_{i\in\Z}A_{\mc G,i}^{(-1)}))$, 
    where $R^{\smallN}\{1\}$ denotes the stack $R^{\smallN}$ with the $B\bG_{m,W}$-structure $R^{\smallN}\to B\bG_{m,W}$ given by the line bundle $\mc O_{\smallN}\{1\}$ (with the canonical $W$-structure). 
    This structure map corresponds to the $\Gm$-torsor $\Spec(\bigoplus_{i\in \Z}\Fil^{-i}_\mr{Nyg}\cdot L^{i})$ over $\mc R(\Fil^\bullet_\mr{Nyg})$ via Convention \ref{conve: BGm} together with Proposition \ref{prop:rees-equiv}. 
    Thus, this group is identified with the group of $\Gm$-equivariant $W$-maps $\Spec(\bigoplus_{i\in \Z}\Fil_\mr{Nyg}^{-i}\cdot L^{i})\to \Spec(\bigoplus_{i\in\Z}A_{\mc G,i}^{(-1)})$,\footnote{As $R^{\smallN}$ is identified with the \emph{completed} Rees stack $\wh{\mc R}(\Fil_\mr{Nyg}^\bullet)$, the group $\mc G\{\mu\}(R^{\smallN})$ is, a priori, identified with the group of compatible systems of $\bb{G}_m$ equivariant maps $\Spec(\bigoplus_{i\in \Z}(\Fil_\mr{Nyg}^{-i}\cdot L^{i})/(p,I)^n)\to \Spec(\bigoplus_{i\in\Z}A_{\mc G,i}^{(-1)})$. But, since each $\Fil^{-i}_\mr{Nyg}\cdot L^{i}$ is complete (see Lemma \ref{lem:Nygaard-closed}), giving such a system is equivalent to giving a $\Gm$-equivariant map as claimed.}  i.e., the group of graded $W$-algebra homomorphisms $\bigoplus_{i\in\Z}A_{\mc G,i}^{(-1)}\to\bigoplus_{i\in \Z}(\Fil_\mr{Nyg}^{-i}\cdot L^{i})$, which is identified with the group $\mc G'_{\mu,{\smallN}}(R)$. 
    The other isomorphisms are obtained similarly. 
\end{proof}

We now consider the sheaf of sets $\mathcal A$ on $R_\qrsp$ defined by the association
\begin{equation*}R'\mapsto {\mr{Isom}}((\sigma^*\mc P_{\mu,{\smallN}})|_{{R'}^\smallprism},(\tau^*\mc P_{\mu,{\smallN}})|_{{R'}^\smallprism}).
\end{equation*}
It admits a natural right action of $\mc G'_{\mu,{\smallN}}$ defined by the rule $X'\cdot g'\defeq {g'}^{-1}X'\phi(g')$ where $g'$ belongs to $\mc G'_{\mu,{\smallN}}(R)$ and $X'$ is in $\mc A(R)$.

\begin{defn}\label{defn: choice free displays}
    Let $R$ be a qrsp ring. By a \emph{choice-free prismatic $\mc G$-$F$-gauge} on $R$ of type $\mu$ we mean a pair $\mathscr{Q}=(\mathscr{Q},\alpha_\mathscr{Q})$ consisting of a $\mc G'_{\mu,{\smallN}}$-torsor $\mathscr{Q}$ on $R_\mr{qrsp}$ and a $\mc G'_{\mu,{\smallN}}$-equivariant map of sheaves $\alpha_\mathscr{Q}\colon \msr Q\to \mc A$. 
\end{defn}

Set $\mc G\hyphen F\hyphen\cat{Gauge}'_\mu(R)$ to be the groupoid of choice-free prismatic $\mc G$-$F$-gauges on $R$ of type $\mu$. 
In the following we let $\mu'$ denote the twist $\mu^{(-1)}$. 

\begin{prop}\label{prop: choice free display equals display}
    Let $R$ be a qrsp $\mc{O}$-algebra and choose a generator $\pi$ in $L$. Let $d_\pi$ denote the element $\phi(\pi)/\pi$ in $\Prism_R$. 
    \begin{enumerate}
        \item The following diagram commutes 
        \bx{
        \mc G(\Prism_R)\ar[d]^-{\mu'(\pi)\text{-}\mr{conj.}}
        & \mc G_{\mu,{\smallN}}(R) \ar[r]^{\sigma_{\mu,{\smallN},d_\pi}}\ar[l]_-{\mr{can}}\ar[d]^-{\mu'(\pi)\text{-}\mr{conj.}}
        & \mc G(\Prism_R) \ar[d]^-{\mu(\pi)\text{-}\mr{conj.}}
        \\ \mc G'_{\mu,\smallprism}(R)
        & \mc G'_{\mu,{\smallN}}(R)\ar[r]^\phi\ar[l]
        &\mc G''_{\mu,\smallprism}(R),
        }\ex
        where $\sigma_{\mu,{\smallN},d_\pi}$ is the map from \S\ref{sss: G F gauges}. 
        \item The choice of $\pi$ induces an \emph{isomorphism} of categories 
        \begin{eqnarray*}
            \mc G\hyphen F\hyphen\cat{Gauge}(R)
            \isomto
            \mc G\hyphen F\hyphen\cat{Gauge}'(R).
        \end{eqnarray*}
    \end{enumerate}
\end{prop}

\begin{proof}
    Claim (1) follows from the definition of the weight decomposition. Indeed, the natural map $\mu'(\pi)g\mu'(\pi)^{-1}\colon A_\mc G\to \Prism[\nicefrac{1}{L}]$ sends $a$ in $A_{\mc G,i}^{{(-1)}}$ to $\pi^{i}g^*(a)$. 

    To see Claim (2), we identify $\mc G(\Prism_R)_{\smallN}\defeq\varprojlim_{d}\mc G(\Prism_R)_{{\smallN},d}$ with $\mc A(R)$ as follows. We have the canonical projection $\mc G(\Prism_R)_{\smallN}\isomto \mc G(\Prism)_{{\smallN},d_\pi}$, so it suffices to identify $\mc G(\Prism_R)_{{\smallN},d_\pi}$ with $\mc A(R)$. 
    Let $\theta_{d_\pi}\colon \mc G(\Prism_R)_{{\smallN},d_\pi}=\mc G(\Prism_R)\isomto \mc A(R)$ 
    be the map sending an element $X$ of $\mc G(\Prism_R)$ to the isomorphism $X'\colon \sigma^*\mc P_{\mu,{\smallN}}\isomto \tau^*\mc P_{\mu,{\smallN}}$ defined as composite
    \begin{eqnarray*}
        \sigma^*\mc P_{\mu,{\smallN}}=\mu_*(\mc O_\smallprism\{1\})\xrightarrow{\mu_*(\pi^{-1}\cdot)}\mu_*(\mc O_\smallprism)=\mc G_{\Prism_R}\xrightarrow{X\cdot}\mc G_{\Prism_R}=\mu'_*(\mc O_\smallprism)\xrightarrow{\mu'_*(\pi\cdot)}\mu'_*(\mc O_\smallprism\{1\})=\tau^*\mc P_{\mu,{\smallN}}.
    \end{eqnarray*}
    By Claim (1), this bijection is equivariant for the action of $\mc G_{\mu,{\smallN}}(R)$ on the left-hand side and that of $\mc G'_{\mu,{\smallN}}(R)$ on the right-hand side via the identification $\mc G'_{\mu,{\smallN}}(R)\isomto\mc G_{\mu,{\smallN}}(R)$ from Claim (1). 
\end{proof}    

\begin{proof}[Proof of Proposition \ref{prop: BT equals Ito's}]
    By Proposition \ref{prop: choice free display equals display} and quasi-syntomic descent, it suffices to construct an equivalence $\cat{Tors}_\mc G^\mu(R^\syn)\isomto \mc G\hyphen F\hyphen\cat{Gauge}'_\mu(R)$ for qrsp rings $R$ in a functorial way. 
    
    Suppose first that $\mc P$ is an object of $\cat{Tors}_\mc G^\mu(R^\syn)$ presented via an isomorphism $\varphi_\mc{P}\colon \sigma^\ast \mc{P}_\smallN\to \tau^\ast\mc{P}_\smallN$ for $\mc{P}_\smallN$ a $\mc{G}$-torsor on $R^\smallN$. Say that $\mc{P}$ is \emph{banal} if $\mc P_{\smallN}$ is isomorphic to $\mc P_{\mu,\smallN}$. Suppose $\mc{P}$ is banal and choose an isomorphism $\theta\colon \mc{P}_{\mu,\smallN}\isomto \mc{P}_\smallN$. Consider the element $X$ in $\mc{A}(R)$ that defines $\varphi_\mc P$ via $\theta$. We then define a choice-free prismatic $\mc G$-$F$-gauge $\mathscr{Q}=(\msr Q,\alpha_\msr Q)$ of type $\mu$ as follows. Define $\msr Q$ to be the trivial $\mc G'_{\mu,{\smallN}}$-torsor on $R_\mr{qrsp}$. Further define a $\mc G'_{\mu,{\smallN}}$-equivariant map $\alpha_\msr Q\colon\msr Q=\mc G'_{\mu,{\smallN}} \to \mc A$ by sending $1$ to $X$. By the definition of the action of $\mc G'_{\mu,{\smallN}}(R)$ on $\mc A(R)$, the object $(\msr Q,\alpha_\msr Q)$ is independent of the choice of $\theta$ in the sense that, if we pick another isomorphism $\theta'\colon \mc P_{\mu,\smallN}\isomto\mc P_{\smallN}$ and $\alpha'_\msr Q$ denotes the similarly defined map, then ${\theta'}^{-1}\circ\theta$, viewed as an element of $\mc G_{\mu,{\smallN}}'(R)$ via Lemma \ref{lem: global section of G mu}, gives an isomorphism $(\msr Q,\alpha_\msr Q)\isomto (\msr Q,\alpha'_\msr Q)$. 
    One may check this defines an equivalence $\cat{Tors}^{\mu}_\mc G(R^\syn)_\mr{banal}\isomto \mc G\hyphen F\hyphen\cat{Gauge}'_\mu(R)_\mr{banal}$. This equivalence is $2$-functorial in $R$, and so we obtain the desired equivalence passing to the quasi-syntomic stackifications.
\end{proof}

\subsubsection{Summary of equivalences} We summarize the above results as follows.

\begin{prop}\label{prop: summary display comparison}
    Let $\mf X$ be a quasi-syntomic $p$-adic formal scheme. 
    \begin{enumerate}
    \item 
    There exists a commutative diagram of categories 
\begin{equation}\label{eq:summary}\begin{tikzcd}[sep=large]
	{\cat{Tors}_\mc{G}^\mu(\mf{X}^\syn)} &  
    {\mc{G}\text{-}F\text{-}\mathbf{Gauge}_\mu(\mf{X})} \\
	{\cat{Tors}^{\varphi,\mu}_\mc{G}(\mf{X}_\smallprism)} &{\mc{G}\text{-}\mathbf{Disp}_\mu(\mf{X})}
	\arrow["\sim", from=1-1, to=1-2]
	\arrow["{\ref{prop: BT equals Ito's}}"', from=1-1, to=1-2]
	\arrow["{\mathrm{R}_\mf{X}}"', from=1-1, to=2-1]
	\arrow["{\mathrm{R}_\mf{X}^\mr{disp}}", from=1-2, to=2-2]
	\arrow["\sim", from=2-1, to=2-2]
	\arrow["{\ref{prop:display-torsor-equiv}}"', from=2-1, to=2-2]
\end{tikzcd}
\end{equation}
$2$-bi-functorial in $\mc{G}$ and $\mf{X}$, where the horizontal arrows are equivalences and the vertical arrows are fully faithful. 
\item Assume further that $\mf X$ is either, (a) $\Spf(R)$ for a perfectoid ring $R$ or, (b) a base formal $W$-scheme. Then the vertical arrows in \eqref{eq:summary} are equivalences.
\end{enumerate}
\end{prop}
\begin{proof}
    The only remaining part which is potentially unclear is the $2$-commutativity of \eqref{eq:summary}. It suffices to work with objects $\mc P=(\mc P_{\smallN},\varphi_\mc P)$ over $\mf X=\Spf(R)$ with $R$ a qrsp ring admitting a map from $\mc O=\Z_p\langle \mu_{p^\infty}\rangle$ and which are banal, i.e., there exists an isomorphism $\mc{P}_{\mu,\smallN}\isomto \mc{P}_\smallN$. 
    Choose such an isomorphism $\theta\colon \mc P_{\mu,{\smallN}}\isomto\mc P_\smallN$ and let $X_\theta$ denote the element of $\mr{Isom}(\sigma^*\mc P_{\mu,{\smallN}},\tau^*\mc P_{\mu,{\smallN}})$ that defines $\varphi_\mc P$ via $\theta$. 
    Then the Frobenius structure of the associated prismatic $\mc{G}$-torsor with $F$-structure is, via $\theta$ given by the composition 
    \begin{equation*} \phi^*\tau^*\mc P_{\mu,{\smallN}}[\nicefrac{1}{I}]\xrightarrow{\mr{can.}}\sigma^*\mc P_{\mu,{\smallN}}[\nicefrac{1}{I}]\xrightarrow{X_\theta}\tau^*\mc P_{\mu,{\smallN}}[\nicefrac{1}{I}].
    \end{equation*}
    Choose a generator $\pi$ of the ideal $L_R$ and set $d_\pi\defeq\phi(\pi)/\pi$. Then, via the identification 
    \begin{eqnarray*}
        \mc G_{\Prism_R}=\mu_*(\mc O_\smallprism)\xrightarrow{\mu'_*(\pi\cdot)}\mu_*(\mc O_\smallprism\{1\})=\tau^*\mc P_{\mu,{\smallN}},
    \end{eqnarray*} 
    the above composition is identified with the composition 
    \begin{equation*} \phi^*\mc G[\nicefrac{1}{I}]\simeq \mc G[\nicefrac{1}{I}]\xrightarrow{\mu(d_\pi)}\mc G[\nicefrac{1}{I}]\xrightarrow{X_\theta^\pi}\mc G[\nicefrac{1}{I}],
    \end{equation*}
    where $X_\theta^\pi\defeq\mu'(\pi)^{-1}X_\theta \mu(\pi)$. 
    
    On the other hand, the associated object $(\msr Q,\alpha_{\msr Q})$ in $\mc G\hyphen F\hyphen\cat{Gauge}_\mu(R)$ is canonically isomorphic to $(\mc G_{\mu,{\smallN}}\to \mc G(\Prism_R):1\mapsto X^\pi_\theta)$, and hence the associated object of $\mc G\hyphen\cat{Disp}_\mu(R)$ is given by $(\mc G_{\mu,\smallprism}\to \mc G(\Prism_R)\colon 1\mapsto \phi(X_\theta^\pi))$, whose associated prismatic $\mc G$-torsor is canonically isomorphic to $(\mc G_{\Prism_R},\varphi)$ where $\varphi$ corresponds to left multiplication by $\mu(d_\pi)\phi(X_\theta^\pi)$.  
    Left multiplication by $X_\theta^\pi$ on $\mc G_\smallprism$ gives the desired commutativity isomorphism. 
\end{proof}

\section{Prismatic realization functors on Shimura varieties of abelian type}\label{s:applications-to-Shimura-varieties}
In this section we construct an object $\omega_{\mathsf{K}^p,\smallprism}$ of $\GVect^\varphi((\wh{\ms{S}}_{\mathsf{K}^p})_\smallprism)$, called the \emph{prismatic realization functor}, where $\ms{S}_{\mathsf{K}^p}$ is the integral canonical model of a Shimura variety of abelian type, and show it can be upgraded to a \emph{syntomic realization functor} $\omega_{\mathsf{K}^p,\syn}$ in $\GVect((\wh{\ms{S}}_{\mathsf{K}^p})^\mr{syn})$.

\subsection{Notation and basic definitions}\label{ss:Shim-Var-Notation} Throughout this section, we fix the following.
\begin{nota}\label{nota:shim-var} Define the following notation:
\begin{itemize}[label=$\diamond$]
    \item $\mb{G}$ is a reductive group over $\Q$,
    \item $\mb{Z}$ denotes the center $Z(\mb{G})$ of $\mb{G}$,
    \item $\bb{S}\defeq \mathrm{Res}_{\C/\R}\,\mathbb{G}_{m,\C}$ is the Deligne torus,
    \item $(\mb{G},\mathbf{X})$ is a Shimura datum (see \cite[Definition 5.5]{MilneShimura}),
    \item $\mathbf{E}=\mathbf{E}(\mb{G},\mathbf{X})\subseteq \bb{C}$ denotes the reflex field of $(\mb{G},\mb{X})$ (see \cite[Definition 12.2]{MilneShimura}),
    \item $\mathsf{K}\subseteq \mb{G}(\mbb{A}_f)$ is a (variable) neat (cf.\@ \cite[p.\@ 288]{MilneShimura}) compact open subgroup.
\end{itemize}
\end{nota}
As in \cite{DeligneModulaire} (cf.\@ \cite{Moonen}), associated to this data is the (canonical model of the) Shimura variety $\Sh_{\mathsf{K}}(\mb{G},\mb{X})$, which is a smooth and quasi-projective $\mathbf{E}$-scheme. For $\mathsf{K}$ and $\mathsf{K}'$ of $\mb{G}(\A_f)$, and $g$ in $\mb{G}(\A_f)$ such that $g^{-1}\mathsf{K}g\subseteq \mathsf{K}'$, denote by $t_{\mathsf{K},\mathsf{K}'}(g)$ the unique finite \'etale morphism of $\mathbf{E}$-schemes $\Sh_{\mathsf{K}}(\mb{G},\mb{X})\to \Sh_{\mathsf{K}'}(\mb{G},\mb{X})$ given on $\C$-points by
\begin{equation*}
    t_{\mathsf{K},\mathsf{K}'}(g)\left(\mb{G}(\Q)(x,g')\mathsf{K}\right)=\mb{G}(\Q)(x,g'g)\mathsf{K}'.
\end{equation*}
We shorten $t_{\mathsf{K},\mathsf{K}'}(\id)$ to $\pi_{\mathsf{K},\mathsf{K}'}$ and $t_{\mathsf{K},g^{-1}\mathsf{K}g}(g)$ to $[g]_K$. The morphisms $\pi_{\mathsf{K},\mathsf{K}'}$ form a projective system $\{\Sh_{\mathsf{K}}(\mb{G},\mb{X})\}$ with finite \'etale transition maps, and the morphisms $[g]_K$ endow
\begin{equation*}
    \Sh(\mb{G},\mb{X})\defeq \varprojlim_\mathsf{K} \Sh_{\mathsf{K}}(\mb{G},\mb{X})
\end{equation*}
(cf.\@ \stacks{01YX}) with a continuous action of $\mb{G}(\A_f)$ (in the sense of \cite[2.7.1]{DeligneModulaire}).

We shall often fix the following additional data/notation/assumptions:
\begin{itemize}[label=$\diamond$]
    \item $p$ is a rational prime and $\mf{p}$ a prime of $\mathbf{E}$ lying over $p$,
    \item $E$ is the completion $\mathbf{E}_\mf{p}$, $\mc{O}_E$ its ring of integers, and $k$ its residue field,
    \item $G\defeq \mb{G}_{\Q_p}$, and $\mc{G}$ is a parahoric model of $G$ over $\Z_p$,
    \item $\mathsf{K}_0\subseteq G(\Q_p)$ the parahoric subgroup given by $\mc{G}(\Z_p)$,
    \item  ${\mathsf{K}^p}\subseteq \mb{G}(\bb{A}^p_f)$ a neat compact open subgroup.
\end{itemize}
The triple $(\mb{G},\mb{X},\mc{G})$ is a \emph{parahoric Shimura datum}, and is an \emph{unramified Shimura datum} if $\mc{G}$ is reductive. For an unramified Shimura datum, the extension $E/\Q_p$ is unramified (see \cite[Corollary 4.7]{MilShmot}) and we identify $\mc{O}_E$ with $W=W(k)$. Moreover, $G$ is quasi-split and split over $\breve{E}$. We shorten $\Sh_{\mathsf{K}}(\mb{G},\mb{X})_E$ (resp.\@ $\Sh(\mb{G},\mb{X})_E$) to $\Sh_{\mathsf{K}}$ (resp.\@ $\Sh$). 

Let $(\mb{G},\mb{X},\mc{G})$ be a parahoric Shimura datum. Associated to $\mb{X}$ is a unique conjugacy class of coharacters $\bb{G}_{m,\C}\to \mb{G}_\C$ (see \cite[p.\@ 344]{MilneShimura}) whose field of definition is $\mb{E}$. Using \cite[Lemma 1.1.3]{KotShtw} this corresponds to a unique conjugacy class $\bm{\mu}_h$ of cocharacters $\bb{G}_{m,\ov{E}}\to G_{\ov{E}}$ which one checks has field of definition $E$. If $(\mb{G},\mb{X},\mc{G})$ is unramified then one may use loc.\@ cit.\@ to show the existence of a unique conjugacy class $\bbmu_h$ of cocharacters $\bb{G}_{m,\breve{\Z}_p}\to \mc{G}_{\breve{\Z}_p}$ modeling $\bm{\mu}_{h}$.

We often denote other Shimura data with numerical subscripts (e.g.\@ $(\mb{G}_1,\mb{X}_1)$) and use the same numerical subscripts to denote the objects defined above (or below) for this Shimura datum (e.g.\@ $\Sh_{\mathsf{K}_{0,1}\mathsf{K}^p_1}$ or $\mc{G}_1$). A morphism of Shimura data $\alpha\colon (\mb{G}_1,\mb{X}_1)\to (\mb{G},\mb{X})$ is a morphism of group $\Q$-schemes $\alpha\colon \mb{G}_1\to\mb{G}$ such that $\alpha_\R(\mb{X}_1)\subseteq \mb{X}$, and is an \emph{embedding} if $\alpha$ is a closed embedding. By \cite[\S5]{DeligneModulaire}, for a morphism $\alpha$ one has $\mathbf{E}\subseteq \mathbf{E}_1$ and there is a morphism $\Sh(\mb{G}_1,\mb{X}_1)\to \Sh(\mb{G},\mb{X})_{\mathbf{E}_1}$ of $\mb{E}_1$-schemes equivariant for $\alpha\colon \mb{G}_1(\A_f)\to \mb{G}(\A_f)$ and such that if $\alpha(\mathsf{K}_1)\subseteq \mathsf{K}$ then the induced map $\alpha_{\mathsf{K}_1,\mathsf{K}}\colon \Sh_{\mathsf{K}_1}(\mb{G},\mb{X})\to \Sh_{\mathsf{K}}(\mb{G},\mb{X})_{\mathbf{E}_1}$ is given by 
\begin{equation*}
    \alpha_{\mathsf{K}_1,\mathsf{K}}\left(\mb{G}_1(\Q)(x,g_1)\mathsf{K}_1\right)=\mb{G}(\Q)(\alpha\circ x,\alpha(g_1))\mathsf{K}
\end{equation*}
on $\C$-points. If the induced map $\alpha\colon G_1^\der\to G^\der$ is an isogeny, then each $\alpha_{\mathsf{K}_1,\mathsf{K}}$ is finite \'etale, as can be checked on connected components (cf.\@ \cite[p.\@ 6620]{ShenPerfectoid}).

By a morphism $\alpha\colon (\mb{G}_1,\mb{X}_1,\mc{G}_1)\to (\mb{G},\mb{X},\mc{G})$ of parahoric Shimura data we mean a morphism $\alpha\colon (\mb{G}_1,\mb{X}_1)\to (\mb{G},\mb{X})$ of Shimura data together with a specified model $\mc{G}_1\to\mc{G}$ of $G_1\to G$, which  we also denote $\alpha$.  We say that $\alpha$ is an embedding if $\mc{G}_1\to\mc{G}$ is a closed embedding.

\subsection{Integral canonical models}\label{ss:integral-models}

We consider the following objects:
\begin{multicols}{2}
    \begin{itemize}[label=$\diamond$]
    \item a symplectic space $\bm{\Lambda}_0$ over $\Z_{(p)}$,
    \item set $\mb{V}_0\defeq \bm{\Lambda}_0\otimes_\Z \Q$, 
    \item set $\Lambda_0\defeq \mb{\Lambda}_0\otimes_{\Z_{(p)}}\Z_p$,
    \item set $V_0\defeq \mb{V}_0\otimes_\Q \Q_p=\Lambda_0[\nicefrac{1}{p}]$.
\end{itemize}
\end{multicols}
\noindent We then have the \emph{Siegel Shimura datum} $(\GSp(\mb{V}_0),\mf{h}^{\pm})$ (see \cite[\S6]{MilneShimura}) with reflex field $\Q$. For a neat compact open subgroup $\mathsf{L}\subseteq \GSp(\mb{V}_0)(\A_f)$ there is an identification of $\mb{Sh}_\mathsf{L}(\GSp(\mb{V}_0),\mf{h}^{\pm})$ with Mumford's moduli space of principally polarized abelian schemes with level $\mathsf{L}$-structure (see \cite[\S4]{DelTS}). Set $\mathsf{L}_0=\GSp(\Lambda_0)$. Then, $\Sh_{\mathsf{L}_0\mathsf{L}^p}$ admits a smooth model $\ms{M}_{\mathsf{L}^p}(\Lambda_0)$ over $\Z_p$ with a similar moduli description (see loc.\@ cit.\@). 

Recall that $(\mb{G},\mb{X})$ is of \emph{Hodge type} if there exists an embedding (called a \emph{Hodge embedding}) $(\mb{G},\mb{X})\hookrightarrow (\GSp(\mb{V}_0),\mf{h}^{\pm})$ for some symplectic space $\mb{V}_0$ over $\Q$, and of \emph{abelian type} if there exists $(\mb{G}_1,\mb{X}_1)$ of Hodge type and an isogeny $\mb{G}_1^\der\to \mb{G}^\der$ inducing an isomorphism of adjoint Shimura data $(\mb{G}_1^\ad,\mb{X}_1^\ad)\to (\mb{G}^\ad,\mb{X}^\ad)$. As in \cite[2.5.14]{LoveringModels} (cf.\@ the proof of \cite[Corollary 3.4.14]{KisIntShab}), if $(\mb{G},\mb{X},\mc{G})$ is an unramified Shimura datum of abelian type then $(\mb{G}_1,\mb{X}_1,\mc{G}_1)$ may be further chosen so that $G_1^\der\to G^\der$ admits a central isogeny model $\mc{G}_1^\der\to\mc{G}^\der$. For such well-chosen data, we say that $(\mb{G}_1,\mb{X}_1)$ (resp.\@ $(\mb{G}_1,\mb{X}_1,\mc{G}_1)$) is \emph{adapted} to $(\mb{G},\mb{X})$ (resp.\@ $(\mb{G},\mb{X},\mc{G})$).

Suppose now that $(\mb{G},\mb{X},\mc{G})$ is an unramified Shimura datum of abelian type. Set 
\begin{equation*}
    \Sh_{\mathsf{K}_0}=\varprojlim_{{\mathsf{K}^p}}\Sh_{\mathsf{K}_0\mathsf{K}^p}=\Sh/\mathsf{K}_0,
\end{equation*}
which is a scheme with a continuous action of $\mb{G}(\A_f^p)$. In \cite{KisIntShab}, there is constructed an $\mc{O}_E$-scheme $\ms{S}=\ms{S}_{K_0}$ with a continuous action of $\mb{G}(\A_f^p)$ whose generic fiber recovers $\Sh_{\mathsf{K}_0}$ with its $\mb{G}(\A_f^p)$-action. For a neat compact open subgroup ${\mathsf{K}^p}\subseteq \mb{G}(\A_f^p)$ write $\ms{S}_{{\mathsf{K}^p}}\defeq \ms{S}/{\mathsf{K}^p}$, and for neat compact open subgroups ${\mathsf{K}^p}$ and ${\mathsf{K}^{'p}}$ of $\mb{G}(\A_f^p)$, and an element $g^p$ of $ \mb{G}(\A_f^p)$ such that $(g^p)^{-1}{\mathsf{K}^p} g^p\subseteq {\mathsf{K}^{'p}}$ denote by $t_{{\mathsf{K}^p},{\mathsf{K}^{'p}}}(g^p)$ the induced map $\ms{S}_{{\mathsf{K}^p}}\to\ms{S}_{{\mathsf{K}^{'p}}}$, subject to the same notational shortenings as in the generic fiber case. Then, $\ms{S}$ is a so-called \emph{integral canonical model}: the $\mc{O}_E$-schemes $\ms{S}_{{\mathsf{K}^p}}$ are smooth (and quasi-projective), the maps $t_{{\mathsf{K}^p},{\mathsf{K}^{'p}}}(g^p)$ are finite \'etale, and for any regular and formally smooth $\mc{O}_E$-scheme $\ms{X}$ any morphism $\ms{X}_\eta\to \Sh_{\mathsf{K}_0}$ of $E$-schemes lifts uniquely to a morphism of $\mc{O}_E$-schemes $\ms{X}\to\ms{S}$ (the \emph{extension property}).

\begin{example}When $(\mb{G},\mb{X})=(\GSp(\mb{V}_0),\mf{h}^{\pm})$, and $\mathsf{L}_0=\GSp(\Lambda_0)$, then the integral canonical model is precisely the system $\{\ms{M}_{\mathsf{L}^p}(\Lambda_0)\}$ (cf.\@ \cite[Corollary 3.8]{Moonen}).
\end{example}

If $\alpha\colon (\mb{G}_1,\mb{X}_1,\mc{G}_1)\to (\mb{G},\mb{X},\mc{G})$ is a morphism of unramified Shimura data of abelian type, then the morphism $\Sh_{\mathsf{K}_{0,1}}\to(\Sh_{\mathsf{K}_0})_{E_1}$ has a unique model $\ms{S}_1\to\ms{S}_{\mc{O}_{E_1}}$ equivariant for the map $\mb{G}_1(\A_f^p)\to \mb{G}(\A_f^p)$ . If $\alpha({\mathsf{K}^p_1})\subseteq {\mathsf{K}^p}$ we denote by $\alpha_{{\mathsf{K}^p_1},{\mathsf{K}^p}}$ the induced morphism $\ms{S}_{{\mathsf{K}^p_1}}\to(\ms{S}_{{\mathsf{K}^p}})_{\mc{O}_{E_1}}$.

\begin{lem}\label{lem:isogeny-finite-etale}
    If $\alpha\colon \mc{G}_1^\der\to \mc{G}^\der$ is a central isogeny, then each $\alpha_{\mathsf{K}_1^p,\mathsf{K}^p}$ is finite \'etale.
\end{lem}
\begin{proof}
    It suffices to show the maps $\mathscr{S}_{\mathsf{K}_1^p}(\mathcal{G}_1^{\der},\mb{X}_1^+) \to \mathscr{S}_{\mathsf{K}^p}(\mathcal{G}^{\der},\mb{X}^+)$ (with notation as in \cite[(3.4.9)]{KisIntShab}) are finite \'etale. Let $(\mb{G}_2,\mb{X}_2,\mc{G}_2)$ be an unramified Shimura datum of Hodge type adapted to $(\mb{G}_1,\mb{X}_1,\mc{G}_1)$ and thus to $(\mb{G},\mb{X},\mc{G})$ and fix a sufficiently small neat compact open subgroup $\mathsf{K}^p_2$. As the map $\mathscr{S}_{\mathsf{K}_1^p}(\mathcal{G}_1^{\der},\mb{X}_1^+) \to \mathscr{S}_{\mathsf{K}^p}(\mathcal{G}^{\der},\mb{X}^+)$ fits into a commutative triangle with maps of the form $\mathscr{S}_{\mathsf{K}_2^p}(\mathcal{G}_2^{\der},\mb{X}_2^+) \to \mathscr{S}_{\mathsf{K}_1^p}(\mathcal{G}_1^{\der},\mb{X}_1^+)$ and $\mathscr{S}_{\mathsf{K}_2^p}(\mathcal{G}_2^{\der},\mb{X}_2^+) \to \mathscr{S}_{\mathsf{K}^p}(\mathcal{G}^{\der},\mb{X}^+)$ it suffices to show these maps are finite \'etale. But, this follows from \cite[2.5.14]{LoveringModels} as the group $\Delta^N$ is finite and acts freely by \cite[Proposition 2.5.9 and Lemma 2.5.10]{LoveringModels}.
\end{proof}

For an unramified Shimura datum $(\mb{G},\mb{X},\mc{G})$ of Hodge type, an \emph{integral Hodge embedding} is an embedding $\iota\colon (\mb{G},\mb{X},\mc{G})\hookrightarrow (\GSp(\mb{V}_0),\mf{h}^{\pm},\GSp(\Lambda_0))$. By \cite[3.3.1]{KimUnif}, such an integral Hodge embedding always exists. As each $\ms{M}_{\mathsf{L}^p}(\Lambda_0)$ is a fine moduli space of principally polarized abelian varieties it has a universal abelian scheme $\ms{A}_{\mathsf{L}^p}$ compatible in $\mathsf{L}^p$. If $\iota(\mathsf{K}^p)\subseteq \mathsf{L}^p$, we (suppressing $\iota$ from the notation) denote by $\ms{A}_{{\mathsf{K}^p}}\to \ms{S}_{{\mathsf{K}^p}}$ the pullback of $\ms{A}_{\mathsf{L}^p}$ along $\iota_{{\mathsf{K}^p},\mathsf{L}^p}$. Denote by $\wh{\ms{A}}_{{\mathsf{K}^p}}\to \wh{\ms{S}}_{{\mathsf{K}^p}}$ its $p$-adic completion (equiv.\@ the pullback of $\ms{A}_{{\mathsf{K}^p}}$ along $\wh{\ms{S}}_{{\mathsf{K}^p}}\to \ms{S}_{{\mathsf{K}^p}}$), and by $A_{{\mathsf{K}^p}}\to\Sh_{{\mathsf{K}_0\mathsf{K}^p}}$ the generic fiber of $\ms{A}_{{\mathsf{K}^p}}\to\ms{S}_{{\mathsf{K}^p}}$.

We finally observe that the connected components of $\ms{S}$ are homogeneous in a suitable sense.

\begin{lem}[{cf.\@ \cite[Lemma 2.2.5]{KisIntShab}}]\label{lem:transitivity-on-conn-comp} The action of $\mb{G}(\A_f^p)$ on $\pi_0(\ms{S}_{\mc{O}_{\breve{E}}})$ is transitive.
\end{lem}

\subsection{\'Etale realization functors}\label{ss:etale-fiber-functor-def} Following \cite[Definition 1.5.4]{KSZ}, for a multiplicative $\Q$-group $\mb{T}$ denote by $\mb{T}_\mathrm{a}$ the largest $\Q$-anisotropic subtorus of $\mb{T}$, and by $\mb{T}_\ac$ the smallest subtorus of $\mb{T}_\mathrm{a}$ whose base change to $\R$ contains the maximal split subtorus of $(\mb{T}_a)_\R$. For a reductive $\Q$-group $\mb{G}$ denote by $\mb{G}^c$ the $\Q$-group $\mb{G}/\mb{Z}_\ac$, and by $G^c$ the group $\mb{G}^c_{\Q_p}$. 

Fix a parahoric Shimura datum $(\mb{G},\mb{X},\mc{G})$. There is a canonical map of Bruhat--Tits buildings $B(G,F)\to B(G^c,F)$. Let $x$ denote a point of $B(G,F)$ corresponding to $\mc{G}$, and $x^c$ its image in $B(G^c,F)$ (see \cite[\S1.1--1.2]{KisinPappas} and the references therein). Set $\mc{G}^c$ to be the parahoric group scheme associated to $x^c$ (denoted by $\mc{G}_{x^c}^\circ$ in \cite[\S1.2]{KisinPappas}). By \cite[Proposition 1.1.4]{KisinPappas}, $\mc{G}$ is a central extension of $\mc{G}^c$, and so one is reductive if and only if the other is.\footnote{For the reader less familiar with Bruhat--Tits theory, \cite[Proposition 1.1.4]{KisinPappas} shows that when $\mc{G}$ is reductive, $\mc{G}^c=\mc{G}/\mc{Z}$, where $\mc{Z}$ is the Zariski closure of $(\mb{Z})_{\Q_p}$.} Denote by $\bm{\mu}^c_h$ the conjugacy class of cocharacters of $G^c$ induced by $\bm{\mu}_h$, and if $(\mb{G},\mb{X},\mc{G})$ is unramified let $\bbmu_h^c$ be conjugacy class of cocharacters of $\mc{G}^c_{\breve{\Z}_p}$ induced by $\bbmu_h$.

\begin{lem}
    A morphism $\alpha\colon (\mb{G}_1,\mb{X}_1)\to (\mb{G},\mb{X})$ (resp.\@ $\alpha\colon (\mb{G}_1,\mb{X}_1,\mc{G}_1)\to (\mb{G},\mb{X},\mc{G})$) of Shimura data (resp.\@ parahoric Shimura data) induces a morphism $\mb{G}_1^c\to \mb{G}^c$ (resp.\@ $\mc{G}_1^c\to\mc{G}^c$).
\end{lem}
\begin{proof} The claim concerning Shimura data would follow from $\alpha(\mb{Z}_{1,\ac})\subseteq \mb{Z}_\ac$. To show this, it suffices to show that if $\mb{S}=(\alpha^{-1}(\mb{Z}_\ac)\cap \mb{Z}_{1,\ac})^\circ$, then $\mb{S}_\R$ contains the split component of $(\mb{Z}_{1,\mathrm{a}})_\R$. Suppose not and that $\R^\times\subseteq (\mb{Z}_{1,\mathrm{a}})(\R)$ is not contained in $\mb{S}(\R)$. If $\alpha(\R^\times)$ is not contained in $\mb{Z}_\R$ then we arrive at a contradiction as in the proof of \cite[Lemma 3.1.3]{LoveringModels}. As $\R^\times\subseteq (\mb{Z}_{1,\mathrm{a}})(\R)$ this then implies that $\alpha(\R^\times)\subseteq \mb{Z}_\ac(\R)$ which is again a contradiction. The claim concerning parahoric Shimura data then follows by applying \cite[Proposition 1.1.4]{KisinPappas}.
\end{proof}

\begin{rem}If $(\mb{G},\mb{X})$ is of Hodge type, then $\mb{G}$ is equal to $\mb{G}^c$. Indeed, this can be checked explicitly for Siegel datum, and follows by functoriality for arbitrary $(\mb{G},\mb{X})$. Shimura data of abelian type need not enjoy this equality in general.
\end{rem}

For a Shimura datum $(\mb{G},\mb{X})$, and a neat compact open $\mathsf{K}=\mathsf{K}_p\mathsf{K}^p\subseteq\mb{G}(\A_f^p)$, the map
\begin{equation*}
    \varprojlim_{\scriptscriptstyle\mathsf{K}_p'\subseteq \mathsf{K}_p}\Sh_{\mathsf{K}_p'\mathsf{K}^p}\to \Sh_{\mathsf{K}}
\end{equation*}
is a $\mathsf{K}_p/\mb{Z}(\Q)^{-}_{\mathsf{K}}$-torsor on $(\Sh_{\mathsf{K}_p\mathsf{K}^p})_\proet$, where $\mb{Z}(\Q)^{-}_{\mathsf{K}}$ is the closure of $\mb{Z}(\Q)\cap \mathsf{K}$ in $\mathsf{K}$ (see \cite[\S1.5.8]{KSZ}). If $K_p^c$ denotes the image of $K_p$ in $G^c(\Q_p)$, loc.\@ cit.\@ shows that $K_p\to K_p^c$ factorizes through $K_p/\mb{Z}(\Q)_{K}^{-}$. Denote by $\mathsf{T}_\mathsf{K}$ the $K_p^c$-torsor obtained by pushing forward $\varprojlim_{\scriptscriptstyle\mathsf{K}_p'\subseteq \mathsf{K}_p}\Sh_{\mathsf{K}_p'\mathsf{K}^p}$ along $\mathsf{K}_p/\mb{Z}(\Q)^{-}_{\mathsf{K}}\to K_p^c$. We obtain an object $\nu_{\mathsf{K},\et}$ of  $G^c\text{-}\cat{Loc}_{\Q_p}(\Sh_{{\mathsf{K}_p\mathsf{K}^p}})$ given by sending $\rho\colon G^c\to\GL(V)$ to the pushforward of $\mathsf{T}_\mathsf{K}$ along $\rho\colon K_p^c\to \GL(V)$. Fix $g$ in $\mb{G}(\A_f)$, and suppose $g^{-1}\mathsf{K} g\subseteq \mathsf{K}'$. If $g=g_pg^p$, and $\Int(g_p^c)$ is the inner automorphism of $G^c$ associated to the image $g_p^c$ of $g_p$ in $G^c(\Q_p)$, then
\begin{equation}\label{eq:Hecke-action-local-system-compat-rat}
  t_{\mathsf{K},\mathsf{K}^{'}}(g)^\ast(\nu_{{\mathsf{K}^{'}},\et}(\rho)))=\nu_{{\mathsf{K}},\et}(\rho\circ \mathrm{Int}((g_p^c)^{-1})). 
\end{equation}
We call the system $\nu_\et\defeq \{\nu_{{\mathsf{K}_p\mathsf{K}^p},\et}\}$ the \emph{(rational $p$-adic) \'etale realization functor} on $\Sh_{{\mathsf{K}_p}}$.

Let $\alpha\colon (\mb{G}_1,\mb{X}_1)\to (\mb{G},\mb{X})$ be a morphism of Shimura data. If $\alpha(\mathsf{K}_1)\subseteq \mathsf{K}$ one obtains a morphism $\mathsf{T}_{\mathsf{K}_1}\to \mathsf{T}_{\mathsf{K}}\times_{(\Sh_{\mathsf{K}})_{E_1}}\Sh_{{\mathsf{K}_1}}$ equivariant for $\alpha^c\colon K_{p,1}^c\to K_p^c$ and, thus an isomorphism of $K_p^c$-torsors $\alpha_\ast^c(\mathsf{T}_{\mathsf{K}_1})\to \mathsf{T}_{\mathsf{K}}\times_{(\Sh_{\mathsf{K}})_{E_1}}\Sh_{{\mathsf{K}_1}}$. This is compatible in ${\mathsf{K}}$ in the obvious way. Equivalently, for $\rho$ in $\cat{Rep}_{\Q_p}(G^c)$ there is an identification 
\begin{equation}\label{eq:shimura-data-morphism-local-system-compat-rat}
    \alpha_{{\mathsf{K}_1},{\mathsf{K}}}^\ast(\nu_{\mathsf{K},\et}(\rho)_{E_1})=\nu_{{\mathsf{K}_1},\et}(\rho\circ\alpha^c),
\end{equation}
compatible in ${\mathsf{K}_1}$, ${\mathsf{K}}$, and $\xi$ in the obvious sense.

For a parahoric Shimura datum $(\mb{G},\mb{X},\mc{G})$, and $\mathsf{K}=\mathsf{K}_0\mathsf{K}^p$ (recall $K_0=\mc G(\Z_p)$), there are analogous integral objects. Again by \cite[\S1.5.8]{KSZ}, $\mb{Z}(\Q)_{\mathsf{K}}\subseteq \mb{Z}_{\ac}(\Q)$ and so $\mathsf{K}_0\to\mc{G}^c(\Z_p)$ factorizes through $\mathsf{K}_0/\mb{Z}(\Q)^{-}_{\mathsf{K}}$.

Denote by $\mathsf{S}_{\mathsf{K}^p}$ the  push forward of $\varprojlim_{\scriptscriptstyle\mathsf{K}_p\subseteq \mathsf{K}_0}\Sh_{\mathsf{K}_p\mathsf{K}^p}$ along $\mathsf{K}_0/\mb{Z}(\Q)^{-}_{\mathsf{K}}\to \mc{G}^c(\Z_p)$. From the contents of \cite[\S2.1.1]{IKY1}, we obtain an associated object of $\mc{G}^c\text{-}\cat{Loc}_{\Z_p}(\Sh_{{\mathsf{K}_0\mathsf{K}^p}})$:
\begin{equation*}
\omega_{{\mathsf{K}^p},\et}\colon \cat{Rep}_{\Z_p}(\mc{G}^c)\to \cat{Loc}_{\Z_p}(\Sh_{{\mathsf{K}_0\mathsf{K}^p}}). 
\end{equation*}
Fix $g=g_p g^p$ in $K_0\mb{G}(\A_f^p)$, and suppose $(g^p)^{-1}\mathsf{K}^pg^p\subseteq \mathsf{K}^{'p}$. If $\Int(g_p^c)$ is the inner automorphism of $\mc{G}^c$ associated to the image $g_p^c$ of $g_p$ in $\mc{G}^c(\Z_p)$, then
\begin{equation}\label{eq:Hecke-action-local-system-compat}
  t_{{\mathsf{K}^p}\mathsf{K}_0,{\mathsf{K}^{'p}}\mathsf{K}_0}(g)^\ast(\omega_{{\mathsf{K}^{'p}},\et}(\xi)))=\omega_{{\mathsf{K}^p},\et}(\xi\circ \mathrm{Int}((g_p^c)^{-1})).
\end{equation}
We call the system $\{\omega_{{\mathsf{K}^p},\et}\}$ the \emph{(integral) \'etale realization functor} on $\Sh_{{\mathsf{K}_0}}$. That this is an integral model of $\nu_{\mathsf{K}_0\mathsf{K}^p,\et}$ is made precise by the observation that 
\begin{equation}\label{eq:integral-model-etale-realization}
    \omega_{\mathsf{K}^p,\et}[\nicefrac{1}{p}]=\nu_{\mathsf{K}_0\mathsf{K}^p,\et}
\end{equation}
compatibly in a neat compact open subgroup $\mathsf{K}^p\subseteq\mb{G}(\A_f^p)$.

Let $\alpha\colon (\mb{G}_1,\mb{X}_1,\mc{G}_1)\to (\mb{G},\mb{X},\mc{G})$ be a morphism of parahoric Shimura data. If $\alpha(\mathsf{K}_1^p)\subseteq {\mathsf{K}^p}$ one obtains a morphism $\mathsf{S}_{\mathsf{K}_{0,1}\mathsf{K}^p_1}\to \mathsf{S}_{\mathsf{K}^p}\times_{(\Sh_{\mathsf{K}_0\mathsf{K}^p})_{E_1}}\Sh_{{\mathsf{K}_{0,1}\mathsf{K}^p_1}}$ equivariant for $\alpha^c\colon \mc{G}_1^c(\Z_p)\to\mc{G}^c(\Z_p)$ and, thus an isomorphism of $\mc{G}^c(\Z_p)$-torsors $\alpha_\ast^c(\mathsf{S}_{\mathsf{K}^1_p})\to \mathsf{S}_{\mathsf{K}^p}\times_{(\Sh_{\mathsf{K}_0\mathsf{K}^p})_{E_1}}\Sh_{{\mathsf{K}_{0,1}\mathsf{K}^p_1}}$. This is compatible in ${\mathsf{K}^p}$ in the obvious way. Equivalently, for $\xi$ in $\cat{Rep}_{\Z_p}(\mc{G}^c)$ there is an identification 
\begin{equation}\label{eq:shimura-data-morphism-local-system-compat}
    \alpha_{{\mathsf{K}^p_1},{\mathsf{K}^p}}^\ast(\omega_{\mathsf{K}^p,\et}(\xi)_{E_1})=\omega_{{\mathsf{K}^p_1},\et}(\xi\circ\alpha^c),
\end{equation}
compatible in ${\mathsf{K}^p_1}$, ${\mathsf{K}^p}$, and $\xi$ in the obvious sense.

\begin{prop}\label{prop:G-local-system-de-Rham} The $\mc{G}^c(\Z_p)$-local system $\omega_{\mathsf{K}^p,\et}$ belongs to $\mc{G}^c\text{-}\cat{Loc}_{\Z_p,\bm{\mu}^c_h}^{\dR}(\Sh_{\mathsf{K}_0\mathsf{K}^p})$. 
\end{prop}
\begin{proof} That $\omega_{\mathsf{K}^p,\et}$ is de Rham follows from \cite[Corollary 4.9]{LiuZhu} and the claim about cocharacters is reduced to the case of special points which follows from \cite[Lemma 4.8]{LiuZhu}.
\end{proof}

Let $(\mb{G},\mb{X},\mc{G})$ be an unramified Shimura datum of Hodge type, and fix an integral Hodge embedding $\iota\colon (\mb{G},\mb{X},\mc{G})\hookrightarrow(\GSp(\mb{V}_0),\mf{h}^{\pm},\GSp(\Lambda_0))$. By \cite[Theorem A.14]{IKY1} there is a tensor package $(\Lambda_0,\mathds{T}_0)$ with $\mc{G}=\mathrm{Fix}(\mathds{T}_0)$ (in the sense of loc.\@ cit.\@). As in \cite[\S3.1.2]{KimUnif}, one may construct from $\mathds{T}_{0}\otimes 1\subseteq V_0^\otimes$ tensors $ \mathds{T}_{0,p}^\et$ on $\mc{H}^1_{\Q_p}(A_{{\mathsf{K}^p}}/\Sh_{{\mathsf{K}_0\mathsf{K}^p}})^\vee$ as an object
of $\cat{Loc}_{\Q_p}(\Sh_{{\mathsf{K}_0\mathsf{K}^p}})$, which are compatible in ${\mathsf{K}^p}$.

\begin{prop}\label{prop:comparison-of-otorsors-Hodge-type} There is an isomorphism of $\Z_p$-local systems $\omega_{{\mathsf{K}^p},\et}(\Lambda_0)\isomto \mc{H}^1_{\Z_p}(A_{{\mathsf{K}^p}}/\Sh_{{\mathsf{K}_0\mathsf{K}^p}})^\vee$ carrying $\omega_{\mathsf{K}^p,\et}(\mathds{T}_0)\otimes 1$ in $\omega_{{\mathsf{K}^p},\et}(\Lambda_0)[\nicefrac{1}{p}]$ to $\mathds{T}_{0,p}^\et$ in $\mc{H}^1_{\Q_p}(A_{{\mathsf{K}^p}}/\Sh_{{\mathsf{K}_0\mathsf{K}^p}})^\vee$.
\end{prop}
\begin{proof} First suppose that $(\mb{G},\mb{X},\mc{G})=(\GSp(\mb{V}_0),\mf{h}^{\pm},\GSp(\Lambda_0))$ and $\iota$ is the identity embedding. Then, by the moduli description of $\Sh_{\mathsf{K}_0\mathsf{K}^p}$ one observes that there is an identification
\begin{equation*}
\mathsf{S}_{\mathsf{K}^p}=\underline{\Isom}\left((\Lambda_0\otimes_{\Z_p}\underline{\Z_p},t_0),(T_p(A_{\mathsf{K}^p}),t)\right), 
\end{equation*}
where $t_0$ is the tensor as in \cite[Example 2.1.6]{KimRZ}, and $t$ is the analogous tensor built from the Weil pairing coming from the principal polarization on $A_{\mathsf{K}^p}$. We deduce a natural identification between $\omega_{\mathsf{K}^p,\et}(\Lambda_0)$ and $\mc{H}^1_{\Z_p}(A_{\mathsf{K}^p}/\Sh_{\mathsf{K}_0\mathsf{K}^p})^\vee$. The desired isomorphism for general $(\mb{G},\mb{X},\mc{G})$ comes from the compatability in \eqref{eq:shimura-data-morphism-local-system-compat}. To prove that the induced isomorphism of $\Q_p$-local systems takes $\omega_{\mathsf{K}^p,\et}(\mathds{T}_0)\otimes 1$ to $\mathds{T}_{0,p}^\et$, we observe that these constructions admit globalizations over $\mb{E}$ in the obvious way, in which case it suffices to check the claim on $\C$-points. But, this then follows from \cite[Theorem 7.4]{MilneShimura}. 
\end{proof}

As a result of Proposition \ref{prop:comparison-of-otorsors-Hodge-type}, we see that $\mathds{T}_{0,p}^\et$ actually lies in the image of the injective map $\mc{H}^1_{\Z_p}(A_{{\mathsf{K}^p}}/\Sh_{{\mathsf{K}_0\mathsf{K}^p}})^\vee\to \mc{H}^1_{\Q_p}(A_{{\mathsf{K}^p}}/\Sh_{{\mathsf{K}_0\mathsf{K}^p}})^\vee$. We deduce from the contents of \cite[\S2.1.1]{IKY1} that $\omega_{{\mathsf{K}^p},\et}$ is the object of $\GLoc_{\Z_p}(\Sh_{{\mathsf{K}_0\mathsf{K}^p}})$ associated to the torsor
\begin{equation*}
    \underline{\Isom}\left((\Lambda_0\otimes_{\Z_p}\underline{\Z_p},\mathds{T}_0\otimes 1),(\mc{H}^1_{\Z_p}(A_{{\mathsf{K}^p}}/\Sh_{{\mathsf{K}_0\mathsf{K}^p}})^\vee,\mathds{T}_{0,p}^\et)\right),
\end{equation*}
which thus  is independent of any choices. Similar claims may be verified for $\nu_{\mathsf{K},\et}$.

We end this section by describing a method, applying ideas from \cite[\S4.6--4.7]{LoveringModels}, which will allow us to reduce statements about Shimura data of abelian type to those of Hodge type.

\begin{lem}\label{lem:Lovering-lem}Let $(\mb{G},\mb{X},\mc{G})$ be an unramified Shimura datum of abelian type, and $(\mb{G}_1,\mb{X}_1,\mc{G}_1)$ an adapted unramified Shimura datum of Hodge type. Then, there exists an unramified Shimura datum $(\mb{T},\{h\},\mc{T})$ of special type and an unramified Shimura datum $(\mb{G}_2,\mb{X}_2,\mc{G}_2)$ of abelian type, both with reflex field $\mb{E}_1$, such that 
\begin{enumerate}
    \item there exists a morphism of unramified Shimura data
    \begin{equation*}
        \alpha=(\alpha^1,\alpha^2)\colon (\mb{G}_2,\mb{X}_2,\mc{G}_2)\to (\mb{G}_1\times \mb{T},\mb{X}_1\times\{h\},\mc{G}_1\times\mc{T}) 
    \end{equation*}
    such that $\alpha^c\colon \mc{G}_2^c\to \mc{G}_1\times \mc{T}^c$ is a closed embedding,
    \item there exists a morphism of unramified Shimura data $\beta\colon (\mb{G}_2,\mb{X}_2,\mc{G}_2)\to(\mb{G},\mb{X},\mc{G})$ such that $\beta_{\mathsf{K}^p_2,\mathsf{K}^p}\colon \ms{S}_{\mathsf{K}^p_2}\to(\ms{S}_{\mathsf{K}^p})_{\mc{O}_{E_1}}$ is finite \'etale.
\end{enumerate}
\end{lem}
\begin{proof} Choose a connected component $\mb{X}_1^+$ of $\mb{X}_1$ and an element $h_G$ of $\mb{X}_1^+$. Let $(\mb{G}_2,\mb{X}_2,\mc{G}_2)$ be the unramifed Shimura datum of abelian type obtained by applying the construction in \cite[\S4.6]{LoveringModels} to $(\mb G_1,\mb X_1,\mc G_1)$ and the choice of $\mb X_1^+$, and set 
\begin{equation*}
    (\mb{T},\{h\},\mc{T})\defeq \left(\mathrm{Res}_{\mb{E}_1/\Q}
    \,\,\mathbb{G}_{m,\mb{E}_1},\{h_E\}, \left(\mathrm{Res}_{\mathcal{O}_{\mb{E}_1}/\mathbb{Z}}\,\,\mathbb{G}_{m,\mathcal{O}_{\mb{E}_1}}\right)_{\Z_p}\right),
\end{equation*} 
with $h_E$ as in \cite[4.6.4]{LoveringModels}. The map $\alpha$ is then constructed from the natural inclusion of $\mb{G}_2=\mb{G}_1\times_{\mb{G}_1^\mathrm{ab}}\mb{T}$ into $\mb{G}_1\times\mb{T}$. To prove that $\alpha^c$ is a closed embedding, observe that as $(\mb{G}_1,\mb{X}_1)$ is of Hodge type that $(\mb{Z}_1)_\ac$ is trivial, and as $\mb{Z}_1\to \mb{G}_1^\mathrm{ab}$ is an isogeny that $(\mb{G}_1^\mathrm{ab})_\ac$ is also trivial. From this we deduce that $\mb{Z}_2^c=\mb{T}_\ac$ and so the claim follows. The map $\beta$ is constructed as in \cite[\S4.7.2]{LoveringModels}, the second claim follows from Lemma \ref{lem:isogeny-finite-etale}.
\end{proof}

\subsection{Prismatic and syntomic 
realization functors}\label{ss:prismatic-realization-functors} For an unramified Shimura datum $(\mb{G},\mb{X},\mc{G})$ of abelian type, and a neat compact open subgroup $\mathsf{K}^p\subseteq\mb{G}(\A_f^p)$, we associate the smooth $p$-adic formal scheme $\wh{\ms{S}}_{\mathsf{K}^p}$, and the open embedding
\begin{equation*}
    \mc{S}_{{\mathsf{K}^p}}\defeq (\wh{\ms{S}}_{{\mathsf{K}^p}})_\eta\subseteq \Sh_{{\mathsf{K}^p}}^\an,
\end{equation*}
with quasi-compact source, which is an isomorphism when $\ms{S}_{\mathsf{K}^p}\to \Spec(\mc{O}_E)$ is proper (see \cite[Remark 4.6 (iv)]{HuberGen}).  The morphisms $t_{{\mathsf{K}^p},{\mathsf{K}^{'p}}}(g^p)$ induce morphisms on the adic spaces $\mc{S}_{{\mathsf{K}^p}}$ compatible with those maps on the $\Sh_{{\mathsf{K}^p}}^\an$, and we use similar notational shortenings for them. We may also consider the functors
\begin{equation*}
    \omega_{{\mathsf{K}^p},\an}\colon \cat{Rep}_{\Z_p}(\mc{G}^c)\to \cat{Loc}_{\Z_p}(\mc{S}_{{\mathsf{K}^p}}),\qquad \Lambda\mapsto \omega_{{\mathsf{K}^p},\et}(\Lambda)^\an|_{\mc{S}_{{\mathsf{K}^p}}},
\end{equation*}
which enjoy the same compatabilities for varying level structure and morphisms of unramified Shimura varieties as the $\mc{G}^c$-local systems $\omega_{\mathsf{K}^p,\et}$. 

\medskip

\paragraph*{Prismatic realization functors} By a \emph{prismatic ($F$-crystal) realization functor} at level $\mathsf{K}^p$, we mean an exact $\Z_p$-linear $\otimes$-functor (unique up to unique isomorphism, as $T_\et$ is fully faithful)
\begin{equation*}
\omega_{{\mathsf{K}^p},\smallprism}\colon \cat{Rep}_{\Z_p}(\mc{G}^c)\to \cat{Vect}^\varphi((\widehat{\ms{S}}_{{\mathsf{K}^p}})_\smallprism),
\end{equation*}
together with $\j_{\mathsf{K}^p}\colon T_\et\circ \omega_{{\mathsf{K}^p},\smallprism}\isomto \omega_{{\mathsf{K}^p},\an}$. If the isomorphisms $\j_{\mathsf{K}^p}$ are chosen compatibly in $\mathsf{K}^p$, we call the collection $\{(\omega_{\mathsf{K}^p,\smallprism},\j_{\mathsf{K}^p})\}$ a \emph{prismatic ($F$-crystal) canonical model} of $\{\omega_{{\mathsf{K}^p},\an}\}$, which is unique up to unique isomorphism. We often omit the data of $\j_{{\mathsf{K}^p}}$ from the notation. 

Fix prismatic canonical models $\{\omega_{\mathsf{K}^p,\smallprism}\}$ and $\{\omega_{\mathsf{K}^p_1,\smallprism}\}$ for unramified Shimura data $(\mb{G},\mb{X},\mc{G})$ and $(\mb{G}_1,\mb{X}_1,\mc{G}_1)$ respectively. If $\alpha\colon (\mb{G}_1,\mb{X}_1,\mc{G}_1)\to (\mb{G},\mb{X},\mc{G})$ is a morphism, then for any $\xi$ in $\cat{Rep}_{\Z_p}(\mc{G}^c)$, and neat compact open subgroups $\mathsf{K}^p\subseteq\mb{G}(\A_f^p)$ and $\mathsf{K}^p_1\subseteq\mb{G}_2(\A_f^p)$ with $\alpha({\mathsf{K}^p_1})\subseteq {\mathsf{K}^p}$, one has canonical, compatible in $\mathsf{K}^p$, $\mathsf{K}^p_1$, and $\xi$, identifications
    \begin{equation}\label{eq:compatibility-prismatic-model-pullback}
        \alpha_{{\mathsf{K}^p_1},{\mathsf{K}^p}}^\ast(\omega_{{\mathsf{K}^p},\smallprism}(\xi)_{\mc{O}_1})=\omega_{{\mathsf{K}^p_1},\smallprism}(\xi\circ \alpha^c).
\end{equation}
This follows by appropriately applying $T_\et^{-1}$ and the isomorphisms $\j_{\mathsf{K}^p}$ and $\j_{\mathsf{K}^p_1}$ to  \eqref{eq:shimura-data-morphism-local-system-compat}.

\begin{thm}\label{thm:main-Shimura-theorem-abelian-type-case} Suppose that $(\mb{G},\mb{X},\mc{G})$ is an unramified Shimura datum of abelian type. Then, for any ${\mathsf{K}^p}$, the functor $\omega_{{\mathsf{K}^p},\an}$ takes values in $\cat{Loc}^\strcrys_{\Z_p}(\mc{S}_{{\mathsf{K}^p}})$. In particular, the collection
\begin{equation*}
    T_\et^{-1}\circ \omega_{{\mathsf{K}^p},\an}\efdeq \omega_{{\mathsf{K}^p},\smallprism}\colon \cat{Rep}_{\Z_p}(\mc{G}^c)\to \cat{Vect}^\varphi((\wh{\ms{S}}_{{\mathsf{K}^p}})_\smallprism)
\end{equation*}
forms a prismatic canonical model of $\{\omega_{{\mathsf{K}^p},\an}\}$.
\end{thm}

\begin{rem}\label{rem:Daniels-comp} If $(\mb{G},\mb{X},\mc{G})$ is of special type this theorem was (implicitly) obtained by Daniels in \cite{Daniels}, and his construction agrees with ours by the unicity of canonical prismatic models.
\end{rem}

We first prove a refined version of this theorem when $(\mb{G},\mb{X},\mc{G})$ is of Hodge type using \cite[Theorem 2.28]{IKY1}. Choose an integral Hodge embedding $\iota\colon (\mb{G},\mb{X},\mc{G})\to (\GSp(\mb{V}_0),\mf{h}^{\pm},\GSp(\Lambda_0))$ and write $\mc{A}_{{\mathsf{K}^p}}\to \mc{S}_{{\mathsf{K}^p}}$ for the generic fiber of $\wh{\ms{A}}_{{\mathsf{K}^p}}\to \wh{\ms{S}}_{{\mathsf{K}^p}}$. Define 
\begin{equation*}
    \mathds{T}_{0,p}^\an\defeq (\mathds{T}_{0,p}^\et)^\an|_{\mc{S}_{{\mathsf{K}^p}}}\subseteq (\mc{H}^1_{\Z_p}(\mc{A}_{{\mathsf{K}^p}}/\mc{S}_{{\mathsf{K}^p}})^\vee)^\otimes.
\end{equation*}
By Proposition \ref{prop:comparison-of-otorsors-Hodge-type}, we have a canonical identification
\begin{equation*}
    (\omega_{{\mathsf{K}^p},\an}(\Lambda_0),\omega_{{\mathsf{K}^p},\an}(\mathds{T}_0))\isomto (\mc{H}^1_{\Z_p}(\mc{A}_{{\mathsf{K}^p}}/\mc{S}_{{\mathsf{K}^p}})^\vee,\mathds{T}_{0,p}^\an).
\end{equation*}
Combining \cite[Corollary 4.64]{AnschutzLeBrasDD} and \cite[Theorem 1.10 (i)]{GuoReinecke}, we deduce that $\mc{H}^1_{\Z_p}(\mc{A}_{{\mathsf{K}^p}}/\mc{S}_{{\mathsf{K}^p}})^\vee$ has prismatically good reduction with a canonical identification
\begin{equation*}
    T_\et^{-1}(\mc{H}^1_{\Z_p}(\mc{A}_{{\mathsf{K}^p}}/\mc{S}_{{\mathsf{K}^p}})^\vee)=\mc{H}^1_\smallprism(\wh{\ms{A}}_{{\mathsf{K}^p}}/\wh{\ms{S}}_{{\mathsf{K}^p}})^\vee,
\end{equation*}
compatible in $\mathsf{K}^p$. Applying $T_\et^{-1}$ to $\mathds{T}^\an_{0,p}$ gives rise to a set $\mathds{T}_{0,p}^\smallprism$ of tensors on the object $\mc{H}^1_\smallprism(\wh{\ms{A}}_{{\mathsf{K}^p}}/\wh{\ms{S}}_{{\mathsf{K}^p}})^\vee$ of $\cat{Vect}^\varphi((\wh{\ms{S}}_{{\mathsf{K}^p}})_\smallprism)$. The following is a consequence of \cite[Theorem 2.28]{IKY1}, and immediately implies Theorem \ref{thm:main-Shimura-theorem-abelian-type-case} for $(\mb{G},\mb{X},\mc{G})$ of Hodge type by \cite[Proposition 1.28]{IKY1}.

\begin{thm}\label{thm:main-Shimura-theorem-Hodge-type-case}
    Suppose that $(\mb{G},\mb{X},\mc{G})$ is an unramified Shimura datum of Hodge type. Then, 
    \begin{equation*}
        \underline{\Isom}\left((\Lambda_0\otimes_{\Z_p}\mc{O}_{\smallprism},\mathds{T}_0\otimes 1),(\mc{H}^1_\smallprism(\wh{\ms{A}}_{{\mathsf{K}^p}}/\wh{\ms{S}}_{{\mathsf{K}^p}})^\vee,\mathds{T}^\smallprism_{0,p})\right)
    \end{equation*}
is a prismatic $\mc{G}$-torsor with $F$-structure on $(\wh{\ms{S}}_{{\mathsf{K}^p}})_\smallprism$, compatible in ${\mathsf{K}^p}$. 
\end{thm}

To reduce from the abelian type case to the Hodge type case, we require a simple lemma concerning prismatically good reduction local systems. We use the notation from \S\ref{s:G-objects-prismatic-crystals}.

\begin{lem}\label{lem:strongly-crystalline-check-on-finite-etale-cover-of-components}
    Suppose that $\mf{X}_2\to\mf{X}_1$ is a finite \'etale cover where $\mf{X}_1\to\Spf(\mc{O}_K)$ is smooth. Then, an object $\bb{L}_1$ of $\cat{Loc}_{\Z_p}(X_1)$ has prismatically good reduction if and only if $\bb{L}_2\defeq \bb{L}_1|_{X_2}$ does.
\end{lem}
\begin{proof} It suffices to prove the if condition. We may assume that $\mf{X}_i=\Spf(R_i)$ where $R_i$ are (framed) small $\mc{O}_K$-algebras. That $\bb{L}_1$ is crystalline is clear. Let $(\mc{V}^i,\varphi_{\mc{V}^i})$ denote the object $T^{-1}_{\mf{X}_i}(\bb{L}_i)$ of $\cat{Vect}^{\an,\varphi}((\mf{X}_i)_\smallprism)$. Then,  by the flatness of $\mf{X}_2\to\mf{X}_1$, we have that
\begin{equation*}
    (j_{(\mf{S}_{R_1},(E))})_\ast \mc{V}^1_{(\mf{S}_{R_1},(E))}\otimes_{\mf{S}_{R_1}}\mf{S}_{R_2}=(j_{(\mf{S}_{R_2},(E))})_\ast \mc{V}^2_{(\mf{S}_{R_2},(E))}.
\end{equation*}
The right-hand side is a vector bundle by \cite[Proposition 1.26]{IKY1} and thus so is the sheaf $(j_{(\mf{S}_{R_1},(E))})_\ast \mc{V}^1_{(\mf{S}_{R_1},(E))}$. Thus, $\bb{L}_1$ has prismatically good reduction again by loc.\@ cit.
\end{proof}

\begin{proof}[Proof of Theorem \ref{thm:main-Shimura-theorem-abelian-type-case}] We freely use notation from Lemma \ref{lem:Lovering-lem}. We first prove the claim for $(\mb{G}_2,\mb{X}_2,\mc{G}_2)$. Choosing a faithful representation $\xi_2'=\xi_1\otimes \xi_t$ of $\mc{G}_1\times\mc{T}^c$, where $\xi_1$ (resp.\@ $\xi_t$) is a faithful representation of $\mc{G}_1$ (resp.\@ $\mc{T}^c$), we obtain the faithful representation $\xi_2\defeq \xi_2'\circ\alpha^c$ of $\mc{G}_2^c$. Choosing neat compact open subgroups $\mathsf{K}_1^p\subseteq\mb{G}_1(\A_f^p)$ and $\mathsf{K}^p_t\subseteq \mb{T}(\A_f^p)$ such that $\alpha(\mathsf{K}^p_2)\subseteq {\mathsf{K}^p_1}\times {\mathsf{K}^p_t}$, we see from \eqref{eq:shimura-data-morphism-local-system-compat} that
\begin{equation*}
    \omega_{\mathsf{K}^p_2,\an}(\xi_2)=(\alpha^1_{\mathsf{K}^p_2K_{0,2},{\mathsf{K}^p_1}\mathsf{K}_{0,1}})^\ast(\omega_{{\mathsf{K}^p_1},\an}(\xi_1))\otimes_{\underline{\Z_p}}(\alpha^2_{\mathsf{K}^p_2 K_{0,2},\mathsf{K}^p_t\mathsf{K}_{0,t}})^\ast(\omega_{\mathsf{K}^p_t,\an}(\xi_t)).
\end{equation*}
But, $\omega_{{\mathsf{K}^p},\an}(\xi_1)$ has prismatically good reduction by Theorem \ref{thm:main-Shimura-theorem-Hodge-type-case}. Moreover, as $\wh{\ms{S}}_{\mathsf{K}^p_t}$ is of the form $\coprod \Spf(\mc{O}_{E'})$, for connected finite \'etale $\mc{O}_E$-algebras $\mc{O}_{E'}$ (e.g.\@ see \cite[Proposition 3.22]{DanielsYoucis}), $\omega_{\mathsf{K}^p_t,\an}(\xi_t)$ has prismatically good reduction by \cite[Proposition 3.7]{GuoReinecke}. Thus, as having prismatically good reduction is preserved by pullbacks and tensor products, the claim follows from \cite[Corollary 2.30]{IKY1}.

Now, to prove the claim for $(\mb{G},\mb{X},\mc{G})$ it suffices to prove that for each compact open subgroup $\mathsf{K}^p$ and each connected component $\ms{C}$ of $\ms{S}_{\mathsf{K}^p}$ that $\omega_{\mathsf{K}^p,\an}(\xi)|_{\wh{\ms{C}}_\eta}$ has prismatically good reduction. By Lemma \ref{lem:transitivity-on-conn-comp} and Equation \eqref{eq:Hecke-action-local-system-compat}, we may assume that there exists some neat compact open subgroup $\mathsf{K}_2^p\subseteq \mb{G}_2(\A_f^p)$ such that $\beta(\mathsf{K}^p_2)\subseteq \mathsf{K}^p$ and $\ms{C}$ lies in the image of $\beta_{\mathsf{K}^p_2,\mathsf{K}^p}$. As $\beta_{\mathsf{K}^p_2,\mathsf{K}^p}$ is finite \'etale the claim follows from \eqref{eq:shimura-data-morphism-local-system-compat} and Lemma \ref{lem:strongly-crystalline-check-on-finite-etale-cover-of-components}. 
\end{proof}

\medskip

\paragraph*{Syntomic realization functor} We now discuss the existence of an upgrade of $\omega_{\mathsf{K}^p,\smallprism}$ to a \emph{syntomic realization functor}. By a \emph{syntomic realization functor} at level $\mathsf{K}^p$, we mean an object $\omega_{\mathsf{K}^p,\mr{syn}}$ of $\mc{G}^c\text{-}\cat{Vect}((\wh{\ms{S}}_{\mathsf{K}^p})^\mr{syn})$ (which is unique up to unique isomorphism, as $\mr{R}_{\wh{\ms{S}}_{\mathsf{K}^p}}$ is fully faithful) with $\j_{\mathsf{K}^p}\colon \mr{R}_{\wh{\ms{S}}_{\mathsf{K}^p}}\circ \omega_{{\mathsf{K}^p},\mr{syn}}\isomto \omega_{{\mathsf{K}^p},\smallprism}$. Such $\j_{\mathsf{K}^p}$ are unique if they exist, and so we often omit them from the notation. The collection $\{\omega_{\mathsf{K}^p,\syn}\}$ is called a \emph{syntomic canonical model} of $\{\omega_{\mathsf{K}^p,\an}\}$.

By Proposition \ref{prop:F-gauge-lff-equiv} to show such $\omega_{\mathsf{K}^p,\mr{syn}}$ exist it suffices to show that $\omega_{\mathsf{K}^p,\smallprism}$ takes values in $\cat{Vect}^{\varphi,\mr{lff}}((\wh{\ms{S}}_{\mathsf{K}^p})_\smallprism)$. This is true (see Corollary \ref{cor:prismatic-realization-lff}) but we delay its proof until \S\ref{ss:comparison-to-shim-vars}.

\begin{thm}\label{thm:prismatic-F-gauge-realization} Suppose that $(\mb{G},\mb{X},\mc{G})$ is an unramified Shimura datum of abelian type. Then, for any ${\mathsf{K}^p}$, the functor $\omega_{{\mathsf{K}^p},\smallprism}$ is of type $-\mu_h^c$, and so takes values in $\cat{Vect}^{\varphi,\mr{lff}}((\wh{\ms{S}}_{\mathsf{K}^p})_\smallprism)$. In particular, the collection
\begin{equation*}
    \mr{R}_{\wh{\ms{S}}_{\mathsf{K}^p}}^{-1}\circ \omega_{{\mathsf{K}^p},\smallprism}\efdeq \omega_{{\mathsf{K}^p},\mr{syn}}\colon \cat{Rep}_{\Z_p}(\mc{G}^c)\to \cat{Vect}((\wh{\ms{S}}_{\mathsf{K}^p})^\mr{syn})
\end{equation*}
forms a syntomic canonical model of $\{\omega_{{\mathsf{K}^p},\an}\}$.
\end{thm}

\subsection{Potentially crystalline loci and comparison of stratifications}\label{ss:pot-crys-strat}

Let $K$ be a complete discrete valuation field with perfect residue field, and let $X$ be a quasi-separated adic space locally of finite type over $K$, and $\Sigma$ be either $\Z_p$ or $\Q_p$. For an object $\bb{L}$ of $\cat{Loc}_{\Sigma}(X)$ we call a point $x$ of $|X|^\mathrm{cl}$ \emph{(potentially) crystalline} for $\bb{L}$ if $\bb{L}_x$ is a (potentially) crystalline representation of $\Gamma_{k(x)}$. 

There exists at most one quasi-compact open subset $U\subseteq X$ such that $|U|^\mathrm{cl}$ is the set of potentially crystalline points of $\bb{L}$ (cf.\@ \cite[Corollary 4.3]{HuberCV}). In this case we call $U$ the \emph{potentially crystalline locus} of $\bb{L}$. For a $K$-scheme $S$ locally of finite type, and an object $\bb{L}$ of $\cat{Loc}_{\Sigma}(S)$, if we speak of the potentially crystalline locus of $\bb{L}$ we mean the potentially crystalline locus of $\bb{L}^\mathrm{an}$. These definitions apply equal well for $\mc{G}$-objects (or $G$-objects) $\omega$ in these categories.\footnote{Although as observed in \cite[Proposition 2.20]{IKY1}, this will coincide with the set of potentially crystalline points for the value of $\omega$ on any faithful representation of $\mc{G}$ or $G$.} 

Observe that potentially crystalline points satisfy pullback stability: for a map $f\colon X'\to X$, a classical point $x'$ of $X'$ is potentially crystalline for $f^\ast(\bb{L})$ if and only if $x=f(x')$ is a potentially crystalline point for $\bb{L}$. If $k(x')/k(x)$ is unramified (e.g.\@ $f=\mf{f}_\eta$ for a finite \'etale model $\mf{f}\colon \mf{X}'\to\mf{X}$), one may replace `potentially crystalline' by `crystalline'.

For a Shimura datum $(\mb{G},\mb{X})$ of (pre-)abelian type, and a neat compact open subgroup $\mathsf{K}\subseteq \mb{G}(\A_f)$, the existence of a potentially crystalline locus $U_\mathsf{K}\subseteq \Sh_\mathsf{K}^\an$ for $\nu_{\mathsf{K},\et}$ was established in \cite[Theorem 5.17]{ImaiMieda} (see \cite[Remark 2.12]{ImaiMieda} and \cite[Theorem 1.2]{LiuZhu}). If $(\mb{G},\mb{X},\mc{G})$ is an unramified Shimura datum of abelian type, and $\mathsf{K}=\mathsf{K}_0\mathsf{K}^p$, we abbreviate $U_{\mathsf{K}}$ to $U_{\mathsf{K}^p}$ which coincides with the potentially crystalline locus of $\omega_{\mathsf{K}^p,\et}$.

We now describe $U_{\mathsf{K}^p}$ for unramified Shimura data of abelian type, generalizing results of Imai--Mieda in the PEL setting (see \cite[Corollary 2.11 and Proposition 5.4]{ImaiMieda} and \cite[\S7]{ImaiMiedaRIMS}).

\begin{prop}\label{prop:crystalline-locus}
    Let $(\mb{G},\mb{X},\mc{G})$ be an unramified Shimura datum of abelian type, and $\mathsf{K}^p\subseteq\mb{G}(\A_f^p)$ a neat compact open subgroup. Then, $U_{\mathsf{K}^p}=(\wh{\ms{S}}_{\mathsf{K}^p})_\eta$ and all the classical points of $U_{\mathsf{K}^p}$ are crystalline for $\omega_{{\mathsf{K}^p},\et}$. 
\end{prop}
\begin{proof}
We know that all the classical points of $\mc{S}_{{\mathsf{K}^p}}$ are crystalline for $\omega_{{\mathsf{K}^p},\et}$ by Theorem 
\ref{thm:main-Shimura-theorem-abelian-type-case}. 
Moreover, the full claim holds in the Siegel-type case by \cite[Theorem 5.17]{ImaiMieda} and its proof.

Assume that $(\mb{G},\mb{X},\mc{G})$ is of Hodge type and choose an integral Hodge embedding $\iota\colon (\mb{G},\mb{X},\mc{G})\hookrightarrow (\GSp(\mb{V}_0),\mf{h}^{\pm},\GSp(\Lambda_0))$, and a level $\mathsf{L}^p$ with $\iota(\mathsf{K}^p)\subseteq \mathsf{L}^p$ and $\iota_{\mathsf{K}_0\mathsf{K}^p,\mathsf{L}_0\mathsf{L}^p}$ is a closed embedding. By the construction of $\ms{S}_{\mathsf{K}^p}$ (see \cite[Theorem (2.3.8)]{KisIntShab}), $\ms{S}_{\mathsf{K}^p}$ is obtained as the normalization of a closed subscheme of $\ms{M}_{\mathsf{L}^p}(\Lambda_0)$, and so finite over $\ms{M}_{\mathsf{L}^p}(\Lambda_0)$.\footnote{More precisely, as a closed subscheme of $\ms{M}_{\mathsf{L}^p}(\Lambda_0)$ is finite type over $\mc{O}_E$, it is excellent (see \stacks{07QW}), and thus Nagata (see \stacks{07QV}) and so one may apply \stacks{035S}.} So, if $x$ is a classical point of $\Sh_{{\mathsf{K}^p}}^{\an}$ not in $\mc{S}_{\mathsf{K}^p}$ then the image of $x$ in $\Sh_{\mathsf{L}^p}^{\an}$ is a point outside $\mc{S}_{{\mathsf{L}^p}}$ (see \cite[Proposition 1.9.6]{HuberEC}). Hence $x$ is not a potentially crystalline point by the Siegel case, and pullback stability.

Assume now that $(\mb{G},\mb{X},\mc{G})$ is of abelian type. 
We use notation from Lemma \ref{lem:Lovering-lem}. 
Let $x$ be a classical point of $\Sh_{{\mathsf{K}^p}}^{\an}$ not in $\mc{S}_{{\mathsf{K}^p}}$. By Lemma \ref{lem:transitivity-on-conn-comp} and Equation \eqref{eq:Hecke-action-local-system-compat}, we may assume that 
there is a lift $x_2$ of $x$ in $\Sh_{\mathsf{K}_2^p}^{\an}$, but not in $\mc{S}_{{\mathsf{K}_2^p}}$, for some neat compact open subgroup $\mathsf{K}_2^p \subset \mb{G}_2 (\A_f^p)$. The image $x_1$ of $x_2$ in $\Sh_{\mathsf{K}_1^p}^{\an}$ for an appropriate $\mathsf{K}_1^p \subset \mb{G}_1 (\A_f^p)$ is a point outside $\mc{S}_{{\mathsf{K}_1^p}}$ because $\ms{S}_{\mathsf{K}_2^p}$ is finite over $\ms{S}_{\mathsf{K}_1^p}$ (cf.\@ \cite[Proposition 1.9.6]{HuberEC}). Take a faithful representation $\xi_1^{\ad}$ of $\mb{G}_1^{\ad}$, inducing faithful representations $\xi^{\ad}$ and $\xi_2^{\ad}$ of $\mb{G}^{\ad}$ and $\mb{G}_2^{\ad}$ respectively as $\mb{G}^{\ad} \cong  \mb{G}_2^{\ad} \cong \mb{G}_1^{\ad}$. By \cite[Corollary 2.11 and Proposition 5.4]{ImaiMieda}, $\omega_{{\mathsf{K}^p_1},\an}(\xi_1^{\ad})_{x_1}$ is not potentially crystalline. Hence $\omega_{{\mathsf{K}^p},\an}(\xi^{\ad})_{x}$ is not potentially crystalline too by pullback stability, as the pullbacks of $\omega_{{\mathsf{K}^p},\an}(\xi^{\ad})_{x}$ and $\omega_{{\mathsf{K}^p_1},\an}(\xi_1^{\ad})_{x_1}$ to $x_2$ are both isomorphic to $\omega_{\mathsf{K}_2^p,\an}(\xi_2^{\ad})_{x_2}$. Thus, $x$ is not potentially crystalline.
\end{proof}

For neat compact open subgroups $K^p\subseteq \mb{G}(\A_f^p)$ and $K_p\subseteq G(\Q_p)$, we may define functions
\begin{equation*}
    \begin{aligned}\Sigma_{\mathsf{K}_p\mathsf{K}^p} &\colon |U_{\mathsf{K}_p\mathsf{K^p}}|^\mathrm{cl}\to B(G^c)\\ \bigg(\text{resp. }\Sigma^\circ_{\mathsf{K}^p} &\colon |U_{\mathsf{K}_0\mathsf{K}^p}|^\mathrm{cl}\to C(\mc{G}^c)
    \bigg)\end{aligned}
\end{equation*}
(where $C(\mc{G}^c)$ is the quotient of $G^c(\breve{\Q}_p)$ by the action of $\mc{G}^c(\breve{\Z}_p)$ by $\sigma$-conjugacy),  associating to $x$ the element of $B(G^c)$ (resp.\@ $C(\mc{G}^c)$) associated with the $F$-isocrystal (resp.\@ $F$-crystal) with $G^c$-structure (resp.\@ $\mc{G}^c$-structure) given by $\underline{D}_\crys\circ (\nu_{\mathsf{K}_p\mathsf{K}^p,\et})_x$ (resp.\@ $\underline{\bb{D}}_\crys\circ(\omega_{\mathsf{K}^p,\et})_x$) (see \cite[Example 1.5]{IKY3} for this latter notation). These functions are equivariant via the map $\Gamma_E\to \Gamma_k$, when the source (resp.\@ target) is endowed with the natural action of $\Gamma_E$ (resp.\@ $\Gamma_k$). On the other hand, we may define functions 
 \begin{equation*}
    \begin{aligned}\ov{\Sigma}_{\mathsf{K}^p} &\colon \ms{S}_{\mathsf{K}^p}(\ov{k})\to B(G^c)\\ \bigg(\text{resp. }\ov{\Sigma}^\circ_{\mathsf{K}^p} &\colon \ms{S}_{\mathsf{K}^p}(\ov{k})\to C(\mc{G}^c)
    \bigg)\end{aligned}
\end{equation*}
in the analogous way using the the $\mc{G}$-object in $\cat{Vect}^\varphi((\ms{S}_{\mathsf{K}^p,k})_\crys)$ given by $\underline{\bb{D}}_\crys\circ\omega_{\mathsf{K^p},\smallprism}$ which is equivariant with respect to the actions of $\Gamma_k$.

In the following, we use the notion of an \emph{overconvergent} (also known as \emph{wide}, \emph{partially proper}, or \emph{Berkovich}) open subset of a rigid $E$-space, as in \cite[Chapter II, \S4.3]{FujiwaraKato}
.

\begin{prop}\label{prop:factorization-through-specialization} The functions $\Sigma_{\mathsf{K}_p\mathsf{K}^p}$ and $\Sigma^\circ_{\mathsf{K}^p}$ are overconvergent locally constant.\footnote{i.e., for every classical point $x$ there exists an overconvergent open neighborhood $U_x\subseteq U_{\mathsf{K}_p\mathsf{K}^p}$ such that these functions are constant on $|U_x|^\mathrm{cl}$.\label{footnote:oc-local-constancy}} Moreover, 
\begin{equation}\label{eq:specialization-equations}
    \Sigma_{\mathsf{K}_0\mathsf{K}^p}=\overline{\Sigma}_{\mathsf{K}^p}\circ \sp,\quad \Sigma^\circ_{\mathsf{K}^p}=\overline{\Sigma}^\circ_{\mathsf{K}^p}\circ\sp,
\end{equation}
where $\sp\colon |(\wh{\ms{S}}_{\mathsf{K}^p})_\eta|^\mathrm{cl}=|U_{\mathsf{K}_0\mathsf{K}^p}|^\mathrm{cl}\to \ms{S}_{\mathsf{K}^p}(\ov{k})$ is the specialization map.
\end{prop}
\begin{proof} The second equality in \eqref{eq:specialization-equations} follows essentially by construction. To prove the first equality, it suffices to check that for a point $x$ of $|U_{\mathsf{K}_0\mathsf{K}^p}|^\mr{cl}$ the $F$-isocrystals with $G$-structure given by $\underline{D}_\crys\circ (\nu_{\mathsf{K}_0\mathsf{K}^p})_x$ and that induced by $\underline{\bb{D}}_\mathrm{crys}\circ(\omega_{\mathsf{K}^p,\et})_x$ agree. It suffices to find an isomorphism between their values at $\Lambda_0$ matching the tensors $\mathds{T}_0$. But, this can be reduced to the second equality in \eqref{eq:specialization-equations} considering \eqref{eq:Hecke-action-local-system-compat}. 

The claim concerning overconvergent local constancy for $\Sigma_{\mathsf{K}_0\mathsf{K}^p}$ and $\Sigma^\circ_{\mathsf{K}^p}$ follows from the equations in \eqref{eq:specialization-equations} as the the tube open subsets $\sp^{-1}(x)^\circ$, for $x$ in $\ms{S}_{\mathsf{K}^p}(\ov{k})$, are overconvergent open and contain every classical point of $(\wh{\ms{S}}_{\mathsf{K}^p})_\eta$ (see \cite[Proposition 2.13]{ALYSpecialization}). To prove that $\Sigma_{\mathsf{K}_p\mathsf{K}^p}$ is overconvergent locally constant for all $\mathsf{K}_p$, observe that for $\mathsf{K}_p'\subseteq\mathsf{K}_p$ we have that \begin{equation*}
    \Sigma_{\mathsf{K}_p'\mathsf{K}^p}=\Sigma_{\mathsf{K}_p\mathsf{K}^p}\circ\pi_{\mathsf{K}_p'\mathsf{K}^p,\mathsf{K}_p\mathsf{K}^p}.
\end{equation*}
Using this, and that $\pi_{\mathsf{K}_p'\mathsf{K}^p,\mathsf{K}_p\mathsf{K}^p}$ is finite \'etale and so preserves overconvergent opens under both preimage and image (cf.\@ \cite[p.\@ 427 (a)]{HuberEC}), one reduces to the previous case $\mathsf{K}_p=\mathsf{K}_0$.
\end{proof}

\begin{rem} Beware that the overconvergent local constancy in Proposition \ref{prop:factorization-through-specialization} does not imply constancy on each connected component as the overconvergent subset $\bigcup_x U_x$, with notation as in Footnote \ref{footnote:oc-local-constancy}, need not equal $U_{\mathsf{K}_p\mathsf{K}^p}$ (and may not even have the same number of components). This is because an overconvergent open subset of a rigid space $U$ which contains all classical points need not be all of $U$. For example, it is possible that this union could be $\mathrm{sep}^{-1}(U^\mathrm{Berk}-\{y\})$, where $\mathrm{sep}\colon U\to U^\mathrm{Berk}$ is the separation map as in \cite[Chapter 0, \S2.3.(c)]{FujiwaraKato}, and $y$ is a non-classical rank $1$ point. For example, compare with the overconvergent open subset of the closed unit ball $\mathbb{B}^1_{\mathbb{Q}_p}$ given by the complement of the \emph{closure} of the Gauss point, which is not even a connected subset.
\end{rem}

\subsection{Comparison with work of Lovering}\label{ss:lovering-comp} In this subsection we compare our work to that in \cite{LoveringFCrystals}, and derive several consequences about Shimura varieties of abelian type when $p>2$.

\subsubsection{Comparison result}\label{sss:Lovering-comp} In \cite{LoveringFCrystals}, Lovering constructs a so-called \emph{crystalline canonical model} $\{\omega_{{\mathsf{K}^p},\crys}\}$ of the system $\{\omega_{{\mathsf{K}^p},\an}\}$ for an unramified Shimura datum of abelian type. More precisely, with notation as in \cite[\S2.1.2]{IKY3}, he constructs exact $\Z_p$-linear $\otimes$-functors
\begin{equation*}
    \omega_{{\mathsf{K}^p},\crys}\colon \cat{Rep}_{\Z_p}(\mc{G}^c)\to \cat{VectF}^{\varphi,\mr{div}}((\wh{\ms{S}}_{{\mathsf{K}^p}})_\crys)
\end{equation*}
compatible in ${\mathsf{K}^p}$, together with compatible identifications of filtered objects of $\mc{G}^c\text{-}\cat{MIC}(\mc{S}_{{\mathsf{K}^p}})$:
\begin{equation*}
    i_{\mathsf{K}^p}\colon D_\dR\circ\omega_{{\mathsf{K}^p},\an}[\nicefrac{1}{p}]\isomto \omega_{{\mathsf{K}^p},\crys}[\nicefrac{1}{p}].
\end{equation*}
Moreover, he shows that for all finite unramified $E'/E$, and points $x\colon \Spf(\mc{O}_{E'})\to \wh{\ms{S}}_{{\mathsf{K}^p}}$ that:
\begin{enumerate}[leftmargin=2cm]
    \item[\textbf{(ICM1)}] for all $\xi$ in $\cat{Rep}_{\Z_p}(\mc{G}^c)$, the morphism of isocrystals on $\Spa(E')$
    \begin{equation*}
        i_{{\mathsf{K}^p},x}\colon ( D_\crys\circ\omega_{{\mathsf{K}^p},\an}[\nicefrac{1}{p}])(\xi)_x \isomto \omega_{{\mathsf{K}^p},\crys}[\nicefrac{1}{p}](\xi)_x,
    \end{equation*} 
    is Frobenius equivariant,\footnote{We are implicitly using the fact that for a $\Z_p$-local system $\bb{L}$ on a smooth rigid space $X$ with a smooth formal model over an unramified base, the underlying vector bundles $D_\dR(\bb{L})$ and $D_\crys(\bb{L})$ on $(X,\mc{O}_X)$ are the same. }
    \item[\textbf{(ICM2)}] for all $\xi$ in $\cat{Rep}_{\Z_p}(\mc{G}^c)$, the morphism $ i_{{\mathsf{K}^p},x}$ matches the lattice $\omega_{{\mathsf{K}^p},\crys}(\xi)_x$ with 
    \begin{equation*}
        M/uM\hookrightarrow ( D_\crys\circ\omega_{{\mathsf{K}^p},\an}[\nicefrac{1}{p}])(\xi)_x, 
    \end{equation*}
    where $M=\phi^\ast \mf{M}(\omega_{\mathsf{K}^p,\an}(\xi)_x)$, and this embedding is as in \cite[Theorem (1.2.1)]{KisIntShab}.
\end{enumerate}
As explained in \cite[Proposition 3.1.6]{LoveringFCrystals}, these conditions uniquely characterize $\{\omega_{{\mathsf{K}^p},\crys}\}$.

\begin{thm}\label{thm:prismatic-crystalline-comparison}
    There is an identification $\bb{D}_\crys\circ\omega_{{\mathsf{K}^p},\smallprism}\isomto \omega_{{\mathsf{K}^p},\crys}$ compatible in ${\mathsf{K}^p}$.
\end{thm}

To prove this, we will require the fact that $\omega_{\mathsf{K}^p,\smallprism}$ takes values in $\cat{Vect}^{\varphi,\mr{lff}}((\wh{\ms{S}}_{\mathsf{K}^p})_\smallprism)$ the proof of which, as mentioned before Theorem \ref{thm:prismatic-F-gauge-realization}, we delay until \S\ref{ss:comparison-to-shim-vars} (see Corollary \ref{cor:prismatic-realization-lff}).

\begin{proof}[Proof of Theorem \ref{thm:prismatic-crystalline-comparison}] It suffices to show that $\{\bb{D}_\crys\circ\omega_{{\mathsf{K}^p},\smallprism}\}$ is a crystalline canonical model. Using Corollary \ref{cor:prismatic-realization-lff}, the fact that it is an exact tensor functor valued in $\cat{VectF}^{\varphi,\mr{div}}(\mf{X}_\crys)$ follows from \cite[Proposition 2.16]{IKY3}. Furthermore, by \cite[Theorem 2.10]{IKY3}, there are isomorphisms in $\cat{IsocF}^\varphi((\wh{\ms{S}}_{\mathsf{K}^p})_\crys)$
\begin{equation}
    (\bb{D}_\crys\circ \omega_{{\mathsf{K}^p},\smallprism})[\nicefrac{1}{p}]\isomto D_\crys\circ (T_\et\circ\omega_{{\mathsf{K}^p},\smallprism})[\nicefrac{1}{p}]\xrightarrow{\j_{{\mathsf{K}^p}}} D_\crys \circ \omega_{{\mathsf{K}^p},\an}[\nicefrac{1}{p}],
\end{equation}
and we denote the inverses by $\j^\crys_{{\mathsf{K}^p}}$, which are compatible in ${\mathsf{K}^p}$. As $\j^\crys_{{\mathsf{K}^p}}$ is an isomorphism of filtered $F$-isocrystals, condition \textbf{(ICM1)} is automatic. Condition \textbf{(ICM2)} follows from the compatability of $T_\et$ with pullbacks, and \cite[Example 2.12]{IKY3}.
\end{proof}

We provide an explication of this result when $(\mb{G},\mb{X},\mc{G})$ is of Hodge type. Fix an integral Hodge embedding $\iota\colon (\mb{G},\mb{X},\mc{G})\to (\GSp(\mb{V}_0),\mf{h}^{\pm},\GSp(\Lambda_0))$ and a tensor package $(\Lambda_0,\mathds{T}_0)$ with $\mathrm{Fix}(\mathds{T}_0)=\mc{G}$. By \cite[\S3.1.2]{KimUnif} and \cite[Corollary 2.3.9]{KisIntShab}, one may construct tensors
\begin{equation*}
    \mathds{T}_{0,p}^\dR\subseteq (\mc{H}^1_\dR(\ms{A}_{{\mathsf{K}^p}}/\ms{S}_{{\mathsf{K}^p}})^\vee)^\otimes,
\end{equation*}
which are compatible in ${\mathsf{K}^p}$. By pulling back $\mathds{T}_{0,p}^\dR$, we obtain a set $\wh{\mathds{T}}_{0,p}^\dR$ of tensors on $\mc{H}^1_\dR(\wh{\ms{A}}_{{\mathsf{K}^p}}/\wh{\ms{S}}_{{\mathsf{K}^p}})^\vee$. Using the canonical isomorphism from \cite[Theorem 7.23 and Summary 7.26.3]{BerthelotOgus},
\begin{equation*}
   \mc{H}^1_\dR(\wh{\ms{A}}_{{\mathsf{K}^p}}/\wh{\ms{S}}_{{\mathsf{K}^p}})^\vee\isomto  \mc{H}^1_\crys(\wh{\ms{A}}_{{\mathsf{K}^p}}/\wh{\ms{S}}_{{\mathsf{K}^p}})^\vee_{\wh{\ms{S}}_{\mathsf{K}^p}}, 
\end{equation*}
of vector bundles with connection, we obtain tensors $\mathds{T}^\crys_{0,p}$ in $\mc{H}^1_\crys(\wh{\ms{A}}_{{\mathsf{K}^p}}/\wh{\ms{S}}_{{\mathsf{K}^p}})^\vee$. By \cite[Proposition 3.3.7]{KimUnif} these are tensors in $\mc{H}^1_\crys(\wh{\ms{A}}_{{\mathsf{K}^p}}/\wh{\ms{S}}_{{\mathsf{K}^p}})^\vee$ considered as an object of $\textbf{VectF}^\varphi(\wh{\ms{S}}_{{\mathsf{K}^p}})$. 

\begin{prop}\label{prop:prismatic-crystalline-tensor-matching} There is an isomorphism 
\begin{equation*}
    \bb{D}_\crys(\omega_{\mathsf{K}^p,\smallprism}(\Lambda_0))\isomto\mc{H}^1_\crys(\wh{\ms{A}}_{{\mathsf{K}^p}}/\wh{\ms{S}}_{{\mathsf{K}^p}})^\vee
\end{equation*}
in $\cat{VectF}^{\varphi,\mr{div}}(\wh{\ms{S}}_{\mathsf{K}^p})$, compatible in ${\mathsf{K}^p}$ and carrying $\bb{D}_\crys(\omega_{\mathsf{K}^p,\smallprism}(\mathds{T}_{0}))$ to $\mathds{T}_{0,p}^\crys$.
\end{prop}
\begin{proof} By Theorem \ref{thm:main-Shimura-theorem-Hodge-type-case}, there is an isomorphism $\omega_{\mathsf{K}^p,\smallprism}(\Lambda_0)\isomto \mc{H}^1_\smallprism(\wh{\ms{A}}_{{\mathsf{K}^p}}/\wh{\ms{S}}_{{\mathsf{K}^p}})^\vee$ in $\cat{Vect}^\varphi((\wh{\ms{S}}_{{\mathsf{K}^p}})_\smallprism)$, compatible in ${\mathsf{K}^p}$, and carrying $\omega_{\mathsf{K}^p,\smallprism}(\mathds{T}_0)$ to $\mathds{T}_{0,p}^\smallprism$. Thus, it suffices to construct an isomorphism 
\begin{equation*}
    \bb{D}_\crys\left(\mc{H}^1_\smallprism(\wh{\ms{A}}_{{\mathsf{K}^p}}/\wh{\ms{S}}_{{\mathsf{K}^p}})\right)^\vee\isomto \mc{H}^1_\crys(\wh{\ms{A}}_{{\mathsf{K}^p}}/\wh{\ms{S}}_{{\mathsf{K}^p}})^\vee,
\end{equation*}
in $\cat{VectF}^\varphi(\wh{\ms{S}}_{{\mathsf{K}^p}})$ carrying $\bb{D}_\crys(\mathds{T}_{0,p}^\smallprism)$ to $\mathds{T}_{0,p}^\crys$. That there is an isomorphism in $\cat{VectF}^\varphi(\wh{\ms{S}}_{\mathsf{K}^p})$ follows from 
\cite[Theorem 4.6.2]{AnschutzLeBrasDD} and \cite[(3.3.7.3)]{BBMDieuII} via \cite[Theorem 4.8]{IKY3}. 

Thus, it suffices to show that this isomorphism carries $\bb{D}_\crys(\mathds{T}_{0,p}^\smallprism)$ to $\mathds{T}_{0,p}^\crys$. As $\mc{H}^1_\crys(\wh{\ms{A}}_{{\mathsf{K}^p}}/\wh{\ms{S}}_{{\mathsf{K}^p}})^\vee$ is a vector bundle on $\wh{\ms{S}}_{{\mathsf{K}^p}}$, there is an injection 
\begin{equation*}
    \Gamma\left(\wh{\ms{S}}_{{\mathsf{K}^p}},\mc{H}^1_\crys(\wh{\ms{A}}_{{\mathsf{K}^p}}/\wh{\ms{S}}_{{\mathsf{K}^p}})^\vee\right) \to \Gamma\left(\mc{S}_{{\mathsf{K}^p}},\mc{H}^1_\crys(\wh{\ms{A}}_{{\mathsf{K}^p}}/\wh{\ms{S}}_{{\mathsf{K}^p}})^\vee_\eta\right),
\end{equation*}
and so it suffices to show that the images of these two sets of tensors agree. But, by \cite[Theorem 2.10]{IKY3} the image of $\bb{D}_\crys(\mathds{T}_{0,p}^\smallprism)$ may be identified with $D_\crys(T_\et(\mathds{T}_{0,p}^\smallprism))=D_\crys(\mathds{T}_{0,p}^\et)$. The claimed matching is then given by \cite[Proposition 3.3.7]{KimUnif}.
\end{proof}

By Theorem \ref{thm:main-Shimura-theorem-Hodge-type-case}, $\omega_{{\mathsf{K}^p},\smallprism}$ is associated with the prismatic $\mc{G}$-torsor with $F$-structure
\begin{equation*}\underline{\Isom}\left((\Lambda_0\otimes_{\Z_p}\mc{O}_{(\wh{\ms{S}}_{{\mathsf{K}^p}})_\smallprism},\mathds{T}_0\otimes 1), (\mc{H}^1_\smallprism(\wh{\ms{A}}_{{\mathsf{K}^p}}/\wh{\ms{S}}_{{\mathsf{K}^p}})^\vee,\mathds{T}^\smallprism_{0,p})\right),
\end{equation*}
compatibly in ${\mathsf{K}^p}$. Thus, by Proposition \ref{prop:prismatic-crystalline-tensor-matching}, we may identify $\bb{D}_\crys\circ\omega_{{\mathsf{K}^p},\smallprism}$ with 
\begin{equation}\label{eq:crystalline-realization}
     \underline{\Isom}\left((\Lambda_0\otimes_{\Z_p}\mc{O}_{\wh{\ms{S}}_{{\mathsf{K}^p}}/\mc{O}},\mathds{T}_0\otimes 1), (\mc{H}^1_\crys(\wh{\ms{A}}_{{\mathsf{K}^p}}/\wh{\ms{S}}_{{\mathsf{K}^p}})^\vee,\mathds{T}^\crys_{0,p})\right),
\end{equation}
with Frobenius and Rees structure (see \cite[\S2.4]{LoveringFCrystals}) inherited from $\mc{H}^1_\crys(\wh{\ms{A}}_{{\mathsf{K}^p}}/\wh{\ms{S}}_{{\mathsf{K}^p}})^\vee$, compatibly in ${\mathsf{K}^p}$. But, this is Lovering's construction of $\omega_{{\mathsf{K}^p},\crys}$ in the Hodge type case.

\subsubsection{Cohomological consequences} To obtain cohomological implications of Theorem \ref{thm:prismatic-crystalline-comparison}, it is useful to recall that the group-theoretic description of when a Shimura variety is proper. 

\begin{prop} For a neat compact open subgroup $\mathsf{K}^p\subseteq\mb{G}(\A_f^p)$, the following are equivalent:
\begin{enumerate}
    \item $\mb{G}^\ad$ is $\Q$-anisotropic,
    \item $\Sh_{\mathsf{K}_0\mathsf{K}^p}\to\Spec(E)$ is proper,
    \item $\ms{S}_{\mathsf{K}^p}\to\Spec(\mc{O}_E)$ is proper.
\end{enumerate}
\end{prop}
\begin{proof} The equivalence of of the first two conditions is classical (e.g.\@ see \cite[Lemma 3.1.5]{Paugam}). The equivalence of the first and third conditions is \cite[Corollary 4.1.7]{MadTorHod} when $(\mb{G},\mb{X},\mc{G})$ is of Hodge type, and one quickly reduces to this case using Lemma \ref{lem:isogeny-finite-etale}.
\end{proof}

So combining \cite[Corollary 3.11]{IKY3} and Corollary \ref{cor:prismatic-realization-lff} reproves \cite[Theorem 3.6.1]{LoveringFCrystals}.

\begin{prop}[{\cite[Theorem 3.6.1]{LoveringFCrystals}}]\label{prop:coholomogy-crystalline} Suppose that $\mb{G}^\der$ is $\Q$-anisotropic. Then, for any object $\xi$ of $\cat{Rep}_{\Z_p}(\mc{G}^c)$, the Galois representation $H^i_\et((\Sh_{\mathsf{K}_0\mathsf{K}^p})_{\ov{\Q}_p},\omega_{\mathsf{K}^p,\et}(\xi)[\nicefrac{1}{p}])$ of $E$ is crystalline, and there is a canonical isomorphism of filtered $F$-isocrystals
\begin{equation}\label{eq:shim-var-matching}D_\crys\left(H^i_\et((\Sh_{\mathsf{K}_0\mathsf{K}^p})_{\ov{\Q}_p},\omega_{\mathsf{K}^p,\et}(\xi)[\nicefrac{1}{p}])\right)\isomto H^i_\crys\left(((\ms{S}_{\mathsf{K}^p})_{\breve{\Z}_p}/\breve{\Z}_p)_\crys,\omega_{\mathsf{K}^p,\crys}(\xi)[\nicefrac{1}{p}]\right). 
\end{equation}
If the $\bm{\mu}^c_h$-weights of $\xi[\nicefrac{1}{p}]$ are in $[0,p-3-i]$, then the isomorphism in \eqref{eq:shim-var-matching} sends the lattice $H^i_\et((\Sh_{\mathsf{K}_0\mathsf{K}^p})_{\ov{\Q}_p},\omega_{\mathsf{K}^p,\et}(\xi))$ to the lattice $H^i_\crys\left(((\ms{S}_{\mathsf{K}^p})_{\breve{\Z}_p}/\breve{\Z}_p)_\crys,\omega_{\mathsf{K}^p,\crys}(\xi)\right)$.
\end{prop}

 We obtain a syntomic refinement of the above result. To this end, let us say that a $\Z/p^n$-representation of $\Gal(\ov{E}/E)$, with $n$ in $\bb{N}\cup\{\infty\}$, is of \emph{syntomically good reduction} if it belongs to the essential image of the functor $T_\et\colon \cat{Vect}(\mc{O}_E^\syn/p^n)\to \cat{Rep}_{\Z/p^n}(\Gal(\ov{E}/E))$, which when $n=\infty$ is a refinement of the crystalliness condition (cf.\@ \cite[Theorem 6.6.13]{BhattNotes} and Footnote \ref{footnote:syntomically-good-reduction}). The following is an immediate corollary of \cite[Proposition 3.12]{IKY3}.

 \begin{thm}\label{thm:gauge-propert-coh} Suppose that $\mb{G}^\der$ is $\Q$-anisotropic. Then, for any object $\xi$ of $\cat{Rep}_{\Z_p}(\mc{G}^c)$ with $\bm{\mu}^c_h$-weights of $\xi[\nicefrac{1}{p}]$ in $[0,p-3-i]$ the $\Gal(\ov{E}/E)$-representation $H^i_\et((\Sh_{\mathsf{K}_0\mathsf{K}^p})_{\ov{\Q}_p},\omega_{\mathsf{K}^p,\et}(\xi)/p^n)$ has syntomically good reduction for any $n$ in $\bb{N}\cup\{\infty\}$. In fact, there is an isomorphism
 \begin{equation*}
     H^i_\et((\Sh_{\mathsf{K}_0\mathsf{K}^p})_{\ov{\Q}_p},\omega_{\mathsf{K}^p,\et}(\xi)/p^n)\simeq T_\et\left(\mc{H}^i_\mr{syn}(\wh{\ms{S}}_{\mathsf{K}^p}/\mc{O}_E,\omega_{\mathsf{K}^p,\mr{syn}}(\xi))/p^n)\right).
 \end{equation*}
 \end{thm}

\subsubsection{Comparison of stratifications} In \cite{ShenZhang}, Newton stratifications and central leaves are defined on the special fiber of $\ms{S}_{\mathsf{K}^p}$, extending all previously known cases (see the references in op.\@ cit.\@). This gives functions 
\begin{equation*}
    \Upsilon_{\mathsf{K}^p} \colon \ms{S}_{\mathsf{K}^p}(\ov{k})\to B(G^c),\qquad\Upsilon^\circ_{\mathsf{K}^p} \colon \ms{S}_{\mathsf{K}^p}(\ov{k})\to C(\mc{G}^c).
\end{equation*}
equivariant with respect to the actions of $\Gamma_k$. By the results of \cite[\S5.4.2 and \S5.4.5]{ShenZhang}, $\Upsilon_{\mathsf{K}_p\mathsf{K}^p}$ agrees with the function to $B(G^c)$ defined using $\omega_{\mathsf{K}^p,\crys}$, and agrees with the function $C(\mc{G}^c)$ defined using $\omega_{\mathsf{K}^p,\crys}$ when $(\mb{G},\mb{X})$ is of Hodge type or $Z(\mb{G})$ is connected. Combining this with Proposition \ref{prop:factorization-through-specialization} and Theorem \ref{thm:prismatic-crystalline-comparison} then gives the following corollary.

\begin{cor}\label{cor:relation-to-SZ} Suppose that $(\mb{G},\mb{X},\mc{G})$ is an unramified Shimura datum of abelian type (resp.\@ of Hodge type or of abelian type and $Z(\mb{G})$ is connected). Then, for any neat compact open subgroup $\mathsf{K}_p\subseteq \mathsf{K}_0$ and neat compact open subgroup $\mathsf{K}^p\subseteq\mb{G}(\A_f^p)$ we have that 
\begin{equation*}
    \Sigma_{\mathsf{K}_p\mathsf{K}^p}=\Upsilon_{\mathsf{K}^p}\circ \sp\circ \pi_{\mathsf{K}_p\mathsf{K}^p,\mathsf{K}_0\mathsf{K}^p},\quad \bigg(\emph{resp. }\Sigma^\circ_{\mathsf{K}^p}=\Upsilon_{\mathsf{K}^p}^\circ\circ \sp\bigg).
\end{equation*}
\end{cor}

\begin{rem} While some of the rational (i.e.\@, $B(G^c)$ related) results in Proposition \ref{prop:factorization-through-specialization} and Corollary \ref{cor:relation-to-SZ} could have been proven using results in \cite{LoveringFCrystals}, the integral (i.e.\@, $C(\mc{G}^c)$ related) results could not, as Lovering is only able to establish matching between the lattices $\omega_{\mathsf{K}^p,\et}(\Lambda)$ and $\omega_{\mathsf{K}^p,\crys}(\Lambda)$ for low Hodge--Tate weights. 
\end{rem}

\section{Prismatic characterizations of integral models and Serre--Tate theory}\label{ss:Pappas--Rapoport--Daniels}

We formulate and prove two characterizations of integral canonical models for unramified Shimura data of abelian type: one in terms of prismatic $F$-crystals and the other in terms of prismatic $F$-gauges. The first is both stronger and less sophisticated, but the second is more conceptual and naturally leads us to a version of the Serre--Tate deformation theorem for such integral Shimura varieties. Throughout this section we assume that $p>2$.

\subsection{The universal deformation spaces of Ito}\label{ss:display} To formulate our prismatic characterization of Shimura varieties, we need a deformation space of prismatic $F$-crystals with $\mc{G}$-structure of type $\mu$. This is furnished by a construction of Ito which we now recall.\footnote{Ito's work is in terms of prismatic $\mc{G}$-$\mu$-displays. That said, given the contents of \S\ref{p:G-mu-displays-and-prismatic-G-torsors-with-F-structure} we glibly state his work in terms of prismatic $\mc{G}$-torsors with $F$-structure.}

\begin{nota}\label{nota:deformation} We fix the following notation:
\begin{itemize}
    \item $k$ is a perfect field of characteristic $p$,
    \item $W=W(k)$,
    \item $\mc{G}$ is a smooth group $\Z_p$-scheme (assumed reductive in \S\ref{sss:comparison-to-Faltings}),
    \item $\mu \colon \bb{G}_{m,W}\to \mc{G}_W$ is a $1$-bounded cocharacter (see \cite[Definition 6.3.1]{LauHigher}),
    \item $\mc{C}_W$ is the category of complete Noetherian local rings $R$ equipped with a local ring map $W\to R$ with $k\to R/\mf{m}_R$ an isomorphism, and with morphisms maps of local $W$-algebra,
    \item $\mc{C}_W^\mr{reg}$ is the full subcategory of $\mc{C}_W$ consisting of regular local rings.
\end{itemize} 
\end{nota}
We consider objects $\mc{C}_W$ of as $p$-adic topological rings unless stated otherwise.

\subsubsection{Universal deformation spaces}\label{ss:universal-deformation}  Denote by $\mc{U}_{\mu}$ the unipotent group scheme over $W$ associated to $\mu$ via the dynamic method (see \cite[Theorem 4.1.7]{ConradReductive}). Set $R_{\mc{G},\mu}\defeq \mc{O}(\wh{\mc{U}}_{\mu})$, which is a $p$-adically complete ring non-canonically isomorphic to $W\ll t_1,\ldots, t_d\rr$ for some $d$ (see \cite[Lemma 4.2.6]{Ito1}). If $f\colon \mc{G}_1\to\mc{G}$ is a morphism of reductive groups over $\Z_p$ mapping $\mu_1$ to $\mu$ then we obtain an induced continuous morphism of $W$-algebras $R_{\mc{G},\mu}\to R_{\mc{G}_1,\mu_1}$. Furthermore, if the map $\mc{G}_1^\ad\to \mc{G}^\ad$ induced by $f$ is an isomorphism, then the map $R_{\mc{G},\mu}\to R_{\mc{G}_1,\mu_1}$ is an isomorphism.

Fix an element $b$ of $\mc{G}(W)\mu(p)^{-1}\mc{G}(W)$. The pair $(\mc{G}_W,\varphi_b)$, where $\varphi_b$ corresponds to left multiplication by $b$, defines an element of $\cat{Tors}^{\varphi,-\mu}_{\mc{G}}(k_\smallprism)$. As in \cite[Definition 1.1.2]{Ito2}, for an object $R$ of $\mc{C}^\mathrm{reg}_W$, a \emph{deformation} of $(\mc{G}_W,\varphi_b)$ over $R$ is a pair $((\mc{A},\varphi_\mc{A}),\gamma)$ where $(\mc{A},\varphi_\mc{A})$ is an object of $\cat{Tors}_{\mc{G}}^{\varphi,-\mu}(R_\smallprism)$, and $\gamma\colon (\mc{A},\varphi_\mc{A})|_{k_\smallprism}\isomto (\mc{G}_W,\varphi_b)$ is an isomorphism. In \cite[Theorem 1.1.3]{Ito2} (using Proposition \ref{prop:display-torsor-equiv}), Ito shows that there is a \emph{universal deformation} $((\mc{A}^\univ_b,\varphi_{\mc{A}^\univ_b}),\gamma^\univ)$ over $R_{\mc{G},\mu}$, which means that for any other deformation $((\mc{A},\varphi_\mc{A}),\gamma)$ over an object $R$ of $\mc{C}_W^\mr{reg}$, there exists a unique morphism $h\colon R_{\mc{G},\mu}\to R$ such that $h^\ast((\mc{A}^\univ_b,\varphi_{\mc{A}^\univ_b}),\gamma^\univ)$ is isomorphic to $((\mc{A},\varphi_\mc{A}),\gamma)$. Define $\omega_b^\univ$ to be the object of $\GVect^\varphi(R_\smallprism)$ associated to $(\mc{A}_b^\univ,\varphi_{\mc{A}_b^\univ})$.

\begin{lem}\label{lem:ito-compatability}
For a morphism of reductive groups $f\colon \mc{G}_1\to \mc{G}$ sending $\mu_1$ to $\mu$, and $b_1$ to $b$, the induced map $f\colon \Spf (R_{\mc{G}_1,\mu_1}) \to \Spf (R_{\mc{G},\mu})$ satisfies the following, with $\xi$ an object of $\cat{Rep}_{\Z_p}(\mc{G})$: 
\begin{equation*}\label{eq:ito-compatability}
   (\mc{A}_{b_1}^\univ ,\varphi_{\mc{A}_{b_1}^\univ})\times^{\mc{G}_{1,\smallprism}}\mc{G}_{\smallprism}\isomto  f^\ast(\mc{A}_b^\univ,\varphi_{\mc{A}_b^\univ}),\quad f^\ast(\omega_b^\univ(\xi))=\omega_{b_1}^\univ(\xi\circ f). 
\end{equation*}
\end{lem}
\begin{proof}
It suffices to show the first isomorphism. 
By \cite[Theorem 6.1.3]{Ito1}, it suffices to construct an isomorphism after evaluating at $(R_{\mc{G}_1,\mu_1}\ll t\rr,(p-t))$. By the construction in the proof of \cite[Theorem 4.4.2]{Ito2}, both evaluations are isomorphic to $\mathcal{Q}_{f(X_1^{\mathrm{univ}})}$, where $X_1^{\mathrm{univ}}$ is $X^{\mathrm{univ}}$ in the proof of \cite[Theorem 4.4.2]{Ito2}, when applied to $(\mc{G}_W,\varphi_{b_1})$, $\mc{G}_1$ and $\mu_1$.   
\end{proof}
If $g$ is in $\mc{G}(W)$ then conjugation by $g^{-1}$ induces an isomorphism $c_g\colon R_{\mc{G},\mu}\to R_{\mc{G},g\mu g^{-1}}$ and $c_g^\ast(\omega_b^\univ)=\omega_{gbg^{-1}}^\univ$. If $b'$ is another element of $\mc{G}(W[\nicefrac{1}{p}])$ and $g$ is an element of $\mc{G}(W)$ such that $b'=gb\phi(g)^{-1}$ then left multiplication by $g$ induces an isomorphism $(\mc{G}_W,\varphi_{b})\to (\mc{G}_W,\varphi_{b'})$ and thus there is an isomorphism $\omega_b^\univ\isomto \omega_{b'}^\univ$. Thus, the pair $(R_{\mc{G},\mu},\omega_b^\univ)$ only depends, up to isomorphism, on the $\mc{G}(W)$-conjugacy class of $\mu$, and the $\sigma$-$\mc{G}(W)$-conjugacy class of $b$.

\subsubsection{Comparison to Faltings universal deformations with Tate tensors}\label{sss:comparison-to-Faltings} We now wish to use the functor $\bb{D}_\crys$ to compare the work of Ito to that of Faltings on universal deformations of $p$-divisible groups with Tate tensors in the reductive case. Below we use the notation for Dieudonn\'e theory as in \cite[\S4]{IKY3}. Most specifically, $\bb{D}(H)$ denotes the filtered Dieudonn\'e crystal of a $p$-divisible group $H$ (see \cite[Definition 4.6]{IKY3}).

We use the notation \ref{ss:universal-deformation}, but now assume that $k$ is algebraically closed and $\mc{G}$ is reductive. We write $\omega_{b,\crys}^\univ$ for the composition $\bb{D}_\crys\circ\omega_b^\univ$, which is a $\mc{G}$-object in $\cat{VectF}^\varphi((R_{\mc{G},\mu})_\crys)$.

\medskip

\paragraph{Universal deformation of $p$-divisible groups} Fix a $p$-divisible group $H_0$ over $k$, and choose any lift $\wt{H}_0$ of $H_0$ over $W$. Set $M_0\defeq \underline{\bb{D}}(H_0)(W\twoheadrightarrow k)$, and let  $b_0$ in $\GL(M_0[\nicefrac{1}{p}])$ such that $\sigma (b_0)$ corresponds to the Frobenius on $M_0$. From $\wt H_0$ we obtain the Hodge filtration:
\begin{equation*}
    \Fil^1_{\wt{H}_0,\mr{Hodge}}\subseteq \bb{D}(\wt H_0)(\id\colon W\to W)=\bb{D}(H_0)(W\twoheadrightarrow k).
\end{equation*}
From the Cartan decomposition, we know that $b_0$ lies in $\mc{G}(W)\mu_0^{-1}(p)\mc{G}(W)$ for a cocharacter $\mu_0\colon \bb{G}_{m,W}\to \GL(M_0)$ uniquely determined up to conjugacy. We take the unique cocharacter $\mu_0$ such that $\Fil^1_{\wt H_0,\mr{Hodge}}$ is induced by $\mu_0$ in the sense of \cite[Definition 2.2.1]{KimRZ}.

Choose an isomorphism $R_{\GL(M_0),\mu_0}\isomto W\ll t_1,\ldots,t_d\rr$ and equip it with the usual Frobenius $\phi_0$. From the above considerations we obtain the following data on $R_{\GL(M_0),\mu_0}$:
\begin{equation*}
    \mb{M}_{b_0}^\univ\defeq R_{\GL(M_0),\mu_0}\otimes_W M_0,\quad \Fil^1{\mb{M}_{b_0}^\univ}\defeq 1\otimes \Fil^1{M_0},\qquad \varphi_{\mb{M}_{b_0}^\univ}\defeq u_t^{-1}\circ(1\otimes b_0),
\end{equation*}
where $u_t$ corresponds to the tautological element of $\wh{\mc{U}}_{\mu_0}(R_{\GL(M_0),\mu_0})$, and $\varphi_{\mb{M}_{b_0}^\univ}$ is considered as a map $\phi_0^\ast \mb{M}_{b_0}^\univ\to \mb{M}_{b_0}^\univ$. As explained in \cite[\S4.5]{Moonen}, Faltings produced a (unique) $\varphi_{\mb{M}_{b_0}^\univ}$-horizontal integrable connection $\nabla_{\mb{M}_{b_0}^\univ}$ on $\mb{M}_{b_0}^\univ$ with $(\mb{M}_{b_0}^\univ,\varphi_{\mb{M}_{b_0}^\univ},\Fil^\bullet_{\mb{M}_{b_0}^\univ},\nabla_{\mb{M}_{b_0}^\univ})$ a filtered Dieudonn\'e crystal on $R_{\GL(M_0),\mu_0}$ and so corresponds to an object $H^\univ_{b_0}$ of $\mr{BT}_p(R_{\GL(M_0),\mu_0})$. 

As explained in \cite[\S3.3]{KimRZ}, this notation is not misleading as $H^\univ_{b_0}$ is a universal deformation of $H_0$ (in the sense that it (pro)represents the functor in \cite[Definition 3.1]{KimRZ}).

\begin{prop}\label{prop:faltings-ito-comp-GL_n} There is a natural isomorphism
\begin{equation*}
    \omega_{b_0,\crys}^\univ(M_0)(R_{\GL(M_0),\mu_0})\isomto (\mb{M}_{b_0}^\univ,\varphi_{\mb{M}_{b_0}^\univ},\Fil^\bullet_{\mb{M}_{b_0}^\univ}).
\end{equation*}

\end{prop}
\begin{proof} This follows from \cite[Theorem 4.8]{IKY3} and \cite[Theorem 6.2.1]{Ito2}.
\end{proof}

\medskip

\paragraph{Deformations with Tate tensors} Fix a triple $(\mc{G},b,\mu)$ as in \S\ref{ss:universal-deformation}. We assume that $(\mc{G},b,\mu)$ is of \emph{Hodge type}: there exists a faithful representation $\iota\colon \mc{G}\to\GL(\Lambda_0)$ such that $\iota \circ \mu$ has only weights $0$ and $1$. Set $(b_0,\mu_0)\defeq (\iota(b),\iota\circ \mu)^\vee$. Further fix isomorphisms 
\begin{equation*}
R_{\GL(\Lambda_0^\vee),\mu_0}\isomto W\ll t_1\ldots,t_d\rr,\qquad R_{\mc{G},\mu}\isomto W\ll s_1,\ldots,s_k\rr
\end{equation*}
such that the natural map $R_{\GL(\Lambda_0^\vee),\mu_0}\to R_{\mc{G},\mu}$ is Frobenius equivariant when the source and target are given the (usual) Frobenii induced by these isomorphisms.\footnote{This is possible, for instance, by the discussion in \cite[\S4.2]{Ito1}, which shows that $U_{\mu}$ is isomorphic to $\mathrm{Lie}(U_{\mu})$ as $W$-schemes, and $\mathrm{Lie}(U_{\mu})$ is a direct summand of $\mathrm{Lie}(\mathrm{GL}(\Lambda_0^\vee))$.} Finally, fix a tensor package (in the sense in \cite[\S A.5]{IKY1}) $(\Lambda_0,\mathds{T}_0)$ with $\mc{G}=\mathrm{Fix}(\mathds{T}_0)$. 

As explained in \cite[\S2.5]{KimRZ} associated to $(\mc{G},b,\mu)$ and $\iota$ is a $p$-divisible group $H_{b_0}$ over $k$ together with an identification $\bb{D}(H_{b_0})(W)=\Lambda_0^\vee$ where Frobenius acts by $b_0$. Moreover, under this identification the set $\mathds{T}_0$ is a set of tensors on $\bb{D}(H_{b_0})$ as an $F$-crystal. Set
\begin{equation*}
(\mb{M}_{b}^\univ,\varphi_{\mb{M}_b^\univ},\Fil^1_{\mb{M}_b^\univ})\defeq (\mb{M}_{{b_0}}^\univ,\varphi_{\mb{M}_{b_0}^\univ},\Fil^1_{\mb{M}_{b_0}^\univ})\otimes_{R_{\GL(\Lambda_0^\vee),\mu_0}}R_{\mc{G},\mu},
\end{equation*}
and let $H_b^\univ$ be the pullback of $H_{b_0}^\univ$ to $R_{\mc{G},\mu}$. Then, 
\begin{equation*}
    \bb{D}(H_b^\univ)(R_{\mc{G},\mu})=(\mb{M}_b^\univ,\varphi_{\mb{M}_b^\univ},\Fil^1_{\mb{M}_b^\univ},\nabla_{\mb{M}_b^\univ}),
\end{equation*}
for some connection $\nabla_{\mb{M}_b^\univ}$. Observe that $\mathds{T}_0$ naturally defines a set of tensors on $\mb{M}_{b_0}^\univ$ and (by base change) on $\mb{M}_b^\univ$, which we denote $\mathds{T}_0^{\Fal,'}$ and $\mathds{T}_0^\Fal$, respectively. The tensors $\mathds{T}_0^\Fal$ on $\bb{D}(H_b^\univ)$ are Frobenius equivariant and lie in the $0^\text{th}$-part of the filtration (see \cite[\S3.5]{KimRZ}). 

By work of Faltings (see \cite[Theorem 3.6]{KimRZ}) $H_b^\univ$ satisfies a universality property. Suppose that $R_0=W\ll u_1,\ldots,u_r\rr$ for some $r$, and $X$ is a $p$-divisible group over $R_0$ deforming $H_{b_0}$. By the universality of $H_{b_0}^\univ$, there exists a unique map $f_X\colon R_{\GL(\Lambda_0^\vee),\mu_0}\to R_0$ such that $f_X^\ast(H_{b_0}^\univ)$ is isomorphic (as a deformation) to $X$. Then, $f_X$ factorizes through $R_{\GL(\Lambda_0^\vee),\mu_0}\to R_{\mc{G},\mu}$ if and only if there exists tensors $\{t_\alpha\}$ on $\bb{D}(X)(R_0)$ lifting those on $\mathds{T}_0$ on $\bb{D}(H_{b_0})$, and which are Frobenus equivariant and lie in the $\text{0}^\text{th}$-part of the filtration. In this case $\{t_\alpha\}=f_X^\ast(\mathds{T}_0^\mathrm{Fal,'})$. 

Consider now  the obvious morphism $f\colon (\mc{G},b,\mu)\to (\GL(\Lambda_0^\vee),b_0,\mu_0)$. Combining Lemma \ref{lem:ito-compatability} and Proposition \ref{prop:faltings-ito-comp-GL_n}, there is a canonical identification 
\begin{equation*}
    \begin{aligned} \omega^\univ_{b,\crys}(\Lambda_0^\vee)(R_{\mc{G},\mu})&\isomto  f^\ast(\bb{D}(H_{b_0}^\univ))(R_{\mc{G},\mu})\\ &=\bb{D}(H_b^\univ)(R_{\mc{G},\mu})\\ &=(\mb{M}_{b}^\univ,\varphi_{\mb{M}_b^\univ},\Fil^1_{\mb{M}_b^\univ}),\end{aligned}
\end{equation*} 
of naive filtered $F$-crystals.

\begin{prop}\label{prop:ito-faltings-match} The isomorphism of naive filtered $F$-crystals on $R_{\mc{G},\mu}$
\begin{equation*}
    \omega_{b,\crys}^\univ(\Lambda_0^\vee)(R_{\mc{G},\mu})\to (\mb{M}_{b}^\univ,\varphi_{\mb{M}_b^\univ},\Fil^1_{\mb{M}_b^\univ})
\end{equation*} 
carries $\omega_{b,\crys}^\univ(\mathds{T}_0)(R_{\mc{G},\mu})$ to $\mathds{T}_0^\Fal$.
\end{prop}
\begin{proof}

Under this identification $\omega_{b,\crys}^\univ(\mathds{T}_0)(R_{\mc{G},\mu})$ constitutes a Frobenius-equivariant tensors lying in the $0^\text{th}$-part of the of the filtration and lifting those on $\bb{D}(H_{b_0})$. Thus, by the universality statement from above, $\omega_{b,\crys}^\univ(\mathds{T}_0)(R_{\mc{G},\mu})$ must be equal to $f^\ast(\mathds{T}_0^{\Fal,'})=\mathds{T}_0^\Fal$. 
\end{proof}

 \subsection{Comparison to Shimura varieties}\label{ss:comparison-to-shim-vars} We now show that for the integral canonical model $\ms{S}_{\mathsf{K}^p}$, the prismatic realization functor $\omega_\smallprism$ recovers Ito's universal prismatic $\mc{G}^c$-torsor with $F$-structure at the completion of $\ms{S}_{\mathsf{K}^p}$ at each point of $\ms{S}_{\mathsf{K}^p}(\ov{\bb{F}}_p)$. This may be seen as a prismatic refinement of \cite[(1.3.9) Proposition]{KisinModp}, in the general abelian type setting.
 
Suppose that $(\mb{G},\mb{X},\mc{G})$ is an unramified Shimura datum of abelian type, and $x$ is a point of $\ms{S}_{\mathsf{K}^p}(\ov{\bb{F}}_p)$. As in \S\ref{ss:pot-crys-strat} we have the associated element $\bm{b}_{x,\crys}\defeq \underline{\bb{D}}_\mr{crys}\circ (\omega_{\mathsf{K}^p,\et})_x$ in $C(\mc{G}^c)$. Additionally, choose an element $\mu^c_h$ in our conjugacy class $\bbmu_h^c$ of cocharacters $\mbb{G}_{m,\breve{\Z}_p}\to \mc G^c_{\breve{\Z}_p}$.

\begin{lem}\label{lem:C(G)-containment} The element $\bm{b}_{x,\crys}$ lies in the image of the map
\begin{equation*}
    \mc{G}^c(\breve{\Z}_p)\sigma(\mu_h^c(p))^{-1}\mc{G}^c(\breve{\Z}_p)\to C(\mc{G}^c). 
\end{equation*}
\end{lem}
\begin{proof}
If $(\mb{G},\mb{X})$ is of Hodge type this follows from \cite[Lemma 3.3.14]{KimUnif}. In the special-type case, we may assume that the torus in the Shimura datum is cuspidal following the argument given on \cite[pp.\@ 31--33]{Daniels}. From there the claim follows from \cite[Proposition 4.3.14 and Corollary 4.4.12]{KSZ}. 

For the abelian type case, we take $(\mb{G}_1,\mb{X}_1,\mc{G}_1)$, $(\mb{T},\{h\},\mc{T})$ and $(\mb{G}_2,\mb{X}_2,\mc{G}_2)$ as in Lemma \ref{lem:Lovering-lem}. The question is reduced to the case of $(\mb{G}_2,\mb{X}_2,\mc{G}_2)$ by functoriality. 
For the reduction to the case of $(\mb{G}_1\times \mb{T},\mb{X}_1 \times \{h\},\mc{G}_1 \times \mc{T})$, it suffices to show the injectivity of 
\begin{equation*} 
\mathcal{G}_2^c(\breve{\mathbb{Z}}_p)\backslash \mathcal{G}_2^c(\breve{\mathbb{Q}}_p) / \mathcal{G}_2^c(\breve{\mathbb{Z}}_p) \to \mathcal{G}_3^c(\breve{\mathbb{Z}}_p)\backslash \mathcal{G}_3^c(\breve{\mathbb{Q}}_p) / \mathcal{G}_3^c(\breve{\mathbb{Z}}_p),
\end{equation*}
where we put $\mathcal{G}_3=\mathcal{G}_1 \times \mathcal{T}$. 
Fix a Borel pair $\mathcal{T}_1 \subset \mathcal{B}_1 \subset \mathcal{G}_1$. This gives Borel pairs $\mathcal{T}_2 \subset \mathcal{B}_2 \subset \mathcal{G}_2^c$ and $\mathcal{T}_3 \subset \mathcal{B}_3 \subset \mathcal{G}_3^c$ by the constructions of $\mathcal{G}_2^c$ and $\mathcal{G}_3^c$. 
The injectivity of this map of double cosets follows from the Cartan decomposition since the unipotent radials of $\mathcal{B}_2$ and $\mathcal{B}_3$ are same. 
\end{proof}

Choose an element $b_x$ in $\mc{G}^c(\breve{\Z}_p)\mu^c_h(p)^{-1}\mc{G}^c(\breve{\Z}_p)$ 
such that $\sigma (b_x)$ maps to $\bm{b}_{x,\crys}$ in $C(\mc{G}^c)$.

\begin{thm}\label{thm:ito-shim-comp} There exists an isomorphism $i_x\colon R_{\mc{G}^c,\mu_h^c}\isomto \wh{\mc{O}}_{\ms{S}_{\mathsf{K}^p},x}$ such that 
\begin{equation*}
    i_x^\ast(\omega_{b_x}^\univ)\cong (\omega_{\mathsf{K}^p,\smallprism})|_{\wh{\mc{O}}_{\ms{S}_{\mathsf{K}^p},x}}.
\end{equation*}
\end{thm}

Note while the pair $(R_{\mc{G}^c,\mu^c_h},\omega_{b_x}^\univ)$ depends on the choice of $b_x$ and $\mu_h^c$, the isomorphism type of the pair $(R_{\mc{G}^c,\mu^c_h},\omega_{b_x}^\univ)$ does not, and therefore neither does the statement of Theorem \ref{thm:ito-shim-comp}.

\begin{proof}[Proof of Theorem \ref{thm:ito-shim-comp}] We perform a devissage to the Hodge and special type cases.

\medskip

\paragraph*{Hodge type case} Let $(\mb{G},\mb{X},\mc{G})\hookrightarrow (\GSp(\mb{V}_0),\mf{h}^\pm, \GSp(\Lambda_0))$ be an integral Hodge embedding, and $\ms{A}_{\mathsf{K}^p}\to \ms{S}_{\mathsf{K}^p}$ the associated abelian scheme. Let $\mathds{T}_{0,p}^\crys$ be the tensors on the filtered $F$-crystal 
\begin{equation*}
\mc{H}^1_\crys(\wh{\ms{A}}_{\mathsf{K}^p}/\wh{\ms{S}}_{\mathsf{K}^p})=\bb{D}(\ms{A}_{\mathsf{K}^p}[p^\infty])
\end{equation*} 
(see \cite[(3.3.7.2)]{BBMDieuII} for this identification), as in \S\ref{ss:lovering-comp}. The triple $(\mc{G},b_x,\mu_h)$ 
is of Hodge type relative to the embedding $\iota\colon \mc{G}\to \GL(\Lambda_0)$. By work of Kisin, there exists an isomorphism $i_x\colon R_{\mc{G},\mu_h}\isomto \wh{\mc{O}}_{\ms{S}_{\mathsf{K}^p},x}$ together with an isomorphism $i_x^\ast(H_{b_x}^\univ,\mathds{T}_0^\Fal)\cong(\ms{A}_{\mathsf{K}^p}[p^\infty],\mathds{T}_{0,p}^\crys)_{\wh{\mc{O}}_{\ms{S}_{\mathsf{K}^p},x}}$ (see \cite[Proposition 4.1.6]{KimUnif}). 
We claim that $i_x^\ast(\omega_{b_x}^\univ)$ is isomorphic to $\omega\defeq (\omega_{\mathsf{K}^p,\smallprism})|_{\wh{\mc{O}}_{\ms{S}_{\mathsf{K}^p},x}}$. We consider the following isomorphisms 
\begin{equation*}
    \omega(\Lambda_0^\vee)\isomto \mc{M}_\smallprism(\ms{A}[p^\infty]_{\wh{\mc{O}}_{\ms{S}_{\mathsf{K}^p},x}})\isomto i_x^*\mc M_\smallprism(H^\univ_{b_x})\isomfrom i_x^\ast(\omega_{b_x}^\univ)(\Lambda_0^\vee),
\end{equation*}
where the first isomorphism is from Theorem \ref{thm:main-Shimura-theorem-Hodge-type-case}, the second is the one obtained by applying $\mc M_\smallprism$ to the above isomorphism of Kisin, and the third one is obtained as follows. By \cite[Theorem 6.2.1]{Ito2} we have a canonical isomorphism $\mc M_\smallprism(H^\univ_{\iota(b_{x})^\vee})\isomfrom \omega_{\iota(b_{x})^\vee}^\univ(\Lambda_0^\vee)$. Then, by Lemma \ref{lem:ito-compatability}, we get the third isomorphism as the restriction of this isomorphism along 
\begin{equation*} R_{\GL(\Lambda_0^\vee),\iota(\mu_h)^\vee}\to R_{\mc G,\mu_h}\xrightarrow{i_x}\wh{\mc O}_{\mathscr{S}_{\mathsf K^p,x}}.
\end{equation*}
By \cite[Proposition 1.28]{IKY1} it suffices to show the above composite carries $\omega(\mathds{T}_0)$ to $i_x^\ast(\omega_{b_x}^\univ)(\mathds{T}_0)$

It further suffices to show $\bb{D}_\crys(\omega(\mathds{T}_0))(R_{\mc{G},\mu_h})$ is matched to $\bb{D}_\crys(i_x^\ast(\omega_{b_x}^\univ)(\mathds{T}_0))(R_{\mc{G},\mu_h})$ as 
\begin{equation*}
    \cat{Vect}((R_{\mc{G},\mu_h})_\smallprism)\to \cat{Vect}(R_{\mc{G},\mu_h}),\quad \mc{E}\mapsto \mc{E}^\crys(R_{\mc{G},\mu_h}) 
\end{equation*}
is faithful as follows from \cite[Corollary 2.2.3]{deJongCrystalline} and the second equivalence in \cite[Equation (2.3.2)]{IKY1}. But, this matching follows by combining Proposition \ref{prop:prismatic-crystalline-tensor-matching} and Proposition \ref{prop:ito-faltings-match}.

\medskip

\paragraph*{Special type case} Write $(\mb{T},\mb{X},\mc{T})$ for the unramified Shimura datum. In this case $\ms{S}_{\mathsf{K}^p}$ is a disjoint union of schemes of the form $\Spec(\mc{O}_{E'})$, for a finite unramified extension $E'$ of $E$ (see \cite[Proposition 3.22]{DanielsYoucis}), and so there is a tautological identification $R_{\mc G^c,\mu^c_h}=\breve{\Z}_p=\wh{\mc{O}}_{\ms{S}_{\mathsf{K}^p},x}$. We show that under this identification that the prismatic $\mc{T}^c$-torsors with $F$-structure are matched.

By Remark \ref{rem:Daniels-comp}, Lemma \ref{lem:ito-compatability}, and the argument given in \cite[pp.\@ 31--33]{Daniels} we are reduced to showing the following. Let $\mc{T}=\mathrm{Res}_{\mc{O}_{E'}/\Z_p}\, \bb{G}_{m,\mc{O}_{E'}}$, $T\defeq \mc{T}_{\Q_p}$, and $\mu$ the $\Q^\mr{ur}_p$-cocharacter of $T$ with weights $(1,0,\ldots,0)$. Then, $b_0=(p^{-1},1,\ldots,1)$ represents the unique class in the image of $\mc{T}(\breve{W})\mu^{-1}(p)\mc{T}(\breve{W})$. 
Then, we must show that $T_\et\circ \omega_{b_0}^\univ$ restricted to the inertia subgroup $\Gamma_{E',0}$ agrees with the Lubin--Tate character $\alpha_0\colon \Gamma_{E',0}\to \mc{T}(\Z_p)$ (see the discussion before \cite[Proposition 4.9]{Daniels}), or equivalently that their compositions with embedding $\iota\colon \mc{T}(\Z_p)\to \GL(\mc{O}_{E'})$ are equal. But, as $(\mc{T},b_0,\mu^{-1})$ is the Lubin--Tate triple, we know by \cite[Theorem 6.2.1]{Ito2} that this $\omega^\univ_{b_0}(\iota)$ is $\mathcal{M}_\smallprism(X_{\mathrm{LT}})$, if $X_\mathrm{LT}$ is the $p$-divisible group with $\mc{O}_E$-structure over $\Spf(\breve{W})$ coming from Lubin--Tate theory. Thus, the composition of the character $\Gamma_{E',0}\to\mc{T}(\Z_p)\to \GL(\mc{O}_E)$ corresponding to $T_\et\circ \omega_{b_0}^\univ$ is $T_\et(\mc{M}_\smallprism(X_\mathrm{LT}))=T_p(X_\mathrm{LT})$ (see \cite[Proposition 3.35]{DLMS}). But, this is the composition of $\alpha_0$ with $\mc{T}(\Z_p)\to \GL(\mc{O}_E)$ (see the proof of \cite[Proposition 4.9]{Daniels}).

\medskip

\paragraph*{Abelian type case} Let $(\mb{G}_1,\mb{X}_1,\mc{G}_1)$ be an unramified Shimura datum of Hodge type adapted to $(\mb{G},\mb{X},\mc{G})$. Consider the morphism of Shimura data obtained in Lemma \ref{lem:Lovering-lem}. Then, as the map $\alpha^c\colon (\mc{G}_2^c)^\der\to (\mc{G}_1\times\mc{T}^c)^\der$ is an isogeny, it induces an isomorphism
\begin{equation*}
    R_{\mc{G}_2,\mu_{h,2}^c}\isomto R_{\mc{G}_1\times\mc{T}^c,\mu_{h,1}^c\times \mu_{h,\mc{T}}^c}. 
\end{equation*}
Moreover, for the same reason, for any $x_2$ in $\ms{S}_{\mathsf{K}_p^2}(\ov{\bb{F}}_p)$ we obtain an induced isomorphism 
\begin{equation*}
    \alpha^c\colon \wh{\mc{O}}_{\ms{S}_{\mathsf{K}^p_2},x_2}\to \wh{\mc{O}}_{\ms{S}_{\mathsf{K}_1^p\times\mathsf{K}_\mc{T}^p},(x_1,x_\mc{T})},
\end{equation*}
as $\alpha_{\mathsf{K}^p_2,\mathsf{K}^p}$ is finite \'etale by Lemma \ref{lem:isogeny-finite-etale}. Thus, we may use Lemma \ref{lem:ito-compatability}, together with the claims in the case of Hodge and special type, to deduce the existence of an isomorphism $i_{x_2}\colon R_{\mc{G}_2^c,\mu_{h,2}^c}\isomto \wh{\mc{O}}_{\ms{S}_{\mathsf{K}^p_2},x_2}$ 
such that the prismatic $\mc{G}$-torsors with $F$-structure $\omega_{\mathsf{K}^p_2,\smallprism}|_{\wh{\mc{O}}_{\ms{S}_{\mathsf{K}^p_2},x_2}}$ and $i_{x_2}^\ast(\omega_{b_{x_2}}^\univ)$ for $\mc{G}_2^c$ agree when pushed forward along $\alpha^c\colon \mc{G}_2^c\to \mc{G}_1\times\mc{T}^c$. As they have further natural identifications when pushed forward to $\mc{G}^\mr{ab}$, this implies they are isomorphic (see \cite[Proposition 2.1.6]{DanielsYoucis}). Finally, as the morphism $\beta\colon \mc{G}_2^\der\to\mc{G}^\der$ is an isogeny, it again induces isomorphisms $R_{\mc{G}^c_2,\mu_{h,2}^c}\to R_{\mc{G}^c,\mu_h^c}$ and $\wh{\mc{O}}_{\ms{S}_{\mathsf{K}^p_2},x_2}\to \wh{\mc{O}}_{\ms{S}_{\mathsf{K}^p},x}$, where $x$ is the image of $x_2$. If $i_x\colon R_{\mc{G}^c,\mu^c_h}\isomto \wh{\mc{O}}_{\ms{S}_{\mathsf{K}^p},x}$ is the resulting isomorphism, then there are isomorphisms between  $\omega_{\mathsf{K}^p,\smallprism}|_{\wh{\mc{O}}_{\ms{S}_{\mathsf{K}^p},x}}$ and $i_{x}^\ast(\omega_{b_{x}}^\univ)$. While arbitrary $x$ may not be the image of such an $x_2$, we may reduce to this case by Lemma \ref{lem:transitivity-on-conn-comp} and \eqref{eq:Hecke-action-local-system-compat}.
\end{proof}

We observe an important corollary of the above proof, which was used previously several times (see Theorem \ref{thm:prismatic-F-gauge-realization} and Theorem \ref{thm:prismatic-crystalline-comparison}).

\begin{cor}\label{cor:prismatic-realization-lff} The prismatic realization functor $\omega_{\mathsf{K}^p,\smallprism}$ belongs to $\cat{Tors}_{\mc{G}^c}^{\varphi,-\mu_h^c}((\wh{\ms{S}}_{\mathsf{K}^p})_\smallprism)$. In particular, $\omega_{\mathsf{K}^p,\smallprism}$ takes values in $\cat{Vect}^{\varphi,\mr{lff}}((\wh{\ms{S}}_{\mathsf{K}^p})_\smallprism)$.
\end{cor}
\begin{proof} It suffices to show that for each small open subset $\Spf(R)$ of $\widehat{\mathscr{S}}_{\mathsf{K}^p}$, and for every point $x$ of the special fiber of $\Spf(R)$, there exists a $p$-adically etale neighborhood $\Spf(S)\to \Spf(R)$ such that the Frobenius for $\omega_{\mathsf{K}^p,\smallprism}$ over $\mathfrak{S}_S$ is in the double coset $\mathcal{G}^c(\mathfrak{S}_S)\mu^c_h(E)^{-1}\mathcal{G}^c(\mathfrak{S}_S)$. 

By moving to an \'etale neighborhood if necessary, we may assume without loss of generality that the underlying $\mathcal{G}^c$-torsor is trivial on $\Spf(R)$. Write $g$ for the element of $\mathcal{G}^c(\mathfrak{S}_R[1/E])$ corresponding to the Frobenius for $\omega_{\mathsf{K}^p,\smallprism}$ on $\mf{S}_R=R\ll t\rr$. Consider the functor
\begin{equation*}
F\colon \cat{Alg}_{R\ll t\rr}\to\mathbf{Set},\quad A\mapsto \{(h,h')\in\mathcal{G}^c(A)\times\mathcal{G}^c(A):hgh'=\mu^c_h(E)^{-1}\in \mathcal{G}^c(A[1/E])\}.
\end{equation*}
Let $y$ be the point of $\Spec(R\ll t\rr)$ equal to $(x,t)$, with the obvious meaning. Observe that we have the equality $\widehat{\mathcal{O}}_{\Spec(R\ll t\rr),y}=\widehat{\mathcal{O}}_{\Spf(R),x}\ll t\rr$. Thus, as a result of Theorem \ref{thm:ito-shim-comp} (and the description of the universal deformation in \cite[Theorem 4.4.2]{Ito2}), we have that $F(\widehat{\mathcal{O}}_{\Spec(R\ll t\rr),y})$ is non-empty. The claim then follows from Artin approximation.

More precisely, first note that $R\ll t\rr$ is excellent (see \cite[Proposition 1.12]{IKY1}). Moreover, the functor $F$ is clearly limit-preserving as $\mathcal{G}^c$ is. Thus, by Artin approximation for an excellent base (see \cite[Theorem 3.4]{AlperHallRydh}), there exists some affine etale neighborhood $\Spec(B)\to \Spec(R\ll t\rr)$ containing $y$ in its image and with $F(\Spec(B))$ non-empty. Let $A$ be the $(p,t)$-adic completion of $B$, so that $\Spf(A)\to\Spf(R\ll t\rr)$ is a $(p,t)$-adically etale neighborhood of $y$. Set $S=A/tA$. Then, $\Spf(S)\to\Spf(R)$ is a $p$-adically etale map, and there is a unique deformation (by the topological invariance of the etale site of a formal scheme) to a $(p,t)$-adically etale map over $\Spf(R\ll t\rr)$ and, in fact, it must be $\Spf(S\ll t\rr)$. Thus, in fact, $A=S\ll t\rr$ where $\Spf(S)\to\Spf(R)$ is a $p$-adically etale neighborhood of $x$. Observe then that, by set-up, $F(S\ll t\rr)$ is non-empty, but this means precisely that the Frobenius is in the double coset of $\mu_h^c(E)^{-1}$ over $S$ as desired.
\end{proof}

\subsection{A prismatic characterization of integral models}

Throughout this section we fix notation and conventions as in \S\ref{s:applications-to-Shimura-varieties}, and in particular fix $(\mb{G},\mb{X},\mc{G})$ to be an unramified Shimura datum of abelian type.

\subsubsection{Characterization of integral canonical models}\label{ss:characterization} Throughout this subsection let us fix a neat compact open subgroup $\mathsf{K}^p\subseteq \mb{G}(\A_f^p)$. Recall from \S\ref{ss:pot-crys-strat} that there exists a potentially crystalline locus $U_{\mathsf{K}^p}\subseteq \Sh_{\mathsf{K}_0\mathsf{K}^p}^\an$ of $\nu_{\mathsf{K}_0\mathsf{K}^p,\et}$ or, equivalently $\omega_{\mathsf{K}^p,\et}$.

Consider a smooth formal $\mc{O}_E$-model $\mf{X}_{\mathsf{K}^p}$ of $U_{\mathsf{K}^p}$, and a prismatic model $\zeta_{\mathsf{K}^p}$ of $\omega_{\mathsf{K}^p,\an}$, i.e.\@, an object of $\mc{G}^c\text{-}\cat{Vect}^\varphi((\mf{X}_{\mathsf{K}^p})_\smallprism)$ with $T_\et\circ \zeta_{\mathsf{K}^p}$ isomorphic to $\omega_{\mathsf{K}^p,\an}$. For $x$ in $\mf{X}_{\mathsf{K}^p}(\ov{\bb{F}}_p)$, there is an element $\bm{b}_{x,\crys}$ in $C(\mc{G}^c)$ associated to the $F$-crystal with $\mc{G}^c$-structure given by $\underline{\bb{D}}_\crys\circ (\zeta_{\mathsf{K}^p})_x$. Fix $\mu_h^c$ in $\bm{\mu}_h^c$. Then, we have the following property of $\bm{b}_{x,\crys}$.

\begin{lem} The element $\bm{b}_{x,\crys}$ lies in the image of the map
\begin{equation*}
    \mc{G}^c(\breve{\Z}_p)\sigma (\mu_h^c(p))^{-1}\mc{G}^c(\breve{\Z}_p)\to C(\mc{G}^c). 
\end{equation*}
\end{lem}
\begin{proof} As $\mf{X}$ is smooth over $\mc{O}_E$, we know that the specialization map $\mr{sp}\colon |U_{\mathsf{K}^p}|^\mr{cl}\to \mf{X}_{\mathsf{K}^p}(\ov{\bb{F}}_p)$ is surjective. Let $y$ be a point of $|U_{\mathsf{K}^p}|^\mr{cl}$ such that $\mr{sp}(y)=x$. Then, a simple specialization argument shows that $\bm{b}_{x,\crys}$ is equal to the element associated to the isocrystal with $\mc{G}$-structure associated to $\underline{\bb{D}}_\crys\circ (\omega_{\mathsf{K}^p,\et})_y$. But, the claim then follows from Lemma \ref{lem:C(G)-containment}.
\end{proof}

Choose an element $b_x$ in $\mc{G}^c(\breve{\Z}_p)\mu_h^c(p)^{-1}\mc{G}^c(\breve{\Z}_p)$ 
such that $\sigma (b_x)$ maps to $\bm{b}_{x,\crys}$ in $C(\mc{G}^c)$.

\begin{defn}\label{defn:integral-canonical-model} A \emph{prismatic integral canonical model} of $U_{\mathsf{K}^p}$ is a smooth and separated formal $\mc{O}_E$-model $\mf{X}_{\mathsf{K}^p}$ such that there exists a prismatic model $\zeta_{\mathsf{K}^p}$ of $\omega_{\mathsf{K}^p,\an}$ with the following property: 
for each $\mathsf{K}^p$ and each $x$ in $\mf{X}_{\mathsf{K}^p}(\ov{\bb{F}}_p)$ there exists an isomorphism $\Theta_x^\smallprism\colon R_{\mc{G}^c,\mu^c_h}\isomto \wh{\mc{O}}_{\mf{X}_{\mathsf{K}^p},x}$ such that $(\Theta_x^\smallprism)^*(\omega_{b_x}^\univ)$ is isomorphic to the pullback of $\zeta_{\mathsf{K}^p}$ to $\wh{\mc{O}}_{\mf{X}_{\mathsf{K}^p},x}$.
\end{defn}

We now aim to show that the unique prismatic integral canonical model of $U_{\mathsf{K}^p}$ is $\wh{\ms{S}}_{\mathsf{K}^p}$. Our proof relies on providing a slight extension of the fully-faithfulness portion of \cite[Theorem A]{GuoReinecke} to certain semi-stable formal schemes, mimicking the argument in \cite[Theorem 3.29]{DLMS}. 

To state this let us fix a complete discrete valuation ring $\mc{O}_K$ with fraction field $K$ and perfect residue field $k$. Set $W$ to be $W(k)$, and fix a uniformizer $\varpi$ of $K$.

\begin{prop}\label{prop: full faithfulness for semistable rings}
    Set $R$ to be $\mc O_K\lbb x_1,\ldots,x_d\rbb/(x_1x_2\cdots x_m-\varpi)$, where $d$ and $m$ are integers with $1\leqslant m\leqslant d$. Then the \'etale realization functor 
    \begin{equation*}
        T_{\et}\colon \cat{Vect}^\varphi(R_\smallprism)\to \cat{Loc}_{\Z_p}(R[1/p])
    \end{equation*}
    is fully faithful. 
\end{prop}

As in \cite{Ito1}, we consider the following Breuil--Kisin type prism. Let $\mf S_R$ be the ring $W\lbb x_1,\ldots,x_d\rbb$ equipped with a Frobenius lift $\phi$ determined by $\phi(x_i)=x_i^p$, and $E$ in  $\mf S_R$ be the polynomial $E_\varpi(x_1x_2\cdots x_m)$, where $E_\varpi$ is the minimal polynomial of $\varpi$ relative to $\mathrm{Frac}(W)$. Then the pair $(\mf S_R,(E))$ defines an object of $R_\smallprism$. 

\begin{lem}\label{lem: semistable BK prism covers}
    The object $(\mf S_R,(E))$ covers the final object $\ast$ of $\cat{Sh}(R_\smallprism)$. 
\end{lem}
\begin{proof}
    Similarly to \cite[Proposition 1.16]{IKY1}, the assertion follows from \cite[Proposition 1.11]{IKY1} (using \cite[Proposition 5.8]{AnschutzLeBrasDD} in place of \cite[Lemma 1.15]{IKY1}).
\end{proof}

Thus, we can regard a prismatic $F$-crystal on $R_\smallprism$ as a finite free Breuil--Kisin module equipped with a descent datum. More precisely, we let $\mf S^{(1)}_R$ be $(\mf S_R\wh{\otimes}_{\Z_p}\mf S_R)\left\{\frac{J}{E}\right\}_\delta^\wedge$ where
\begin{equation*} 
J\defeq\ker\left(\mf S_R\wh{\otimes}_{\Z_p}\mf S_R\to \mf S_R\to \mf S_R/(E)\isomto R\right).
\end{equation*}
As in \cite[Example 3.4]{DLMS}, it represents the self-product of $\mf S_R$ over $\ast$ in $\cat{Sh}(R_\smallprism)$. 

For $i=1,\ldots,d$, let $\varepsilon_i$ denote the product of $x_j$ for $1\leqslant j\leqslant d$ excluding $i$. Denote the ring 
$R[\nicefrac{1}{\varepsilon_i}]_p^\wedge$  by $R_i$, which 
is a base ring in the sense in \cite[\S1.1.5]{IKY1}. Using the map $\mc O_K\langle x_j^{\pm1};j\ne i\rangle\to R_i$ given by sending $x_j$ to $x_j$ as a formal framing, we obtain the relative Breuil--Kisin ring $\mf S_{R_i}$ which we denote by $\mf S_i$. Then we have a morphism $(\mf S_R,(E))\to (\mf S_i,(E_\varpi))$ sending $x_i$ to $\tfrac{u}{\varepsilon_i}$ and $x_j$ to $x_j$ for $j\ne i$.

We let $\mc O_\mc E$ (resp.\@ $\mc O_{\mc E,i}$) be the $p$-adic completion of $\mf S_R[\nicefrac{1}{E}]$ (resp.\@ $\mf S_i[\nicefrac{1}{E_\varpi}]$). 
We use the following lemma to reduce Proposition \ref{prop: full faithfulness for semistable rings} to the case of $R_i$. 
\begin{lem}\label{lem: intersection is S}
    Let $\mf S_{(i)}$ denote the intersection $\mf S_i\cap \mc O_\mc E$ in the ring $\mc O_{\mc E,i}$. Then the inclusion $\mf S_R\subseteq \bigcap_{i=1}^d\mf S_{(i)}$ in $\mc O_\mc E$ is an equality.
\end{lem}

\begin{proof}
    We put $\mf S'\defeq \bigcap_{i=1}^d\mf S_{(i)}$. Since $\mf{S}_R$ and $\mf S'$ are $p$-adically complete, it suffices to show that the containment $\mf{S}_R\subseteq \mf{S}'$ is an equality modulo $p$. 
    We consider the commutative diagram
    \bx{
    \mf S_R/(p)=k\lbb x_1,\ldots,x_d\rbb \ar[r]\ar[d]
    &\left(k\lbb x_j;j\ne i\rbb[\nicefrac{1}{x_j};j\ne i]\right)\ll u\rr =\mf S_i/(p) \ar[d]
    \\ \mc O_\mc E/(p)=k\lbb x_1,\ldots,x_d\rbb[\nicefrac{1}{x_1\cdots x_m}]\ar[r]
    &\left(k\lbb x_j;j\ne i\rbb[\nicefrac{1}{x_j};j\ne i]\right)\lbb u\rbb[\nicefrac{1}{u}]=\mc O_{\mc E,i}/(p),
    }\ex
    in which all the maps are injective. Setting $\overline{\mf S}_{(i)}$ to be the intersection of $\mf S_i/(p)$ and $\mc O_\mc E/(p)$ in $\mc O_{\mc E,i}/(p)$, we have $\mf S_R/(p)=\bigcap_{i=1}^d\overline{\mf S}_{(i)}$. 
    We claim that the induced surjection $\mf S'/(p)\to \bigcap_{i=1}^d\ov{\mf S}_{(i)}$ is an isomorphism. 
    This is equivalent to the equality 
    \be
        p\mc O_\mc E\cap \bigcap_i\mf S_{(i)}=p\cdot \bigcap_i\mf S_{(i)}.
    \ee
    We note that the right hand side is equal to $\bigcap_i(p\mf S_{(i)})$ as $p$ is a nonzerodivisor in $\mc O_\mc E$. Hence, it suffices to show the equality $p\mc O_\mc E\cap\mf S_{(i)}=p\mf S_{(i)}$ for all $i$. But this follows from the injectivity of the right vertical map in the above diagram. 
\end{proof}

\begin{proof}[Proof of Proposition \ref{prop: full faithfulness for semistable rings}]
    By the proof of Lemma \ref{lem: semistable BK prism covers} the faithfulness portion of the claim is reduced to the case of a perfectoid base, which is clear. Thus, it suffices to check fullness. Let $\mc{F}$ and $\mc{F}'$ be two objects of $\cat{Vect}^\varphi(R_\smallprism)$ and let $T_\et(\mc{F})\to T_\et(\mc{F})$ be a morphism in $\cat{Loc}_{\Z_p}(R[1/p])$, which, by \cite[Corollary 3.7]{BhattScholzeCrystals}, corresponds to a morphism $\mc{F}[\nicefrac{1}{\mc I_\smallprism}]_p^\wedge\to \mc{F}'[\nicefrac{1}{\mc I_\smallprism}]_p^\wedge$ of prismatic Laurent $F$-crystals on $R_\smallprism$ (cf.\ \cite[\S2.2]{IKY1}). Let $\mf M$ and $\mf M'$ (resp.\ $\mc M$ and $\mc M'$) denote the evaluation of $\mc{F}$ and $\mc{F}'$ (resp.\ $\mc{F}[\nicefrac{1}{\mc I_\smallprism}]_p^\wedge$ and $\mc{F}[\nicefrac{1}{\mc I_\smallprism}]_p^\wedge$) at the prism $(\mf S_R,(E))$. 

    Since the \'etale realization functor 
    \begin{equation*} 
    \cat{Vect}^\varphi(R_{i,\smallprism})\to \cat{Vect}(R_{i,\smallprism},\mc O_\smallprism[\nicefrac{1}{\mc I_\smallprism}]_p^\wedge)\isomto\cat{Loc}_{\Z_p}(R_i[1/p])
    \end{equation*}
    is fully faithful by \cite[Theorem 3.29 (1)]{DLMS}, the restriction of $\mc{F}[\nicefrac{1}{\mc I_\smallprism}]_p^\wedge\to \mc{F}'[\nicefrac{1}{\mc I_\smallprism}]_p^\wedge$ to $R_{i,\smallprism}$ induces a morphism $\mc{F}|_{R_{i,\smallprism}}\to \mc{F}'|_{R_{i,\smallprism}}$. 
    In particular, the map 
    \begin{equation*} 
    \mc M_i\defeq\mc M\otimes_{\mc O_\mc E}\mc O_{\mc E,i}\to \mc M'_i\defeq\mc M'\otimes_{\mc O_\mc E}\mc O_{\mc E,i}
    \end{equation*}
    sends $\mf M_i\defeq\mf M\otimes_\mf S\mf S_i$ into $\mf M'_i\defeq\mf M'\otimes_\mf S\mf S_i$. 
    Then, by Lemma \ref{lem: intersection is S}, we get that the map $\mc M\to \mc M'$ sends $\mf M$ into $\mf M'$. 

    Since the prism $(\mf S_R,(E))$ is a cover of the final object by Lemma \ref{lem: semistable BK prism covers}, it suffices to show that the map $\mf M\to \mf M'$ is compatible with the descent data for $\mc{F}$ and $\mc{F}'$. 
    To see this, we observe that the natural map $\mf S_R^{(2)}\to \mf S_R^{(2)}[\nicefrac{1}{E}]_p^\wedge$ is injective. Indeed, it can be checked after passing modulo $p$, where it is reduced to showing that $E$ is a nonzerodivisor in $\mf S_R^{(2)}/(p)$, which follows from the flatness of $\mf S_R\to \mf S_R^{(2)}$, cf.\ \cite[Lemma 3.5]{DLMS}). Thus, the assertion follows from the compatibility of the map $\mc M\to \mc M'$ with the descent data for $\mc{F}[\nicefrac{1}{\mc I_\smallprism}]_p^\wedge$ and $\mc{F}'[\nicefrac{1}{\mc I_\smallprism}]_p^\wedge$.     
\end{proof}

We are now ready to prove our uniqueness claim concerning prismatic integral canonical models of $U_{\mathsf{K}^p}$. We roughly follow the strategy employed in \cite[Theorem 7.1.7]{Pappas}, with some key differences owing to the more formal geometry and $p$-adic Hodge theory nature of our setup. 

\begin{thm}\label{thm:prismatic-characterization-completion} The unique prismatic integral canonical model of $U_{\mathsf{K}^p}$ is $\wh{\ms{S}}_{\mathsf{K}^p}$. 
\end{thm}
\begin{proof} That $\wh{\ms{S}}_{\mathsf{K}^p}$ is a prismatic integral canonical model of $U_{\mathsf{K}^p}$ follows from combining Theorem \ref{thm:main-Shimura-theorem-abelian-type-case}, Proposition \ref{prop:crystalline-locus}, Proposition \ref{thm:ito-shim-comp}. Thus, it suffices to show that if $\mf{X}_{\mathsf{K}^p}$ and $\mf{X}'_{\mathsf{K}^p}$ are two prismatic integral canonical models of $U_{\mathsf{K}^p}$, then they are isomorphic.

Denote by $\mf{X}''_{\mathsf{K}^p}$ the normalization of $\mf{X}_{\mathsf{K}^p}\times_{\Spf(\mc{O}_E)}\mf{X}'_{\mathsf{K}^p}$ in $U_{\mathsf{K}^p}$. More precisely, we set $\mf{X}''_{\mathsf{K}^p}$ to be the relative formal spectrum $\underline{\Spf}(\mc{A})\to \mf{X}_{\mathsf{K}^p}\times_{\Spf(\mc{O}_E)}\mf{X}'_{\mathsf{K}^p}$, where $\mc{A}$ is the integral closure of $\mc{O}_{\mf{X}_{\mathsf{K}^p}\times_{\Spf(\mc{O}_E)}\mf{X}'_{\mathsf{K}^p}}$ in $s_\ast(\mc{O}_{U_{\mathsf{K}^p}}^+)$, where $s\colon (U_{\mathsf{K}^p},\mc{O}^+_{\mathsf{U}_{\mathsf{K}^p}})\to \mf{X}_{\mathsf{K}^p}\times_{\Spf(\mc{O}_E)}\mf{X}'_{\mathsf{K}^p}$ is the composition of the following map of locally ringed spaces
\begin{equation*}
    (U_{\mathsf{K}^p},\mc{O}_{U_{\mathsf{K}^p}}^+)\xrightarrow{\Delta}(U_{\mathsf{K}^p}\times_{\Spa(E)}U_{\mathsf{K}^p},\mc{O}_{U_{\mathsf{K}^p}\times_{\Spa(E)}U_{\mathsf{K}^p}}^+)\xrightarrow{\mr{sp}} \mf{X}_{\mathsf{K}^p}\times_{\Spf(\mc{O}_E)}\mf{X}'_{\mathsf{K}^p}.
\end{equation*}
As $E$ is a discrete valuation field, and therefore the local rings of each formal scheme and rigid space are excellent, these normalizations are finite over their original base and so topologically of finite type and normal (cf.\@ \stacks{0AVK} and \stacks{035L}). Let $\pi\colon \mf{X}''_{\mathsf{K}^p}\to \mathfrak{X}_{\mathsf{K}^p}$ and $\pi'\colon \mf{X}''_{\mathsf{K}^p}\to \mf{X}'_{\mathsf{K}^p}$ be the natural projection maps. We show that $\pi$ and $\pi'$ are isomorphisms. 

To prove this, fix a point $x''$ in $\mf{X}''_{\mathsf{K}^p}(\ov{\bb{F}}_p)$ and let $x$ and $x'$ be their images in $\mf{X}_{\mathsf{K}^p}$ and $\mf{X}'_{\mathsf{K}^p}$, respectively. Let $\bm{b}_{x,\crys}$ and $\bm{b}_{x',\crys}$ be as in the definition of a prismatic integral canonical model. Observe that $\bm{b}_{x,\crys}$ actually equals $\bm{b}_{x',\crys}$. Indeed, from the diagram of isomorphisms
\begin{equation*}
    \Spec(k(x))\isomfrom \Spec(k(x''))\isomto \Spec(k(x')),
\end{equation*}
and the identification of $T_\et\circ \zeta_{\mathsf{K}^p}$ and $T_\et\circ \zeta_{\mathsf{K}^p}'$ with $\omega_{\mathsf{K}^p,\an}$, we obtain an isomorphism
\begin{equation*}
    T_\et\circ (\zeta_{\mathsf{K}p})_{x''}\cong T_\et\circ (\zeta'_{\mathsf{K}p})_{x''}, 
\end{equation*}
from where the claim follows by \cite[Theorem A]{GuoReinecke}. Denote this common class of $\mb{b}_x$, $\mb{b}_{x'}$ by $\mb{b}_{x'',\crys}$, and choose an element $b_{x''}$ of $\mc{G}^c(\breve{\Z}_p)\mu_h^c(p)^{-1}\mc{G}^c(\breve{\Z}_p)$ such that $\sigma (b_{x''})$ maps to $\mb{b}_{x'',\crys}$.

Let $\mc{O}$, $\mc{O}'$, and $\mc{O}''$ be the complete local rings of $x$, $x'$, and $x''$ of their respective formal schemes. By excellence each of these complete local rings is normal and formally of finite type over $\mc{O}_E$
(see \stacks{0C23}). Choose isomorphisms $\Theta_x^\smallprism$ and $\Theta^\smallprism_{x'}$ as in the definition of a prismatic integral canonical model. We claim that the following diagram commutes:
\begin{equation}\label{eq:first-charcterization-eq}
\begin{tikzcd}[sep=scriptsize]
	& {\Spf(\mc{O}'')} \\
	{\Spf(\mc{O})} && {\Spf(\mc{O}')} \\
	& {\Spf(R_{\mc{G}^c,\mu_h^c}).}
	\arrow["\pi"', from=1-2, to=2-1]
	\arrow["{\pi'}", from=1-2, to=2-3]
	\arrow["{\Theta_x^{\smallprism}}"', from=2-1, to=3-2]
	\arrow["{\Theta_{x'}^\smallprism}", from=2-3, to=3-2]
\end{tikzcd}
\end{equation}
To prove this, we first make the following observation.

\vspace*{5 pt}

\noindent\textbf{Claim:} There exists an epimorphism of formal schemes of the form $\Spf(R)\to \Spf(\mc{O}'')$, where
\begin{equation*}
    R=\mc{O}_K\llbracket t_1,\ldots,t_n,x_1,\ldots,x_m\rrbracket/(x_1\cdots x_m-\varpi),
\end{equation*} 
with notation as in Proposition \ref{prop: full faithfulness for semistable rings}.
\begin{proof} Let $\Spf(A)$ be an affine open neighborhood of $x''$ in $\mf{X}''_{\mathsf{K}^p}$. As $\Spf(A)_\eta$ is an open subset of the smooth rigid space $U_{\mathsf{K}^p}$, it is smooth, and so by Elkik's algebraization theorem (see \cite[Th\'eor\`eme 7]{Elkik}) there exists some finite type smooth morphism $\Spec(B)\to \Spec(\mc{O}_E)$ such that $A$ is isomorphic to the $p$-adic completion of $B$ over $\mc O_E$. Let $f\colon Y\to \Spec(B)$ be a strictly semi-stable over $\mc{O}_K$ (for some finite extension $K$ of $E$) alteration of $\Spec(B)$ as in \cite[Theorem 6.5]{deJongAlteration}, and choose any closed point $y$ of of $Y$ mapping to $x''$. Then $\wh{\mc{O}}_{Y,y}$ is isomorphic to
\begin{equation*}
    \mc{O}_K\llbracket x_1,\ldots,x_d\rrbracket/(x_1\cdots x_m-\varpi)
\end{equation*}
over $\mc O_K$ for some integers $d$ and $m$ with $1\leqslant m\leqslant d$. We then claim that the induced map $f\colon \Spf(\wh{\mc{O}}_{Y,y})\to \Spf(\mc{O}'')$ is an epimorphism, from where the conclusion will follow. But, the map $f\colon \Spec(\mc{O}_{Y,y})\to \Spec(B_x)$ is dominant by assumption, and thus induces an injection $B_x\to \mc{O}_{Y,y}$. As both the source and target are regular local rings we deduce from \cite[I, Corollaire 3.9.8]{EGASpringer} that $\mc{O}''\to\wh{\mc{O}}_{Y,y}$ is an injection. Since $\wh{\mc{O}}_{Y,y}$ and $\mc{O}''$ are complete local rings, the claim follows.
\end{proof}

To prove that Equation \eqref{eq:first-charcterization-eq} commutes, it thus suffices to show that the outer square of the following diagram commutes
\begin{equation*}
    \begin{tikzcd}[sep=scriptsize]
	& {\Spf(R)} \\
	& {\Spf(\mc{O}'')} \\
	{\Spf(\mc{O})} && {\Spf(\mc{O}')} \\
	& {\Spf(R_{\mc{G}^c,\mu_h^c}),}
	\arrow["\pi"', from=2-2, to=3-1]
	\arrow["{\pi'}", from=2-2, to=3-3]
	\arrow["{\Theta_x^{\smallprism}}"', color={black}, from=3-1, to=4-2]
	\arrow["{\Theta_{x'}^\smallprism}", color={black}, from=3-3, to=4-2]
	\arrow[from=1-2, to=2-2]
	\arrow["f"', color={black}, curve={height=12pt}, from=1-2, to=3-1]
	\arrow["{f'}", color={black}, curve={height=-12pt}, from=1-2, to=3-3]
\end{tikzcd}
\end{equation*}
where $f$ and $f'$ are defined to make the triangle diagrams they sit in commute. But, observe that as $R$ is a complete regular local ring, it suffices by the universality condition of $\Spf(R_{\mc{G}^c,\mu_h^c})$, and our definition of a prismatic integral canonical models, to show that $f^\ast(\zeta_{\mathsf{K}^p})$ is isomorphic to $(f')^\ast(\zeta_{\mathsf{K}^p}')$. But, by setup we know that 
\begin{equation*}
    T_\et\circ f^\ast(\zeta_{\mathsf{K}^p})\cong T_\et\circ (f')^\ast(\zeta_{\mathsf{K}^p}'),
\end{equation*}
and so the claim follows from Proposition \ref{prop: full faithfulness for semistable rings}.

Given the commutativity of \eqref{eq:first-charcterization-eq}, we can now argue as in \cite[Proposition 6.3.1 (b)]{Pappas} to show that $\pi\colon \Spf(\mc{O}'')\to \Spf(\mc{O})$ and $\pi'\colon \Spf(\mc{O}'')\to\Spf(\mc{O}')$ are isomorphisms. Indeed, the commutativity of \eqref{eq:first-charcterization-eq} is equivalent to $\pi\otimes\pi'\colon\mc O\wh{\otimes}_{\mc O_{\breve E}}\mc O'\to \mc O''$ factorizing through the map 
\begin{equation*}
\mc O\wh{\otimes}_{\mc O_{\breve E}}\mc O'\xrightarrow{a\otimes\id_{\mc O'}} \mc O'\wh{\otimes}_{\mc O'_{\breve E}}\mc O'\xrightarrow{\Delta} \mc O',
\end{equation*}
where $a$ denotes the isomorphism $(\Theta^\smallprism_{x'})^{-1}\circ\Theta^\smallprism_x$. Let $R$ denote the image of $\pi\otimes \pi'$. Then, we see that this factorization gives rise to a surjection $\mc{O}'\to R$.

We claim that $\mc{O}''$ has the same Krull dimension as $\mc O'$. Observe that $\dim(\mc{O}'')=\dim(\mathcal{O}_{\mf{X}'',x''})$ and $\dim(\mc{O}')=\dim(\mc{O}_{\mf{X}',x'})$. As these are closed points on integral formal schemes of finite type over $\mc{O}_K$, they have the same dimension as $\mf{X}''$ and $\mf{X}'$, respectively. But, $\dim(\mf{X}'')$ and $\dim(\mf{X}')$ each decrease by $1$ when passing to the rigid generic fiber, but these generic fibers are isomorphic. 

On the other hand, the dimension of $R$ is equal to the dimension of $\mc{O}''$, as $R\to \mc{O}''$ is an integral embedding (see \stacks{00OK}). Thus, combining these two claims we deduce that the dimension of $R$ and $\mc{O}'$ are the same. Thus, the surjection $\mc{O}'\to R$ must be an isomorphism, being a surjection of integral domains of the same finite dimension.

We then get a finite map $\mc O'\isomto R\to \mc O''$. We claim that this map is an isomorphism. This follows from taking $A=\mathcal{O}_{\mf{X}',x'}$ and $B=\mc{O}_{\mf{X}'',x''}$ in the following lemma.

\begin{lem} Let $(A,\mf{m})$ and $ (B,\mf{n})$ be normal local Noetherian rings flat over $\mathbb{Z}_{(p)}$. Suppose that $(A,\mf{m})\to (B,\mf{n})$ satisfies: (1) $\widehat{A}_\mf{m}\to\wh{B}_\mf{n}$ is finite, (2) $A[\nicefrac{1}{p}]\to B[\nicefrac{1}{p}]$ is an isomorphism. Then, $\wh{A}_\mf{m}\to\wh{B}_\mf{n}$ is an isomorphism. 
\end{lem}
\begin{proof} As $\wh{A}_\mf{m}\to\wh{B}_\mf{n}$ is a finite map between normal domains, it suffices to show that the map $\wh{A}_\mf{m}[\nicefrac{1}{p}]\to\wh{B}_\mf{n}[\nicefrac{1}{p}]$ is an isomorphism. Let us begin by observing that the map $A\to B$ is automatically injective as the source and target are both $\mathbb{Z}_{(p)}$-flat and the map $A[\nicefrac{1}{p}]\to B[\nicefrac{1}{p}]$ is injective. As $(A,\mf{m})$ is a normal domain, we deduce from \cite[I, Corollaire 3.9.8]{EGASpringer} that $\wh{A}_\mf{m}\to\wh{B}_\mf{n}$ is injective, and thus that $\wh{A}_\mf{m}[\nicefrac{1}{p}]\to \wh{B}_\mf{n}[\nicefrac{1}{p}]$ is injective. Thus, it suffices to show that $\wh{A}_\mf{m}[\nicefrac{1}{p}]\to \wh{B}_\mf{n}[\nicefrac{1}{p}]$ is surjective. But, observe that as $A[\nicefrac{1}{p}]\to B[\nicefrac{1}{p}]$ is an isomorphism that the map $\widehat{A}_\mf{m}[\nicefrac{1}{p}]\to (\wh{A}_\mf{m}\otimes_A B)[\nicefrac{1}{p}]$ is an isomorphism. As one has a factorization 
\begin{equation*}
    \wh{A}_\mf{m}[\nicefrac{1}{p}]\to (\wh{A}_\mf{m}\otimes_A B)[\nicefrac{1}{p}]\to \wh{B}_\mf{n}[\nicefrac{1}{p}],
\end{equation*}
with the second map being the obvious one, it suffices to show that the map $\wh{A}_\mf{m}\otimes_A B\to \wh{B}_\mf{n}$ is surjective. But, by Nakayama's lemma, using the fact that $\wh{B}_\mf{n}$ is a finite $\wh{A}_\mf{m}$-module, it suffices to show this surjectivity modulo $\mf{m}$. But, as $\wh{B}_\mf{n}$ is a finite $\wh{A}_\mf{m}$-module, its topology agrees with the $\mf{m}$-adic one. Thus, one has that $B/\mf{m}B$ is naturally equal to $\wh{B}_\mf{n}/\mf{m}\wh{B}_\mf{n}$, and thus the surjectivity of $\wh{A}_\mf{m}\otimes_A B\to \wh{B}_\mf{n}$ modulo $\mf{m}$ is clear.
\end{proof}

 From the above we deduce that the map $\pi'\colon\Spf(\mc{O}'')\to\Spf(\mc{O}')$ is an isomorphism and, by symmetry, the same holds for $\pi$. We are then done by Lemma \ref{lem:formal-scheme-isom} below.
\end{proof}

\begin{lem}\label{lem:formal-scheme-isom} Let $\alpha\colon \mf{Y}_1\to \mf{Y}_2$ be a morphism of finite type flat formal $\mc{O}_E$-schemes such that: (a)  $\alpha_\eta$ is an isomorphism of rigid $E$-spaces, (b) for every point $y_1$ of $\mf{Y}_1(\ov{\bb{F}}_p)$ with $y_2=\alpha(y_1)$ the induced map $\wh{\mc{O}}_{\mf{Y}_2,y}\to \wh{\mc{O}}_{\mf{Y}_1,y_1}$ is an isomorphism. Then, $\alpha$ is an isomorphism.
\end{lem}
\begin{proof} For each $n\geqslant 0$ let $\alpha_n\colon \mf{Y}_{1,n}\to \mf{Y}_{2,n}$ denote the reduction of $\alpha$ modulo $p^{n+1}$. It suffices to show that $\alpha_n$ is an isomorphism for all $n$. Indeed, we first observe that as the flat locus of each $\alpha_n$ is open, and contains every closed point of the scheme $\mf{Y}_{1,n}$, which are evidently dense by consideration of the variety $\mf{Y}_{1,0}$, we deduce that $\alpha_n$ is flat. Thus, to prove that it's \'etale, it suffices to prove this claim for $\alpha_0$ (see \stacks{06AG}). But, in this case the fact that (b) implies $\alpha$ is \'etale is classical. To prove that $\alpha_0$, and thus each $\alpha_n$ (see loc.\@ cit.\@), is an isomorphism it suffices to show that the fiber over each $\ov{\bb{F}}_p$-point $y_2$ of $\mf{Y}_2(\ov{\bb{F}}_p)$ is a singleton (see \stacks{02LC}). But, as $\mf{Y}_2$ is a finite type and flat over $\mc{O}_E$, the specialization map $\mr{sp}\colon |(\mf{Y}_2)_\eta|^\mr{cl}\to \mf{Y}_{2,0}(\ov{\bb{F}}_p)$ is surjective. Thus, there exists some finite extension $E'$ of $E$ and a morphism $\Spf(\mc{O}_{E'})\to \mf{Y}_2$ whose special fiber is the underlying point of $y_2$. As $\alpha$ is \'etale, we know that $\alpha^{-1}(y_2)$ is a disjoint union of copies of $\Spec(\ov{\bb{F}}_p)$, and so by the topological invariance of the \'etale site, this implies that $\mf{Y}_1\times_{\mf{Y}_2}\Spf(\mc{O}_{E'})$ is a disjoint union of copies of $\Spf(\mc{O}_{E'})$. As $\alpha_\eta$ is an isomorphism though, this number of copies must be one. The claim follows.
\end{proof}

There is also a characterization of the scheme $\ms{S}_{\mathsf{K}^p}$ itself, using the notion of a prismatic $F$-crystal $(\omega',\zeta,\iota)$ on a $p$-adic scheme as in \cite[\S3.3.4]{IKY1}, which we use freely below.

\begin{defn}\label{defn:prismatic-model-local-system} Let $X$ be a finite type separated $E$-scheme, $U\subseteq X^\mr{an}$ an open adic subspace, and $\omega_\et$ a $\mc{G}$-object in de Rham $\Z_p$-local systems on $X^\mr{an}$ (resp.\@ $X$).
\begin{itemize}[leftmargin=.3in]
    \item A locally of finite type separated flat $\mc{O}_E$-scheme $\ms{X}$ is a \emph{model} of $(X,U)$ if there is an isomorphism $X\isomto \ms{X}_E$ carrying $U$ isomorphically onto $\wh{\ms{X}}_\eta$.
    \item A \emph{prismatic model (of type $\bm{\mu}$)} of $\omega$ (resp.\@ $\omega^\mr{an}$) is a $\mc{G}$-object in prismatic $F$-crystals (one of type $\bm{\mu}$) $(\omega',\zeta,\iota)$ on $\ms{X}$ with $\omega'$ isomorphic to $\omega$ (resp.\@ $\omega^\an$).
\end{itemize}
\end{defn}

We often identify a prismatic model $(\omega',\zeta,\iota)$ of $\omega$ with just its prismatic $F$-crystal component on $\wh{\ms{X}}$, i.e., with just $\zeta$. So, we informally speak of $\zeta$ being a prismatic $F$-crystal model of $\omega$.

\begin{defn}\label{defn:prismatic-integral-canonical-model} A smooth and separated $\mc{O}_E$-model $\ms{X}_{\mathsf{K}^p}$ of $\Sh_{\mathsf{K}_0\mathsf{K}^p}$ is called a \emph{prismatic integral canonical model} if it is a model of $(\Sh_{\mathsf{K}_0\mathsf{K}^p},U_{\mathsf{K}^p})$ and there exists a prismatic model $\zeta_{\mathsf{K}^p}$ of $\omega_{\mathsf{K}^p,\et}$ such that for every point $x$ of $\ms{X}_{\mathsf{K}^p}(\ov{\bb{F}}_p)$ there exists an isomorphism $    \Theta_x^\smallprism\colon R_{\mc{G}^c,\mu^c_h}\isomto \wh{\mc{O}}_{\ms{X}_{\mathsf{K}^p},x}$ such that $(\Theta_x^\smallprism)^*(\omega_{b_x}^\univ)$ is isomorphic to the pullback of $\zeta_{\mathsf{K}^p}$ to $\wh{\mc{O}}_{\ms{X}_{\mathsf{K}^p},x}$.
\end{defn}

Said differently, a smooth separated $\mc{O}_E$-model $\ms{X}_{\mathsf{K}^p}$ of $\Sh_{\mathsf{K}_0\mathsf{K}^p}$ is a prismatic integral canonical model if $\wh{\ms{X}}_{\mathsf{K}^p}$ is a prismatic integral canonical model of $U_{\mathsf{K}^p}$.

The following is an immediately corollary of \cite[Proposition 3.6]{IKY1} and Theorem \ref{thm:prismatic-characterization-completion}.

\begin{cor}\label{cor:scheme-int-can-model-charac} The unique prismatic integral canonical model of $\Sh_{\mathsf{K}_0\mathsf{K}^p}$ is $\ms{S}_{\mathsf{K}^p}$.
\end{cor}

\subsubsection{Relationship to work of Pappas and Rapoport}\label{ss:PR}
We now discuss the relationship between our work and that in \cite{PappasRapoportI}. Below we shall refer to the conjunction of \cite[Conjecture 4.2.2]{PappasRapoportI} and \cite[Conjecture 4.5]{Daniels} as the \emph{Pappas--Rapoport conjecture}.

We begin by formulating a version of a prismatic integral canonical model as in Definition \ref{defn:prismatic-integral-canonical-model}, but for the entirety of the system $\{\Sh_{\mathsf{K}_0\mathsf{K}^p}\}_{\mathsf{K}^p}$. Namely, by a \emph{smooth $\mb{G}(\A_f^p)$-model} of $\{\Sh_{\mathsf{K}^p\mathsf{K}_0}\}_{\mathsf{K}^p}$, we mean a collection $\{\ms{X}_{\mathsf{K}^p}\}_{\mathsf{K}^p}$ of separated smooth $\mc{O}_E$-schemes together with finite \'etale morphisms $f_{\mathsf{K}^p,\mathsf{K}^{p'}}(g^p)$ modeling $t_{\mathsf{K}^p,\mathsf{K}^{p'}}(g^p)$. 

\begin{defn} A \emph{prismatic integral canonical model} of $\{\Sh_{\mathsf{K}_0\mathsf{K}^p}\}_{\mathsf{K}^p}$ is a smooth $\mb{G}(\A_f^p)$-model $\{\ms{X}_{\mathsf{K}^p}\}_{\mathsf{K}^p}$ such that $\ms{X}_{\mathsf{K}^p}$ is a prismatic integral canonical model of $\Sh_{\mathsf{K}_0\mathsf{K}^p}$ for all $\mathsf{K}^p$.
\end{defn}

We have the following which is an essentially trivial corollary of Theorem \ref{thm:prismatic-characterization-completion}.

\begin{thm}\label{thm:prismatic-characterization} The system $\{\ms{S}_{\mathsf{K}^p}\}_{\mathsf{K}^p}$ is the unique prismatic integral canonical model of $\{\Sh_{\mathsf{K}_0\mathsf{K}^p}\}_{\mathsf{K}^p}$. 
\end{thm}

That said, independent of Theorem \ref{thm:prismatic-characterization-completion} we can show a prismatic integral canonical model satisfies the conditions of the Pappas--Rapoport conjecture.

\begin{prop}\label{prop:compatibility-PR-conjecture}
    Suppose $\{\ms{X}_{\mathsf{K}^p}\}_{\mathsf{K}^p}$ is a prismatic integral canonical model of $\{\Sh_{\mathsf{K}^p\mathsf{K}_0}\}_{\mathsf{K}^p}$. Then, $\{\ms{X}_{\mathsf{K}^p}\}_{\mathsf{K}^p}$ satisfies the conditions of the Pappas--Rapoport conjecture.
\end{prop}
\begin{proof} That condition (a) of \cite[Conjecture 4.5]{Daniels} holds for $\{\ms{X}_{\mathsf{K}^p}\}_{\mathsf{K}^p}$ follows by combining \cite[Proposition 3.6]{IKY1} and the N\'eron--Ogg--Shafarevich criterion (see Lemma \ref{lem:NOS-criterion} below). To show that condition (b) of loc.\@ cit.\@ holds, set $(\ms{P}_{\mathsf{K}^p},\varphi_{\ms{P}_{\mathsf{K}^p}})=T_\sht(\zeta_{\mathsf{K}^p})$, with notation as in \cite[\S3.3]{IKY1}, an object of $\mc{G}^c\text{-}\Sht_{\bm{\mu}_h^c}(\ms{X}_{\mathsf{K}^p})$. By definition, we have that $(\ms{P}_{\mathsf{K}^p})_E=U_\sht(\omega_{\mathsf{K}^p,\et}^\mr{an})\cong \ms{P}_{\mathsf{K}^p,E}$, and thus condition (b) is satisfied. Finally, to verify condition (c) of loc.\@ cit.\@, it suffices by setup to show that there exists an isomorphism 
\begin{equation*}
    \Theta_x\colon \wh{\mc{M}}^\mathrm{int}_{(\mc{G}^c,b_x,\mu_h^c)/x_0}\isomto \Spd(R_{\mc{G}^c,\mu_h^c}), 
\end{equation*}
with the property that $T_\sht(\Theta_x^\ast(\omega_{b_x}^\univ))$ is isomorphic to the universal shutka. Here $\mc{M}^\mathrm{int}_{(\mc{G}^c,b_x,\mu_h^c)}$ is the integral moduli space of shtukas as in \cite[Definition 25.1]{ScholzeBerkeley}, and $\wh{\mc{M}}^\mathrm{int}_{(\mc{G}^c,b_x,\mu_h^c)/x_0}$ is the completion at the neutral point $x_0$ in the sense of \cite{GleasonSpecialization}. But, the existence of such an isomorphism follows from \cite[Theorem 5.3.5]{Ito2} and its proof.
\end{proof}

\begin{lem}[N\'eron--Ogg--Shafarevich criterion]\label{lem:NOS-criterion} Let $R$ be a discrete valuation ring over $\mc{O}_E$ of mixed characteristic $(0,p)$. For an element $\{x_{\mathsf{K}^p}\}$ of $\varprojlim_{\mathsf{K}^p}\Sh_{\mathsf{K}_0\mathsf{K^p}}(R[\nicefrac{1}{p}])$, each induced morphism 
\begin{equation*} 
x_{\mathsf{K}^p}^\an\colon \Spa(R[\nicefrac{1}{p}])\to \Sh_{\mathsf{K}_0\mathsf{K}^p}^\an
\end{equation*}
factorizes through $U_{\mathsf{K}^p}$.
\end{lem}
\begin{proof} As $\{\ms{S}_{\mathsf{K}^p}\}$ satisfies the extension property, there exists a unique element $\{y_{\mathsf{K}^p}\}$ in $\varprojlim_{\mathsf{K}^p}\ms{S}_{\mathsf{K}^p}(R)$ with $(y_{\mathsf{K}^p})_\eta=x_{\mathsf{K}^p}$. Let $\wh{y}_{\mathsf{K}^p}\colon \Spf(\wh{R})\to \ms{S}_{\mathsf{K}^p}$ denote the completion of $y_{\mathsf{K}^p}$. Then, we observe that $x_{\mathsf{K}^p}^\an=(\wh{y}_{\mathsf{K}^p})_\eta$. But, $(\wh{y}_{\mathsf{K}^p})_\eta$ takes values in $U_{\mathsf{K}^p}$ by Proposition \ref{prop:crystalline-locus}
\end{proof}

\begin{rem} The usage of the system $\{\ms{S}_{\mathsf{K}^p}\}$ and its extension property in the proof of Lemma \ref{lem:NOS-criterion} is not strictly necessary. One could also use the method in the proof of Proposition \ref{prop:crystalline-locus} to reduce to the Siegel-type case, and thus to the classic N\'eron--Ogg--Shafarevich theorem.
\end{rem}

\subsection{A syntomic characterization of integral models and Serre--Tate theory} We now discuss how the material from \S\ref{ss:comparison-to-shim-vars} can be upgraded to the realm of prismatic $F$-gauges. We then use this to realize an expectation of Drinfeld and in doing so produce a syntomic characterization of integral canonical models and an analogue of the Serre--Tate theorem.

\subsubsection{The stack \texorpdfstring{$\mr{BT}^{\mc{G},\mu}_\infty$}{of truncted Barsotti--Tate groups with G-structure}}\label{ss:BTGmu} We begin by recalling Gardner--Madapusi--Mathew's representability of the moduli space of prismatic $F$-gauges with $\mc{G}$-structure of type $\mu$, which is based off of previous ideas of Drinfeld as in \cite{Drinfeld}. 

\begin{nota}\label{nota:BT} Fix notation as in Notation \ref{nota:deformation} and additionally set $P_{-\mu}\subseteq \mc{G}$ to be the parabolic given by the dynamic method for the cocharacter $\mu^{-1}$ (see \cite[Theorem 4.1.7]{ConradReductive}).
\end{nota}
For a bounded $p$-adic ring $R$ and an element $n$ of $\bb{N}\cup\{\infty\}$, the notion of an $n$-truncated prismatic $F$-gauge with $\mc{G}$-structure of type $\mu$ over $R$ is defined as in Definition \ref{defn: F-gauge of G mu structure}.

\begin{defn}[{cf.\@ \cite{GMM}}]
   A $p$-adically complete ring $R$ has \emph{$p$-finite differentials} if the $R/p$-module $\Omega^1_{(R/p)/\bb{F}_p}$ is finitely generated. 
\end{defn}
For example, this condition is satisfied when $R$ is a base $W$-algebra (see \cite[Lemma 1.3.3]{deJongCrystalline}), $R$ is an object of $\mc{C}_W$, $k$ is a characteristic $p$ field with finite $p$-basis (see loc.\@ cit.\@ for the definition), or $R$ is perfectoid. We let $\cat{Alg}_{\Z_p}^{p\text{-fin}}$ denote the category of $p$-adically complete rings with $p$-finite differentials, which we equip with the flat topology.

\begin{defn}[{cf.\@ \cite{GMM}}] For $n$ in $\bb{N}\cup\{\infty\}$ we consider the derived prestack
\begin{equation*}
    \mr{BT}^{\mc{G},\mu}_n\colon \cat{Alg}_{\Z_p}^{p\text{-fin},\mr{op}}\to \cat{Grpd}_\infty, \qquad R\mapsto \mr{BT}^{\mc{G},\mu}_n(R),
\end{equation*} of $n$-truncated prismatic $F$-gauges with $\mc{G}$-structure of type $\mu$.
\end{defn}

We then have the following remarkable theorem of Gardner--Mathew. In the following, we use the notion of Weil restriction as in \cite{GMM}: for a prestack $Y$ on $\cat{Alg}_{\Z_p}^{p\text{-fin}}$, set $Y^{(n)}$ to be the prestack on $\cat{Alg}_{\Z_p}^{p\text{-fin}}$ given by $R\mapsto Y(R\otimes^\bb{L}_{\mathbb{Z}_p} \mathbb{Z}/p^n)$.

\begin{thm}[{\cite{GMM}}]\label{thm:GM-main} Fix $n$ in $\bb{N}\cup\{\infty\}$.
\begin{enumerate} 
\item If $n$ is finite, the prestack $\mr{BT}^{\mc{G},\mu}_n$ is a quasi-compact smooth $p$-adic formal Artin stack of dimension $0$ with affine diagonal. 
\item For finite $n$, the natural map $\mr{BT}^{\mc{G},\mu}_{n+1}\to \mr{BT}^{\mc{G},\mu}_n$ is smooth and surjective.
\item{\emph{[Grothendieck--Messing theory]}} For $(R'\to R,\gamma)$ a (nilpotent) divided power thickening in $\cat{Alg}^{p\emph{-fin}}_{W}$ there is a Cartesian diagram of groupoids
\begin{equation}\label{eq:GR-1}
    \begin{tikzcd}
	{\mr{BT}^{\mc{G},\mu}_n(R')} & {BP_{-\mu}^{(n)}(R')} \\
	{\mr{BT}^{\mc{G},\mu}_n(R)} & {X^\mu_n(R'\to R)}
	\arrow[from=1-1, to=1-2]
	\arrow[from=1-1, to=2-1]
	\arrow[from=1-2, to=2-2]
	\arrow[from=2-1, to=2-2]
\end{tikzcd}
\end{equation}
where $X^\mu_n(R'\to R)\defeq B\mc{G}^{(n)}(R')\times_{B\mc{G}^{(n)}(R)} BP_{-\mu}^{(n)}(R)$.
\end{enumerate}
\end{thm}
\begin{proof} The only thing not addressed by \cite{GMM} is Grothendieck--Messing theory for $n=\infty$. But, as a limit of Cartesian diagrams is a Cartesian diagram, this follows by passing to the limit: for a(n affine smooth) $\Z_p$-group scheme $\mc H$ and a $p$-complete ring $R$, the natural maps $B\mc H(R)\to  B\mc H(\Spf(R))\to \lim_nB\mc H^{(n)}(R)$ are isomorphisms (see \cite[Proposition 2.1.4]{HLP}).  
\end{proof}

In diagram \eqref{eq:GR-1}, the vertical maps are the natural ones, the top horizontal map is the restriction along the Nygaard de Rham point (see e.g., see \cite[Definition 1.11]{IKY3}), and the lower horizontal map is such that the projection, denoted $\alpha_\gamma$, to $B\mc G^{(n)}(R')$ is constructed using the map $\wt{x}_{\mr{dR},\gamma}\colon \Spec(R')\to R^\smallprism$ from \cite[Lemma 6.8.1]{GMM}. The map $\wt{x}_{\mr{dR},\gamma}\colon \Spf(R')\to R^\smallprism$ is an extension of prismatic de Rham point (e.g., see \cite[Definition 1.11]{IKY3}). The extension $\wt{x}_{\dR,\gamma}$ depends on $\gamma$.

From Theorem \ref{thm:GM-main} we see that $\mr{BT}^{\mc{G},\mu}_\infty=\varprojlim \mr{BT}^{\mc{G},\mu}_n$ is the projective limit of smooth quasi-compact $p$-adic formal Artin stacks, with smooth surjective transition maps. This is just as in the case of the formal $\Z_p$-stack $\mr{BT}^{h,d}_{p,\infty}$ of $p$-divisible groups of height $h$ and dimension $d$, which is the limit of the smooth Artin formal stack over $\mathbb{Z}_p$s $\mr{BT}_{p,n}^{h,d}$ of $n$-truncated $p$-divisible groups of height $h$ and dimension $d$ (cf.\@ \cite[\S1]{WedhornBT}). This is reasonable as the following result shows.

\begin{thm}[{\cite[Theorem A]{GMM}, \cite[Theorem 1.11]{MondalClassification}}]\label{Madapusi--Mathew--Mondal} Let $\mu_d$ be the cocharacter $(1^d,0^{h-d})$ of $\GL_{h,\Z_p}$. Then, there is a canonical isomorphism 
\begin{equation*}
    \mc{M}_\mr{syn}\colon \mr{BT}^{h,d}_{p,\infty}\isomto \mr{BT}^{\GL_{h,\Z_p},\mu_d}_{\infty}
\end{equation*}
which agrees with Anschutz--Le Bras's filtered Dieudonn\'e functor $\mc{M}_\smallprism$ (cf.\@ \cite{AnschutzLeBrasDD} and specifically \cite[Remark 1.9]{AnschutzLeBrasDD}, as well as \cite[Proposition 3.45]{MondalClassification}) on qrsp rings $R$. Moreover, the functor $\mc{M}_\syn$ preserves duals.
\end{thm}

\begin{eg}[{see \cite[\S11.6]{GMM}}]\label{eg:Siegel} Let us say that $(\mc{G},\mu)$ is of \emph{Siegel type} if it is of the form $(\mr{GSp}(\Lambda_0),\mu_g)$ for some symplectic $\Z_p$-lattice $\Lambda_0$ of rank $2g$. For a $p$-nilpotent ring $R$, we then define $p\text{-}\mr{Div}^{\mc{G},\mu}(R)$ to be the groupoid of \emph{quasi-polarized $p$-divisible groups} over $R$. By definition, such a quasi-polarized $p$-divisible group is a triple $(H,L,\lambda)$ where $H$ is a $p$-divisible group over $R$, $L$ is a rank $1$ $\Z_p$-local system on $R$, and  $\lambda\colon H\isomto H^\vee\otimes L$ such that under the canonical double-duality isomorphism $H\simeq H^{\vee\vee}$ we have that $\lambda^\vee$ corresponds to $-\lambda$.\footnote{Note that precisely $H\otimes L$ denotes the tensor product in fppf $\Z_p$-modules, which is still a $p$-divisible group as this can be checked locally, which reduces us to the case when $L$ is trivial.}

One may then deduce from from Theorem \ref{Madapusi--Mathew--Mondal} that $\mc{M}_\syn$ induces an equivalence of categories $\mr{BT}^{\mc{G},\mu}_\infty(R)\to p\text{-}\mr{Div}^{\mc{G},\mu}(R)$ given by $\mc{P}\mapsto (H,L,\lambda)$. More precisely, $(H,L)$ is uniquely defined so that $f_\ast\mc{P}=(\mc{M}_\syn(H),\mc{M}_\syn(L))$ where $f\colon \mc{G}\to \GL(\Lambda_0)\times \bb{G}_{m,\Z_p}$ is the tautological map, and where in the expression $\mc{M}_\syn(L)$ we are abusing notation and implicitly identifying $L$ with its associated \'etale $p$-divisible group.
\end{eg}

\subsubsection{Deformation theory of prismatic \texorpdfstring{$\mathcal{G}$-$F$-gauges}{G-F-gauges}}\label{ss:defm-theory-G-F-Gauges}

We now show that Ito's universal deformation $(\mc{A}_b^\mr{univ},\varphi_{\mc{A}_b^\mr{univ}})$ from \S\ref{ss:universal-deformation}  is universal not only for $\mc{C}_W^\mr{reg}$ but the larger category $\mc{C}_W$, when interpreted using prismatic $F$-gauges with $\mc{G}$-structure of type $\mu$. 

We continue to use the notation from Notation \ref{nota:BT} and further fix an element $b$ of $\mc{G}(W)\mu(p)^{-1}\mc{G}(W)$. The pair $(\mc{G}_W,\varphi_b)$, where $\varphi_b$ corresponds to left multiplication by $b$, defines an element of $\cat{Tors}^{\varphi,-\mu}_{\mc{G}}(k_\smallprism)$ which by Proposition \ref{prop: F gauge type mu equals F crystal type mu} is equivalent to an object $\mc{P}_b$ of $\mr{BT}^{\mc{G},-\mu}_\infty(k)$. 

Given a functor of groupoids $F\colon \ms{G}_1\to\ms{G}_2$ and an object $x$ of $\ms{G}_2$ we denote by $\mr{fib}(\ms{G}_1\to\ms{G}_2;x)$ the \emph{fiber} over $x$ which, by definition, means the groupoid of pairs $(y,\iota)$ where $y$ is an object of $\ms{G}_2$ and $\iota\colon F(y)\isomto x$ is an isomorphism, where a morphism $(y,\iota)\to (y',\iota')$ is an isomorphism $\sigma\colon y\to y'$ such that $\iota'\circ F(\sigma)=\iota$.

\begin{defn} Let $\mc{Y}$ be a formal stack over $\mathbb{Z}_p$. Let $y$ a point of $\mc{Y}(k)$ for some $k$ perfect extension of $\F_p$. Set $W=W(k)$. Define the \emph{deformation functor} associated to the pair $(\mc{Y},y)$ to be the contravariant functor
\begin{equation*}
    \mr{Def}_{y}\colon \mc{C}_{W}^\mr{op}\to\cat{Grpd},\quad R\mapsto \mr{fib}(\mc{Y}(R)\to \mc{Y}(k);y)
\end{equation*}
We say $\mc{Y}$ has \emph{discrete deformation theory} if for each $p$-nilpotent ring $R$, nilpotent ideal $I$ of $R$, and point $y$ of $\mc Y(R/I)$, the fiber $\mr{fib}(\mc{Y}(R)\to \mc{Y}(R/I);y)$ is a discrete groupoid. 
\end{defn}

For $\mc{Y}=\mr{BT}^{\mc{G},-\mu}_\infty$ and $b$ as above, we shorten the notation $\mr{Def}_{\mc{P}_b}$ to $\mr{Def}_b$. We then have the following result which underlies the rest that follows.

\begin{prop}\label{prop:disc-def-theory} The formal stack $\mr{BT}^{\mc{G},-\mu}_\infty$ over $\Z_p$ has discrete deformation theory.
\end{prop}

Before we prove Proposition \ref{prop:disc-def-theory} we give the following interpretation of Grothendieck--Messing theory that will be used many times below. Fix a PD thickening $(R'\twoheadrightarrow R,\gamma)$ of objects of $\cat{Alg}_{\Z_p}^{p{\text{-fin}}}$. Let $x'$ be an object of $B\mc G(R')$ and set 
\begin{equation*} \mr{BT}^{\mc G,-\mu}_\infty(R')_{x'}\defeq \mr{BT}^{\mc G,-\mu}_\infty(R')\times_{x_{\dR,R'}^\ast,B\mc G(R')}x',\quad \mr{BT}^{\mc G,-\mu}_\infty(R)_{x'}\defeq \mr{BT}^{\mc G,-\mu}_\infty(R)\times_{\alpha_\gamma,B\mc G(R')}x'.
\end{equation*}
Then, we have the following natural Cartesian diagram
\begin{equation}\label{eq:GM-reintrepret}
    \begin{tikzcd}
	{\mr{BT}^{\mc{G},-\mu}_\infty(R')_{x'}} & {\mc{Q}_{x'}/P_{\mu}(R')} \\
	{\mr{BT}^{\mc{G},-\mu}_\infty(R)_{x'}} & {\mc{Q}_{x'}/P_{\mu}(R)}
	\arrow[from=1-1, to=1-2]
	\arrow[from=1-1, to=2-1]
	\arrow[from=1-2, to=2-2]
	\arrow[from=2-1, to=2-2]
\end{tikzcd}
\end{equation}
Here $\mc{Q}_{x'}$ is our shorthand for the $\mc{G}$-torsor $\alpha_\gamma(x')$ on $R'$, which carries a natural action of $P_{\mu}$. 
Here we are interpreting $\mc{Q}_{x'}/P_{\mu}$ as a quotient stack. We observe that \eqref{eq:GM-reintrepret} being Cartesian implies a natural isomorphism of groupoids
\begin{equation}\label{eq:GR-2}
    \mr{fib}\left(\mr{BT}^{\mc{G},-\mu}_\infty(R')\to \mr{BT}^{\mc{G},-\mu}_\infty(R);z\right)\isomto \mr{fib}\left((\mc{Q}_{x'}/P_{\mu})(R')\to (\mc{Q}_{x'}/P_{\mu})(R);t_{z}\right),
\end{equation}
for $z$ in $\mr{BT}^{\mc{G},-\mu}(R)_{x'}$ and $t_z$ the image of $z$ under the bottom arrow in \eqref{eq:GM-reintrepret}.

This is quite helpful as $\mc{Q}_{x'}/P_{\mu}$ is actually a smooth algebraic space (cf.\@ \stacks{06PH} and \stacks{0AHE}). Moreover, by \stacks{06GE}, every surjective morphism $R\to R'$ of Artinian objects in $\mc{C}_W$ may be factorized as a sequence of small surjections in the sense of \stacks{06GD}. But, such a small surjection is, in particular, a square-zero thickening and so has a unique (and so canonical) PD structure $(\gamma_i)$ with $\gamma_i=0$ for $i\geqslant 2$, which we denote by $\gamma$ in all cases.

\begin{proof}[Proof of \ref{prop:disc-def-theory}] Applying \stacks{06GE} to $R\to R/I$ it suffices to show that if $R'\to R$ is a small extension and an object $x$ of $\mr{BT}^{\mc{G},-\mu}_\infty(R)$, then $\mr{fib}\left(\mr{BT}^{\mc{G},-\mu}_\infty(R')\to \mr{BT}^{\mc{G},-\mu}_\infty(R);x\right)$ is discrete. But, given \eqref{eq:GR-2} this is clear as the right-hand side is a discrete groupoid.
\end{proof}

Consider the ring $R_{\mc{G},\mu}$ from \S\ref{ss:universal-deformation}. We abuse notation and conflate $\Spf(R_{\mc{G},\mu})$ (where $R_{\mc{G},\mu}$ is endowed with the topology defined by the maximal ideal) with the natural functor $\mc{C}_W^\mr{op}\to \cat{Set}$ it represents. Using Proposition \ref{prop: summary display comparison}, the pair $(\mc{A}_b^\mr{univ},\varphi_{\mc{A}^\mr{univ}_b})$  from \S\ref{ss:universal-deformation} upgrades to a deformation $(\mc{A}_{b,\mr{syn}}^\mr{univ},\varphi_{\mc{A}^\mr{univ}_{b,\mr{syn}}})$ of $\mc{P}_b$ and so defines a morphism $\rho_b\colon \Spf(R_{\mc{G},\mu})\to \mr{Def}_b$. 

\begin{prop}\label{prop:prorepresentable} The morphism 
\begin{equation*} \rho_b\colon \Spf(R_{\mc{G},\mu})\to \mr{Def}_b
\end{equation*}
is an isomorphism. In particular, the pair $(\Spf(R_{\mc{G},\mu}),(\mc{A}_{b,\mr{syn}}^\mr{univ},\varphi_{\mc{A}_{b,\mr{syn}}^\mr{univ}}))$ represents $\mr{Def}_b$.
\end{prop}
\begin{proof} We use the results of \cite{GMM} to show that $\mr{Def}_b$ is represented by an object $R_b^\mr{univ}$ of $\mc{C}_W^\mr{reg}$. This is sufficient, as the morphism $\rho_b$ induces a bijection when evaluated on any object $R$ of $\mc{C}_W^\mr{reg}$ by the universal property of the pair $(\Spf(R_{\mc{G},\mu}),\mc{A}_{b,\mr{syn}}^\mr{univ},\varphi_{\mc{A}^\mr{univ}_{b,\mr{syn}}})$ (again utilizing Proposition \ref{prop: summary display comparison}). As the source and target of $\rho_b$ belong to $\mc{C}_W^{\mr{reg},\mr{op}}$, this implies that $\rho_b$ is an isomorphism by the Yoneda lemma. To show that $\mr{Def}_b$ is (pro)representable it suffices by Schlessinger's criterion (cf.\@ \stacks{06JM}) to show that $\mr{Def}_b$ satisfies condition $\mr{(RS)}$ as in \stacks{06J2} and that $\dim_k T\mr{Def}_b$ is finite, where this tangent space is as in \stacks{06I2}.\footnote{Note that condition (a) in Schlessinger's criterion is equivalent to condition $\mr{(RS)}$ by definition, as evidently $\mr{Def}_b$ is a predeformation category in the sense of \stacks{06GT} as $\mr{Def}_b(k)$ is a singleton. Moreover, condition (c) is vacuous as $W$ surjects onto $k$.} 

To show that $\mr{Def}_b$ satisfies condition $\mr{(RS)}$ we show that whenever we have a diagram of Artinian objects of $\mc{C}_W$ $R_2\rightarrow R\leftarrow R_1$ such that the left arrow is a small surjection, that 
\begin{equation}\label{eq:Schlessinger}
    \mr{Def}_b(R_1\times_R R_2)\to \mr{Def}_b(R_1)\times_{\mr{Def}_b(R)}\mr{Def}_b(R_2)
\end{equation}
is a bijection of sets. This is indeed sufficient by \stacks{06J5}. As \eqref{eq:Schlessinger} is a map of sets over $\mr{Def}_b(R_1)$ it suffices to show that for each $s$ in this set that the map in \eqref{eq:Schlessinger} induced on fibers over $s$ is a bijection. But, this is naturally interpreted as the map of sets
\begin{equation*}
    \mr{fib}\left(\mr{BT}^{\mc{G},-\mu}_\infty(R_1\times_R R_2)\to \mr{BT}^{\mc{G},-\mu}_\infty(R_1);s\right)\to \mr{fib}\left(\mr{BT}^{\mc{G},-\mu}_\infty(R_2)\to \mr{BT}_\infty^{\mc{G},-\mu}(R);\ov{s}\right),
\end{equation*}
where $\ov{s}$ is the image of $s$ in $\mr{Def}_b(R)$.  But, note that $R_1\times_R R_2\to R_1$ is a small surjection by \stacks{06GH}. So, using \eqref{eq:GR-2}, if $x'=\alpha_\gamma(s)$ and $\ov{x}'=\alpha_\gamma(\ov{s})$, we must show that the natural map
\begin{equation*}
    \mr{fib}\left((\mc{Q}_{x'}/P_{\mu})(R_1\times_R R_2)\to (\mc{Q}_{x'}/P_{\mu})(R_1);t_s\right)\to \mr{fib}\left((\mc{Q}_{\ov{x}'}/P_{\mu})(R_2)\to (\mc{Q}_{\ov{x}'}/P_{\mu})(R);t_{\ov{s}}\right) 
\end{equation*}
is a bijection, where we are using the fact that $\mc{Q}_{x'}$ base changes to $\mc{Q}_{\ov{x}'}$. But, this is clear as $\mc{Q}_{x'}/P_{\mu}$ is a smooth algebraic space.

For $\dim_k T\mr{Def}_b<\infty$, it suffices to observe that, as a special case of \eqref{eq:GR-2}, we have an isomorphism of $k$-spaces $T\mr{Def}_b\to T\mr{Def}_{t_b}$, where $\mr{Def}_{t_b}$ is the deformations of $t_b$ in $(\mc{Q}_b/P_{\mu})(k)$ inside $\mc{Q}_b/P_{\mu}$, with $\mc{Q}_b\defeq x_{\dR,k}^\ast\mc{P}_b$. But, as $\mc{Q}_b/P_{\mu}$ is a smooth algebraic space over $k$ this is clear. 

Finally, to show that the universal deformation ring $R_b^\mr{univ}$ is an object of $\mc{C}_W^\mr{reg}$, it suffices by \stacks{0DYL} to show that $\mr{Def}_b$ is unobstructed in the sense of \stacks{06HP}. But, by \stacks{06HH}, we may restrict to small extensions. We may then apply \eqref{eq:GR-2} again to deduce unobstructedness as each $\mc{Q}_{x'}/P_{\mu}$ is a smooth algebraic space.
\end{proof}

\begin{rem}Given Proposition \ref{prop:prorepresentable} we shall often conflate $(\mc{A}_b^\mr{univ},\varphi_{\mc{A}^\mr{univ}_b})$ and $(\mc{A}_{b,\mr{syn}}^\mr{univ},\varphi_{\mc{A}^\mr{univ}_{b,\mr{syn}}})$, only using the extra decoration when clarity is needed.
\end{rem}

\subsubsection{Serre--Tate theory in the abelian-type case}
We now verify an expectation of Drinfeld using results obtained so far. This allows us to further establish an analogue of the Serre--Tate theorem in the setting of abelian-type Shimura varieties at hyperspecial level. 

Throughout this section we adopt the notation and conventions of Notation \ref{nota:shim-var} and \S\ref{ss:comparison-to-shim-vars}. In particular, we fix $(\mb{G},\mb{X},\mc{G})$ to be an unramified Shimura datum of abelian type. We further fix an element $\mu^c_h$ of the conjugacy class from $\bbmu_h^c$ from \S\ref{ss:Shim-Var-Notation}.

\medskip

\paragraph{Serre--Tate theory}\label{p:Serre--Tate-theory} Let us begin by observing that the syntomic realization functor $\omega_{\mathsf{K}^p,\mr{syn}}$ on $\wh{\ms{S}}_{\mathsf{K}^p}$ as constructed in \S\ref{ss:prismatic-realization-functors} is a prismatic $F$-gauge with $\mc{G}^c$-structure of type $\mu_h^c$ by Corollary \ref{cor:prismatic-realization-lff} and Proposition \ref{prop: F gauge type mu equals F crystal type mu}. Thus, we obtain a canonical morphism
\begin{equation*}
    \rho_{\mathsf{K}^p}\colon \wh{\ms{S}}_{\mathsf{K}^p}\to \mr{BT}^{\mc{G}^c,-\mu_h^c}_\infty,
\end{equation*}
of formal stacks over $\Z_p$. 
The existence of such a morphism was expected by Drinfeld (see \cite[\S4.3]{DrinfeldTowardsShimurian}).
The main theorem of this section is the following result. 

\begin{thm}[{Serre--Tate theory}]\label{thm:Drinfeld-conjecture} The morphism 
\begin{equation*}
    \rho_{\mathsf{K}^p}\colon \wh{\ms{S}}_{\mathsf{K}^p}\to \mr{BT}^{\mc{G}^c,-\mu_h^c}_\infty 
\end{equation*}
is formally \'etale, i.e., if $R$ is a $p$-nilpotent ring and $R\to R/I$ a nilpotent thickening, then \begin{equation*}
    \begin{tikzcd}[sep=2.25em]
	{\ms{S}_{\mathsf{K}^p}(R)} & {\mr{BT}^{\mc{G}^c,-\mu_h^c}_\infty(R)} \\
	{\ms{S}_{\mathsf{K}^p}(R/I)} & {\mr{BT}^{\mc{G}^c,-\mu_h^c}_\infty(R/I)}
	\arrow["{\rho_{\mathsf{K}^p}}", from=1-1, to=1-2]
	\arrow[from=1-1, to=2-1]
	\arrow[from=1-2, to=2-2]
	\arrow["{\rho_{\mathsf{K}^p}}"', from=2-1, to=2-2]
\end{tikzcd}
\end{equation*}
is Cartesian.
\end{thm}

To prove this theorem we need some preliminary setup. To this end, fix $x$ to be a point of $\ms{S}_{\mathsf{K}^p}(\ov{\bb{F}}_p)$. As in \S\ref{ss:characterization} we have the associated element $\bm{b}_{x,\crys}$ in $C(\mc{G}^c)$. Additionally, recall that we have the associated conjugacy class $\bbmu_h^c$ of cocharacters $\mbb{G}_{m,\breve{\Z}_p}\to \mc G^c_{\breve{\Z}_p}$. Then, by Lemma \ref{lem:C(G)-containment} the element $\bm{b}_{x,\crys}$ lies in the image of the map
\begin{equation*}
    \mc{G}^c(\breve{\Z}_p)\sigma(\mu_h^c(p))^{-1}\mc{G}^c(\breve{\Z}_p)\to C(\mc{G}^c). 
\end{equation*}
Fix $b_x$ in $\mc{G}^c(\breve{\Z}_p)\mu^c_h(p)^{-1}\mc{G}^c(\breve{\Z}_p)$ 
with $\sigma (b_x)$ mapping to $\bm{b}_{x,\crys}$ in $C(\mc{G}^c)$, e.g., $b_x=\rho_{\mathsf{K}^p}(x)$.

Consider the functor
\begin{equation*}
    \mr{Def}_x\colon \mc{C}_W\to \mb{Set},\qquad 
    \mr{Def}_x\defeq \ms{S}_{\mathsf{K}^p}|_{\mc{C}_W}.
\end{equation*}
Note that $\mr{Def}_x$ parameterizes the deformations of $x$ within $\ms{S}_{\mathsf{K}^p}$ and is (pro)represented by the formal spectrum $\Spf(\wh{\mc{O}}_{\ms{S}_{\mathsf{K}^p},x})=\Spf(\wh{\mc{O}}_{\wh{\ms{S}}_\mathsf{K}^p,x})$. 

The following is an immediate consequence of Theorem \ref{thm:ito-shim-comp} and Proposition \ref{prop:prorepresentable}. Indeed, Proposition \ref{prop: F gauge type mu equals F crystal type mu} implies that (with notation as in these results) that $\rho_{\mathsf{K}^p,_x}=\rho_{b_x}\circ i_{x}$, as both pull back the universal object over $\mr{Def}_{b_x}$ to isomorphic deformations of $\mc{P}_{b_x}=(\omega_{\mathsf K^p,\syn})|_x$.

\begin{lem}\label{lem:def-isom} The induced map
\begin{equation*}
    \rho_{\mathsf{K}^p,x}\colon \mr{Def}_x\to \mr{Def}_{b_x}
\end{equation*}
is an isomorphism.
\end{lem}

\begin{proof}[Proof of Theorem \ref{thm:Drinfeld-conjecture}] By inducting on the minimal $n\geqslant 2$ such that $I^n=(0)$, we may assume without loss of generality that $n=2$. 

So, suppose that $R$ is a $p$-nilpotent ring and $I\subseteq R$ is square-zero ideal, then we must show that for any $R/I$-point $x$ of $\ms{S}_{\mathsf{K}^p}$ that the natural map
\begin{equation}\label{eq:serre-tate}
    \mr{fib}\left(\ms{S}_{\mathsf{K}^p}(R)\to \ms{S}_{\mathsf{K}^p}(R/I);x\right)\to \mr{fib}\left(\mr{BT}^{\mc{G}^c,-\mu_h^c}_\infty(R)\to \mr{BT}^{\mc{G}^c,-\mu_h^c}_\infty(R/I);\rho_{\mathsf{K}^p}(x)\right),
\end{equation}
is a bijection of sets (where the right-hand is discrete by Proposition \ref{prop:disc-def-theory}). The proof will proceed in several steps. But, before we start those steps, we make the following observation.

\begin{lem}\label{lem:BT-commutes-with-colimts} Suppose that $(\mc{G},\mu)$ is as in Notation \ref{nota:deformation}. Then if $R=\colim_i R_i$ is a filtered colimit of rings and $I=\colim I_i$ an ideal with $I$ and each $I_i$ square-zero. Then the commutative diagram 
\bx{\colim_i \mr{BT}^{\mc G,\mu}_\infty(R_i)\ar[r]\ar[d]
&\mr{BT}^{\mc G,\mu}_\infty(R)\ar[d]
\\ \colim_i\mr{BT}^{\mc G,\mu}_\infty(R_i/I_i)\ar[r]
& \mr{BT}^{\mc G,\mu}_\infty(R/I)
}\ex
is Cartesian.
\end{lem}
\begin{proof} The natural map from $\colim_i \mr{BT}^{\mc{G},\mu}_\infty(R_i)$ to the fiber product is evidently injectivity. To show surjectivity, fix an object $\mc{P}$ of $\mr{BT}^{\mc{G},\mu}_\infty(R)$ and consider the $\mc{G}$-torsor $x=x_{\dR,R}^\ast\mc{P}$ which is the image of some $x_i$ in $B\mc G(R_i)$ for sufficiently large $i$ (see \cite[Lemma 2.1]{CesnaviciusPT}). Consider 
\begin{equation*}
    \begin{tikzcd}[sep=2.25em]
	{\displaystyle \twocolim_{j\geqslant i}\mr{BT}^{\mc{G},\mu}_\infty(R_j)_{x_i}} & {\mr{BT}^{\mc{G},\mu}_\infty(R)_x} & {(\mc{Q}_x/P_{-\mu})(R)} \\
	{\displaystyle \twocolim_{j\geqslant i}\mr{BT}^{\mc{G},\mu}_\infty(R_j/I_j)_{x_i}} & {\mr{BT}^{\mc{G},\mu}_\infty(R/I)_x} & {(\mc{Q}_x/P_{-\mu})(R/I).}
	\arrow[from=1-1, to=1-2]
	\arrow[from=1-1, to=2-1]
	\arrow[from=1-2, to=1-3]
	\arrow[from=1-2, to=2-2]
	\arrow[from=1-3, to=2-3]
	\arrow[from=2-1, to=2-2]
	\arrow[from=2-2, to=2-3]
\end{tikzcd}
\end{equation*}
Here the left-hand square is the obvious ones, and the right-hand square is the Cartesian square as in \eqref{eq:GM-reintrepret}. Observe that $\mc{Q}_x/P_{-\mu}=(\mc{Q}_{x_i}/P_{-\mu})_R$ and $\mc{Q}_{x_i}/P_{-\mu}$ is of finite presentation over $R_i$, we have that the outer rectangular diagram (ignoring the central nodes) is obtained by passing to the $2$-colimit over $j\geqslant i$ of the Cartesian diagram as in \eqref{eq:GM-reintrepret} applied for $R_j\twoheadrightarrow R_j/I_j$. As $2$-colimits of Cartesian diagrams of groupoids is Cartesian, we deduce the outer rectangle is also Cartesian. Thus, we deduce the left-hand square is Cartesian as desired.
\end{proof}

\paragraph*{Step 1: restriction to $\breve{\Z}_p$-algebras} Observe that it suffices to show that $(\wh{\ms{S}}_{\mathsf{K}^p})_{\breve{\Z}_p}\to\mr{BT}^{\mc{G}^c,-\mu_h}_\infty$ is formally \'etale. Indeed, this follows from the following general observation  
as well as the fact that $\Z_p\to \breve{\Z}_p$ is ind-\'etale modulo $p^n$ for all $n$. 

\begin{lem}\label{lem:formally-etale-triangle} 
Let 
\begin{equation*}
    \begin{tikzcd}[sep=2.25em]
	X \\
	Y & {\mathcal{F}}
	\arrow["f"', from=1-1, to=2-1]
	\arrow["p", from=1-1, to=2-2]
	\arrow["q"', from=2-1, to=2-2]
\end{tikzcd}
\end{equation*}
be a diagram of fpqc-stacks, where $X$ and $Y$ are schemes. If $f$ is surjective and pro-\'etale and $p$ is formally \'etale, then $q$ is formally \'etale.
\end{lem}
\begin{proof} We begin by making the following observation. Consider a commutative diagram
\begin{equation}\label{eq:formally-etale-diag-1}
    \begin{tikzcd}
	Y & {\mc{F}} \\
	{\Spec(R/I)} & {\Spec(R).}
	\arrow["q", from=1-1, to=1-2]
	\arrow["a",from=2-1, to=1-1]
	\arrow["i",from=2-1, to=2-2]
	\arrow["b",from=2-2, to=1-2]
\end{tikzcd}
\end{equation}
As $f$ is surjective and pro-\'etale, there exists an ind-\'etale cover $\ov{c}\colon R/I\to \ov{C}$ and a map $\gamma\colon \Spec(\ov{C})\to X$ such that $f\circ\gamma=a\circ\ov{c}$. By \stacks{097P} there exists an ind-\'etale $R$-algebra $C$ such that $\ov{C}=C/IC$, which necessarily induces a surjective map $c\colon \Spec(C)\to\Spec(R)$. We thus obtain the following commutative diagram
\begin{equation}\label{eq:formally-etale-diag-2}
    \begin{tikzcd}
	X \\
	Y & {\mc{F}} \\
	{\Spec(R/I)} & {\Spec(R)} \\
	{\Spec(C/IC)} & {\Spec(C).}
	\arrow["f"', from=1-1, to=2-1]
	\arrow["p", from=1-1, to=2-2]
	\arrow["q", from=2-1, to=2-2]
	\arrow["a",from=3-1, to=2-1]
	\arrow["i",from=3-1, to=3-2]
	\arrow["b",from=3-2, to=2-2]
	\arrow["\gamma",curve={height=-40pt}, from=4-1, to=1-1]
	\arrow["\ov{c}",from=4-1, to=3-1]
	\arrow["i'",from=4-1, to=4-2]
	\arrow["c",from=4-2, to=3-2]
\end{tikzcd}
\end{equation} 
We use of this diagram twice below.

We first show that $q$ is formally unramified. Suppose that $\alpha,\beta\colon \Spec(R)\to Y$ are two morphisms whose addition to \eqref{eq:formally-etale-diag-2} preserves commutativity. Using the fact that $f$ is pro-\'etale and surjective we may (after possibly enlarging $C$) produce $\alpha',\beta'\colon \Spec(C)\to X$ whose addition to \eqref{eq:formally-etale-diag-2} preserves commutativity. But, by the formal unramifiedness of $p$ we deduce that $\alpha'$ equals $\beta'$, and thus that 
\begin{equation*}
    \alpha\circ c=f\circ \alpha'=f\circ\beta'=\beta\circ c.
\end{equation*} 
As $c$ is an epimorphism of schemes, we deduce that $\alpha$ equals $\beta$ as desired.

We now show that $q$ is formally \'etale. As we know that $q$ is formally unramified, it suffices to show that there exists a map $\alpha\colon \Spec(R)\to Y$ such that $\alpha\circ i=a$ and $q\circ \alpha=b$. 
Now, as $p$ is formally \'etale there exists a unique morphism $\delta\colon \Spec(C)\to X$ whose addition to \eqref{eq:formally-etale-diag-2} preserves commutativity. By a simple diagram chase, it suffices to show that the map $f\circ\delta\colon \Spec(C)\to Y$ descends to a morphism $\Spec(R)\to Y$. As $\Spec(C)\to \Spec(R)$ is an fpqc cover and $Y$ is a sheaf for the fpqc topology on $S$, it suffices to show that $f\circ\delta$ equalizes the two natural maps $\pi_1,\pi_2\colon\Spec(C\otimes_R C)\to \Spec(C)$. But, reducing $f\circ \delta\circ \pi_i$ modulo $I$ gives the same map, as the map $a\circ\ov{c}\colon \Spec(C/IC)\to Y$ does descend to a map $\Spec(R/I)\to Y$. Thus, by the formal unramifiedness of $Y\to \mc{F}$ it suffices to show that $q\circ f\circ \delta\circ \pi_1=q\circ f\circ \delta\circ \pi_2$. But, this is clear as the map $q\circ f\circ \delta$ does descend to a map $\Spec(R)\to \mc{F}$ (namely $b$).
\end{proof}

Moreover, let us observe that if $\Spec(R/I)\to(\wh{\ms{S}}_{\mathsf{K}^p})_{\breve{\Z}_p}$ is a morphism, then we have a map $\Spec(R/I)\to \Spec(\breve{\Z}_p)$, and by the formal \'etaleness of $\Z_p/p^n\to\breve{\Z}_p/p^n$ for all $n$ there exists a unique extension to a map $\Spec(R)\to\Spec(\breve{\Z}_p)$. Thus, we see that without loss of generality, we may assume that $R$ is a $\breve{\Z}_p$-algebra.

\medskip

\paragraph*{Step 2: reduction to finitely generated $\breve{\Z}_p$-algebras} Suppose that $R$ is an arbitrary $\breve{\Z}_p$-algebra and $I$ is a square-zero ideal. Observe that we may write $R$ as a filtered colimit $R=\colim_i R_i$ where $R_i$ ranges over all finitely generated $\breve{\Z}_p$-subalgebras of $R$. If $I_i\defeq I\cap R_i$ then $I=\colim_i I_i$. Moreover, $I_i^2\subseteq I^2=(0)$, so $I_i^2$ is also square-zero. Suppose that we have shown that the map in \eqref{eq:serre-tate} is a bijection for all finite-type $\breve{\Z}_p$-algebras. Then, in particular
\begin{equation*}
    \colim_i \mr{fib}\left(\ms{S}_{\mathsf{K}^p}(R_i)\to \ms{S}_{\mathsf{K}^p}(R_i/I_i);x\right)\to \colim_i \mr{fib}\left(\mr{BT}^{\mc{G}^c,-\mu_h^c}_\infty(R_i)\to \mr{BT}^{\mc{G}^c,-\mu_h^c}_\infty(R_i/I_i);\rho_{\mathsf{K}^p}(x)\right)
\end{equation*}
is a bijection. But, the source is $\mr{fib}\left(\ms{S}_{\mathsf{K}^p}(R)\to \ms{S}_{\mathsf{K}^p}(R/I);x\right)$ by the finite presentation of $\ms{S}_{\mathsf{K}^p}$ and the latter is $\mr{fib}\left(\mr{BT}^{\mc{G}^c,-\mu_h^c}_\infty(R)\to \mr{BT}^{\mc{G},-\mu_h^c}_\infty(R/I);\rho_{\mathsf{K}^p}(x)\right)$ by Lemma \ref{lem:BT-commutes-with-colimts}. The claim follows.

\medskip

\paragraph*{Step 3: the Hodge-type case} We now assume that $(\mb{G},\mb{X},\mc{G})$ is of Hodge type. We will then apply the following lemma for an integral Hodge embedding.

\begin{lem}\label{lem:closed-emb-lifting} Let $\mf{X}\hookrightarrow \mf{Y}$ be a closed embedding of topologically of finite type formal $\Z_p$-schemes and let $R$ be a $p$-nilpotent finite type $\breve{\Z}_p$-algebra. Then a morphism $\Spec(R)\to \mf{Y}$ factorizes through $\mf{X}$ if and only if for every maximal ideal $\mf{m}$ of $R$ the composition $\Spf(\wh{R}_{\mf{m}})\to\Spec(R)\to\mf{Y}$ factorizes through $\mf{X}$.
\end{lem}
\begin{proof} Let $\mc{J}\subseteq \mc{O}_{\mf{Y}}$ be the ideal sheaf corresponding to $\mf{X}$. It suffices to show that $\mc{J}\mc{O}_{\Spec(R)}$ is zero. But, for this it suffices to check this vanishing on each $R_\mf{m}$. Moreover, as $R_\mf{m}\to \wh{R}_\mf{m}$ is faithfully flat (see \stacks{00MC}) it further suffices to check that $\mc{J}\mc{O}_{\Spf(\wh{R}_\mf{m})}$ is zero. But, this follows from the existence of a factorization of $\Spf(\wh{R}_{\mf{m}})\to\Spec(R)\to\mf{Y}$ through $\mf{X}$.
\end{proof}

Choose now an integral Hodge embedding $\iota\colon (\mb{G},\mb{X},\mc{G})\hookrightarrow (\mb{H},\mf{h}^{\pm},\mc{H})$ and let $\mu_h^\mr{s}$ be a (choice of) integral Hodge cocharacter for the Siegel-type Shimura datum base changed to $W$. By \cite[Theorem 1.1.1]{Xu} this integral Hodge embedding induces a closed embedding 
\begin{equation*}
     \wh{\ms{S}}_{\mathsf{K}^p}(\mb{G},\mb{X})\hookrightarrow \wh{\ms{S}}_{\mathsf{L}^p}(\mb{H},\mf{h}^{\pm})
\end{equation*}
for an appropriately chosen neat compact open subgroup $\mathsf{L}^p\subseteq \mb{H}(\bb{A}_f^p)$. By \textbf{Step 2} it suffices to prove \eqref{eq:serre-tate} for $R$ finite type over $\breve{\Z}_p$. But using Example \ref{eg:Siegel} we may interpret the composition 
\begin{equation*}
    \Spec(R)\to\mr{BT}^{\mc{G},-\mu_h}_\infty\to\mr{BT}^{\mc{H},-\mu_h^\mr{s}}_\infty
\end{equation*}
as a quasi-polarized $p$-divisible group $(H,L,\lambda)$ over $R$ deforming $(A_0[p^\infty],L_0,\lambda_0)$, where we write $(A_0,L_0,\lambda_0,\alpha_0)$ for the point of $\wh{\ms{S}}_{\mathsf{L}^p}(\mb{H},\mf{h}^{\pm})(R/I)$ corresponding to the composition
\begin{equation*}
    \Spec(R/I)\to \wh{\ms{S}}_{\mathsf{K}^p}(\mb{G},\mb{X})\hookrightarrow \wh{\ms{S}}_{\mathsf{L}^p}(\mb{H},\mf{h}^{\pm}),
\end{equation*}
where $\alpha_0$ denotes the $\mathsf{K}^p$-level structure. By the classical version of Serre--Tate theory (e.g., see \cite[\S1]{Katz}) there exists a unique deformation $(A,L,\lambda,\alpha)$ in $\wh{\ms{S}}_{\mathsf{L}^p}(\mb{H},\mf{h}^{\pm})(R)$ of $(A_0,\lambda_0,L_0,\alpha_0)$ such that $(A[p^\infty],L,\lambda)=(H,L,\lambda)$. 
By Lemma \ref{lem:closed-emb-lifting} we will be done if we can show that for any maximal ideal $\mf{m}$ of $R$ the restriction of $(A[p^\infty],L,\lambda)$ is in the image of 
\begin{equation*}
    \wh{\ms{S}}_{\mathsf{K}^p}(\mb{G},\mb{X})(\wh{R}_\mf{m}) \hookrightarrow \wh{\ms{S}}_{\mathsf{L}^p}(\mb{H},\mf{h}^{\pm})(\wh{R}_\mf{m})
\end{equation*}
Note that $\wh{R}_{\mf{m}}$ is an object of $\mc{C}_{\breve{\Z}_p}$. As we have a morphism $\Spf(\wh{R}_{\mf{m}})\to \mr{BT}^{\mc{G},-\mu_h}_\infty$ extending the map $\Spec(R/\mf{m})\to \wh{\ms{S}}_{\mathsf{K}^p}$ we deduce from Lemma \ref{lem:def-isom} that such an extension exists.

\medskip

\paragraph*{Step 4: the special type case} Suppose now that $(\mb{G},\mb{X},\mc{G})$ is of special type. Then the bijectivity of \eqref{eq:serre-tate} may be checked by hand. Indeed, in this case $\ms{S}_{\mathsf{K}^p}(\mb{G},\mb{X})$ is isomorphic to a disjoint union of the form $\Spec(\mc{O}_F)$, where $F$ is an an unramified extension of $\Q_p$ (e.g., see \cite[Proposition 3.22]{DanielsYoucis}). Working component by component, we may replace $\wh{\ms{S}}_{\mathsf{K}^p}$ with $\Spf(\mc{O}_F)$. As $\Spf(\mc{O}_F)\to\Spf(\Z_p)$ is formally \'etale, we see that the morphism $\Spec(R/I)\to \Spf(\mc{O}_F)$ for the $\breve{\Z}_p$-algebra $R$, uniquely lifts to a morphism $\Spec(R)\to\Spf(\mc{O}_F)$. We will thus be done if we can show that the composition $\Spec(R)\to\Spf(\mc{O}_F)\to \mr{BT}^{\mc{G}^c,-\mu_h^c}_\infty$ agrees with the given one from \eqref{eq:serre-tate}. But, it's clear that the two maps induce the same map $\Spec(R/I)\to \mr{BT}^{\mc{G}^c,-\mu_h^c}_\infty$. So, the claim follows from \eqref{eq:GR-2} as, in this case, we have that $\mc{Q}_{x'}/P_{\mu_h^c}$ is trivial, as $\mc{G}^c$ is a torus, and so there is a unique lift of this $\Spec(R/I)\to\mr{BT}^{\mc{G}^c,-\mu_h^c}_\infty$ to a point $\Spec(R)\to \mr{BT}^{\mc{G}^c,-\mu_h^c}_\infty$.

\medskip

\paragraph*{Step 5: the abelian type case} Suppose now that $(\mb{G},\mb{X},\mc{G})$ is of abelian type. Consider the objects as in Lemma \ref{lem:Lovering-lem}. For an appropriate neat compact open subgroup and $\mathsf{M}^p\subseteq \mb{G}_1(\A_f^p)\times \mb{T}(\A_f^p)$ the map $\alpha$ induces a closed embedding 
\begin{equation}\label{eq:shim-closed-emb}
    \alpha\colon \ms{S}_{\mathsf{K}^p_2}(\mb{G}_2,\mb{X}_2)\hookrightarrow \ms{S}_{\mathsf{M}^p}(\mb{G}_1\times\mb{T},\mb{X}_1\times\{h\}).
\end{equation}
To prove this it suffices to pass to $\breve{\Z}_p$. Now, it's evident that $(\mathbf{G}_1,\mb{X}_1,\mc{G}_1)$ is an adapted Hodge type datum for both $(\mb{G}_2,\mb{X}_2,\mc{G}_2)$ and $(\mb{G}_1\times\mb{T},\mb{X}_1\times\{h\},\mc{G}_1\times\mc{T})$. Thus, by construction of integral canonical models as in \cite[(3.4.11)]{KisIntShab}, we have that at infinite level, the choice of a connected component $\ms{S}^+$ of $(\ms{S}_{K_{p,1}})_{\breve{\Z}_p}$ gives an identification of the map $\ms{S}_{K_{p,2}}\to \ms{S}_{K_{p,1}\times L_p}$ (base changed to $\breve{\Z}_p$) with the map
\begin{equation*}
    [\ms{S}^+\times \ms{A}(\mc{G}_2)]/\ms{A}(\mc{G}_2)^\circ \to [\ms{S}^+\times \ms{A}(\mc{G}_1\times\mc{T})]/\ms{A}(\mc{G}_1\times\mc{T})^\circ,
\end{equation*}
with notation as in loc.\@ cit. We claim this map is a closed embedding. Indeed, from the fact that $\mc{G}_2\to \mc{G}_1\times\mc{T}$ and $\mb{G}_2\to \mb{G}_1\times \mb{T}$ are closed embeddings inducing isomorphisms on derived subgroups (and so induce isomorphisms on adjoint groups and has the property that $Z(\mb{G}_1)= Z(\mb{G}_1\times \mb{T})\cap \mb{G}_1$ and the integral analogue) we see that the map $\ms{A}(\mc{G}_2)\to \ms{A}(\mc{G}_1\times\mc{T})$ is a closed embedding. We are then done as the group $\ms{A}(\mc{B})^\circ$ only depends on $\mc{B}^\mr{der}$ (e.g., see \cite[Lemma 4.6.4 (2)]{KisinPappas}), so that $\ms{A}(\mc{G}_2)^\circ=\ms{A}(\mc{G}_1\times\mc{T})^\circ$. We then deduce the existence of the desired $\mathsf{M}^p$ by arguing as in \cite[Proposition 1.15]{DelTS}. 

Fix a finite index subgroup of $\mathsf{M}^p$ of the form $\mathsf{K}_1^p\times\mathsf{L}^p$, and so we obtain a finite \'etale morphism 
\begin{equation*}
    \ms{S}_{\mathsf{K}^p_1}(\mb{G}_1,\mathbf{X}_1)\times\ms{S}_{\mathsf{L}^p}(\mb{T},\{h\})\simeq \ms{S}_{\mathsf{K}_1^p\times\mathsf{L}^p}(\mb{G}_1\times\mb{T},\mb{X}_1\times\{h\})\to \ms{S}_{\mathsf{M}^p}(\mb{G}_1\times\mb{T},\mb{X}_1\times\{h\}).
\end{equation*}
As we have already verified formal \'etaleness for Hodge and special type, it follows that formal \'etaleness holds also for the source. Applying Lemma \ref{lem:formally-etale-triangle} we deduce also that formal \'etaleness holds at level $\mathsf{M}^p$. One may then apply the same argument as in \textbf{Step 3} to the closed embedding in \eqref{eq:shim-closed-emb} to deduce the formal \'etaleness result holds for $(\mb{G}_2,\mb{X}_2,\mc{G}_2)$.

Finally, to deduce the result for $\ms{S}_{\mathsf{K}^p}(\mb{G},\mb{X})$, we observe that we have a finite \'etale morphism
\begin{equation*}
    \beta\colon \ms{S}_{\mathsf{K}^p_2}(\mb{G}_2,\mb{X}_2)\to \ms{S}_{\mathsf{K}^p}(\mb{G},\mb{X}),
\end{equation*}
for an appropriate choice of neat compact open subgroup $\mathsf{K}^p_2\subseteq \mb{G}_2(\A_f^p)$. Moreover, we know that $\rho_{\mathsf{K}^p}\circ\beta$ agrees with $\pi\circ\rho_{\mathsf{K}^p_2}$ where $\pi\colon \mr{BT}^{\mc{G}_2^c,-\mu_{h2}^c}_\infty\to \mr{BT}^{\mc{G}^c,-\mu_h^c}_\infty$ is the natural map. We already know that $\rho_{\mathsf{K}^p_2}$ is formally \'etale, and we claim that $\pi$ is as well. Indeed, this follows from the identification in \eqref{eq:GR-2} as $\mc{G}_2^c\to\mc{G}^c$ is a central isogeny, and so the natural morphism $\mc{Q}_{x'_2}/P_{\mu_{h2}^c}\to \mc{Q}_{x'}/P_{\mu_h^c}$, where the notation has the obvious meaning is an isomorphism. Indeed, this may be checked over an \'etale cover, in which case we may assume that $\mc{Q}_{x'_2}$ and $\mc{Q}_{x'}$ are trivial, so that this is the natural map $\mc{G}_2^c/P_{\mu_{h2}^c}\to \mc{G}^c/P_{\mu_h^c}$. But, this is obviously an isomorphism as the flag variety does not change under central isogenies.\footnote{Indeed, let $f\colon \mc{G}\to \mc{H}$ be a central isogeny of reductive group schemes over a base scheme $S$ and let $Z\subseteq Z(\mc{G})$ be its kernel. Then, for a cocharacter $\mu$ of $\mc{G}$ the morphism $P_{f\circ \mu}\to \mc{H}$ is the result of quotienting $P_{\mu}\to \mc{G}$ by $Z$. Thus, as $\mc{G}/P_{\mu}=(\mc{G}/Z)/(P_{\mu}/Z)$ this implies the desired claim.}

We deduce that $\pi\circ \rho_{\mathsf{K}^p_2}$ is formally \'etale, and so $\rho_{\mathsf{K}^p}$ is formally \'etale by Lemma \ref{lem:formally-etale-triangle} on the connected components lying in $\mr{im}(\beta)$. But, we may then deduce the formal \'etaleness of $\rho_{\mathsf{K}^p}$ over all connected components, using a translation argument (see the proof of Theorem \ref{thm:main-Shimura-theorem-abelian-type-case}).
\end{proof}

\medskip

\paragraph{Syntomic characterization}

We now formulate a more conceptual, but a priori stronger, version of the notion of prismatic integral model from Definition \ref{defn:prismatic-integral-canonical-model}. We use the notion of a prismatic $F$-gauge with $\mc{G}$-structure of type $\bm{\mu}$ modeling a $\mc{G}$-object in de Rham local systems on $\ms{X}_E$, for a (formal) scheme $\ms{X}$ over $\mc{O}_E$. This is a straightforward generalization of that in Definition \ref{defn:prismatic-integral-canonical-model}, and we follow similar conventions as in the prismatic $F$-crystal setting.

\begin{defn}\label{defn:syntomic-integral-canonical-model} A smooth separated model $\mf{X}_{\mathsf{K}^p}$ (resp.\@ $\ms{X}_{\mathsf{K}^p}$) of $U_{\mathsf{K}^p}$ (resp.\@ $\Sh_{\mathsf{K}_0\mathsf{K}^p}$) is called a \emph{syntomic integral canonical model} if there exists a prismatic $F$-gauge with $\mc{G}^c$-structure of type $-\mu_h^c$ modeling $\omega_{\mathsf{K}^p,\mr{an}}$ (resp.\@ $\omega_{\mathsf{K}^p,\et}$) such that the induced map
\begin{equation*}
    \mf{X}_{\mathsf{K}^p}\to\mr{BT}^{\mc{G}^c,-\mu_h^c}_\infty,\quad\bigg(\text{resp. }\wh{\ms{X}}_{\mathsf{K}^p}\to\mr{BT}^{\mc{G}^c,-\mu_h^c}_\infty\bigg)
\end{equation*}
is formally \'etale (resp.\@ and is also a model of $(\Sh_{\mathsf{K}_0\mathsf{K}^p},U_{\mathsf{K}^p})$ in the sense of Definition \ref{defn:prismatic-model-local-system}).
\end{defn}

From Proposition \ref{prop: F gauge type mu equals F crystal type mu} and Proposition \ref{prop:prorepresentable} any syntomic integral canonical model is a prismatic integral canonical model. So, the following is a consequence of Theorem \ref{thm:prismatic-characterization-completion} and Theorem \ref{thm:prismatic-characterization}.

\begin{thm}\label{thm:syntomic-characterization} Let $(\mb{G},\mb{X},\mc{G})$ be an unramified Shimura datum of abelian type and $\mathsf{K}^p$ a neat-compact open subgroup of $\mb{G}(\A_f^p)$. Then $\ms{S}_{\mathsf{K}^p}$ (resp.\@ $\wh{\ms{S}}_{\mathsf{K}^p}$) is the unique syntomic integral canonical model of $\Sh_{\mathsf{K}_0\mathsf{K}^p}$ (resp.\@ $U_{\mathsf{K}^p}$).
\end{thm}

\subsection{Applications to the theory of \texorpdfstring{$\mc{G}$-zips}{G-zips}}

In \cite{OortStratifications}, Oort defined stratifications of the special fiber of integral canonical models of Siegel-type Shimura varieties in terms of the group-theoretic properties of the $p$-torsion of the universal abelian scheme. To generalize this idea to integral canonical models $\ms{S}_{\mathsf{K}^p}$ for other unramified Shimura data, in \cite{PWZ}, the authors defined the Artin stack $\mc{H}\text{-}\mathsf{Zip}^{\mu}$ of \emph{$(\mc{H},\mu)$-zips} for a reductive group $\mc{H}/\mathbb{F}_q$ with cocharacter $\mu$.

The generalization of Oort's ideas should take the form of a map $\zeta_{\mathsf{K}^p}\colon \overline{\ms{S}}_{\mathsf{K}^p}\to \mc{G}^c\text{-}\mathsf{Zip}^{-\mu_h^c}$ where $\overline{\ms{S}}_{\mathsf{K}^p}$ is the reduction modulo $p$ of $\ms{S}_{\mathsf{K}^p}$. Such a map was constructed by Viehmann--Wedhorn in \cite{ViehmannWedhorn} for PEL-type Shimura varieties, and by Zhang in \cite{ZhangEO} for Hodge-type Shimura varieties. In fact, Zhang was able to show that the map $\zeta_{\mathsf{K}^p}$ is smooth.

One can use the syntomic realization functor $\omega_{\mathsf{K}^p,\mr{syn}}$ to generalize such results to the abelian-type setting. Namely, by reducing modulo $p$ we obtain a formally \'etale
\begin{equation}\label{eq:Gzip-1}
\overline{\omega}_{\mathsf{K}^p,\mr{syn}}\colon \ov{\ms{S}}_{\mathsf{K}^p}\to \overline{\mr{BT}}^{\mc{G}^c,-\mu_h^c}_\infty,
\end{equation}
where again we are using $\overline{\mr{BT}}^{\mc{G}^c,-\mu_h^c}_\infty$ to denote the reduction of ${\mr{BT}}^{\mc{G}^c,-\mu_h^c}_\infty$ modulo $p$. By passing to the $1$-truncation (i.e., pulling back a $\mc{G}$-bundle along $R^\mr{syn}\otimes\bb{F}_p\to R^\mr{syn}$), we obtain a morphism
\begin{equation}\label{eq:Gzip-2}
\overline{\mr{BT}}^{\mc{G}^c,-\mu_h^c}_\infty\to \overline{\mr{BT}}_1^{\mc{G}^c,-\mu_h^c}
\end{equation}
which is formally smooth by \cite[Theorem D]{GMM}. By \cite[Theorem E]{GMM} there is a morphism
\begin{equation}\label{eq:Gzip-3}
\overline{\mr{BT}}_1^{\mc{G}^c,-\mu_h^c}\to \mc{G}^c\text{-}\mathsf{Zip}^{-\mu_h^c},
\end{equation}
which, in fact, is a gerbe for an explicit group scheme (and so smooth). Composing the above maps, we arrive at the following result.

\begin{thm}\label{thm:GZip-map} Suppose that $p>2$ and $(\mc{G},\mb{G},\mb{X})$ is an unramified Shimura datum of abelian type. Then, composing \eqref{eq:Gzip-1}, \eqref{eq:Gzip-2}, and \eqref{eq:Gzip-3} gives a smooth morphism
\begin{equation*}
 \zeta_{\mathsf{K}^p}\colon \overline{\ms{S}}_{\mathsf{K}^p}\to \mc{G}^c\text{-}\mathsf{Zip}^{-\mu_h^c}.
\end{equation*}
\end{thm}


\begin{thebibliography}{DvHKZ24}
	\providecommand{\url}[1]{\texttt{#1}}
	\providecommand{\urlprefix}{URL }
	\providecommand{\eprint}[2][]{\url{#2}}
	
	\bibitem[AHR23]{AlperHallRydh}
	J.~Alper, J.~Hall and D.~Rydh, The \'etale local structure of algebraic stacks,
	2023, \eprint{1912.06162}.
	
	\bibitem[ALB23]{AnschutzLeBrasDD}
	J.~Ansch\"{u}tz and A.-C. Le~Bras, Prismatic {D}ieudonn\'{e} {T}heory, Forum
	Math. Pi 11 (2023), Paper No. e2.
	
	\bibitem[ALY22]{ALYSpecialization}
	P.~Achinger, M.~Lara and A.~Youcis, Specialization for the pro-\'{e}tale
	fundamental group, Compos. Math. 158 (2022), no.~8, 1713--1745.
	
	\bibitem[BBM82]{BBMDieuII}
	P.~Berthelot, L.~Breen and W.~Messing, Th\'{e}orie de {D}ieudonn\'{e}
	cristalline. {II}, vol. 930 of Lecture Notes in Mathematics, Springer-Verlag,
	Berlin, 1982.
	
	\bibitem[Bha23]{BhattNotes}
	B.~Bhatt, Prismatic $F$-gauges, 2023, unpublished course notes
	\href{https://www.math.ias.edu/~bhatt/teaching/mat549f22/lectures.pdf}.
	
	\bibitem[BL22a]{BhattLurieAbsolute}
	B.~Bhatt and J.~Lurie, Absolute prismatic cohomology, 2022,
	\eprint{2201.06120}.
	
	\bibitem[BL22b]{BhattLuriePrismatization}
	B.~Bhatt and J.~Lurie, The prismatization of $p$-adic formal schemes, 2022,
	\eprint{2201.06124}.
	
	\bibitem[BMS19]{BMS-THH}
	B.~Bhatt, M.~Morrow and P.~Scholze, Topological {H}ochschild homology and
	integral {$p$}-adic {H}odge theory, Publ. Math. Inst. Hautes \'{E}tudes Sci.
	129 (2019), 199--310.
	
	\bibitem[BO78]{BerthelotOgus}
	P.~Berthelot and A.~Ogus, Notes on crystalline cohomology, Princeton University
	Press, Princeton, N.J.; University of Tokyo Press, Tokyo, 1978.
	
	\bibitem[BS22]{BhattScholzePrisms}
	B.~Bhatt and P.~Scholze, Prisms and prismatic cohomology, Ann. of Math. (2) 196
	(2022), no.~3, 1135--1275.
	
	\bibitem[BS23]{BhattScholzeCrystals}
	B.~Bhatt and P.~Scholze, Prismatic {$F$}-crystals and crystalline {G}alois
	representations, Camb. J. Math. 11 (2023), no.~2, 507--562.
	
	\bibitem[BW04]{BurgosWildehaus}
	J.~I. Burgos and J.~Wildeshaus, Hodge modules on {S}himura varieties and their
	higher direct images in the {B}aily-{B}orel compactification, Ann. Sci.
	\'{E}cole Norm. Sup. (4) 37 (2004), no.~3, 363--413.
	
	\bibitem[{\v{C}}es15]{CesnaviciusPT}
	K.~{\v{C}}esnavi\v{c}ius, Poitou-{T}ate without restrictions on the order,
	Math. Res. Lett. 22 (2015), no.~6, 1621--1666.
	
	\bibitem[Con14]{ConradReductive}
	B.~Conrad, Reductive group schemes, in Autour des sch\'{e}mas en groupes.
	{V}ol. {I}, vol. 42/43 of Panor. Synth\`eses, pp. 93--444, Soc. Math. France,
	Paris, 2014.
	
	\bibitem[{\v{C}}S24]{CesnaviciusScholze}
	K.~{\v{C}}esnavi\v{c}ius and P.~Scholze, Purity for flat cohomology, Ann. of
	Math. (2) 199 (2024), no.~1, 51--180.
	
	\bibitem[Dan22]{Daniels}
	P.~Daniels, Canonical integral models for Shimura varieties defined by tori,
	2022, \eprint{2207.09513}.
	
	\bibitem[Del71]{DelTS}
	P.~Deligne, Travaux de {S}himura  (1971), 123--165. Lecture Notes in Math.,
	Vol. 244.
	
	\bibitem[Del79]{DeligneModulaire}
	P.~Deligne, Vari\'{e}t\'{e}s de {S}himura: interpr\'{e}tation modulaire, et
	techniques de construction de mod\`eles canoniques, in Automorphic forms,
	representations and {$L$}-functions ({P}roc. {S}ympos. {P}ure {M}ath.,
	{O}regon {S}tate {U}niv., {C}orvallis, {O}re., 1977), {P}art 2, Proc. Sympos.
	Pure Math., XXXIII, Amer. Math. Soc., Providence, R.I., 1979 pp. 247--289.
	
	\bibitem[dJ95]{deJongCrystalline}
	A.~J. de~Jong, Crystalline {D}ieudonn\'{e} module theory via formal and rigid
	geometry, Inst. Hautes \'{E}tudes Sci. Publ. Math.  (1995), no.~82, 5--96
	(1996).
	
	\bibitem[dJ96]{deJongAlteration}
	A.~J. de~Jong, Smoothness, semi-stability and alterations, Inst. Hautes
	\'{E}tudes Sci. Publ. Math.  (1996), no.~83, 51--93.
	
	\bibitem[DLMS24]{DLMS}
	H.~Du, T.~Liu, Y.~S. Moon and K.~Shimizu, Completed prismatic {$F$}-crystals
	and crystalline {$Z_p$}-local systems, Compos. Math. 160 (2024), no.~5,
	1101--1166.
	
	\bibitem[Dri24a]{Drinfeld}
	V.~Drinfeld, Prismatization, 2024, \eprint{2005.04746}.
	
	\bibitem[Dri24b]{DrinfeldTowardsShimurian}
	V.~Drinfeld, Toward Shimurian analogs of Barsotti-Tate groups, 2024,
	\eprint{2309.02346}.
	
	\bibitem[DvHKZ24]{DvHKZ2}
	P.~Daniels, P.~van Hoften, D.~Kim and M.~Zhang, Igusa Stacks and the Cohomology
	of {S}himura Varieties, 2024, \eprint{2408.01348}.
	
	\bibitem[DY24]{DanielsYoucis}
	P.~Daniels and A.~Youcis, Canonical Integral Models of Shimura Varieties of
	Abelian Type, 2024, \eprint{2402.05727}.
	
	\bibitem[Elk73]{Elkik}
	R.~Elkik, Solutions d'\'{e}quations \`a coefficients dans un anneau
	hens\'{e}lien, Ann. Sci. \'{E}cole Norm. Sup. (4) 6 (1973), 553--603 (1974).
	
	\bibitem[FK18]{FujiwaraKato}
	K.~Fujiwara and F.~Kato, Foundations of rigid geometry. {I}, EMS Monographs in
	Mathematics, European Mathematical Society (EMS), Z\"{u}rich, 2018.
	
	\bibitem[GD71]{EGASpringer}
	A.~Grothendieck and J.~A. Dieudonn\'{e}, \'{E}l\'{e}ments de g\'{e}om\'{e}trie
	alg\'{e}brique. {I}, vol. 166 of Grundlehren der mathematischen
	Wissenschaften [Fundamental Principles of Mathematical Sciences],
	Springer-Verlag, Berlin, 1971.
	
	\bibitem[GL23]{GuoLi}
	H.~Guo and S.~Li, Frobenius height of prismatic cohomology with coefficients,
	2023, \eprint{2309.06663}.
	
	\bibitem[Gle22]{GleasonSpecialization}
	I.~Gleason, Specialization maps for Scholze's category of diamonds, 2022,
	\eprint{2012.05483}.
	
	\bibitem[GM24]{GMM}
	Z.~Gardner and K.~Madapusi, An algebraicity conjecture of Drinfeld and the
	moduli of $p$-divisible groups, 2024, \eprint{2412.10226}.
	
	\bibitem[GR24]{GuoReinecke}
	H.~Guo and E.~Reinecke, A prismatic approach to crystalline local systems,
	Invent. Math. 236 (2024), no.~1, 17--164.
	
	\bibitem[HLP23]{HLP}
	D.~Halpern-Leistner and A.~Preygel, Mapping stacks and categorical notions of
	properness, Compos. Math. 159 (2023), no.~3, 530--589.
	
	\bibitem[Hub93]{HuberCV}
	R.~Huber, Continuous valuations, Math. Z. 212 (1993), no.~3, 455--477.
	
	\bibitem[Hub94]{HuberGen}
	R.~Huber, A generalization of formal schemes and rigid analytic varieties,
	Math. Z. 217 (1994), no.~4, 513--551.
	
	\bibitem[Hub96]{HuberEC}
	R.~Huber, \'{E}tale cohomology of rigid analytic varieties and adic spaces,
	Aspects of Mathematics, E30, Friedr. Vieweg \& Sohn, Braunschweig, 1996.
	
	\bibitem[IKY24]{IKY1}
	N.~Imai, H.~Kato and A.~Youcis, A Tannakian framework for prismatic
	$F$-crystals, 2024, \eprint{2406.08259}.
	
	\bibitem[IKY25]{IKY3}
	N.~Imai, H.~Kato and A.~Youcis, An integral analogue of Fontaine's crystalline
	functor, 2025, \eprint{2504.16282}.
	
	\bibitem[IM13]{ImaiMiedaRIMS}
	N.~Imai and Y.~Mieda, Toroidal compactifications of {S}himura varieties of
	{PEL} type and its applications, in Algebraic number theory and related
	topics 2011, RIMS K\^{o}ky\^{u}roku Bessatsu, B44, pp. 3--24, Res. Inst.
	Math. Sci. (RIMS), Kyoto, 2013.
	
	\bibitem[IM20]{ImaiMieda}
	N.~Imai and Y.~Mieda, Potentially good reduction loci of {S}himura varieties,
	Tunis. J. Math. 2 (2020), no.~2, 399--454.
	
	\bibitem[Ino25]{Inoue}
	K.~Inoue, Log Prismatic Dieudonn\'{e} theory and its application to Shimura
	varieties, 2025, \eprint{2503.07379}.
	
	\bibitem[Ito23]{Ito1}
	K.~Ito, Prismatic $G$-display and descent theory, 2023, \eprint{2303.15814}.
	
	\bibitem[Ito25]{Ito2}
	K.~Ito, Deformation theory for prismatic G-displays, Forum of Mathematics,
	Sigma 13 (2025), e61.
	
	\bibitem[Kat81]{Katz}
	N.~Katz, Serre-{T}ate local moduli, in Algebraic surfaces ({O}rsay, 1976--78),
	vol. 868 of Lecture Notes in Math., pp. 138--202, Springer, Berlin, 1981.
	
	\bibitem[Kim18a]{KimRZ}
	W.~Kim, Rapoport-{Z}ink spaces of {H}odge type, Forum Math. Sigma 6 (2018),
	Paper No. e8, 110.
	
	\bibitem[Kim18b]{KimUnif}
	W.~Kim, Rapoport-{Z}ink uniformization of {H}odge-type {S}himura varieties,
	Forum Math. Sigma 6 (2018), Paper No. e16, 36.
	
	\bibitem[Kis10]{KisIntShab}
	M.~Kisin, Integral models for {S}himura varieties of abelian type, J. Amer.
	Math. Soc. 23 (2010), no.~4, 967--1012.
	
	\bibitem[Kis17]{KisinModp}
	M.~Kisin, {${\rm mod}\,p$} points on {S}himura varieties of abelian type, J.
	Amer. Math. Soc. 30 (2017), no.~3, 819--914.
	
	\bibitem[Kot84]{KotShtw}
	R.~E. Kottwitz, Shimura varieties and twisted orbital integrals, Math. Ann. 269
	(1984), no.~3, 287--300.
	
	\bibitem[Kot90]{KottwitzAnnArbor}
	R.~E. Kottwitz, Shimura varieties and {$\lambda$}-adic representations, in
	Automorphic forms, {S}himura varieties, and {$L$}-functions, {V}ol. {I}
	({A}nn {A}rbor, {MI}, 1988), vol.~10 of Perspect. Math., pp. 161--209,
	Academic Press, Boston, MA, 1990.
	
	\bibitem[KP18]{KisinPappas}
	M.~Kisin and G.~Pappas, Integral models of {S}himura varieties with parahoric
	level structure, Publ. Math. Inst. Hautes \'{E}tudes Sci. 128 (2018),
	121--218.
	
	\bibitem[KSZ21]{KSZ}
	M.~Kisin, S.~W. Shin and Y.~Zhu, The stable trace formula for Shimura varieties
	of abelian type, 2021, \eprint{2110.05381}.
	
	\bibitem[Lau21]{LauHigher}
	E.~Lau, Higher frames and {$G$}-displays, Algebra Number Theory 15 (2021),
	no.~9, 2315--2355.
	
	\bibitem[Lee21]{Lee}
	S.~Y. Lee, Eichler-Shimura Relations for Shimura Varieties of Hodge Type, 2021,
	\eprint{2006.11745}.
	
	\bibitem[LM24]{MadapusiLee}
	S.~Y. Lee and K.~Madapusi, {$p$}-isogenies with {$G$}-structure and their
	applications, 2024, in preparation.
	
	\bibitem[Lov17a]{LoveringFCrystals}
	T.~Lovering, Filtered F-crystals on Shimura varieties of abelian type, 2017,
	\eprint{1702.06611}.
	
	\bibitem[Lov17b]{LoveringModels}
	T.~Lovering, Integral canonical models for automorphic vector bundles of
	abelian type, Algebra Number Theory 11 (2017), no.~8, 1837--1890.
	
	\bibitem[LS18]{LanStrohII}
	K.-W. Lan and B.~Stroh, Nearby cycles of automorphic \'{e}tale sheaves, {II},
	in Cohomology of arithmetic groups, vol. 245 of Springer Proc. Math. Stat.,
	pp. 83--106, Springer, Cham, 2018.
	
	\bibitem[LvO96]{LvO}
	H.~Li and F.~van Oystaeyen, Zariskian filtrations, vol.~2 of $K$-Monographs in
	Mathematics, Kluwer Academic Publishers, Dordrecht, 1996.
	
	\bibitem[LZ17]{LiuZhu}
	R.~Liu and X.~Zhu, Rigidity and a {R}iemann-{H}ilbert correspondence for
	{$p$}-adic local systems, Invent. Math. 207 (2017), no.~1, 291--343.
	
	\bibitem[Mad24]{MadapusiDerivedCycles}
	K.~Madapusi, Derived special cycles on {S}himura varieties, 2024, in
	preparation.
	
	\bibitem[Mat80]{MatsumuraCommAlg}
	H.~Matsumura, Commutative algebra, vol.~56 of Mathematics Lecture Note Series,
	Benjamin/Cummings Publishing Co., Inc., Reading, Mass., second edn., 1980.
	
	\bibitem[Mil94]{MilShmot}
	J.~S. Milne, Shimura varieties and motives, in Motives ({S}eattle, {WA}, 1991),
	vol.~55 of Proc. Sympos. Pure Math., pp. 447--523, Amer. Math. Soc.,
	Providence, RI, 1994.
	
	\bibitem[Mil05]{MilneShimura}
	J.~S. Milne, Introduction to {S}himura varieties, in Harmonic analysis, the
	trace formula, and {S}himura varieties, vol.~4 of Clay Math. Proc., pp.
	265--378, Amer. Math. Soc., Providence, RI, 2005.
	
	\bibitem[Mon24]{MondalClassification}
	S.~Mondal, Dieudonn\'e theory via cohomology of classifying stacks II, 2024,
	\eprint{2405.12967}.
	
	\bibitem[Moo98]{Moonen}
	B.~Moonen, Models of {S}himura varieties in mixed characteristics, in Galois
	representations in arithmetic algebraic geometry ({D}urham, 1996), vol. 254
	of London Math. Soc. Lecture Note Ser., pp. 267--350, Cambridge Univ. Press,
	Cambridge, 1998.
	
	\bibitem[MP19]{MadTorHod}
	K.~Madapusi~Pera, Toroidal compactifications of integral models of {S}himura
	varieties of {H}odge type, Ann. Sci. \'{E}c. Norm. Sup\'{e}r. (4) 52 (2019),
	no.~2, 393--514.
	
	\bibitem[Nie21]{NieThesis}
	T.~Nie, Absolute Hodge Cycles in Prismatic Cohomology, 2021.
	
	\bibitem[Oor01]{OortStratifications}
	F.~Oort, A stratification of a moduli space of abelian varieties, in Moduli of
	abelian varieties ({T}exel {I}sland, 1999), vol. 195 of Progr. Math., pp.
	345--416, Birkh\"{a}user, Basel, 2001.
	
	\bibitem[Pap23]{Pappas}
	G.~Pappas, On integral models of {S}himura varieties, Math. Ann. 385 (2023),
	no. 3-4, 2037--2097.
	
	\bibitem[Pau04]{Paugam}
	F.~Paugam, Galois representations, {M}umford-{T}ate groups and good reduction
	of abelian varieties, Math. Ann. 329 (2004), no.~1, 119--160.
	
	\bibitem[PR24]{PappasRapoportI}
	G.~Pappas and M.~Rapoport, {$p$}-adic shtukas and the theory of global and
	local {S}himura varieties, Camb. J. Math. 12 (2024), no.~1, 1--164.
	
	\bibitem[PWZ15]{PWZ}
	R.~Pink, T.~Wedhorn and P.~Ziegler, {$F$}-zips with additional structure,
	Pacific J. Math. 274 (2015), no.~1, 183--236.
	
	\bibitem[Rep24]{Reppen}
	S.~Reppen, Systems of Hecke eigenvalues on subschemes of Shimura varieties,
	arXiv preprint arXiv:2409.11720  (2024).
	
	\bibitem[Sai90]{Saito}
	M.~Saito, Mixed {H}odge modules, Publ. Res. Inst. Math. Sci. 26 (1990), no.~2,
	221--333.
	
	\bibitem[She17]{ShenPerfectoid}
	X.~Shen, Perfectoid {S}himura varieties of abelian type, Int. Math. Res. Not.
	IMRN  (2017), no.~21, 6599--6653.
	
	\bibitem[She24]{ShendeRham}
	X.~Shen, De Rham $F$-gauges and Shimura varieties, 2024, \eprint{2403.01899}.
	
	\bibitem[SP]{StacksProject}
	{Stacks Project Authors}, \textit{Stacks Project},
	\url{http://stacks.math.columbia.edu}, 2023.
	
	\bibitem[SW20]{ScholzeBerkeley}
	P.~Scholze and J.~Weinstein, Berkeley lectures on {$p$}-adic geometry, vol. 207
	of Annals of Mathematics Studies, Princeton University Press, Princeton, NJ,
	2020.
	
	\bibitem[SZ22]{ShenZhang}
	X.~Shen and C.~Zhang, Stratifications in good reductions of {S}himura varieties
	of abelian type, Asian J. Math. 26 (2022), no.~2, 167--226.
	
	\bibitem[Tsu20]{Tsu20}
	T.~Tsuji, Crystalline-Representations and-Representations with Frobenius, in
	p-adic Hodge Theory, pp. 161--319, Springer, 2020.
	
	\bibitem[VW13]{ViehmannWedhorn}
	E.~Viehmann and T.~Wedhorn, Ekedahl-{O}ort and {N}ewton strata for {S}himura
	varieties of {PEL} type, Math. Ann. 356 (2013), no.~4, 1493--1550.
	
	\bibitem[Wed01]{WedhornBT}
	T.~Wedhorn, The dimension of {O}ort strata of {S}himura varieties of
	{PEL}-type, in Moduli of abelian varieties ({T}exel {I}sland, 1999), vol. 195
	of Progr. Math., pp. 441--471, Birkh\"{a}user, Basel, 2001.
	
	\bibitem[Xu20]{Xu}
	Y.~Xu, Normalization in integral models of Shimura varieties of Hodge type,
	arXiv preprint arXiv:2007.01275  (2020).
	
	\bibitem[Yan25]{Yan}
	Q.~Yan, On certain integral Frobenius period maps for Shimura varieties and
	their reductions, 2025, \eprint{2501.06601}.
	
	\bibitem[Zha18]{ZhangEO}
	C.~Zhang, Ekedahl-{O}ort strata for good reductions of {S}himura varieties of
	{H}odge type, Canad. J. Math. 70 (2018), no.~2, 451--480.
	
	\bibitem[Zha23]{ZhangThesis}
	M.~Zhang, A {PEL}-type {I}gusa stack and the $p$-adic Geometry of {S}himura
	varieties, 2023, \eprint{2309.05152}.
	
\end{thebibliography}
\end{document}